\documentclass[twoside]{article}
\usepackage[textsize=scriptsize]{todonotes}

\usepackage{etoolbox}
\usepackage[margin=1in,marginparwidth=0.8in]{geometry}
\usepackage{tikz-cd}
\usepackage{amsmath}
\usepackage{amssymb}
\usepackage{amsthm}
\usepackage{amscd}
\usepackage{enumerate}
\usepackage[pdfusetitle,unicode,hidelinks]{hyperref}
\usepackage{bbm}
\usepackage{etoolbox}

\usepackage[utf8]{inputenc}
\usepackage[T1]{fontenc}

\newtheorem{proposition}{Proposition}
\newtheorem{corollary}[proposition]{Corollary}
\newtheorem{lemma}[proposition]{Lemma}
\newtheorem{theorem}[proposition]{Theorem}
\newtheorem{conjecture}[proposition]{Conjecture}
\newtheorem*{conjecture*}{Conjecture}
\newtheorem*{theorem*}{Theorem}
\newtheorem*{corollary*}{Corollary}
\newtheorem*{proposition*}{Proposition}
\newtheorem*{lemma*}{Lemma}
\theoremstyle{definition}
\newtheorem{definition}[proposition]{Definition}
\newtheorem{construction}[proposition]{Construction}

\newtheorem*{definition*}{Definition}
\newtheorem*{construction*}{Construction}
\newtheorem{remark}[proposition]{Remark}
\newtheorem*{remark*}{Remark}

\newtheorem{example}[proposition]{Example}
\newtheorem*{example*}{Example}

\newcommand{\id}{\operatorname{id}}
\newcommand{\Z}{\mathbb{Z}}

\newcommand{\Q}{\mathbb{Q}}
\newcommand{\F}{\mathbb{F}}

\let\scr=\mathcal
\let\bb=\mathbb

\def\A{\bb A}
\def\P{\bb P}

\newcommand{\SH}{\mathcal{SH}}
\newcommand{\Spt}{\mathcal{S}\mathrm{pt}{}}

\newcommand{\Perf}{\mathrm{Perf}}

\DeclareMathOperator*{\colim}{colim}

\let\lim=\relax
\DeclareMathOperator*{\lim}{lim}
\def\Map{\mathrm{Map}}

\def\map{\mathrm{map}}
\def\CAlg{\mathrm{CAlg}}

\def\Cat{\mathcal{C}\mathrm{at}{}}
\def\Spc{\mathrm{Spc}}

\def\Fun{\mathrm{Fun}}

\newcommand{\Spec}{\mathrm{Spec}}

\DeclareRobustCommand{\ul}{\underline}

\newcommand{\tr}{\mathrm{tr}}
\newcommand{\Hom}{\operatorname{Hom}}

\def\op{\mathrm{op}}

\let\cat=\mathrm
\def\Sm{{\cat{S}\mathrm{m}}}
\def\Sch{\cat{S}\mathrm{ch}{}}

\def\Nis{\mathrm{Nis}}
\def\Zar{\mathrm{Zar}}
\def\cdh{\mathrm{cdh}}

\def\mot{\mathrm{mot}}
\newcommand{\et}{{\acute{e}t}}

\def\ph{\mathord-}

\input{xy}
\xyoption{all}

\newcommand{\Gr}{\operatorname{Gr}}
\newcommand{\Fil}{\operatorname{Fil}}
\newcommand{\gr}{\mathrm{gr}}

\newcommand{\PShv}{\mathrm{PShv}}

\newcommand{\syn}{\mathrm{syn}}

\newcommand{\qcqs}{\mathrm{qcqs}}

\newcommand{\KGL}{\mathrm{KGL}}

\newcommand{\K}{\mathrm{K}}
\newcommand{\KH}{\mathrm{KH}}
\newcommand{\HH}{\mathrm{HH}}
\newcommand{\HC}{\mathrm{HC}}
\newcommand{\HP}{\mathrm{HP}}

\newcommand{\BMS}{\mathrm{BMS}}

\newcommand{\Sp}{\mathrm{Sp}}
\newcommand{\comp}{{{\kern -.5pt}\wedge}}
\newcommand{\Shv}{\operatorname{Shv}}
\renewcommand{\PShv}{\operatorname{PSh}}

\newcommand{\eh}{\text{\'{e}h}}

\newcommand{\h}{\mathrm{h}}

\renewcommand{\H}{\mathrm{H}}

\newcommand{\TC}{\mathrm{TC}}

\newcommand{\fib}{\mathrm{fib}}

\newcommand{\THH}{\mathrm{THH}}
\newcommand{\Pic}{\mathcal{P}\mathrm{ic}}
\renewcommand{\SH}{\mathrm{SH}}
\renewcommand{\Cat}{\opp{Cat}}
\usepackage{relsize}
\usepackage[bbgreekl]{mathbbol}
\usepackage{amsfonts}
\DeclareSymbolFontAlphabet{\mathbb}{AMSb} %to ensure that the meaning of \mathbb does not change
\DeclareSymbolFontAlphabet{\mathbbl}{bbold}

\numberwithin{proposition}{section}
\numberwithin{equation}{section}

\usepackage{framed}
\renewcommand{\bb}[1]{\mathbb{#1}}
\newcommand{\sub}[1]{{\mbox{\rm \scriptsize #1}}}
\newcommand{\roi}{\mathcal{O}}
\newcommand{\isoto}{\stackrel{\simeq}{\to}}
\newcommand{\quis}{\stackrel{\sim}{\to}}
\newcommand{\opp}{\operatorname}
\renewcommand{\hat}{\widehat}

\newcommand{\To}{\longrightarrow}
\renewcommand{\cal}{\mathcal}
\newcommand{\comment}[1]{}

\renewcommand{\Spt}{\mathrm{Sp}}
\newcommand{\xto}{\xrightarrow}
\renewcommand{\tilde}{\widetilde}

\newcommand{\res}{\overline}
\newcommand{\al}{\alpha}
\newcommand{\into}{\hookrightarrow}

\usepackage{sectsty}
\sectionfont{\Large\sc\centering}
\chapterfont{\large\sc\centering}
\chaptertitlefont{\LARGE\centering}
\partfont{\centering}

\usepackage{fancyhdr}

\title{Motivic cohomology of equicharacteristic schemes}
\date{\today}

\author{Elden Elmanto\footnote{University of Toronto}, Matthew Morrow\footnote{CNRS et Laboratoire de Math\'ematiques d'Orsay}}

\begin{document}

\maketitle

\begin{abstract}
We construct a theory of motivic cohomology for quasi-compact, quasi-separated schemes of equal characteristic, which is related to non-connective algebraic $K$-theory via an Atiyah--Hirzebruch spectral sequence, and to \'etale cohomology in the range predicted by Beilinson and Lichtenbaum. On smooth varieties over a field our theory recovers classical motivic cohomology, defined for example via Bloch's cycle complex. Our construction uses trace methods and (topological) cyclic homology.

As predicted by the behaviour of algebraic $K$-theory, the motivic cohomology is in general sensitive to singularities, including non-reduced structure, and is not $\bb A^1$-invariant. It nevertheless has good geometric properties, satisfying for example the projective bundle formula and pro cdh descent.

Further properties of the theory include a Nesterenko--Suslin comparison isomorphism to Milnor $K$-theory, and a vanishing range which simultaneously refines Weibel's conjecture about negative $K$-theory and a vanishing result of Soul\'e for the Adams eigenspaces of higher algebraic $K$-groups. We also explore the relation of the theory to algebraic cycles, showing in particular that the Levine--Weibel Chow group of zero cycles on a surface arises as a motivic cohomology group.
\end{abstract}

\tableofcontents

\section{Introduction}
The vision of motivic cohomology is due to Beilinson and Lichtenbaum \cite{Beilinson1987a, bms-zero, Lichtenbaum1984}. For a reasonable class of schemes $X$ they predicted the existence of natural complexes of abelian groups $\Z(j)^{\mot}(X)$, for $j\ge0$, satisfying various relations to algebraic $K$-theory and \'etale cohomology. Perhaps the most important of these relations is a desired {\em Atiyah--Hirzebruch spectral sequence}
\begin{equation}\label{eq:ahss}
E_2^{ij}=H_{\mot}^{i-j}(X, \Z(-j)) \implies \K_{-i-j}(X),
\end{equation}
relating the {\em motivic cohomology groups} $H^i_{\mot}(X, \Z(j)) := H^i(\Z(j)^{\mot}(X))$ to the algebraic $K$-groups of $X$. (Note that $\Z(j)^{\mot}=0$ for $j<0$; this is either a result or a convention, depending on the context.) They asked that this spectral sequence would degenerate rationally and identify the rationalised motivic cohomology $H^i_{\mot}(X, \Z(j))\otimes_\bb Z\bb Q$ with the Adams eigenspace  $\K_{i-2j}(X)_\bb Q^{(j)}$. Meanwhile, motivic cohomology with finite coefficients $H^i_{\mot}(X, \Z/\ell (j)) := H^i(\Z(j)^{\mot}(X)/\ell)$ was expected to coincide with \'etale cohomology $H^i_\sub{\'et}(X,\mu_{\ell}^{\otimes j})$ in the range $i\le j$, whenever $\ell>0$ is invertible on $X$. Note that any such theory of motivic cohomology must necessarily fail to be $\bb A^1$-invariant for sufficiently singular $X$, i.e., the maps $H^i_\sub{mot}(\bb A_X^1,\bb Z(j))\to H^i_\sub{mot}(X,\bb Z(j))$ are not in general isomorphisms, since algebraic $K$-theory also fails to be $\bb A^1$-invariant on general schemes.

In this article, which builds on our joint work with T.~Bachmann about $\bb A^1$-invariant motivic cohomology \cite{BachmannElmantoMorrow}, we construct such motivic complexes $\bb Z(j)^\sub{mot}(X)$ whenever $X$ is a quasi-compact, quasi-separated scheme of equal characteristic.\footnote{By {\em equal characteristic} we mean that the structure map $X\to\Spec(\bb Z)$ factors through $\Spec(\bb Q)$ or $\Spec(\bb F_p)$ for some prime number $p$. The equal characteristic assumption which runs throughout this paper has been removed by Bouis \cite{Bouis2025, Bouis2025a, Bouis2025b}, whose work therefore provides a robust theory of motivic cohomology for arbitrary quasi-compact, quasi-separated schemes.} For the rest of the introduction let $\bb F$ be a prime field, i.e., $\bb Q$ or $\bb F_p$ for some prime number $p$.

\begin{theorem}\label{thm:main}
There exist finitary Nisnevich sheaves
\[
\Z(j)^{\mot}: \Sch_{\bb F}^{\qcqs,\op} \To \rm D(\Z)
\]
for $j\ge0$ such that the following properties hold for any qcqs $\bb F$-scheme $X$:
\begin{enumerate}
\item There exists a functorial, multiplicative, $\mathbb{N}$-indexed filtration $\mathrm{Fil}_\sub{mot}^{\star}\K(X)$ on the non-connective algebraic $K$-theory $K(X)$, such that the graded pieces are naturally and multiplicatively given by
\[
\mathrm{gr}^j_\sub{mot}\K(X) \simeq \Z(j)(X)^{\mot}[2j]
\]
for $j\ge0$.
In particular, writing $H^i_\sub{mot}(X,\bb Z(j)):=H^i(\bb Z(j)^\sub{mot}(X))$ for the corresponding {\em motivic cohomology groups}, there exists a natural multiplicative Atiyah-Hirzebruch spectral sequence
\[
E_2^{ij}=H_{\mot}^{i-j}(X, \Z(-j)) \implies \K_{-i-j}(X).
\]
If $X$ has finite valuative dimension, then the filtration $\mathrm{Fil}_\sub{mot}^{\star}\K(X)$ is complete and the Atiyah--Hirzebruch spectral sequence is convergent.

\item Rational structure: the Atiyah--Hirzebruch spectral sequence degenerates rationally and there are natural isomorphisms \[H^i_\sub{mot}(X,\bb Z(j))\otimes_\bb Z\bb Q\cong \K_{2j-i}(X)_\bb Q^{(j)}\] for all $i\in\bb Z$ and $j\ge0$, where the right side refers to Adams eigenspaces of rationalised $K$-theory.
\item Relation to \'etale cohomology: for any integer $\ell>0$ invertible in $\bb F$, there are natural equivalences \[\tau^{\le j}(\bb Z(j)^\sub{mot}(X)/\ell)\simeq\tau^{\le j}R\Gamma_\sub{\'et}(X,\mu_\ell^{\otimes j})\] for $j\ge0$.
\item Relation to syntomic cohomology: if $\bb F=\bb F_p$ then for any $r>0$ there are natural equivalences
\[\tau^{\leq j}(\bb Z(j)^\sub{mot}(X)/p^r) \simeq \tau^{\leq j}(\bb Z_p(j)^\sub{syn}(X)/p^r)\] for $j\ge0$, where $\bb Z_p(j)^\sub{syn}(X)$ denotes the weight-$j$ {\em syntomic} cohomology of $X$ in the sense of \cite{AntieauMathewMorrowNikolaus, BhattMorrowScholze2}.
\item Weight zero: there is a natural equivalence \[\bb Z(0)^\sub{mot}(X)\simeq R\Gamma_\sub{cdh}(X,\bb Z)\] where the right side denotes cdh cohomology with coefficients in the constant sheaf $\bb Z$.
\item Weight one: there is a natural first Chern class map \[R\Gamma_\sub{Nis}(X,\bb G_m)[-1]\To \bb Z(1)^\sub{mot}(X),\] which is an equivalence in degrees $\le 3$.
\item Projective bundle formula: For any $j,r \geq 0$, the powers of the first Chern class of the tautological bundle $c_1(\scr O(1)) \in \mathrm{Pic}(\P^r_X) \cong H^{2}_\sub{mot}(\P^r_X, \Z(1))$
induce a natural equivalence
\[
\bigoplus^r_{i=0} \Z(j-i)^{\mot}(X)[-2i] \quis \Z(j)^{\mot}(\P^r_X). 
\]
\item Blow-up formula: Given any regular closed immersion $Y\to X$ (i.e., $X$ admits an open affine cover such that, on each such affine, $Y$ is defined by a regular sequence), then $\bb Z(j)^\sub{mot}$ carries the cartesian square of schemes
\[\xymatrix{
Y\times_X\mathrm{Bl}_Y(X)\ar[d]\ar[r] & \mathrm{Bl}_Y(X)\ar[d]\\
Y\ar[r] & X
}\]
to a cartesian square in $\rm D(\bb Z)$.
\item Finally, suppose $X$ is a smooth scheme over a field. Then there are equivalences \[
\Z(j)^{\mot}(X) \simeq z^j(X,\bullet)[-2j]
\]
for $j\ge0$, where $z^j(X,\bullet)$ is Bloch's cycle complex of $X$. Moreover the filtration $\Fil^\star_\sub{mot}$ on $\K(X)$ is naturally equivalent, as multiplicative filtered spectra, with the filtration coming from Voevodsky's slice filtration as in \cite{Voevodsky2002a}.
\end{enumerate}
\end{theorem}

The original approach to motivic cohomology was that of Bloch \cite{Bloch1986b}, in terms of his cycle complexes $z^j(X,\bullet)$ for algebraic varieties $X$. Ignoring certain technicalities (such as functoriality, multiplicative structure, quasi-projectivity hypotheses,...), the work of Bloch, Bloch--Lichtenbaum \cite{BlochLichtenbaum}, Friedlander--Suslin \cite{friedlander-suslin}, and Levine \cite{levine-tech, Levine2008} show that the complexes $z^j(X,\bullet)[-2j]$ satisfy a variant of the conjectural framework of Beilinson and Lichtenbaum; the crucial difference is that $z^j(-,\bullet)$ is covariant in the algebraic variety and the Atiyah--Hirzebruch spectral sequence converges not to the $K$-theory of $X$ but rather to the $G$-theory. (In terms of Voevodsky's approach \cite{lecture-notes-mot-cohom, voevodsky-triang-motives} via triangulated categories of motives, Bloch's cycle complex appears as Borel--Moore homology.) However, restricting attention to smooth algebraic varieties $X$, the motivic complexes \[\bb Z(j)^\bb A(X)=z^j(X,\bullet)[-2j]\] do have all desired properties (and the technicalities can be overcome using motivic stable homotopy theory and the slice filtration); we will call this theory the {\em classical motivic cohomology} of the smooth algebraic variety $X$; see \S\ref{ss_classical} for a review. Theorem \ref{thm:main}(9) states that the new theory of this paper reduces to the classical theory in the smooth case.

\begin{remark}
Although the focus of this article is to extend motivic cohomology beyond smooth algebraic varieties, our results have applications to the smooth case. For example, we will see in Corollary \ref{corol_smooth_comparison} that Theorem \ref{thm:main}(9) implies that the canonical map $R\Gamma_\sub{\'et}(X,\Omega^j_\sub{log})\to R\Gamma_\sub{\'eh}(X,\Omega^j_\sub{log})$ is an equivalence for any smooth variety $X$ over a field of characteristic $p$. The analogous equivalence between the Nisnevich and cdh cohomologies is contained in the joint work with Bachmann \cite{BachmannElmantoMorrow}. Such results, which are required for example in Geisser's theory of arithmetic cohomology \cite{Geisser2006}, seem to have been previously out of reach without assuming resolution of singularities.
\end{remark}

\subsection{Relation to $\bb A^1$-invariant and cdh-motivic cohomologies}\label{ss_intro_cdh}
This article depends on our joint work with Bachmann \cite{BachmannElmantoMorrow}, in which we revisit the theory of $\bb A^1$-invariant motivic cohomology on arbitrary qcqs schemes. In the following overview of some aspects of \cite{BachmannElmantoMorrow} we restrict to the case of equicharacteristic schemes, both because it is enough for the present paper and because \cite{BachmannElmantoMorrow} simplifies in that case. See \S\ref{ss_cdh_local} for further details.

In \cite{BachmannElmantoMorrow} we introduced {\em $\bb A^1$-motivic cohomology} \[\bb Z(j)^\bb A:\text{Sch}_\bb F^\sub{qcqs,op}\To\mathrm{D}(\bb Z),\qquad j\ge0\] as an extension of classical motivic cohomology from smooth $\bb F$-schemes to arbitrary qcqs $\bb F$-schemes, while retaining $\bb A^1$-invariance. A priori it is defined using Morel--Voevodsky's motivic homotopy theory and Voevodsky's slice filtration (see Remark \ref{rem:def_of_ZjA} for further details), but with Bachmann we show that it admits the following simpler description: it is the cdh sheafification of the left Kan extension of classical motivic cohomology from smooth $\bb F$-schemes to qcqs $\bb F$-schemes. It is therefore a cdh analogue of Bloch's cycle complex and, in the case of singular algebraic varieties assuming strong resolution of singularities, in fact coincides with a motivic cohomology studied by Friedlander and Voevodsky \cite[Definition 9.2]{FriedlanderVoevodsky2000}.

\begin{remark}[Cdh-motivic cohomology]
The order of logic of \cite{BachmannElmantoMorrow} requires introducing and studying the cdh-local left Kan extension of classical motivic cohomology, which we denote by $\bb Z(j)^\sub{cdh}$, before the eventual equivalence $\bb Z(j)^\sub{cdh}\quis \bb Z(j)^\bb A$ can be established. For the sake of the present article, readers who wish to bypass the motivic homotopy theoretic definition of $\bb Z(j)^\bb A$ can simply redefine it to be $\bb Z(j)^\sub{cdh}$; we implicitly did this in the first version of the article by working throughout with $\bb Z(j)^\sub{cdh}$ but now feel that it does not notationally reflect the conceptual dichotomy between non-$\bb A^1$-invariant and $\bb A^1$-invariant motivic cohomologies.
\end{remark}

$\bb A^1$-motivic cohomology fits into an Atiyah--Hirzebruch spectral sequence converging to $\KH$-theory, with finite coefficients it is related to \'etale and syntomic cohomology, and one of the deepest theorems of \cite{BachmannElmantoMorrow} shows that it is represented by the zeroth slice of the unit $1_X\in\SH(X)$. Overall, $\bb A^1$-motivic cohomology satisfies a variant of Beilinson and Lichtenbaum's vision, after imposing $\bb A^1$-invariance everywhere, and we have no doubt that it is the ``correct'' theory of $\bb A^1$-invariant motivic cohomology.
The theory of the present paper is designed so that, for any qcqs equicharacteristic scheme $X$, the canonical map $\K(X)\to\KH(X)$ is compatible with the motivic filtrations on each side, thereby inducing comparison maps \begin{equation}\bb Z(j)^\sub{mot}(X)\To \bb Z(j)^\bb A(X)\label{eqn_nonA1_to_A1}\end{equation} from the new not-necessarily-$\bb A^1$-invariant motivic cohomology to the $\bb A^1$-motivic cohomology. This comparison map has the following properties, thereby refining to the level of motivic cohomology known comparisons between $\K$-theory and $\KH$-theory:

\begin{theorem}[See Theorems~\ref{thm_mot_vs_cdh}, \ref{thm:graded-pieces}, and \ref{thm:graded-pieces_charp}]\label{intro_thm_SH}
Let $j\ge0$ and let $\bb F$ be a prime field.
\begin{enumerate}
\item The map (\ref{eqn_nonA1_to_A1}) identifies $\bb Z(j)^\bb A$ with the $\bb A^1$-localisation and the cdh-sheafification of $\bb Z(j)^\sub{mot}$. That is, on the category of qcqs $\bb F$-schemes, there are natural equivalences of $\rm D(\bb Z)$-valued presheaves: \[L_{\bb A^1}\bb Z(j)^\sub{mot}\simeq \bb Z(j)^\bb A \simeq L_\sub{cdh}\bb Z(j)^\sub{mot}.\]
\item For any qcqs $\bb F$-scheme and integer $\ell>0$ invertible in $\bb F$, the map (\ref{eqn_nonA1_to_A1}) is an equivalence mod $\ell$: \[\bb Z(j)^\sub{mot}(X)/\ell\quis \bb Z(j)^\bb A(X)/\ell.\]
\item For any qcqs $\bb F_p$-scheme $X$, the map (\ref{eqn_nonA1_to_A1}) is an equivalence after inverting $p$: \[\bb Z(j)^\sub{mot}[\tfrac1p]\quis \bb Z(j)^\bb A[\tfrac1p].\]
\item For any regular Noetherian $\bb F$-scheme $X$, the map (\ref{eqn_nonA1_to_A1}) is an equivalence.
\end{enumerate}
\end{theorem}

Part (1) of the theorem refines the fact that $\KH$-theory is both the $\bb A^1$-localisation of $K$-theory (by definition) and its cdh-sheafification (as we will recall at the start of \S\ref{intro_ss_sketch}). Parts (2) and (3) refine in equicharacteristic results of Weibel \cite{Weibel1989a} that $\K(A)/\ell \quis \KH(A)/\ell$ (resp.~$\K(A)[\tfrac1p] \quis \KH(A)[\tfrac1p]$) for rings $A$ in which $\ell$ is invertible (resp.~in which $p=0$). Finally, part (4) refines in equicharacteristic the equivalence between $\K$-theory and $\KH$-theory for regular Noetherian rings.

An input to establishing part (3) of the previous theorem, which is essential to controlling our motivic cohomology in characteristic $p$, is to show that rationalised syntomic cohomology $\bb Q_p(j)^\sub{syn}:=\bb Z_p(j)^\sub{syn}[\tfrac1p]$ is a cdh sheaf on qcqs $\bb F_p$-schemes (see Lemma \ref{lemma_Qpsyn}). Perhaps this can be proved directly, but our approach is rather to reduce it to the aforementioned fact that $\K[\tfrac1p]=\KH[\tfrac1p]$ on such schemes; the reduction argument passes through the cartesian square (\ref{eq:mainsquare}) below and so ultimately depends on trace methods. This extraction of information about cohomology theories from localising invariants is a theme which runs throughout this paper and \cite{BachmannElmantoMorrow}; we will return to it in Remark \ref{re_into_pmf} when discussing the projective bundle formula.

\subsection{The construction of $\bb Z(j)^\sub{mot}$ via trace methods}\label{intro_ss_sketch}
Our construction of $\bb Z(j)^\sub{mot}$ is inspired by trace methods in algebraic $K$-theory. For any qcqs scheme $X$ let $\TC(X)$ denote its topological cyclic homology, and $\K^\sub{inf}(X)$ the fibre of the trace map $\K(X)\to\TC(X)$. The presheaf $\K^\sub{inf}$ is nil-invariant by the Dundas--Goodwillie--McCarthy theorem \cite{Dundas2013}, and even a cdh sheaf by Kerz--Strunk--Tamme \cite{KerzStrunkTamme2018} and Land--Tamme \cite{LandTamme2019}. Coupled with the surprising fact that Weibel's $\KH$-theory is equivalent to the $\cdh$-sheafification of $K$-theory (first proved by Haesemeyer in characteristic zero \cite{Haesemeyer2004} and Kerz--Strunk--Tamme \cite{KerzStrunkTamme2018} in general), we arrive at a cartesian square for any qcqs scheme
\begin{equation}\label{eq:mainsquare}
\begin{tikzcd}
\K(X) \ar{d} \ar{r} & \TC(X) \ar{d}\\
\KH(X) \ar{r} & L_{\cdh}\TC(X),
\end{tikzcd}
\end{equation}
where the bottom map is the $\cdh$-sheafified trace map. We define the motivic filtration $\Fil^{\star}_\sub{mot}$ on $\K(X)$ by glueing existing filtrations on $\KH(X)$, $\TC(X)$, and $L_\sub{cdh}\TC(X)$:

\begin{enumerate}
\item For any qcqs $\bb F_p$-scheme $X$, Bhatt, Scholze, and the second author \cite{BhattMorrowScholze2} have defined a filtration on $\TC(X)$ whose graded pieces are $\bb Z_p(j)^\sub{syn}(X)[2j]$ for $j\ge0$; here $\bb Z_p(j)^\sub{syn}(X)$ is the {\em syntomic cohomology} of $X$, which modulo $p$ is a derived version of the \'etale cohomology of Illusie--Milne's sheaves $\Omega^j_{X,\log}$. Cdh sheafifying this filtration over qcqs $\bb F_p$-schemes induces a filtration on $L_\sub{cdh}\TC(X)$.
\item For any qcqs $\bb Q$-scheme $X$, its topological cyclic homology $\TC(X)$ identifies with its negative cyclic homology $\HC^-(X/\bb Q)$. Antieau \cite{antieau-fil} has defined a filtration on $\HC^-(X/\Q)$, extending previous work of Loday \cite{Loday1989}, Weibel \cite{Weibel1997} and To\"en--Vezzosi \cite{ToenVezzosi2011}, whose graded pieces are $R\Gamma(X, \widehat{L\Omega}_{-/\bb Q}^{\geq j})[2j]$ for $j\in\bb Z$. Here $\widehat{L\Omega}_{-/\bb Q}$ is the Hodge-completed \emph{derived de Rham complex} equipped with its Hodge filtration, as studied notably by Illusie \cite{Illusie1971, Illusie1972} and Bhatt \cite{Bhatt2012}. As in characteristic $p$, cdh sheafifying then induces a compatible filtration on $L_\sub{cdh}\TC(X)=L_\sub{cdh}\HC^-(X/\bb Q)$.
\item As already discussed in \S\ref{ss_intro_cdh}, for any qcqs equicharacteristic scheme $X$, motivic homotopy theory defines a motivic filtration on $\KH(X)$ whose graded pieces are $\bb A^1$-motivic cohomology $\bb Z(j)^\bb A(X)[2j]$ for $j\ge0$.
\end{enumerate}

The following compatibility of these filtrations depends crucially on the cdh-local description of $\bb A^1$-motivic cohomology mentioned in \S\ref{ss_intro_cdh} and is fundamental to our construction; we refer to Corollary \ref{corol_cdh_filtered_trace} and Proposition \ref{prop_cdh_filtered_trace_p} for more precise statements:

\begin{proposition}
For any qcqs equicharacteristic scheme $X$, the cdh-sheafified trace map $\KH(X)\to \TC(X)$ respects the filtrations on each side.
\end{proposition}

We consequently define our motivic filtration $\Fil^\star_\sub{mot}$ on $\K(X)$ by glueing the existing filtrations on the three other corners of the square (\ref{eq:mainsquare}). Passing to graded pieces yields the following description of our motivic cohomology:

\begin{theorem}[See Theorems~\ref{thm:graded-pieces} and \ref{thm:graded-pieces_charp}]\label{intro_thm_main_squares}
For $j\ge0$ and any qcqs scheme $X$ over $\bb Q$ (resp.~$\bb F_p)$, there is a natural cartesian square in $\mathrm{D}(\bb Z)$
\begin{equation}\label{eq:fundamental_squares}
\xymatrix@=1cm{
\Z(j)^{\mot}(X) \ar[r] \ar[d] & R\Gamma(X,\widehat{L\Omega}_{-/\Q}^{\geq j})  \ar[d]\\
\Z(j)^\bb A(X) \ar[r] & R\Gamma_{\cdh}(X,\widehat{L\Omega}_{-/\Q}^{\geq j}).
}
\qquad\text{resp.~}
\xymatrix@=1.3cm{
\Z(j)^{\mot}(X) \ar[r] \ar[d] & \bb Z_p(j)^\sub{syn}(X)  \ar[d]\\
\Z(j)^\bb A(X) \ar[r] & L_\sub{cdh}\bb Z_p(j)^\sub{syn}(X).
}
\end{equation}
(Where the right vertical arrows are cdh sheafification, and the bottom horizontal arrows are cdh versions of de Rham/syntomic cycle class maps: see Remarks \ref{rem:horizontal_map_char0} and \ref{rem:horizontal_map_charp}.)
\end{theorem}

The cartesian squares (\ref{eq:fundamental_squares}) encapsulate the central idea of our construction of motivic cohomology. They say that the motivic complex $\bb Z(j)^\sub{mot}(X)$ is a modification of the $\bb A^1$-motivic cohomology $\bb Z(j)^\bb A(X)$ (which is morally governed by algebraic cycles) by derived de Rham/syntomic~cohomology. In particular, in characteristic zero the left square of (\ref{eq:fundamental_squares}) yields a fibre sequence \[\bb Z(j)^\sub{mot}(X)\To \bb Z(j)^\bb A(X)\To {\rm cofib}\left(R\Gamma(X,L\Omega_{-/\Q}^{< j})\to  R\Gamma_\sub{cdh}(X,\Omega_{-/\Q}^{< j})\right)[-1];\] this plays the role of the weight-$j$ motivic component of the well-known fibre sequence \[\K(X)\To \KH(X)\To{\rm cofib}\big(\HC(X)\to L_\sub{cdh}\HC(X)\big)[1]\] arising from (\ref{eq:mainsquare}), which was used throughout Corti\~nas--Haesemeyer(--Schlichting)--Weibel's work \cite{Cortinas2008, Cortinas2008a} on the $K$-theory of singular varieties in characteristic zero. The present paper may in fact be roughly understood as a refinement of their work from the level of $K$-theory to that of motivic cohomology, as well as providing an extension to finite characteristic.

The squares (\ref{eq:fundamental_squares}) also provide a refinement of the trace map and its main property to the level of motivic cohomology:

\begin{corollary}\label{corol_motivic_DGM}
On the category of qcqs schemes over $\bb Q$ (resp.~$\bb F_p$), there exists for each $j\ge0$ a ``weight-$j$ motivic trace map'' (namely the top horizontal arrow in (\ref{eq:fundamental_squares})) \[\bb Z(j)^\sub{mot}\To R\Gamma(-,\widehat{L\Omega}_{-/\Q}^{\geq j}),\qquad\text{resp.~}\bb Z(j)^\sub{mot}\To \bb Z_p(j)^\sub{syn},\] whose fibre is a cdh sheaf.
\end{corollary}

\begin{remark}[Projective bundle formula]\label{re_into_pmf}
As already stated in Theorem \ref{thm:main}(8), our motivic cohomology satisfies the projective bundle formula. For Grothendieck this was one of the most fundamental desired properties of any cohomology theory, and it means that the motivic cohomology assembles into a motivic spectrum in the sense of Annala--Iwasa \cite{AnnalaIwasa2023}. But it is also an essential input into proving the comparison theorems with classical motivic cohomology (Theorem \ref{thm:main}(9)) and with the $\bb A^1$-invariant theory in the regular case (Theorem \ref{intro_thm_SH}(4)), as we use Gabber's technique \cite{Gabber1994, GrosSuwa1988} axiomatised by Colliot-Th\'el\`ene--Hoobler--Kahn \cite{ColliotThelene-Hoobler-Kahn1997}.

Remarkably, our proof of the projective bundle formula depends on the theory of localising invariants (at least in characteristic $p$ -- in characteristic zero it is sufficient to use strong resolution of singularities). Indeed, exploiting the fact that cdh-motivic cohomology and syntomic cohomology are known to have this property (by \cite[Corollaries~8.23]{BachmannElmantoMorrow} and \cite{BhattLurie2022} respectively; the proof of the former also uses localising invariants), the problem reduces via the right square in (\ref{eq:fundamental_squares}) to showing that cdh-sheafified syntomic cohomology satisfies the projective bundle formula. We prove this in Theorem \ref{thm:cdh-syn-pbf}, using the fact that the square (\ref{eq:mainsquare}) is cartesian and so allows us to upgrade $L_\sub{cdh}\TC$ to a localising invariant.
\end{remark}

\subsection{Relation to Milnor $\K$-theory and lisse motivic cohomology}\label{ss_milnor}
When $F$ is a field, a theorem of Nesterenko--Suslin \cite{Suslin1989}, later reproved by Totaro \cite{Totaro1992}, produces a natural isomorphism between the classical motivic cohomology $H^{j}_{\mot}(F,\Z(j))$ with the Milnor $K$-group $ \K^M_{j}(F)$. On the one hand, this gives a ``cohomological interpretation'' of the Milnor $K$-groups defined via generators and relations. On the other hand it shows that the graded ring $\bigoplus_{j\ge0}H^{j}_{\mot}(F,\Z(j))$ is generated by elements in degree $1$ and with relations in degree $2$.

In his thesis \cite{Kerz2009}, Kerz extended the Nesterenko--Suslin isomorphism to the generality of regular local rings containing an infinite field, thereby settling a conjecture of Beilinson. He later eliminated the hypothesis that the field be infinite, using the improved Milnor $K$-theory $\widehat{\K}^M_j$ which he and Gabber had introduced \cite{Kerz2010}.

Our motivic cohomology satisfies the Nesterenko--Suslin isomorphism for arbitrary local rings containing a field, without any regularity hypotheses:

\begin{theorem}[Nesterenko--Suslin isomorphism; see Thm.~\ref{theorem_NS}]\label{thm:nst} Let $A$ be a local ring containing a field. Then the isomorphism $A^\times\cong H^1_\sub{mot}(A,\bb Z(1))$ of Theorem \ref{thm:main}(6) induces, by multiplicativity, isomorphisms
\begin{equation}\label{eq:nst}
\widehat{\K}^M_{j}(A) \isoto H^{j}_{\mot}(A, \Z(j))
\end{equation}
for all $j\ge1$.
\end{theorem}

The proof of Theorem~\ref{thm:nst} is intertwined with a comparison theorem relating our motivic cohomology to a more naive cohomology obtained by simply left Kan extending classical motivic cohomology. More precisely, for any $\bb F$-algebra $A$ we define $\bb Z(j)^\sub{lse}(A)\in \rm D(\bb Z)$ to be the left Kan extension, from smooth $\bb F$-algebras, of weight-$j$ classical motivic cohomology.  More explicitly, there exists a simplicial resolution $P_{\bullet} \quis A$ whose terms are ind-smooth $\bb F$-algebras and whose face maps are henselian surjections; then $\bb Z(j)^\sub{lse}(A)$ is given by the totalisation of the simplicial complex $\bb Z(j)^\bb A(P_\bullet)$. We call $\bb Z(j)^\sub{lse}(A)$ the weight-$j$, {\em lisse motivic cohomology} of $A$ to emphasise the fact that it controlled by smooth algebras. The complexes $\bb Z(j)^\sub{lse}(A)[2j]$, for $j\ge0$, appear as the graded pieces of a motivic filtration on the connective algebraic $K$-theory $\K^\sub{cn}(A)$; see \S\ref{subsec_lke} for more details. For any $\bb F$-algebra $A$ there is a natural comparison map \begin{equation}\bb Z(j)^\sub{lse}(A)\To \bb Z(j)^\sub{mot}(A),\label{intro:lke_vs_mot}\end{equation} and we prove the following, which in degree $j$ is the Nesterenko--Suslin isomorphism:

\begin{theorem}[see Theorem~\ref{thm_lke_lej}]\label{thm_into_lke_lej}
For any local $\bb F$-algebra $A$ and $j\ge0$, the map (\ref{intro:lke_vs_mot}) induces an equivalence \[\bb Z(j)^\sub{lse}(A)\quis\tau^{\le j}\bb Z(j)^\sub{mot}(A).\]
\end{theorem}

The theorem states that, in degrees less than or equal to the weight, our motivic cohomology is Zariski locally controlled by classical motivic cohomology; in particular, in this range it is closely related to algebraic cycles. This is the next topic we discuss.

\subsection{Relations to algebraic cycles}\label{sec:cycles-intro}
One of the key features of the classical motivic cohomology of smooth algebraic varieties $X$ is its  description in terms of algebraic cycles, via Bloch's cycle complex; this yields in particular isomorphisms
\begin{equation}H^{2j}_\sub{mot}(X,\bb Z(j))\cong \textrm{CH}^j(X)\label{eqn_mot_vs_Chow}\end{equation} for each $j\ge0$.

In the case of a singular algebraic variety $X$, various definitions of Chow groups have been proposed. A first possibility is Fulton's \cite{Fulton1998}, but his theory is a Borel--Moore homology theory related more to $G(X)$ than $\K(X)$. Another is Baum--Fulton--Macpherson's \cite{BaumFultonMacPherson1975} theory of cohomological Chow groups, essentially obtained by left Kan extending $\textrm{CH}^j$ from smooth algebraic varieties to arbitrary varieties; it is thus related, at least superficially, to the lisse motivic cohomology $\bb Z(j)^\sub{lse}$ introduced above in \S\ref{ss_milnor}. Levine \cite{Levine1983} refined Baum--Fulton--Macpherson's idea by (roughly speaking) restricting the class of smooth varieties appearing in the left Kan extension procedure to better control the algebraic cycles. The case of zero cycles is particularly well-developed, and the {\em Levine--Weibel} Chow group $\rm CH^\sub{LW}_0(X)$ of zero cycles \cite{Levine1985a} on a singular variety $X$ has found concrete applications towards $K$-theoretic problems such as the splitting of vector bundles on affine varieties \cite{Krishna2002, Murthy1994}. For a modern text on Levine--Weibel's group, we refer the reader to work of Binda--Krishna \cite{BindaKrishna2018, BindaKrishna2022}. 

In \S\ref{section_cf_cycles} we study the relationship between our motivic cohomology and algebraic cycles. We are not sure what to expect in general, but we can show in the case of surfaces that our theory captures the Levine--Weibel group of zero cycles:

\begin{theorem}[See Thm.~\ref{thm:lw-comparison}]\label{thm_intro_LW}
Let $X$ be a reduced, equi-dimensional, quasi-projective surface over an infinite field $k$; then there is a natural isomorphism \begin{equation}H^4_\sub{mot}(X,\bb Z(2))\cong \mathrm{CH}_0^\sub{LW}(X).\label{eqn_mot_vs_Chow2}\end{equation}
\end{theorem}

Whereas one often views (\ref{eqn_mot_vs_Chow}) as a description of motivic cohomology in terms of algebraic cycles, we suggest adopting the alternative point of view on (\ref{eqn_mot_vs_Chow2}), namely it provides a new description of zero cycles on singular surfaces. Indeed, bearing in mind the main idea presented after Theorem \ref{intro_thm_main_squares}, it says that the Levine--Weibel group of zero cycles of a surface is somehow built from cdh-local zero cycles and derived de Rham/syntomic cohomology.

\begin{example}[cdh-local zero cycles on surfaces]
Let $X$ be as in Theorem \ref{thm_intro_LW}. The proof of the Soul\'e--Weibel vanishing Theorem \ref{intro_Weibel_vanishing} below implies in addition that the canonical map \[\textrm{CH}_0^\sub{LW}(X)=H^{4}_\sub{mot}(X,\bb Z(j))\To H^{4}_\bb A(X,\bb Z(2))=H^{2}_\sub{cdh}(X,\hat \K_2^M)\] is surjective, where the right equality follows from a cdh-local version of the Nesterenko--Suslin isomorphism. In other words, any ``cdh-local zero cycle'' comes from an honest zero cycle. Our results also imply that this map is an isomorphism modulo any integer invertible in the base field $k$, and an isomorphism after inverting $p$ if $k$ has characteristic $p>0$. This may well be known to experts: such comparisons have certainly been established previously for {\em projective} varieties in arbitrary dimensions \cite[Theorems.~1.6 \& 1.7]{BindaKrishna2022}.
\end{example}

Another context in which algebraic cycles appear in motivic cohomology is the theory of Chow groups with modulus, building on Bloch--Esnault's earlier notion of additive Chow groups \cite{Bloch2003}. The set-up of the theory varies, but suppose for simplicity that $X$ is a smooth algebraic variety equipped with an effective divisor $D$ such that $D_\sub{red}$ is a simple normal crossing divisor. The theory defines various ``Chow groups on $X$ with modulus $D$'', which it is hoped will ultimately correspond to a piece of the motivic cohomology of $X$ relative to $D$; these Chow groups with modulus should in particular be related to the $K$-theory of $X$ relative to $D$, but at present the evidence of this relation is limited. A common theme in the subject \cite{Bloch2003, Rulling2007, RullingSaito2018}, already present in the original work of Bloch--Esnault, is that Chow groups with modulus, although they are defined purely in terms of algebraic cycles, often contain groups of differential forms, Witt vectors, or more generally de Rham--Witt groups; that is, the theory offers a cycle-theoretic description of the latter objects. The theory presented in this paper provides a systematic framework to obtain similar descriptions, exemplified as follows, for which the reader should recall that lisse motivic cohomology is described by algebraic cycles:

\begin{example}[See Example~\ref{example_dim_0}]
Let $k$ be a perfect field of characteristic $p>0$, and $j,e\ge0$. Then the lisse motivic cohomology of $k[x]/x^e$ relative to its residue field, i.e., the fibre of $\bb Z(j)^\sub{lse}(k[x]/x^e)\to\bb Z(j)^\sub{lse}(k)$, is naturally equivalent to $(\bb W_{ej}(k)/V^e\bb W_j(k))[-1]$.
\end{example}

\subsection{Negative $K$-groups and Soul\'e--Weibel vanishing}
A major stimulus in the development of the algebraic $K$-theory of singular schemes has been the problem of understanding their negative $K$-groups, in which the central conjecture for many years was Weibel's vanishing conjecture: for a Noetherian scheme $X$ of finite dimension, he predicted that the negative $K$-groups $K_{-n}(X)$ would vanish for $n>\dim X$. Following numerous special cases (see the start of Section \ref{section_Weibel} for references), the conjecture was proved in general by Kerz--Strunk--Tamme \cite{KerzStrunkTamme2018}. Meanwhile, concerning the positive $K$-groups of a Noetherian ring $A$, Soul\'e \cite[Corol.~1]{Soule1985} had proved much earlier the vanishing of the Adams eigenspaces $\K_n(A)_\bb Q^{(j)}$ whenever $n>0$ and $j>n+\dim A$. The following integral motivic vanishing theorem strenghtens and unifies these two results in the equicharacteristic case:

\begin{theorem}[Motivic Soul\'e--Weibel vanishing; see Thm.~\ref{theorem_Weibel_vanishing}]\label{intro_Weibel_vanishing}
Let $j\ge0$ and let $X$ be a Noetherian equicharacteristic scheme of finite dimension. Then $H^i_\sub{mot}(X,\bb Z(j))=0$ for all $i> j+\dim X$.
\end{theorem}

Kerz--Strunk--Tamme's proof of Weibel vanishing depended on first establishing that $K$-theory satisfied pro cdh descent on Noetherian schemes, again following various special cases which had been proved earlier. It seems in fact that pro cdh descent, which is an analogue of the formal functions theorem from coherent cohomology, is one of the most fundamental properties of algebraic $K$-theory. In any case, as well as its appearance in the proof of Weibel vanishing, it has applications to the study of algebraic cycles on singular varieties \cite{Krishna2010, Krishna2002, Morrow_zero_cycles}. We prove that our motivic cohomology also has this property:

\begin{theorem}[Pro cdh descent for motivic cohomology; see Thm.~\ref{theorem_pro_cdh_descent}]\label{intro_pro_cdh_descent}
On the category of Noetherian equicharacteristic schemes, the presheaf $\bb Z(j)^\sub{mot}$ satisfies pro cdh descent for each $j\ge0$. That is, given any abstract blowup square of Noetherian equicharacteristic schemes
\[
\xymatrix{
Y'\ar[r]\ar[d] & X'\ar[d] \\
Y\ar[r] & X
}\]
the associated square of pro complexes 
\[\xymatrix{
\bb Z(j)^\sub{mot}(X) \ar[r]\ar[d] & \bb Z(j)^\sub{mot}(X')\ar[d]\\
\{\bb Z(j)^\sub{mot}(rY)\})_r \ar[r] & \{\bb Z(j)^\sub{mot}(rY')\})_r
}\]
is cartesian.
\end{theorem}

\subsection{Other recent work on motivic cohomology of singular schemes}
\subsubsection{Kelly--Saito's pro cdh-local motivic cohomology}\label{ss_KellySaito}
Kelly and Saito \cite{KellySaito2023} have defined a Grothendieck topology, called the {\em pro cdh topology}, on qcqs schemes (as well as a variant for formal schemes \cite{KellySaito2024}) with the following property: a presheaf $F$ on qcqs schemes, valued in Sp or $\rm{D}(\bb Z)$, is a pro cdh sheaf if and only if it is both a Nisnevich sheaf and, for every abstract blow-up square of qcqs schemes denoted as in Theorem \ref{intro_pro_cdh_descent}, the associated square 
\[\xymatrix{
F(X)\ar[r]\ar[d] & \lim_r F(rY)\ar[d] \\
F(X')\ar[r] & \lim_r F(rY')
}\]
is cartesian. They define {\em pro cdh-local motivic cohomology} \[\bb Z(j)^\sub{pcdh}:\text{Sch}^\sub{qcqs,op}\To\rm{D}(\bb Z)\] to be the pro cdh sheafification of the left Kan extension of classical motivic cohomology from smooth $\bb Z$-schemes to all qcqs schemes. That is, the definition mimics that of $\bb Z(j)^\sub{cdh}$, but replacing the cdh topology by their coarser pro cdh topology. 

For any Noetherian scheme $X$, its pro cdh local motivic cohomology fits into an Atiyah--Hirzebruch spectral sequence converging to $\K(X)$; to prove this one uses that, on Noetherian schemes, $K$-theory is the pro cdh sheafification of connective $\K$-theory.

By combining some of Kelly--Saito's main theorems about their topology (in particular the fact that it has enough points, and the description of the points) with some of our own (including Theorems~\ref{thm:main}(9), \ref{thm_into_lke_lej}, and \ref{intro_pro_cdh_descent}), one obtains the following comparison:

\begin{theorem}[See \cite{KellySaito2023}]\label{theorem_KS}
For any Noetherian equicharacteristic scheme $X$ and $j\ge0$, there is a natural equivalence \begin{equation}\bb Z(j)^\sub{pcdh}(X)\quis \bb Z(j)^\sub{mot}(X).\label{pcdh}\end{equation}
\end{theorem}

Thus, on Noetherian equicharacteristic schemes, Kelly--Saito's approach offers an alternative definition of the same motivic cohomology of this paper; their definition does not require trace methods. On the other hand we are not aware at present whether their approach can be used to establish, for example, the projective bundle formula, the Nesterenko--Suslin isomorphism, the comparisons to $\bb A^1$-invariant motivic cohomology, or the relation to zero cycles on surfaces. In the generality of non-Noetherian schemes, the two theories differ and pro cdh-local motivic cohomology is not finitary. The two sides of (\ref{pcdh}) thus seem to have quite different flavours; we hope that the comparison between then will serve as a useful tool in the future (for example, \cite[\S5]{KellySaito2024} offers a topos-theoretic reinterpretation of the proof of the Soul\'e--Weibel vanishing bound).

\subsubsection{Annala--Hoyois--Iwasa's non-$\bb A^1$-invariant motivic homotopy theory}
Annala, Hoyois, and Iwasa \cite{AnnalaHoyoisIwasa2023, AnnalaHoyoisIwasa2024} have been developing a theory of non-$\bb A^1$-invariant motivic homotopy theory, building on earlier work of Annala--Iwasa \cite{AnnalaIwasa2022, AnnalaIwasa2023}. A theory of motivic cohomology in their framework is provided by forcing the left Kan extension of classical motivic cohomology, from smooth schemes to qcqs schemes, to satisfy the projective bundle bundle, elementary blow-up excision, and Nisnevich descent. We all hope it coincides in the equicharacteristic case with the motivic cohomology constructed in the present paper (and for arbitrary qcqs schemes with the motivic cohomology of Bouis \cite{Bouis2025, Bouis2025a, Bouis2025b}).

\subsubsection{Park's yeni higher Chow groups}
For any algebraic variety $X$, Park \cite{Park2021} has defined complexes of cycles ${\bf z}^j(X,\bullet)$, for $j\ge0$, which coincide with Bloch's cycle complexes $z^j(X,\bullet)$ when $X$ is smooth. His complexes are Zariski locally supported in negative cohomological degrees and therefore their cohomology groups cannot fit into an Atiyah--Hirzebruch spectral sequence converging to the $K$-groups of $X$. From his construction (by locally embedding $X$ into a smooth variety and looking at certain cycles on the formal completion of the embedding), it seems plausible that ${\bf z}^j(X,\bullet)[-2j]$ is Zariski locally an explicit approximation of the lisse motivic cohomology $\bb Z(j)^\sub{lse}(X)$. In more recent work Park offers descriptions of Milnor $K$-groups of certain local rings in terms of algebraic cycles \cite{Park2025a, Park2025b}.

\subsection{Outline of the paper}
We briefly summarize the contents of the paper. After reviewing conventions regarding sheaves and filtrations in \S\ref{sec:notation} we recall previously known constructions of motivic cohomology in~\S\ref{sec:recall}. Of note is the lisse version of motivic cohomology reviewed in \S\ref{subsec_lke} which is produced simply by left Kan extending and the cdh version of motivic cohomology which is jointly produced with Bachmann, recalled in \S\ref{ss_cdh_local}. 

Our construction of motivic cohomology is explained in \S\ref{s_motivic_def}. Specifically, the characteristic zero version is given in Definition~\ref{eq:char0} and the characteristic $p > 0$ version is given in Definition~\ref{def:charp}. We briefly discuss an extension of the theory to derived schemes in \S\ref{ss_cdh_local}. In \S\ref{section_pbf} we prove the projective bundle formula and the blowup formula for motivic cohomology. This is the technical heart of the paper. In particular, we prove a $\P^1$-bundle formula for cdh-sheafified syntomic cohomology in~\S\ref{sec:syn-p1}, adopting techniques that we developed in the joint paper with Bachmann in \cite{BachmannElmantoMorrow}. The projective bundle formula is key in comparing our construction to previous constructions of motivic cohomology in the smooth setting. This is discussed, more generally, when comparing with $\A^1$-invariant versions of motivic cohomology in~\S\ref{section_smooth}. 

The last part of the paper is dedicated to deeper properties of motivic cohomology. In \S\ref{section_lke} we describe a portion of our motivic cohomology using lisse motivic cohomology. Using this, in \S\ref{sec:singular-nst}, we prove the singular Nesterenko--Suslin isomorphism. In \S\ref{section_Weibel} we prove the motivic Soul\'e--Weibel vanishing. The key ingredient is a pro cdh descent result which we establish in \S\ref{sec:pro-cdh}. We then finish off the paper by examining how our theory relates to algebraic cycles in~\S\ref{section_cf_cycles}. 

This paper has two appendices. In Appendix~\ref{app:cdh}, we prove a technical result establishing, under certain hypotheses, that the $\cdh$ sheafification of an \'etale sheaf is an $\eh$ sheaf. In Appendix~\ref{app:chw}, we discuss rational motivic cohomology and prove a spectrum-level, multiplicative refinement of a theorem of Corti\~nas, Haesemeyer and Weibel on the compatibility between Adams operators on rationalized $K$-theory and negative cyclic homology. This appendix is important in controlling the rational parts of our theory. 

\subsection{Acknowledgements}
We are grateful to Toni Annala, Ben Antieau, Tom Bachmann, Bhargav Bhatt, Federico Binda, Dustin Clausen, Frédéric Déglise, Christian Haesemeyer, Lars Hesselholt, Marc Hoyois, Ryomei Iwasa, Shane Kelly, Amalendu Krishna, Marc Levine, Akhil Mathew, Jinhyun Park, Arpon Raksit, Maxime Ramzi, Nick Rozenblyum, Shuji Saito, and Peter Scholze for useful discussions without which this project might not have been realised. We also thank Tomer Schlank for suggesting the terminology ``deflatable.''

Both the authors would like to take this opportunity to thank especially Chuck Weibel. The first author thanks him for his generosity and encouragement, and lack of pretension in mathematics. He has benefitted from conversations with Chuck on motives throughout the years, starting when he was a first year graduate student. For the second author, Chuck's work on $\K$-theory has been a source of motivation to him for many years and has directly inspired many of the main ideas of this article. 

This project has received funding from the European Research Council (ERC) under the European Union's Horizon 2020 research and innovation programme (grant agreement No.~101001474), and from the NSERC grant  RGPIN-2025-07114 ``motivic cohomology: theory and applications.'' 

In the second version of this paper the material on Adams operators has been significantly rewritten to correct some errors and clarify coherent multiplicative structures. We are grateful to Marc Hoyois, Nick Rozenblyum, and Peter Scholze for discussions on this topic.

\section{Some notation and conventions}\label{sec:notation}
In this section, we collect notation and conventions  which we will use throughout the paper. We freely use the language of $\infty$-categories as developed in \cite{LurieHA, Lurie2009}.

\subsection{Sheaves}
Let $(\scr C, \tau)$ be an $\infty$-site and $F:\scr C^{\op} \rightarrow \scr D$ a presheaf, where $\scr D$ is a stable, presentable $\infty$-category with a $t$-structure. The examples of $\scr D$ that will appear in this paper are mainly the category $\textrm{Sp}$ of spectra and the derived category $\rm D(A)$ where $A$ is a discrete coefficient ring, each equipped with the standard $t$-structure; we will also see filtered variants of these categories. One says that $F$ is \emph{discrete} if it factors through the heart $\scr D^{\heartsuit} \subseteq \scr D$; in the previous example $\rm{Sp}$ (resp.~$\rm D(R)$), a discrete presheaf means a discrete presheaf of abelian groups  (resp.~of $R$-modules). We will use the following terminology and notation:
\begin{enumerate}
\item We write $L_{\tau}$ to be the endofunctor $L_{\tau}: \PShv(\scr C) \rightarrow \PShv(\scr C)$ reflecting onto the subcategory of $\tau$-sheaves $\Shv_{\tau}(\scr C)$; this functor is referred to as \emph{sheafification}.
\item If $F$ is a discrete presheaf, then we write $R\Gamma_{\tau}(-,F)$ then for the $\tau$-cohomology of $F$; in other words we have an equivalence of $\tau$-sheaves
\[
L_{\tau}F\simeq R\Gamma_{\tau}(-,F).
\]
\item Given another topology $\tau'$ which is finer than $\tau$ then there is an adjunction
\[
\epsilon^*:\Shv_{\tau} \rightleftarrows \Shv_{\tau'}:\epsilon_*.
\]
whose unit gives rise to a canonical map in $L_{\tau}F \rightarrow \epsilon_*L_{\tau'}F$ in $\Shv_{\tau}(\scr C)$. (We note that in the case $\scr D = \rm D(R)$, then $\epsilon_*$ is often denoted in the literature as $R\epsilon_*$.) Often we regard $L_{\tau}F$ and $L_{\tau'}F$ as presheaves on $\scr C$ and simply write the previous map as $L_{\tau}F \rightarrow L_{\tau'}F$; the context should always make it clear how we are viewing the objects.
\end{enumerate}

\subsection{Filtrations}\label{ss_filtrations}
For $\cal C$ a stable $\infty$-category, the associated stable $\infty$-categories of \emph{filtered objects} and {\em graded objects} are
\[
\Fil\cal C := \Fun((\Z, \geq)^{\op}, \cal C) \qquad \text{and}\qquad \Gr\cal C := \Fun(\Z^{\delta}, \cal C),
\]
where $(\Z, \geq)$ denotes the totally ordered set of the integers and $\Z^{\delta}$ is the discrete category of the integers. Our filtrations are thus, by convention, $\bb Z$-indexed and always decreasing.\footnote{The exception is when we occasionally encounter finite filtrations, in which case we implicitly impose that the filtration be both exhaustive and complete, and we may allow it to be increasing if it makes the indexing easier to follow.} The functor of taking associated graded is written as usual by $
\mathrm{gr}^\star:\Fil\cal C \rightarrow  \Gr\cal C$.

As the notation suggests, we use $\star$ as the placeholder for the variable in graded objects. But at times this appears confusingly pedantic and so, for example, instead of writing $\bb Z(\star)^\sub{mot}$ we sometimes write $\bb Z(j)^\sub{mot}$, it being tacitly understood that $j$ varies.

Similarly, we tend to write filtered objects of $\cal C$ as $\Fil^\star M$, where $M$ is an object of $\cal C$; this notation implicitly means that there is a morphism $\Fil^{-\infty}M:=\colim_{j\to- \infty}\Fil^jM\to M$ in $\cal C$. The filtration is said to be {\em exhaustive} when the latter morphism is an equivalence. The filtration is said to be {\em $\bb N$-indexed} when $\Fil^jM\to M$ is an equivalence for all $j\le 0$ (or, equivalently, the filtration is exhaustive and $\gr^jM=0$ for $j<0$). The filtration is said to be {\em complete} if $\lim_{j\to\infty}\Fil^jM=0$.

When $\cal C=\text D(\bb Z),\,\text{Sp}$, etc., then our filtrations are often complete because they satisfy the stronger property of being {\em bounded} (``uniformly homologically bounded below'' to be more precise): for us this means that there exists $d\ge 0$ such that $\Fil^jM$ is supported in cohomological degrees $\le d-j$ for any $j\in\bb Z$. Then $\gr^jM$ is also supported in cohomological degrees $\le d-j$, and the associated spectral sequence of the filtered complex/spectrum lies in the left half-plane $\{x\le d\}$. Conversely, if the filtration is already known to be complete then boundedness can be checked via the graded pieces: taking the inverse limit, $\gr^jM$ being supported in cohomological degrees $\le d-j$ for all $j\in\bb Z$ implies the same about all $\Fil^jM$.

Assume now that $\cal C$ is presentably symmetric monoidal (our main cases of interest are when $\cal C$ is the $\infty$-category $\Spt$ of spectra or the derived $\infty$-category $\text D(R)$ of modules over some discrete ring $R$). Then $\Fil\cal C$ and $\Gr\cal C$ admit canonical symmetric monoidal structures given by Day convolution, which ensures that taking associated graded promotes to a  strong symmetric monoidal functor. In particular, we have the $\infty$-category of \emph{filtered $\bb E_{\infty}$-algebras} $\CAlg(\Fil\cal C)$ and \emph{graded $\bb E_{\infty}$-algebras} $\CAlg(\Gr\cal C)$ such that $\mathrm{gr}^\star$ promotes to a strong symmetric monoidal functor
$
\mathrm{gr}^\star:\CAlg(\Fil\cal C) \rightarrow  \CAlg(\Gr\cal C).
$
We sometimes summarise this wealth of information by speaking simply of {\em multiplicative filtrations} or {\em multiplicative graded objects}. We also often consider maps between such structured objects and call them maps which are \emph{multiplicative}.

Given an $\bb E_\infty$-algebra $E(\star)\in\Gr\rm D(R)$ in graded complexes, where $R$ is a discrete ring, we may shift the graded pieces to obtain new $\bb E_\infty$-algebras $E(\star)[2\star]$ and $E(\star)[-2\star]$ in $\Gr\rm D(R)$ \cite[Rem.~2.1]{BachmannElmantoMorrow}; we refer to this process as {\em shearing} and will sometimes use it without explicit mention.

\subsection{Left Kan extensions}
Given a fully faithful inclusion of categories $\iota: \cal C\subseteq \cal C'$ and a functor $F:\cal C\to\cal D$ valued in a presentable $\infty$-category $\cal D$, we write $L_{\cal C'/\cal C}F:\cal C'\to\cal D$ for the corresponding left Kan extension. We will use the following standard facts: 
\begin{enumerate}
\item Left Kan extension provides a left adjoint to the restriction functor $\iota^*:\Fun( \cal C', \cal D) \rightarrow \Fun( \cal C, \cal D)$
\[
\iota_!:\Fun( \cal C, \cal D) \rightarrow \Fun( \cal C, \cal D); 
\]
which is furthermore fully faithful  \cite[Prop. 4.3.2.17]{Lurie2009}.
\item Let $R$ be a commutative base ring and consider the special case that $\cal C \subset \cal C'$ is the inclusion $\CAlg^{\Sm}_R \subset \CAlg_R$ of smooth $R$-algebras into all commutative $R$-algebras; we will denote left Kan extension in this context by $L_R^\sub{sm}= L_{\CAlg_R^\sub{sm}/\CAlg_R}$. Given a functor $F:\CAlg^\sub{sm}_R \rightarrow \CAlg(\cal D)$, the left Kan extension $L_R^\sub{sm}F:\CAlg_R\to\cal D$ upgrades to a functor $L_R^\sub{sm}F:\CAlg_R\to\cal \CAlg(\cal D)$. The key point here is that, for an $R$-algebra $S$, the diagram $(\CAlg^\sub{sm}_{R})_{/S}$ is \emph{sifted}, whence the colimit computing the left Kan extension in $\CAlg(\cal D)$ is computed in $\cal D$; see \cite[Corol.~3.2.3.2]{LurieHA} for a reference. This also works for the inclusion $\CAlg^{\Sm}_R \subset \CAlg^\sub{ani}_R$ of smooth $R$-algebra into animated $R$-algebras. Therefore, without further comment, we take for granted that the left Kan extensions appearing in this paper preserves multiplicative structures.
\end{enumerate}

\section{Recollections of other cohomologies}\label{sec:recall}

\subsection{Classical motivic cohomology of smooth schemes}\label{ss_classical}
We briefly recall the $\bb A^1$-invariant theory of motivic cohomology for smooth schemes over a field $k$. We will sometimes call this the {\em classical} case of motivic cohomology for the sake of clarity.

For $j\ge0$ and for any given $X\in \text{Sm}_k$, its weight-$j$ motivic cohomology is in principle given by a shift of Bloch's cycle complex, namely \begin{equation}z^j(X,\bullet)[-2j],\label{eqn_mot=BL}\end{equation} but this is not a good working definition. Indeed, Bloch's cycle complex is a priori only functorial for flat morphisms between smooth $k$-schemes, which is insufficient for our purposes (notably for left Kan extending beyond smooth $k$-schemes), and its multiplicative properties are unclear (especially in mixed characteristic, although that is irrelevant for the present paper). Fortunately it is known how to resolve these problems via motivic homotopy theory as follows; we refer to our joint work with Bachmann \cite[\S4]{BachmannElmantoMorrow} and to Remark \ref{rem:def_of_ZjA} for more details.

Firstly, for any smooth $k$-scheme $X$, Voevodsky's slice filtration \cite{Voevodsky2002} on the category of motivic spectra $\SH(k)$ equips the $K$-theory spectrum $\K(X)$ with a natural, multiplicative, complete $\bb N$-indexed filtration $\Fil^\star_\bb A\K(X)$, which we call the classical motivic filtration. One then defines the weight-$j$ motivic cohomology %\footnote{In the terminology of \cite[Definition~4.15]{BachmannElmantoMorrow}, this is the restriction of $\bb A^1$-motivic cohomology to smooth $k$-schemes.}
 of $X$ by shifting the graded pieces of this motivic filtration, namely \begin{equation}\bb Z(j)^\bb A(X):=(\gr^j_\bb A\K(X))[-2j]\in\mathrm{D}(\bb Z),\qquad j\in\bb Z.\label{eqn:A1vsBloch}\end{equation} Results of Levine \cite{Levine2008} and Voevodsky \cite{Voevodsky2002} state that, for any fixed $X\in\Sm_k$ and $j\ge0$, there is an equivalence $\bb Z(j)^\bb A(X)\simeq z^j(X,\bullet)[-2j]$; also $\bb Z(j)^\bb A(X)=0$ when $j<0$. However, the functoriality issues of Bloch's cycle complex are now fixed: formal properties of the slice filtration show that this classical motivic cohomology (\ref{eqn:A1vsBloch}) assembles to form an $\bb A^1$-invariant Nisnevich sheaf on smooth $k$-schemes valued in $\bb E_\infty$-algebras in graded complexes, i.e., $\bb Z(\star)^\bb A\in \Shv_{\Nis,\bb A^1}(\Sm_k,\CAlg(\Gr\Sp))$.

The classical motivic filtration induces the multiplicative {\em Atiyah--Hirzebruch}, or {\em motivic}, {\em spectral sequence} \[E_2^{ij}=H^{i-j}_\bb A(X,\bb Z(-j))\implies \K_{-i-j}(X)\] functorially in $X\in\text{Sm}_k$, where $H^i_\bb A(X,\bb Z(j)):=H^i(\bb Z(j)^\bb A(X))$ denotes the classical motivic cohomology groups of $X$. The motivic filtration on $\K(X)$ is bounded in that $\Fil^j_{\bb A}\K(X)$ is supported in homological degrees $\ge \dim X-j$ (in particular, $H^i_\bb A(X,\bb Z(j))=0$ when $i> j+\dim X$). Adams operators imply that the motivic filtration splits rationally, i.e., there is a natural equivalence of filtered spectra $\Fil^\star_{\bb A}\K(X)_\bb Q\quis \bigoplus_{j\ge0}\bb Q(j)^{\bb A}(X)[2j]$, and that the Atiyah--Hirzebruch spectral degenerates rationally; see \cite[\S4.3]{BachmannElmantoMorrow} for details.

As part of his resolution of the Bloch--Kato conjecture, Voevodsky related classical motivic cohomology to \'etale cohomology and established the Beilinson--Lichtenbaum equivalence: namely, for any integer $\ell$ invertible in $k$, there are multiplicative equivalences \[\bb Z(j)^\bb A/\ell\simeq L_\sub{Zar}\tau^{\le j}R\Gamma_\sub{\'et}(-,\mu_\ell^{\otimes j}),\] for $j\ge0$, of sheaves on $\Sm_k$. We refer to \cite[\S6.1]{BachmannElmantoMorrow} for an explanation how to deduce a highly structured such equivalence from the classical results of Geisser, Levine, Suslin, and Voevodsky.

\subsection{Lisse motivic cohomology}\label{subsec_lke}
The simplest way to extend classical motivic cohomology to arbitrary algebras over fields is via left Kan extension of the classical theory:

\begin{definition}\label{def:conn}
Fix a prime field $\bb F$ (i.e., $\bb F_p$ for some prime number $p\ge2$ or $\bb Q$) and $j\in\bb Z$. We define \emph{weight-$j$, lisse motivic cohomology}
\[
\Z(j)^\sub{lse}:=L_\bb F^\sub{sm}\Z(j)^\bb A: \text{CAlg}_\bb F \rightarrow \rm D(\Z)
\]
as the left Kan extension of classical motivic cohomology $\Z(j)^\bb A$ of smooth $\bb F$-algebras along the inclusion $\rm{CAlg}_\bb F^\sub{sm} \subseteq \rm{CAlg}_\bb F$ of smooth $\bb F$-algebras into all $\bb F$-algebras.\footnote{\label{footnote_warning_lisse}Warning: In \cite[Definition~7.1]{BachmannElmantoMorrow} lisse motivic cohomology was defined by left Kan extending $\bb A^1$-motivic cohomology from any base ring $B$ which was either a field or a mixed characteristic Dedekind domain, and was denoted by $\bb Z(j)_B^\sub{lse}$ to emphasise the a priori dependence on $B$; the notation $\bb Z(j)^\sub{lse}$ was used for the most generally defined theory, namely $B=\bb Z$. Thus \[\bb Z(j)^\sub{lse}\textrm{ of this paper }\stackrel{\textrm{by def.}}=\textrm{ }\bb Z(j)_\bb F^\sub{lse}\textrm{ of \cite{BachmannElmantoMorrow} }\stackrel{?}\neq \textrm{ the restriction of }\bb Z(j)^\sub{lse}\textrm{ of \cite{BachmannElmantoMorrow} to $\bb F$-algebras.}\] While this might look like a potential source of terrible confusion, we assure the reader that it is not: we only use lisse motivic cohomology of {\em local} $\bb F$-algebras $A$ (a handful of results are stated in non-local cases, but never used in such generality), and in that case it is known that \[\bb Z(j)^\sub{lse}(A)\textrm{ of this paper }=\textrm{ }\bb Z(j)^\sub{lse}(A)\textrm{ of \cite{BachmannElmantoMorrow}.}\] This was already proved for henselian local $\bb F$-algebras in \cite[Theorem~7.5]{BachmannElmantoMorrow}. To remove the henselian assumption we must check that the statement of \cite[Proposition~7.4]{BachmannElmantoMorrow} remains true if we replace Nis by Zar. But the Nisnevich topology was only used in the final sentence of the proof where, on smooth $\bb F$-schemes, the mod-$p$ Beilinson--Lichtenbaum cohomology $\bb F_p(j)^\sub{BL}$ was tautologically identified Nisnevich locally with the truncated syntomic cohomology $\tau^{\le j}\bb F_p(j)^\sub{syn}$; the key observation is now that this identification actually holds Zariski locally by \cite[Remark~5.53]{BachmannElmantoMorrow}, since here we are working with smooth schemes over the field $\bb F$ (rather than possibly over a mixed characteristic Dedekind domain).
}
\end{definition}

We warn the reader that $\Z(j)^\sub{lse}$ is not in general a Zariski sheaf, already in the case $j=1$, as the following example shows.

\begin{example}[$j=1$]\label{example_lke_1}
On smooth $\bb F$-schemes there is a first Chern class equivalence \begin{equation}c_1^\bb A: R\Gamma_\sub{Zar}(-,\bb G_m)[-1]\quis \bb Z(1)^\bb A,\label{eqn_1st_Chern_A}\end{equation} e.g., \cite[Corollary~6.26]{BachmannElmantoMorrow}, where the left side is the same as $(\tau^{\le1}R\Gamma_\sub{Zar}(-,\bb G_m))[-1]$ on smooth $\bb F$-schemes. Since both units and $\text{Pic}$, as functors $\CAlg_\bb F\to\text D(\bb Z)$, are left Kan extended from smooth $\bb F$-algebras, we deduce that there is a natural equivalence $\bb Z(1)^\sub{lse}(A)\simeq(\tau^{\le1}R\Gamma_\sub{Zar}(A,\bb G_m))[-1]$ for any $\bb F$-algebra $A$. However, the truncated presheaf itself is not a Zariski sheaf and the natural map $\tau^{\le1}R\Gamma_\sub{Zar}(-,\bb G_m) \rightarrow R\Gamma_\sub{Zar}(-,\bb G_m)$ witnesses the target as the Zariski sheafification.
\end{example}

Lisse motivic cohomology occurs as the graded pieces of a motivic filtration on connective $K$-theory $\K^\sub{cn}$:

\begin{proposition}\label{prop:mot-filt} Let $A$ be an $\bb F$-algebra. Then there exists a natural, $\bb N$-indexed, multiplicative filtration $\Fil^\star_\sub{lse}\K^\sub{cn}(A)$ on the connective $K$-theory $\K^\sub{cn}(A)$ with graded pieces given naturally and multplicatively by 
\[
\mathrm{gr}^j_\sub{lse}\K^\sub{cn}(A) \simeq  \Z(j)^\sub{lse}(A)[2j]
\]
for $j\ge0$; moreover $\bb Z(j)^\sub{lse}(A)$ is supported in cohomological degrees $\le 2j$. If $A$ is local then $\bb Z(j)^\sub{lse}(A)$ is supported in cohomological degrees $\le j$ and the filtration is bounded.
\end{proposition}
\begin{proof}
The desired filtration follows by left Kan extending the motivic filtration of \S\ref{ss_classical}, since $\K^\sub{cn}:\text{CAlg}_\bb F\to\text{Sp}$ is left Kan extended from smooth $\bb F$-algebras \cite[Example~A.0.6]{EHKSY3}. The bound holds in the smooth case and is preserved by left Kan extension.

If $A$ is local then $\Fil_\sub{lse}^j\K^\sub{cn}(A)$ is supported in homological degrees $\ge j$; indeed, this connectivity bound holds Zariski locally on smooth $\bb F$-algebras by the Gersten conjecture in classical motivic cohomology, and is again preserved by left Kan extension.
\end{proof}

\begin{remark}[$\bb Z(j)^\sub{lse}$ is a cycle complex]\label{remark_lke_as_cycles}
Our interest in the cohomology $\bb Z(j)^\sub{lse}$ is not just as an intermediate tool nor because it is the ``easiest'' extension of motivic cohomology beyond smooth schemes, but because it is defined purely in terms of algebraic cycles. This is already clear from the definition, since it is the left Kan extension of the cycle-theoretic $\bb Z(j)^{\bb A}$ from smooth algebras, but we spell it out more explicitly. Given a $\bb F$-algebra $A$, we may pick a simplicial resolution $P_\bullet\to A$ where each term $P_m$ is an ind-smooth $\bb F$-algebra and each face map $P_{m+1}\to P_m$ is a henselian surjection. Then the formalism of left Kan extension from smooth algebras implies that there is a natural equivalence \[\colim_{m\in\Delta^\sub{op}}\bb Z(j)^\bb A(P_m)\quis \bb Z(j)^\sub{lse}(A)\] (in this line and below we implicitly extend $\bb Z(j)^\bb A$ and $z^j(-,\bullet)$ from smooth to ind-smooth $\bb F$-algebras, by taking filtered colimits). Expanding each $\bb Z(j)^\bb A(P_m)=z^j(P_m,\bullet)[-2j]$ as a complex of cycles, we see that the left side of the previous line is the $[-2j]$-shift of the totalisation of the bicomplex (really bisimplicial abelian group)
\begin{equation}\label{eq:lke-formula}
\xymatrix@=5mm{
&\vdots\ar[d] &\vdots\ar[d] &\vdots\ar[d] \\
\cdots\ar[r] & z^j(P_2,2) \ar[r]\ar[d] & z^j(P_2,1) \ar[r] \ar[d]& z^j(P_2,0) \ar[d]\\
\cdots\ar[r] & z^j(P_1,2) \ar[r]\ar[d] & z^j(P_1,1) \ar[r] \ar[d]& z^j(P_1,0) \ar[d]\\
\cdots\ar[r] & z^j(P_0,2) \ar[r] & z^j(P_0,1) \ar[r] & z^j(P_0,0)
}
\end{equation}
In conclusion $\bb Z(j)^\sub{lse}(A)$ admits a description in terms of various algebraic cycles on the affine schemes $\bb A_{P_m}^n$, for $n,m\ge0$
\end{remark}

We will discuss comparisons between lisse motivic cohomology and our new motivic cohomology in \S\ref{section_lke}, mainly in the case of local rings.

\begin{remark}[Variant: Zariski sheafified lisse motivic cohomology]\label{rem_Zariski_lisse}
As already mentioned, we only really need lisse motivic cohomology in the case of local rings. A more useful variant for arbitrary qcqs $\bb F$-schemes is therefore obtained by taking its Zariski sheafification $L_\sub{Zar}\bb Z(j)^\sub{lse}:\Sch^\sub{qcqs,op}_\bb F\to\rm{D}(\bb Z)$.

Zariski sheafifying Proposition \ref{prop:mot-filt} we see, for any qcqs $\bb F$-scheme $X$, that the Zariski sheafification of connective $K$-theory, i.e., $L_\sub{Zar}\K^\sub{cn}(X)$, admits a natural, $\bb N$-indexed, multiplicative filtration $L_\sub{Zar}\Fil^\star_\sub{lse}\K^\sub{cn}(X)$, with graded pieces given naturally and multiplicatively by $L_\sub{Zar}\bb Z(\star)^\sub{lse}(X)[2\star]$.

This Zariski sheafified lisse motivic cohomology will be used in \S\ref{sec:p1}.
\end{remark}

\subsection{$\bb A^1$-invariant and cdh-local motivic cohomology}\label{ss_cdh_local}
This paper builds on our joint work with Bachmann \cite{BachmannElmantoMorrow}, one of whose goals is to extend the $\bb A^1$-invariant theory of classical motivic cohomology from smooth varieties to arbitrary qcqs schemes, while retaining $\bb A^1$-invariance.

We summarise here the main output of \cite{BachmannElmantoMorrow} in the case of schemes over a fixed prime field $\bb F$:

\begin{theorem}[{\cite{BachmannElmantoMorrow}}]\label{thm:cdh}
There exists a multiplicative family of presheaves of complexes \[\bb Z(j)^\bb A:\Sch_\bb F^\sub{qcqs,op}\To\mathrm{D}(\bb Z),\qquad j\ge0\] (i.e., $\bb Z(\star)^\bb A$ is a presheaf valued in $\bb E_\infty$-algebras in graded complexes) with the following properties:
\begin{enumerate}
\item For any qcqs $\bb F$-scheme $X$, there exists a natural, multiplicative, $\bb N$-indexed filtration, the {\em $\bb A^1$-motivic filtration}, $\mathrm{Fil}_\bb A^{\star}\KH(X)$ on the homotopy invariant $K$-theory $\KH(X)$, such that the graded pieces are naturally and multiplicatively given by
$
\mathrm{gr}_\bb A^j\KH(X)\simeq \Z(j)^\bb A[2j],
$
for $j\ge0$. In particular, writing $H^i_\bb A(X,\bb Z(j)):=H^i(\bb Z(j)^\bb A(X))$ for the corresponding {\em $\bb A^1$-motivic cohomology groups}, there exists a natural, multiplicative Atiyah--Hirzebruch spectral sequence
\[
E_2^{ij}=H_\bb A^{i-j}(X, \Z(-j)) \implies \KH_{-i-j}(X).
\]
(Tautologically, or by convention, we have $\bb Z(j)^\bb A=0$ when $j<0$.) If $X$ has finite valuative dimension $\le d$ then this filtration is bounded: more precisely, $\Fil_\bb A^j\KH(X)$ is supported in cohomological degrees $\le d-j$.
\item Finitariness: $\bb Z(j)^\bb A$ is a finitary cdh sheaf for each $j\in\bb Z$.
\item Low weights: There is a unique multiplicative equivalence $\bb Z(0)^\bb A\simeq R\Gamma_\sub{cdh}(-,\bb Z)$ and a first Chern class equivalence $c_1^{\bb A}:R\Gamma_\sub{cdh}(-,\bb G_m)[-1]\quis\bb Z(1)^\bb A$.
\item Beilinson--Lichtenbaum relation to \'etale cohomology: for any integer $\ell>0$ invertible in $\bb F$, there is a unique equivalence 
\[
\Z(\star)^\bb A/\ell \simeq L_{\cdh}\tau^{\leq \star}R\Gamma_{\et}(-,\mu_{\ell}^{\otimes \star}).
\]
of $\bb E_\infty$-algebras in graded presheaves of complexes on $\Sch_\bb F^\sub{qcqs}$ which is compatible with first Chern classes.
\item Relation to syntomic cohomology: if $\bb F=\bb F_p$ then for any $r\ge0$ there is a unique equivalence 
\[
\Z(\star)^\bb A/p^r \simeq R\Gamma_{\cdh}(-,W_r\Omega^\star_{\log})[-\star].
\]
of $\bb E_\infty$-algebras in graded presheaves of complexes on $\Sch_\bb F^\sub{qcqs}$ which is compatible with first Chern classes.
\item\label{item-A1-invar} $\bb A^1$-invariance: for any qcqs $\bb F$-scheme $X$ and $j\ge0$, the map
\[
\Z(j)^\bb A(X) \To \Z(j)^\bb A(X\times \A^1),
\]
induced by the projection $X \times \A^1 \rightarrow X$, is an equivalence.

\item \label{item-PBF} Projective bundle formula: for any $r,j\ge0$, the powers of the first Chern class of the tautological line bundle $c_1^\bb A(\scr O(1)) \in H^{2}_\bb A(\P^r_X, \Z(1))$
induce a natural equivalence
\begin{equation}\bigoplus_{i=0}^r\bb Z(j-i)^\bb A(X)[-2i]\quis \bb Z(j)^\bb A(\bb P^r_X)\label{eqn_intro_PBF}\end{equation}

\item Rational structure via Adams operators: for any qcqs $\bb F$-scheme $X$ and integer $m\in\bb Z\setminus\{0\}$ there exists a natural, multiplicative, filtered endomorphism $\psi^m$ of $\Fil^\star_\bb A\KH(X)[\tfrac1m]$ whose induced action on the associated graded is multiplication-by-$m^\star$; this causes the filtration $\mathrm{Fil}_\bb A^{\star}\KH(X)_{\bb Q}$ to naturally split, i.e., there is a natural, multiplicative equivalence of filtered spectra
\[
\Fil^\star_\bb A\KH(X)_{\bb Q} \simeq \bigoplus_{j \geq \star} \bb Q(j)^\bb A(X)[2j].
\]
\item Comparison to classical motivic cohomology: for any field $k\supseteq\bb F$ the restriction of $\bb Z(\star)^\bb A$ to smooth $k$-schemes coincides with the classical motivic cohomology of \S\ref{ss_classical}; in particular, for any smooth $k$-scheme $X$ there are equivalences
\[
\Z(j)^\bb A(X) \simeq z^j(X,\bullet)[-2j]
\]
for each $j\ge0$.
\item Cdh-locally left Kan extended from smooth schemes: for each $j\ge0$, the presheaves $\bb Z(j)^\bb A$ and $\Fil^j_\bb A\KH$ on $\Sch^\sub{qcqs}_\bb F$ agree with the cdh sheafifications of their left Kan extensions from smooth $\bb F$-schemes. 
\end{enumerate}
\end{theorem}

To discuss the theorem further we must recall the cdh topology, which appears several times in the statement. An \emph{abstract blowup square} is a cartesian square of qcqs $\bb F$-schemes
\begin{equation}\label{eqn_cdh_square}
\begin{tikzcd}
Y' \ar{r} \ar{d} & X' \ar{d}{p} \\
Y \ar[swap]{r}{i} & X,
\end{tikzcd}
\end{equation}
where $i$ is a finitely presented closed immersion and $p$ is a finitely presented, proper morphism inducing an isomorphism $X'\setminus Y'\isoto X\setminus Y$. On $\Sch^{\qcqs}_\bb F$ the \emph{cdh topology} is the Grothendieck topology generated by the pretopology given by maps $\{ Y \rightarrow X, X' \rightarrow X\}$ for all abstract blow up squares as above and by the Nisnevich pretopology. A result of Voevodsky \cite{Voevodsky2010}, generalised in \cite[Proposition~2.1.5]{ElmantoHoyoisIwasaKelly2021} and \cite[Appendix~A]{BachmannHoyois2021} in the non-Noetherian setting, states that cdh sheaves are exactly those presheaves which convert both Nisnevich squares and abstract blowup squares to cartesian squares. See \cite[\S2.2]{BachmannElmantoMorrow} for further details.

Using the cdh topology we may give an explicit description of the $\bb A^1$-motivic cohomology of the theorem in terms of classical motivic cohomology, as stated in Theorem \ref{thm:main}(10). We first extend the lisse motivic cohomologies $\bb Z(j)^\sub{lse}$ from $\bb F$-algebras to qcqs $\bb F$-schemes which are not necessarily affine: this is done simply by defining $\bb Z(j)^\sub{lse}:\Sch_\bb F^\sub{qcqs,op}\to\rm D(\bb Z)$ to be the left Kan extension of classical motivic cohomology $\bb Z(j)^\bb A$ of smooth $\bb F$-schemes; on affines this agrees with Definition \ref{def:conn} by \cite[Proposition~2.23]{BachmannElmantoMorrow}, which proves that left Kan extending from smooth schemes and then evaluating on an affine agrees with the left Kan extension from smooth affines. Secondly, as in \cite[\S7]{BachmannElmantoMorrow}, we define weight-$j$ {\em cdh-motivic cohomology} of qcqs $\bb F$-schemes by cdh-sheafifying $\bb Z(j)^\sub{lse}$, that is \begin{equation}\mathbb Z(j)^\sub{cdh}:=L_\sub{cdh}\bb Z(j)^\sub{lse}:\opp{Sch}_\bb F^\sub{qcqs, op}\To\rm D(\bb Z)\label{eqn:Z(j)cdh}\end{equation} for $j\in\bb Z$ (though note that $\mathbb Z(j)^\sub{cdh}=0$ if $j<0$, since the same holds for classical motivic cohomology of smooth $k$-schemes). In terms of universal properties, $\bb Z(j)^\sub{cdh}:\opp{Sch}_\bb F^\sub{qcqs, op}\to\rm D(\bb Z)$ is initial among cdh sheaves on $\opp{Sch}_\bb F^\sub{qcqs, op}$ whose restriction to smooth $\bb F$-schemes is equipped with a map from classical motivic cohomology $\bb Z(j)^\bb A$. Assuming resolution of singularities, this construction is essentially due to Friedlander and Voevodsky for finite type $\bb F$-schemes \cite{FriedlanderVoevodsky2000} \cite[Remarks 1.12 \& 7.9]{BachmannElmantoMorrow}.

In light of Theorem \ref{thm:cdh}(2), there is an induced counit comparison map \begin{equation}\bb Z(\star)^\sub{cdh}\To \bb Z(\star)^\bb A\label{eqn_cdh_vs_A1}\end{equation} of presheaves of $\bb E_\infty$-algebras in complexes on qcqs $\bb F$-schemes, and with Bachmannn we prove that \eqref{eqn_cdh_vs_A1} is an equivalence.\footnote{In fact, this follows from parts of Theorem \ref{thm:cdh}: rationally from part (8) and with finite coefficients from parts (4) and (5). But this totally misrepresents the order of logic in \cite{BachmannElmantoMorrow}.} Similarly, each step of the $\bb A^1$-motivic filtration $\Fil^j_\bb A\KH$ can be recovered by restricting it to smooth $\bb F$-schemes, left Kan extending to qcqs $\bb F$-schemes, and then cdh sheafifying. We will summarise these results by simply saying that each $\bb Z(j)^\bb A$ and $\Fil^j_{\bb A}\KH$ are {\em cdh-locally left Kan extended} from smooth $\bb F$-schemes. These cdh-local descriptions of $\bb A^1$-motivic cohomology and the $\bb A^1$-motivic filtration will be absolutely crucial to the construction of non-$\bb A^1$-invariant motivic cohomology in this paper.

\begin{remark}[Definition of $\bb Z(j)^\bb A$]\label{rem:def_of_ZjA}
For the purposes of the present article, and in light of the previous paragraph, the reader is welcome to define $\bb Z(\star)^\bb A$ to be $\bb Z(\star)^\sub{cdh}$, or rather replace every occurrence of the former by the latter, so that the equivalence \eqref{eqn_cdh_vs_A1} becomes a tautology. All that will be lost is the relation to motivic homotopy theory. That is even the point of view we adopted in the first version of this article. But, upon further reflection, we feel that the fact that the equivalent cohomologies $\bb Z(\star)^\bb A$ and $\bb Z(\star)^\sub{cdh}$ appear to be the correct theory of $\bb A^1$-invariant motivic cohomology should be reflected in the notation, and that $\bb Z(\star)^\sub{cdh}$ should rather be viewed (at least a priori) as an auxiliary invariant.

Nevertheless, here is the true definition of $\bb A^1$-motivic cohomology via motivic homotopy theory. For any qcqs scheme $X$ let $\SH(X)$ denote its $\infty$-category of motivic spectra, as introduced by Morel and Voevodsky \cite{Voevodsky1998}; for a modern approach see \cite{Robalo2015} \cite[\S4]{BachmannHoyois2021}. Examples of such motivic spectra include the unit object (or motivic sphere) $1_X$, the motivic spectrum $\KGL_X$ representing homotopy invariant $K$-theory of smooth $X$-schemes, and a motivic Eilenberg--Maclane spectrum $\H \bb Z_X$ constructed by Spitzweck \cite{Spitzweck2018}. Any motivic spectrum $E\in \SH(X)$ may be equipped with a functorial {\em slice filtration} \[\cdots \to \mathrm \Fil_\sub{slice}^{j+1}E\to \mathrm \Fil_\sub{slice}^{j}E\to\mathrm \Fil_\sub{slice}^{j-1}E\to \cdots\to E\] in $\SH(X)$, whose graded pieces are denoted by $s^j E:=\mathrm{cofib}(\mathrm \Fil_\sub{slice}^{j+1}E\to \mathrm \Fil_\sub{slice}^j E)$. The slice filtration on $\KGL_X$ induces a natural, multiplicative, complete $\bb N$-indexed filtration $\Fil^\star_\bb A\KH(X)$, which we tend to call the {\em $\bb A^1$-motivic filtration}, and we define the {\em $\bb A^1$-motivic cohomology} as the shifts of its graded pieces: \begin{equation}\bb Z(j)^\bb A(X):=(\gr^j_\bb A\KH(X))[-2j]\in\mathrm{D}(\bb Z),\qquad j\in\bb Z\label{eqn:A1_mot_coh}\end{equation} In the case of a smooth algebraic variety $X$ over a field $k$, this recovers the classical motivic cohomology of \S\ref{ss_classical} because the pullback functor $\SH(k)\to\SH(X)$ commutes with the slice filtration \cite[Remark 3.41]{BachmannElmantoMorrow}

In \cite[\S9]{BachmannElmantoMorrow} (see also \cite[Remark~4.14]{BachmannElmantoMorrow}), also using some previous work of Bachmann \cite{Bachmann2022}, we establish natural equivalences in $\SH(X)$ \begin{equation}\H \bb Z_X\simeq s^0(1_X)\simeq s^0(\KGL_X),\label{intro_V_conj}\end{equation} thereby settling Conjectures 1, 7, and 10 of Voevodsky \cite{Voevodsky2002a} for arbitrary qcqs schemes.  These equivalences, combined with Theorem \ref{eqn:A1_mot_coh}, offer substantial justification that $\bb Z(\star)^\bb A$ is indeed the  ``right'' theory of $\bb A^1$-invariant motivic cohomology.%, are in fact equivalent. We call this {\em $\bb A^1$-motivic cohomology} and denote it by $\bb Z(j)^{\bb A}$, for $j\in \bb Z$; in the case of smooth varieties it already appeared in \S\ref{ss_classical} as a more coherent and multiplicative version of Bloch's cycle complex.
\end{remark}

\begin{remark}[Leitfaden to proof of Theorem \ref{thm:cdh}]
Particularly since we are restricting attention to qcqs schemes of equal characteristic, Theorem~\ref{thm:cdh} does not require the full force of the results of \cite{BachmannElmantoMorrow}. Here we sketch the necessary arguments. Firstly, and in line with the opening paragraph of Theorem \ref{rem:def_of_ZjA}, we explain how to prove parts (1)--(8) of the theorem after replacing $\bb Z(j)^\bb A$ by $\bb Z(j)^\sub{cdh}$ throughout; this accurately reflects the order of logic in \cite{BachmannElmantoMorrow}.

Part (1) is obtained by cdh sheafifying Propsition \ref{prop:mot-filt}; see \cite[Proposition~7.12(1)]{BachmannElmantoMorrow} for the details. Part (2) is proved by checking that left Kan extension yields a finitary presheaf $\bb Z(j)^\sub{lse}$, and that this property is preserved by cdh sheafification; see \cite[Proposition~2.15]{BachmannElmantoMorrow}.

Parts (4) and (5) correspond to \cite[Theorem~7.16]{BachmannElmantoMorrow} (and the subsequent corollaries), which are obtained by cdh sheafifying the left Kan extension of the highly coherent Beilinson--Lichtenbaum equivalence \cite[Theorem~6.14]{BachmannElmantoMorrow}. The interested reader will also easily check that the proof of the latter theorem simplifies over fields.

For the proof of part (3) we refer to \cite[Proposition~7.12(5)\&(6)]{BachmannElmantoMorrow}, which refers to \cite[Corollary~6.26]{BachmannElmantoMorrow}; over fields the references in the latter proof to \cite[Theorem~6.21(3)]{BachmannElmantoMorrow} can be replaced by \cite[Theorem~6.14]{BachmannElmantoMorrow}, i.e., we do not need \cite[\S6.2--6.3]{BachmannElmantoMorrow}.

Parts (6) and (7) are the deepest parts of the theorem; they are special cases of \cite[Corollaries~8.16 \& 8.23]{BachmannElmantoMorrow}.

The Adams operators in part (8) are obtained by restricting Bachmann--Hopkins' Adams operators \eqref{eqn:BH_adams} to smooth $\bb F$-schemes, left Kan extending, then cdh sheafifying. The resulting rational Adams decomposition follows by the same argument as in Theorem \ref{thm:adams} below.

Having proved parts (1)--(8) for the cdh-motivic cohomology, some yoga of pullbacks in motivic homotopy theory produces the equivalence \eqref{eqn_cdh_vs_A1}: more precisely, parts (6) and (8) show that \cite[Theorem 9.13]{BachmannElmantoMorrow} holds, whence \cite[Theorem 9.1]{BachmannElmantoMorrow} holds, and then the argument of \cite[Corollary 9.4]{BachmannElmantoMorrow} establishes the equivalence \eqref{eqn_cdh_vs_A1}.

For part (9), it was already explained in Remark \ref{rem:def_of_ZjA} that $\bb Z(j)^\bb A$ recovers classical motivic cohomology on smooth varieties (almost tautologically: one must just note that the slice filtration behaves well under smooth maps).
\end{remark}

\begin{remark}[Warning in mixed characteristic]
We defined $\bb A^1$-motivic cohomology in Remark \ref{rem:def_of_ZjA} for arbitrary qcqs schemes. Cdh-locally left Kan extending it from smooth $\bb Z$-schemes, one also obtains an extension of the cdh-motivic cohomology $\bb Z(\star)^\sub{cdh}$ from equicharacteristic qcqs schemes to arbitrary qcqs schemes \cite[\S7]{BachmannElmantoMorrow}. The equivalence $\bb Z(\star)^\sub{cdh}\simeq \bb Z(\star)^\bb A$ of \eqref{eqn_cdh_vs_A1} for equicharacteristic schemes is unfortunately at present unknown in mixed characteristic: more precisely, for $j\ge2$, it is unknown whether $\bb Z(j)^\bb A$, as a presheaf on all qcqs schemes, is cdh-locally left Kan extended from smooth $\bb Z$-schemes. This is only known to hold conditionally on the so-called key hypothesis in \cite{BachmannElmantoMorrow} about the syntomic cohomology of mixed characteristic valuation rings; see \cite[\S9]{BachmannElmantoMorrow} for the conditional results.

In Bouis' extension of this paper to mixed characteristic \cite{Bouis2025, Bouis2025a, Bouis2025b}, the $\bb A^1$-motivic cohomology and cdh-motivic cohomology must therefore be kept separate in the constructions; it then turns out to be the latter which plays the more prominent role, because its compatibility with the trace map can be established.
\end{remark}

\subsection{(Topological) cyclic homology} The constructions of this paper depend on trace methods in algebraic $K$-theory. Recall that if $E$ is a spectrum with $S^1$-action, then we can functorially associate several other spectra: its \emph{homotopy fixed points} $E^{hS^1}$, its \emph{homotopy orbits} $E_{hS^1}$ and its \emph{Tate fixed points} 
\[
E^{tS^1}:= \mathrm{cofib}(\mathrm{Nm}:E_{hS^1}[1] \rightarrow E^{hS^1}).
\]
The same formalism exists if we replace $S^1$ by any of its finite subgroups, for example if $p$ is a prime then one has
\[
E^{tC_p}:= \mathrm{cofib}(\mathrm{Nm}:E_{hC_p} \rightarrow E^{hC_p}).
\]
According to \cite{NikolausScholze2018}, a \emph{($p$-typical) cyclotomic spectrum} is a spectrum with an $S^1$-action $E$ equipped with $S^1$-equivariant maps
\[
\phi_p: E \to E^{tC_p}
\]
for all primes $p$; here $E^{tC_p}$ is given the residual $S^1/C_p \simeq S^1$-action. In the situation of algebraic geometry, we have the functor of topological Hochschild homology landing in the $\infty$-category of cyclotomic spectra:
\[
\THH: \Sch^\sub{qcqs,op} \To \mathrm{CycSp} \qquad X \mapsto \THH(\Perf(X))=:\THH(X).
\]

Let $E$ be a cyclotomic spectrum which is bounded below (which will always be the case in our situations of interest). Firstly, for each prime number $p$ its \emph{$p$-adic topological cyclic homology} is defined to be the $p$-complete spectrum
\[
\TC(E;\Z_p) := \mathrm{fib}\left( \phi_p^{hS^1} - \mathrm{can}: \left(E^\comp_p\right)^{hS^1} \To \left( E^{tC_p} \right)^{hS^1}  \simeq \left( E^\comp_p \right)^{tS^1} \right).
\]
Here we have used that $E^{tC_p}$ is $p$-complete and that $E^{tC_p} \simeq (E^\comp_p)^{tC_p}$ by \cite[Lemma~I.2.9]{NikolausScholze2018}, and $\mathrm{can}$ is the canonical map from fixed points to the Tate construction. We assemble these $p$-adic constructions to define the \emph{integral topological cyclic homology} $\TC(E)$ of $E$ as the pullback
\begin{equation}\label{eq:integral-tc} 
\begin{tikzcd}
\TC(E) \ar{r} \ar{d} & \left( E_{\Q} \right)^{hS^1}\ar{d}\\
\prod_{p} \TC(E;\Z_p) \ar{r} & \prod_{p} \left( E_p^\comp[\tfrac1p] \right)^{hS^1},
\end{tikzcd}
\end{equation}
where the bottom map is the product over $p$ of the compositions $\TC(E;\bb Z_p)\to \left(E^\comp_p\right)^{hS^1} \to (E^\comp_p[\tfrac{1}{p}]))^{hS^1}$.

By a standard abuse of notation, for a scheme $X$, we write $\TC(X)$ in place of $\TC(\THH(X))$, and similarly for the $p$-adic variant.

\begin{remark} The square~\eqref{eq:integral-tc} imitates the original definition of integral topological cyclic homology defined by Goodwillie \cite[Lemma~6.4.3.2]{Dundas2013}. Indeed, \cite[Theorem~II.4.11]{NikolausScholze2018} proves that the definitions agree for bounded below cyclotomic spectra.
\end{remark}

There is a morphism of localizing invariants (in the sense of \cite{BlumbergGepnerTabuada2013}) called the \emph{cyclotomic trace}, or just {\em trace map} for short
\[
\mathrm{tr}: \K \To \TC.
\]
A major result about this map is the Dundas--Goodwillie--McCarthy theorem \cite{Dundas2013}, stating that its fibre $\K^{\inf}$ is insensitive to nilpotent thickenings. In the language of \cite{LandTamme2019}, $\K^\sub{inf}$ is even \emph{truncating}: for any connective $\mathbb{E}_1$-ring $A$, the map $\K^{\inf}(A) \rightarrow \K^{\inf}(\pi_0A)$ is an equivalence. This property implies not only nil-invariance \cite[Corollary~3.5]{LandTamme2019} but even cdh descent, whence one obtains the following fundamental square:

\begin{theorem}[Kerz--Strunk--Tamme, Land--Tamme]\label{thm:mainsq} Let $X$ be a qcqs scheme. The the square of spectra
\begin{equation}\label{eq:mainsq}
\begin{tikzcd}
\K(X) \ar{r}{\mathrm{tr}} \ar{d} & \TC(X) \ar{d} \\
\KH(X) \ar{r}[swap]{L_\sub{cdh}\mathrm{tr}} & L_{\cdh}\TC(X).
\end{tikzcd}
\end{equation}
is cartesian.
\end{theorem}
\begin{proof} This follows from the facts that the canonical map $L_{\cdh}\K (X)\to \KH(X)$ is an equivalence \cite[Theorem~6.3]{KerzStrunkTamme2018} (see also \cite[Remark~3.4]{KellyMorrow2021}) and that $\K^\sub{inf}$ satisfies $\cdh$ descent \cite{LandTamme2019}.
\end{proof} 

To be clear, the bottom horizontal arrow in the previous diagram is obtained by cdh sheafifying the trace map $\K\to\TC$. Indeed, as we commented in the proof, we have $\KH\simeq L_\sub{cdh}\K$ and therefore there is an induced {\em cdh-local trace map} $L_\sub{cdh}\mathrm{tr}:\KH\to L_\sub{cdh}\TC$; it will play an important role in the construction of our motivic cohomology.

\section{Definition of $\bb Z(j)^\sub{mot}(X)$}\label{s_motivic_def}
In this section we introduce our theory of motivic cohomology and the motivic filtration on algebraic $K$-theory. We also establish a number of other fundamental properties, such as finitariness, to justify that the definition is not unreasonable.

\subsection{Characteristic zero} \label{sec:charzero}
We begin with reminders on cyclic homology. First note that for any $X \in \Sch_{\Q}^\sub{qcqs}$ we have 
\[
\THH(X) \simeq \THH(X) \otimes_{\THH(\Q)} \Q \simeq \HH(X/\Q)
\]
where the second equivalence is formal and the first follows from the fact that $\THH(\Q)\simeq\Q$. Similarly, the integral topological cyclic homology $\TC(X)$, as defined by the pullback square \eqref{eq:integral-tc}, coincides with the {\em negative cyclic homology} of $X$; indeed, the latter is defined by $\HC^-(X/\Q) := \left( \HH(X/\Q) \right)^{hS^1}$ and the square \eqref{eq:integral-tc} collapses to an equivalence \[\TC(X)\quis \HC^-(X/\bb Q)\] (since the bottom terms of the square vanish as $\THH(X)$ has vanishing $p$-completion). The cyclotomic trace becomes the more classical {\em Goodwillie trace} \[\mathrm{tr}:\K(X)\To\HC^-(X/\bb Q),\] and Theorem \ref{thm:mainsq} is rewritten as the cartesian square of spectra
\begin{equation}\label{eq:hc-}
\begin{tikzcd}
\K(X) \ar{d} \ar{r} & \HC^-(X/\Q) \ar{d}\\
\KH(X) \ar{r} & L_{\cdh}\HC^-(X/\Q).
\end{tikzcd}
\end{equation}
We remark that the bottom right term in \eqref{eq:hc-} is poor notation, which we will nevertheless continue to use; it should really be written $(L_\sub{cdh}\HC^-(-/\bb Q))(X)$.

\begin{remark}[Replacing $\bb Q$ by a general case $k$]\label{rem_general_k}
More generally, let $k$ be a discrete commutative ring. For any qcqs $k$-scheme $X$, let $\HC^-(X/k):=\HH(X/k)^{hS^1}$ denote its negative cyclic homology relative to $k$.

Then there is a natural map $\TC(X)\to\HC^-(X/k)$ constructed as follows. Firstly, from square (\ref{eq:integral-tc}) we see that $\TC(X)$ naturally maps to the pull back of
\begin{equation}
\begin{tikzcd}
 \ar[dotted]{r} \ar[dotted]{d} & \left( \THH(X)_{\Q} \right)^{hS^1}\ar{d}\\
\prod_{p} (\THH(X)_p^\comp)^{hS^1} \ar{r} & \prod_{p} \left( \THH(X)_p^\comp[\tfrac1p] \right)^{hS^1}.
\end{tikzcd}
\label{eqn_fixed_points}
\end{equation}
Removing the $hS^1$ from the three corners, the pullback of the square is $\THH(X)$; passing to homotopy fixed points preserves pullback squares (and commutes with products), whence the pullback of the square (\ref{eqn_fixed_points}) is the {\em negative topological cyclic homology} $\TC^-(X):=\THH(X)^{hS^1}$. This constructs a natural map $\TC(X)\to\TC^-(X)$, which may then be composed with $S^1$-fixed points of $\THH(X)\to \HH(X/k)$.

Composing with the cyclotomic trace thereby defines a trace map $\K(X)\to \HC^-(X/k)$ relative to $k$; of course it would have been sufficient to define this in the case $k=\bb Z$ and then compose with the canonical map $\HC^-(X/\bb Z)\to \HC^-(X/k)$.
\end{remark}

To construct our motivic filtration on $K$-theory in characteristic zero we first recall the Hochschild--Kostant--Rosenberg filtration on negative cyclic homology, which relies on the theory of derived de Rham cohomology of lllusie \cite{Illusie1972} and Bhatt \cite{Bhatt2012, Bhatt2012a}. Since the following two results do not require any characteristic zero hypothesis, let $k$ be a commutative ring and recall, for any $k$-algebra $R$, the {\em Hodge-completed derived de Rham cohomology} $\hat{L\Omega}_{R/k}\in \text D(k)$ of $R$ and its complete $\bb N$-indexed {\em Hodge filtration} $\hat{L\Omega}_{R/k}^{\ge\star}$. For $j\ge0$ the cofibre of the map $\hat{L\Omega}_{R/k}^{\ge j}\to \hat{L\Omega}_{R/k}$ is $L\Omega_{R/k}^{<j}$, which admits a finite decreasing filtration with graded pieces (in increasing order) 
\[
R, L_{R/k}[-1],L^2_{R/k}[-2],\dots,L^{j-1}_{R/k}[-j+1]. 
\]
By fpqc descent of $L_{-/k}$ and its wedge powers on $k$-algebras \cite[Theorem~3.1]{BhattMorrowScholze2}, right Kan extension defines a unique fpqc sheaf \[\text{Sch}_k^\sub{qcqs,op}\To\text D(k),\qquad X\mapsto R\Gamma(X,\hat{L\Omega}_{-/k})\] whose value on affines $\Spec(R)$ is $\hat{L\Omega}_{R/k}$; similarly for $\hat{L\Omega}_{-/k}^{\ge j}$, $L\Omega_{-/k}^{<j}$, and each wedge power of $L_{-/k}$. Alternatively, the fpqc sheaf $R\Gamma(-,\hat{L\Omega}_{-/k})$ is equivalent to the Nisnevich sheafification of $X\mapsto \hat{L\Omega}_{\roi_X(X)/k}$, and similarly for the variants.

The following is the HKR filtration on negative cyclic homology:

\begin{theorem}[HKR filtration {\cite{antieau-fil,mrt-hkr,raksit-hkr,ToenVezzosi2011}}]\label{thm:hkr} Let $k$ be a commutative ring. For any qcqs $k$-scheme $X$, there exists a functorial, complete, multiplicative filtration $\mathrm{Fil}_\sub{HKR}^\star\HC^-(X/k)$ on $\HC^-(X/k)$ whose graded pieces are given naturally and multiplicatively
\[
\mathrm{gr}_\sub{HKR}^j\HC^-(X/k) \simeq R\Gamma(X, \widehat{L\Omega}_{-/k}^{\geq j})[2j]
\]
for $j\in\bb Z$. Furthermore, if $X$ is quasisyntomic over $k$\footnote{i.e., for each affine open $\Spec(A)\subseteq X$, the cotangent complex $L_{A/k}\in D(A)$ has Tor amplitude in $[-1,0]$.} then this filtration is exhaustive.
\end{theorem}

\begin{remark}\label{rema:hkr} If $k$ is a $\bb Q$-algebra and $X$ is smooth over $k$, then this result is essentially due to Loday \cite{Loday1989}. Dropping the hypothesis that $X$ be smooth, but remaining in characteristic zero, the product decomposition of the previous theorem is due to Weibel, under the name of the ``Hodge decomposition'' and written in Adams operator type notation as ``$\text{HN}(X/k) \simeq \prod_j\text{HN}^{(j)}(X/k)$'' in \cite{Weibel1997, Cortinas2008a} (but see Remark \ref{rem:discuss-adams} for a warning about indexing conventions).
\end{remark}

\begin{remark}[Variant: cdh-local HKR filtration]\label{rem_cdh_local_HKR}
Cdh sheafifying the HKR filtration levelwise we see that, for any qcqs $k$-scheme $X$, there exists a functorial, multiplicative filtration 
\[
\mathrm{Fil}_\sub{HKR}^{\star}L_\sub{cdh}\HC^-(X/k):= L_{\cdh}\mathrm{Fil}_\sub{HKR}^{\star}\HC^-(-/k)(X),
\] on $L_\sub{cdh}\HC^-(X/k)$ whose graded pieces are given naturally and multiplicatively by \[\mathrm{gr}_\sub{HKR}^jL_\sub{cdh}\HC^-(X/k) \simeq R\Gamma_\sub{cdh}(X, \widehat{L\Omega}_{-/k}^{\geq j})[2j]\] for $j\in\bb Z$. Here we denote by \[\text{Sch}_k^\sub{qcqs,op}\To\text D(k), \qquad X\mapsto R\Gamma_\sub{cdh}(X,\hat{L\Omega}_{-/k})\] the cdh sheafification of the presheaf $R\Gamma(-,\hat{L\Omega}_{-/k})$, or equivalently the cdh sheafification of the presheaf $X\mapsto \hat{L\Omega}_{\roi_X(X)/k}$.

Similar notation will be used for $L\Omega_{-/k}^{<j}$ and each wedge power of $L_{-/k}$.
\end{remark}

The following is probably known to experts but we could not find a standalone reference in the required degree of generality:

\begin{lemma}[Cdh descent of derived de Rham cohomology in characteristic zero]\label{lemma_cdh_descent_HP}
For any $\bb Q$-algebra $k$, the two presheaves
\begin{align*}
\Sch^\sub{qcqs,op}_k&\To{\rm D}(k)\\
X&\mapsto R\Gamma(X,\hat{L\Omega}_{-/k})\\
X&\mapsto\HC^-(X/k)/\Fil^0_\sub{HKR}\HC^-(X/k)
\end{align*}
satisfy cdh descent.
\end{lemma}
\begin{proof}
The cited references for Theorem \ref{thm:hkr} also construct an HKR filtration on periodic cyclic homology: for a qcqs $k$-scheme $X$, this is a functorial, complete, multiplicative filtration $\mathrm{Fil}_\sub{HKR}^\star\HP(X/k)$ on $\HP(X/k)$ whose graded pieces for $j\in\bb Z$ are given by
\[
\mathrm{gr}_\sub{HKR}^j\HP(X/k) \simeq R\Gamma(X, \widehat{L\Omega}_{-/k})[2j].
\]
The references show that the canonical map $\HC^-(X/k)\to \HP(X/k)$ respects the HKR filtrations, i.e., naturally upgrades to a filtered map, given on graded pieces by the canonical maps $R\Gamma(X,\hat{L\Omega}^{\ge j}_{-/k})\to R\Gamma(X,\hat{L\Omega}_{-/k})$.

Since $k$ is a $\bb Q$-algebra, the HKR filtration on $\HP(X/k)$ is naturally split \cite{Bals2023}, i.e., there is a natural equivalence $\HP(X/k)\simeq\prod_{n\in\bb Z}R\Gamma(X, \widehat{L\Omega}_{-/k})[2n]$ such that the HKR filtration $\Fil^j_\sub{HKR}$ on the left matches the product filtration $\prod_{n\le -j}$ on the right.

The presheaf $R\Gamma(-, \widehat{L\Omega}_{-/k}): \Sch^\sub{qcqs,op}_k\to\text D(k)$ is thus a direct summand of the presheaf $\HP(-/k)$; but the latter is a cdh sheaf thanks to the theory of truncating invariants \cite[Corol.~3.13]{Cortinas2008} \cite[Cor.~A.6]{LandTamme2019}, so therefore the former is also a cdh sheaf.

The cited references for Theorem~\ref{thm:hkr} also implicity prove that the canonical map $\HC^{-}(-/k) \rightarrow \HP(-/k)$ induces an equivalence $\HC^-(-/k)/\Fil^0_\sub{HKR}\HC^-(-/k) \xrightarrow{\simeq} \HP(-/k)/\Fil^0_\sub{HKR}\HP(-/k)$. By the aforementioned splitting the latter is equivalent to $\prod_{n\leq -1}R\Gamma(-, \widehat{L\Omega}_{-/k})[2n]$ and is therefore a cdh sheaf since we have shown that $R\Gamma(-, \widehat{L\Omega}_{-/k})$ is a cdh sheaf.
\end{proof}

We next prove the following compatibility, informally stating that for any smooth $k$-scheme $X$ the trace map $\K(X)\to\HC^-(X/k)$ naturally carries the classical motivic filtration on the left to the HKR filtration on the right. In fact, it is rather the cdh-local analogue below (Corollary \ref{corol_cdh_filtered_trace}) which is crucial to our construction, but the smooth case is required for the proof of the cdh case and also to formulate the comparison map to classical motivic cohomology (Construction \ref{cons:cla-vs-new}):

\begin{proposition}\label{prop:compat} Let $k_0\to k$ be a quasismooth\footnote{i.e., the cotangent complex $L_{k/k_0}$ is supported in degree $0$ and $\Omega^1_{k/k_0}$ is a flat $k$-module. In this paper we only require the trivial situation that $\bb Q=k_0=k$, but we record the more general statement for future use.} map of rings, where $k$ is a field. Then the trace map $\mathrm{tr}:\K\to\HC^-(-/k_0)$, viewed as a map between spectra-valued presheaves on $\Sm_k$, admits a unique, multiplicative extension to a map of filtered presheaves of spectra $\Fil_\bb A^{\star}\K\to \Fil_\sub{HKR}^{\star}\HC^-(-/k_0)$ on $\Sm_k$.
\end{proposition}
\begin{proof}
There is a $t$-structure on $\Shv_{\Zar}(\Sm_k,\Spt)$, which denotes the stable $\infty$-category of Zariski sheaves of spectra on smooth $k$-schemes. This $t$-structure is described as follows:
\begin{itemize}
\item its non-negative part $\Shv_{\Zar}(\Sm_k,\Spt)_{\geq 0}$ is given by those sheaves of spectra $\scr F$ such that the homotopy sheaves $\underline{\pi}_n\scr F$ vanish for all $n < 0$;
\item its non-positive part $\Shv_{\Zar}(\Sm_k,\Spt)_{\le 0}$ is given by those sheaves of spectra $\scr F$ such that $\pi_n(\cal F(X))$ vanishes for all $X\in\Sm_k$ and all $n > 0$.
\end{itemize}
This is a specialization of a much more general result on $\infty$-topoi as in \cite[Proposition~1.3.2.7]{LurieSAG}.

Now let $j\in\bb Z$ and observe the following facts about the connectivity of the filtrations on $\K$ and $\HC^-$ with respect to the above $t$-structure:

\begin{enumerate}
\item $\Fil^{\geq j}_\bb A\K$ is $j$-connective. This follows from standard vanishing bounds in motivic cohomology, though a little care is required since a priori taking homotopy sheaves might not commute with viewing $K$-theory as a complete filtered object. Let $X$ be a smooth $k$-scheme and set $d:=\dim X$.

Recall first that, for any $i\ge0$, the cohomology sheaves $\cal H^n(\bb Z(i)^\bb A)=\ul\pi_{-n}\bb Z(i)^\bb A$ on $X_\sub{Zar}$ vanish for $n>i$ (by Gersten injectivity to reduce to the case of a field); therefore, the motivic complex $\bb Z(j)^\bb A(X)$ of $X$ itself vanishes in cohomological degrees $> 2j$ (by the Gersten resolution), and also in cohomological degrees $>j+d$ (for dimension reasons). Using the dimension bound we see, for any affine open $\Spec(R)\subseteq X$ in place of $X$ and $i\ge j+d$, that $\gr^i_\bb A\K(R)=\bb Z(i)^\bb A(R)[2i]$ is supported in cohomological degrees $\le -j$; by completeness of the motivic filtration we deduce the same for $\Fil^{\geq j+d}_\bb AK(R)$, for all open affines $\Spec(R)\subseteq X$. In particular, $\Fil^{\geq j+d}_\bb A\K$ is $j$-connective on $X_\sub{Zar}$.

But now the problem reduces, by a finite induction, to checking that $\gr_\bb A^i\K$ is $i$-connective for each $i\ge0$ (in fact, just for $i=j,\dots,j+d-1$), or in other words that the Zariski cohomology sheaves $\cal H^n(\bb Z(i)^\bb A)$ vanish for $n>i$. But this was already explained in the previous paragraph and so completes the proof.

\item On the other hand, $\Fil^{<j}_\sub{HKR}\HC^-(-/k_0):=\HC^-(-/k_0)/\Fil^{j}_\sub{HKR}\HC^-(-/k_0)$ is $j-1$-truncated for the $t$-structure. Indeed, for any smooth $k$-algebra $R$ and $i\in\bb Z$, the $i^\sub{th}$ graded piece of the HKR filtration on $\HC^-(R/k_0)$ is given by
\[
\mathrm{gr}^i_\sub{HKR}\HC^-(R/k_0) = \hat{L\Omega}_{R/k_0}^{\geq i}[2i]\simeq \Omega_{R/k_0}^{\geq i}[2i]
\]
since the composition $k_0\to k\to R$ is quasismooth; the graded piece therefore vanishes in cohomological degrees $< -i$. By induction it follows, for any $i<j$, that the cofibre of $\Fil_\sub{HKR}^{j}\HC^-(R/k_0)\to \Fil_\sub{HKR}^{i}\HC^-(R/k_0)$ vanishes in cohomological degrees $\le -j$. Finally let $i\to\infty$, recalling from Theorem \ref{thm:hkr}(1) that the filtration is exhaustive in this case, to deduce that $\Fil_\sub{HKR}^{< j}\HC^-(R/k_0)$ vanishes in cohomological degrees $\le - j$ (i.e., homological degrees $>j-1$).
\end{enumerate}
Therefore, by general results on $t$-structures, the mapping space \[\Map_{\Shv_{\Zar}(\Sm_k,\Spt)}(\Fil^{\geq j}_\bb A\K, \Fil^{<j}_\sub{HKR}\HC^-(-/k_0))\] is contractible for each $j$. By induction, the trace map $\K \rightarrow \HC^-(-/k_0)$ therefore uniquely refines to compatible maps $\Fil^{j}_\bb A\K\to\Fil^{j}_\sub{HKR}\HC^-(-/k_0)$ for all $j\ge0$, as desired.

To ensure multiplicativity, one uses the Postnikov $t$-structure on Zariski sheaves of filtered spectra as introduced in \cite[Construction~3.3.6-7]{raksit-hkr}. This is a $t$-structure which wraps together the $t$-structure on Zariski sheaves and the $t$-structure on the filtered derived category. The (co-)connective part consists of filtered Zariski sheaves $F^{\star}$ such that $F^j \in \Shv_{\Zar}(\Sm_k;\Spt)_{\geq j}$ (resp. $F^j\in \Shv_{\Zar}(\Sm_k;\Spt)_{\leq j}$) for all $j\in\bb Z$. Furthermore, the truncation functor $\tau^P_{\geq 0}$  admits a lax symmetric monoidal structure such that the counit map $\tau^P_{\geq 0} \rightarrow \id$ is a morphism of lax symmetric monoidal functors. In particular, if $F^\star$ is a filtered, multiplicative sheaf then the map $\tau^P_{\geq 0}F^{\star} \rightarrow F^{\star}$ is multiplicative.

Our proof shows that, for the Postnikov $t$-structure, firstly $\Fil^\star_\bb A\K$ is connective, and secondly the cofibre $\text{cofib}(\Fil^\star_\sub{HKR}\HC^-(-/k_0)\to\HC^-(-/k_0))$ is $-1$-truncated where the target is given the constant filtration; therefore the map $\Fil^\star_\sub{HKR}\HC^-(-/k_0)\to\HC^-(-/k_0)$ is a $\tau^P_{\geq 0}$-equivalence. So we obtain a multiplicative map of filtered objects
\[
\Fil^\star_\bb A\K \xleftarrow{\simeq} \tau^P_{\geq 0}\Fil^\star_\bb A\K \rightarrow \tau^P_{\geq 0}\HC^-(-/k_0) \simeq  \tau^P_{\geq 0}\Fil^\star_\sub{HKR}\HC^-(-/k_0) \rightarrow \Fil^\star_\sub{HKR}\HC^-(-/k_0),
\]
as desired. 
\end{proof}

\begin{remark}\label{rem:promote-compat} In fact the $t$-structure argument in the above proposition proves slightly more: with the same hypotheses as in Proposition~\ref{prop:compat}, {\em any} multiplicative morphism $\K \rightarrow \HC^-(-/k_0)$ on $\Sm_k$ promotes uniquely to a multiplicative filtered map intertwining the classical motivic and HKR filtrations. This will be used in Step 2 of the proof of Theorem \ref{thm:adams} to prove that the Adams operations on $K$-theory and negative cyclic homology are compatible in a filtered sense.
\end{remark}

It now follows that the cdh-local trace map $L_\sub{cdh}\mathrm{tr}:\KH(X)\to L_\sub{cdh}\HC^-(X/\bb Q)$, for qcqs $\bb Q$-schemes $X$, also naturally carries the $\bb A^1$-motivic filtration on the left to the cdh-local HKR filtration on the right:

\begin{corollary}\label{corol_cdh_filtered_trace}
The cdh-local trace map $L_\sub{cdh}\mathrm{tr}:\KH\to L_\sub{cdh}\HC^-(-/\bb Q)$, viewed as a map between spectra-valued presheaves on $\Sch_\bb Q^\sub{qcqs}$, admits a unique multiplicative extension to a map of filtered presheaves \[\Fil^{\star}_{\bb A}\KH\To \Fil^{\star}_\sub{HKR}L_\sub{cdh}\HC^-(-/\bb Q)\] (the filtration on the left being the $\bb A^1$-motivic filtration of Theorem \ref{thm:cdh}(1); the filtration on the right is the cdh-local HKR filtration of Remark \ref{rem_cdh_local_HKR}).
\end{corollary}
\begin{proof}
The map of filtered spectra $\Fil_\sub{HKR}^{\star}\HC^-(X/\bb Q)\to \Fil_\sub{HKR}^{\star}L_\sub{cdh}\HC^-(X/\bb Q)$ is an equivalence for any smooth $\bb Q$-scheme $X$; this is a standard consequence of strong resolution of singularities; we refer to \cite[Lemma.~4.7]{Voevodsky2010} for the key idea, and to \cite[Theorem 2.4 \& Corollary 2.5]{Cortinas2008a} for precise proofs.

Consequently, given any filtered upgrade of the cdh-local trace map, its restriction to smooth $\bb Q$-schemes is necessarily the unique filtered upgrade of the trace map from Proposition \ref{prop:compat}. Conversely, the desired filtered upgrade of the cdh-local trace map is therefore necessarily given by the following composition: \[\Fil_{\bb A}^{\star}\KH\To L_\sub{cdh}L_{\sub{Sch}_{\bb Q}^\sub{qcqs,op}/\sub{Sm}_{\bb Q}^\sub{op}}\Fil^\star_\sub{HKR}\HC^-(-/\bb Q)\To \Fil^\star_\sub{HKR}L_\sub{cdh}\HC^-(-/\bb Q).\] 
Here the first map is obtained by cdh-locally left Kan extending the filtered trace map for smooth $\bb Q$-schemes, using the fact that $\Fil_{\bb A}^{\star}\KH$ is cdh-locally left Kan extended by Theorem~\ref{thm:cdh}(10). The second map is the cdh sheafification of the counit map $L_{\sub{Sch}_{\bb Q}^\sub{qcqs,op}/\sub{Sm}_{\bb Q}^\sub{op}}\Fil^\star_\sub{HKR}\HC^-(-/\bb Q)\to \Fil^\star_\sub{HKR}\HC^-(-/\bb Q)$. Multiplicativity follows from the fact that both left Kan extension and sheafification are lax monoidal.
\end{proof}

We may now construct our motivic cohomology and motivic filtration on qcqs $\bb Q$-schemes:

\begin{definition}[Motivic cohomology in characteristic zero]\label{eq:char0}
For a qcqs $\bb Q$-scheme $X$, let $\Fil^\star_\sub{mot}\K(X)$ be the pullback in $\mathbb{E}_{\infty}$-algebras in filtered spectra of the diagram
\begin{equation}\label{eq:motfilt}
\begin{tikzcd}
\mathrm{Fil}^{\star}_{\mot}\K(X) \ar[dotted]{d} \ar[dotted]{r}    & \mathrm{Fil}^{\star}_\sub{HKR}\HC^-(X/\bb Q)\ar{d} \\
\mathrm{Fil}^\star_{\bb A}\KH(X) \ar{r} & \Fil^{\star}_\sub{HKR}L_\sub{cdh}\HC^-(X/\bb Q).
\end{tikzcd}
\end{equation}
Here the bottom horizontal arrow is the unique multiplicative filtered upgrade of the cdh-local trace map provided by Corollary \ref{corol_cdh_filtered_trace}, and the right vertical arrow is the canonical map.

For $j\in\bb Z$, define the {\em weight-$j$ motivic cohomology} of $X$ to be \[\bb Z(j)^\sub{mot}(X):=(\gr^j_\sub{mot}\K(X))[-2j],\] which we will see in Theorem \ref{thm:graded-pieces} lies in $\rm D(\bb Z)$ and vanishes for $j<0$; in particular $\bb Z(\star)^\sub{mot}(X)$ is naturally an $\bb E_\infty$-algebra in $\rm D(\bb Z)$. The associated motivic cohomology groups, for $i\in\bb Z$, are denoted by $H^i_\sub{mot}(X,\bb Z(j)):=H^i(\bb Z(j)^\sub{mot}(X))$.
\end{definition}

In the following theorem we collect some of the immediate, but fundamental, properties of this motivic cohomology theory for qcqs $\bb Q$-schemes:

\begin{theorem}\label{thm:graded-pieces}
Let $j\in\bb Z$. For any qcqs $\bb Q$-scheme $X$, the weight-$j$ motivic cohomology has the following properties:
\begin{enumerate}
\item $\bb Z(j)^\sub{mot}(X)=0$ for $j<0$.
\item There is a natural pullback square
\[
\begin{tikzcd}
\bb Z(j)^\sub{mot}(X) \ar{r} \ar{d} & R\Gamma(X,\widehat{L\Omega}_{-/\Q}^{\geq j})  \ar{d}\\
\bb Z(j)^{\bb A}(X) \ar{r} & R\Gamma_\sub{cdh}(X,\widehat{L\Omega}_{-/\Q}^{\geq j}).
\end{tikzcd}
\]
(Varying $j$, this is a pullback of $\bb E_\infty$-algebras in graded complexes; see Example \ref{example_00}.)
\item Fundamental fibre sequence: there is a natural fibre sequence \[\bb Z(j)^\sub{mot}(X)\To \bb Z(j)^{\bb A}(X)\To {\rm cofib}\left(R\Gamma(X,L\Omega_{-/\Q}^{< j})\to  R\Gamma_\sub{cdh}(X,\Omega_{-/\Q}^{< j})\right)[-1]\] 
\item For any integer $m\ge1$, the map $\bb Z(j)^\sub{mot}(X)/m\to\bb Z(j)^{\bb A}(X)/m$ is an equivalence.
\item The presheaf $\bb Z(j)^\sub{mot}:\rm{Sch}_\bb Q^\sub{qcqs,op}\to\rm D(\bb Z)$ is a finitary Nisnevich sheaf.
\end{enumerate}
\end{theorem}
\begin{proof}
We obtain the pullback square of part (2) by taking graded pieces in the cartesian square \eqref{eq:motfilt}: as a reminder, the graded pieces of $\Fil^\star_\bb A\KH(X)$ are $\bb A^1$-motivic cohomology by Theorem \ref{thm:cdh}(1); those of $\HC^-(X/\bb Q)$ are Hodge-completed derived de Rham cohomology by Theorem \ref{thm:hkr}; and those of its cdh sheafification are cdh sheafified Hodge-completed derived de Rham cohomology by Remark \ref{rem_cdh_local_HKR}. 

In particular, when $j<0$ (in which case $\bb Z(j)^\bb A=0$ by Theorem~\ref{thm:cdh}(1)) we have established the existence of a cartesian square 
\[
\begin{tikzcd}
\bb Z(j)^\sub{mot}(X) \ar{r} \ar{d} & R\Gamma(X,\widehat{L\Omega}_{-/\Q})  \ar{d}\\
0 \ar{r} & R\Gamma_\sub{cdh}(X,\widehat{L\Omega}_{-/\Q}).
\end{tikzcd}
\]
But the right vertical arrow is an equivalence because Hodge-completed derived de Rham cohomology satisfies cdh descent in characteristic zero by Lemma~\ref{lemma_cdh_descent_HP}. Therefore $\bb Z(j)^\sub{mot}(X)=0$ for $j<0$.

To obtain the fundamental fibre sequence, compute the cofibre of the right vertical arrow in part (2) as follows: compare the fibre sequence \[R\Gamma(X,\widehat{L\Omega}_{-/\Q}^{\ge j})\To R\Gamma(X,\widehat{L\Omega}_{-/\Q})\To R\Gamma(X,L\Omega_{-/\Q}^{<j})\] to its cdh sheafified version, and use the following two facts: firstly, cdh sheafifying the middle term of the fibre sequence does not change it, by Lemma \ref{lemma_cdh_descent_HP}; secondly, the canonical map $R\Gamma_\sub{cdh}(X,L\Omega_{-/\Q}^{<j})\to R\Gamma_\sub{cdh}(X,\Omega_{-/\Q}^{<j})$ is an equivalence, either by resolution of singularities or by Gabber--Ramero's results on the cotangent complex of valuation rings \cite[Theorem~6.5.12 \& Corollary~6.5.21]{GabberRamero2003}.

Part (4) follows from the pullback square of (2), since the complexes on the right side of the square are rational.

For part (5), recall that wedge powers $L^i_{-/\bb Q}$ of the cotangent complex commute with filtered colimits of rings; therefore, by Zariski descent and a finite induction, $R\Gamma(X,L\Omega_{-/\Q}^{<j})$ is finitary. Cdh sheafifying preserves finitariness \cite[Proposition 2.15(2)]{BachmannElmantoMorrow}, so $R\Gamma_\sub{cdh}(-,L\Omega_{-/\Q}^{<j})$ is also finitary. Finally, $\bb Z(j)^{\bb A}$ is finitary by Theorem \ref{thm:cdh}(2). We now deduce finitariness of $\bb Z(j)^\sub{mot}$ from the fundamental fibre sequence.
\end{proof}

\begin{example}[Weight $0$ and $\bb Z$-algebra structure]\label{example_00}
The right vertical arrow in Theorem \ref{thm:graded-pieces}(2) is an equivalence when $j=0$, by cdh descent of Hodge-completed derived de Rham cohomology; therefore the same is true of the left vertical arrow. That is, there is a natural equivalence \[\bb Z(0)^\sub{mot}(X)\quis \bb Z(0)^{\bb A}(X)\simeq R\Gamma_\sub{cdh}(X,\bb Z)\] for any qcqs $\bb Q$-scheme $X$ (the equality in the previous line being Theorem \ref{thm:cdh}(3)). These are equivalences of $\bb E_\infty$-algebras in spectra, thereby upgrading $\bb Z(0)^\sub{mot}(X)$ into an $\bb E_\infty$-algebra in $\rm D(\bb Z)$. By shearing the indexing of the associated graded $\gr_\sub{mot}^\star\K(X)$ as in \cite[Remark~2.1]{BachmannElmantoMorrow}, we naturally upgrade $\bb Z(\star)^\sub{mot}(X)$ to be an $\bb E_\infty$-algebra in $\Gr\rm D(\bb Z)$. We will see in Example \ref{example_0p} that the same works in finite characteristic.
\end{example}

\begin{remark}[Definition of $\bb Z(j)^\sub{mot}$ without trace map]\label{rem:horizontal_map_char0}
The bottom horizontal arrow in Theorem~\ref{thm:graded-pieces}(2) (which by definition is the $j^\sub{th}$ graded piece of the filtered multiplicative upgrade of the cdh sheafified trace map provided by Corollary \ref{corol_cdh_filtered_trace}) admits the following explicit description as a cdh-local de Rham cycle class map. 

Namely, on smooth $\bb Q$-schemes we consider the composition \[\bb Z(j)^\bb A\To R\Gamma_\sub{Zar}(-,\K_j^M)[-j]\xto{\sub{dlog}}R\Gamma(-,\Omega^{\ge j}_{-/\bb Q}),\] where the first map is induced by the facts that $\bb Z(j)^\bb A|_{\Sm_\bb Q}$ is Zariski locally supported in cohomological degrees $\le j$ (by Gersten vanishing), with $j^\sub{th}$ cohomology locally being $\K_j^M$ by Nesterenko--Suslin \cite{Suslin1989}, and the second map is induced by the dlog map $\K_j^M\to\Omega^j_{-/\bb Q}$ landing in closed forms. Left Kan extending to qcqs $\bb Q$-schemes, cdh sheafifying, and composing along the counit back to $R\Gamma(-,\hat{L\Omega}^{\ge j}_{-/\bb Q})$ defines a map of presheaves on qcqs $\bb Q$-schemes \[\bb Z(j)^\bb A\To R\Gamma(-,\hat{L\Omega}^{\ge j}_{-/\bb Q}), \] which we claim is precisely the bottom horizontal arrow in Theorem \ref{thm:graded-pieces}(2).

To check this one reduces by multiplicativity to the case $j=1$, and then runs through the proofs of Proposition \ref{prop:compat} (with $\bb Q=k_0=k$) and Corollary \ref{corol_cdh_filtered_trace} to deduce the claim from the fact that the composition \[\K_1(\bb Q[t^{\pm1}])\xto{\mathrm{tr}}\HC^-_1(\bb Q[t^{\pm1}]/\bb Q)\stackrel\simeq\leftarrow \pi_1(\Fil^1_\sub{HKR}\HC^-(\bb Q[t^{\pm1}]/\bb Q))\isoto \pi_1(\gr^1_\sub{HKR}\HC^-(\bb Q[t^{\pm1}]/\bb Q))=\Omega^1_{\bb Q[t^{\pm1}]/\bb Q)}\] sends $t$ to $\tfrac{dt}{t}$.\footnote{This calculation of the trace map on units in fact concerns only the Dennis trace, since the canonical map $\pi_1(\gr^1_\sub{HKR}\HC^-(\bb Q[t^{\pm1}]/\bb Q))\to \pi_1(\gr^1_\sub{HKR}\HH(\bb Q[t^{\pm1}]/\bb Q))=\Omega^1_{\bb Q[t^{\pm1}]/\bb Q)}$ is an isomorphism; then it follows from the classical explicit definition of the Dennis trace map, e.g., \cite{Loday}.}
\end{remark}

Next we state some of the fundamental properties of the motivic filtration: namely, $\Fil^\star_\sub{mot}\K(X)$ is indeed a filtration on $\K(X)$, as suggested by the notation, and so there is the desired Atiyah--Hirzebruch spectral sequence:

\begin{theorem}\label{theorem_AH_SS_0}
Let $X$ be a qcqs $\bb Q$-scheme. Then the $\bb E_\infty$-algebra in filtered spectra $\Fil^\star_\sub{mot}\K(X)$ is $\bb N$-indexed and satisfies $\Fil^0_\sub{mot}\K(X)=\K(X)$. Moreover:
\begin{enumerate}
\item If $X$ has finite valuative dimensional then the filtration is bounded and it induces a bounded multiplicative Atiyah--Hirzebruch spectral sequence \[E_2^{ij}=H^{i-j}_\sub{mot}(X,\bb Z(-j))\implies \K_{-i-j}(X)\] which degenerates rationally.
\item The filtration naturally splits rationally, i.e., there is a natural, multiplicative equivalence of filtered spectra \[\Fil^\star_\sub{mot}\K(X)_\bb Q\simeq \bigoplus_{j\ge\star}\bb Q(j)^\sub{mot}(X)[2j].\]
\end{enumerate}
\end{theorem}
\begin{proof}
Theorem \ref{thm:graded-pieces}(1) already shows that the filtered spectrum $\Fil^\star_\sub{mot}\K(X)$ is $\bb N$-graded. By definition $\Fil^0_\sub{mot}\K(X)$ is defined via a pullback square
\[\begin{tikzcd}
\Fil^0_\sub{mot}\K(X) \ar{r} \ar{d} & \Fil^0_\sub{HKR}\HC^-(X/\bb Q)  \ar{d}\\
\Fil^0_{\bb A}\KH(X) \ar{r} & L_\sub{cdh}\Fil^0_\sub{HKR}\HC^-(X/\bb Q),
\end{tikzcd}
\]
which admits a map to the pullback square (\ref{eq:hc-}). Since $\Fil^0_{\bb A}\KH(X)\quis\KH(X)$, the claim reduces to checking that the square
\[\begin{tikzcd}
\Fil^0_\sub{HKR}\HC^-(X/\bb Q) \ar{r} \ar{d} & \HC^-(X/\bb Q)  \ar{d}\\
L_\sub{cdh}\Fil^0_\sub{HKR}\HC^-(X/\bb Q) \ar{r} & L_\sub{cdh}\HC^-(X/\bb Q),
\end{tikzcd}
\]
is a pullback, i.e., that the cofibre $\HC^-(-/\bb Q)/\Fil^0_\sub{HKR}\HC^-(-/\bb Q)$ satisfies cdh descent on $\text{Sch}_\bb Q^\sub{qcqs}$. This was explained in Lemma~\ref{lemma_cdh_descent_HP}.

We now prove part (1) apart from the rational degeneration; so assume that $X$ has finite valuative dimension $\le d$. We know from Theorem \ref{thm:cdh}(1) that $\Fil^j_{\bb A}\KH(X)$ is supported in homological degrees $\ge j-d$. Now, $\widehat{L\Omega}^{\geq j}_{-/\bb Q}$ is supported in cohomological degrees $\le j$; by Zariski or cdh sheafifying, it follows that $\Fil^j_\sub{HKR}\HC^-(X/\bb Q)$ and $\Fil^j_\sub{HKR}L_\sub{cdh}\HC^-(X/\bb Q)$ are both supported in homological degrees $\ge j-d$. From the defining pullback square (\ref{eq:motfilt_charp}), we then see that $\Fil^j_\sub{mot}\K(X)$ is supported in homological degrees $\ge j-d-1$, which is good enough to prove the desired boundedness (but not the optimal bound: see \S\ref{section_Weibel}). The spectral sequence follows from standard machinary of filtered spectra (e.g. \cite[\S 1.2.2]{LurieHA}). Its rational degeneration will follow from part (2).

We now drop the hypothesis that $X$ has finite valuation dimension and turn to (2). The splitting follows from the result on Adams operators in Theorem~\ref{thm:adams} by an argument which is standard but looks technically complicated in our context since we are keeping track of multiplicative structure. Fix $m\ge 2$. We start with the multiplicative maps of graded spectra 
\[
\gr^\star_{\mot}\K(X) \longleftarrow \Fil^{\star}_{\mot}\K(X) \longrightarrow \K(X),
\]
where the right-most term is regarded as a constant graded object and the middle term is the graded spectrum $j\mapsto \Fil^{j}_{\mot}\K(X)$ (i.e., $R(\Fil^{\star}_{\mot}\K(X))$ in terms of the upcoming notation). Rationalising and taking Adams eigenspectra, we obtain a diagram of graded $\bb E_\infty$-algebras
\begin{equation}\label{eq:gr-to-k}
\begin{tikzcd}
  (\gr^\star_{\mot}\K(X)_{\Q})^{\psi^{m} \simeq m^\star} \ar{d} &  \ar[swap]{l}{\simeq} (\Fil^{\star}_{\mot}\K(X)_{\Q})^{\psi^{m} \simeq m^\star} \ar{r} & \K(X)_{\Q}\\
\ar[dotted, bend left]{u} \gr^\star_{\mot}\K(X)_{\Q} \ar[bend right, swap]{d}{m^{\star}} \ar[bend left]{d}{\psi^m}& & \\
\gr^\star_{\mot}\K(X)_{\Q}. 
\end{tikzcd}
\end{equation}
Here the superscript $^{\psi^m\simeq m^\star}$ denotes the equaliser of $\psi^m$ and multiplication-by-$m^\star$ on the indicated object (so the column is an equaliser diagram) and we must explain the equivalence and the dotted arrow. The equivalence is the assertion that the canonical map $(\Fil^j_\sub{mot}\K(X)_\bb Q)^{\psi^m\simeq m^j}\to (\gr^j_\sub{mot}\K(X)_\bb Q)^{\psi^m\simeq m^j}$ is an equivalence of complexes for all $j\ge0$, or in other words that $(\Fil^{j+1}_\sub{mot}\K(X)_\bb Q)^{\psi^m\simeq m^j}$ vanishes. To prove this we may reduce by taking a filtered colimit (using Theorem \ref{thm:graded-pieces}(5)) to the case that $X$ has finite valuative dimension; then the motivic filtration on $\K(X)_\bb Q$ is bounded by part (1), hence complete, and so the vanishing follows from the fact that $\psi^m$ acts as $m^i$ on $\gr^i_\sub{mot}\K(X)_\bb Q$ for all $i\ge j+1$ (Theorem~\ref{thm:adams}). Similarly, the natural homotopy between $\psi^m$ and multiplication-by-$m^\star$ on $\gr^\star_\sub{mot}\K(X)_\bb Q$ induces the dotted arrow: it is a natural map of graded $\bb E_\infty$-algebras which is split by the vertical arrow.

Composing from the middle of the column to the top right defines maps $\gr^j_\sub{mot}\K(X)_\bb Q\to\K(X)_\bb Q$ for all $j\ge0$, which will yield the desired splitting; but additional effort is required to pay attention to multiplicative structure.

For any presentably symmetric monoidal $\infty$-category $\scr C$ we have a monoidal adjunction\footnote{More precisely this adjunction is obtained as follows: we have an inclusion of wide subcategories $\iota:\bb Z^{\delta} \subset (\bb Z, \geq)^{\op}$ which induces the restriction functor $\iota^*: \Fil\scr C \rightarrow \Gr \scr C$. The adjunction of interest is given by
\[
\iota_!: \Gr \scr C \rightleftarrows \Fil \scr C: \iota^*.
\] This description makes the strong monoidality of the left adjoint (and therefore the lax monoidality of the right adjoint) evident.} $L:\Gr\scr C \rightleftarrows \Fil \scr C: R$ where $L$ send a graded object  $M^{\star}$ to the split filtered object \[\bigoplus_{n \geq \star} M^n = (\cdots \to\bigoplus_{n \geq \star+1} M^n \xrightarrow{\rm incl.} \bigoplus_{n \geq \star} M^n \cdots\to),\] and $R$ send a filtered object $F^{\star}M$ to the graded object \[F^{\star}M = (\cdots, F^{-1}M, F^0M, F^1M, \cdots)\] i.e., forget the transition maps. We compose the dotted arrow of \eqref{eq:gr-to-k} with the ``inclusion'' map $(\Fil^{\star}_{\mot}\K(X)_{\Q})^{\psi^{m} \simeq m^\star}  \rightarrow \Fil^{\star}_{\mot}\K(X)_{\Q}$, then apply $L$ everywhere, and finally compose along the counit back to $\Fil^{\star}_{\mot}\K(X)_{\Q}$ as a filtered $\bb E_\infty$-algebra. The conclusion is a natural map of $\bb N$-indexed filtered $\bb E_{\infty}$-algebras
\[
\bigoplus_{n \geq \star} \gr^n_{\mot}\K(X)_{\Q}  \To \bigoplus_{n \geq \star} \Fil^{n}_{\mot}\K(X)_{\Q} \xrightarrow{\rm counit} \Fil^{\star}_{\mot}\K(X)_{\Q}.
\]
which is an equivalence on graded pieces. We wish to show it is an equivalence. But both sides commute with filtered colimits by Theorem \ref{thm:graded-pieces} (note that finitariness of $\gr^j_\sub{mot}\K$ for all $j$ and of $\K$ itself implies that of $\Fil^j_\sub{mot}\K$ for all $j$), so we may assume that $X$ has finite valuative dimension; then the filtrations on both sides are bounded by part (1), whence complete, and so the desired equivalence follows from the equivalence on graded pieces.
\end{proof}

The splitting and degeneration parts of the parts of the previous theorem were consequences of a finer result, namely the existence of Adams operators acting in the expected way on motivic cohomology as asserted in the following theorem. We note that this result depends on Appendix~\S\ref{app:chw} about Adams operators.

\begin{theorem}[Existence of filtered Adams operators on $K$-theory of $\bb Q$-schemes]\label{thm:adams}
Fix $m\in\bb Z\setminus\{0\}$. For any qcqs $\bb Q$-scheme $X$, there exists a natural endomorphism $\psi^m$ on the $\bb E_\infty$-algebra in filtered spectra $\Fil^\star_\sub{mot}\K(X)[\tfrac1m]$ with the property that the induced endomorphism of the $\bb E_\infty$-algebra in graded complexes $\gr^\star_\sub{mot}\K(X)[\tfrac1m]=\bb Z(\star)^\mot(X)[\tfrac1m][2\star]$ is naturally homotopic to its endomorphism\footnote{The reader should see Remark~\ref{rem:mult_by_m}  of Appendix B for a precise definition of this multiplicative endomorphism.} multiplication-by-$m^\star$.

Moreover, the Adams operator $\psi^m$ is compatible with existing Adams operators on $K$-theory and negative cyclic homology, in the following ways:
\begin{enumerate}
\item The diagram
\[
\begin{tikzcd}
\Fil^\star_\sub{mot}\K[\tfrac{1}{m}] \ar[swap]{d}{\psi^m} \ar{r} & \Fil^\star_{\bb A}\KH[\tfrac{1}{m}] \ar{d}{\psi^m}\\
\Fil^\star_\sub{mot}\K[\tfrac{1}{m}] \ar[swap]{r}  & \Fil^\star_{\bb A}\KH[\tfrac{1}{m}],
\end{tikzcd}
\]
in $\PShv(\Sch^\sub{qcqs}_\bb Q,\CAlg(\Fil\Sp))$ commutes, where the $\psi^m$ on the right vertical arrow is Bachmann--Hopkins' Adams operator \eqref{eqn:BH_adams} on $\KH$-theory with its $\bb A^1$-motivic filtration.
\item The diagram
\[
\begin{tikzcd}
\Fil^\star_\sub{mot}\K[\tfrac{1}{m}] \ar[swap]{d}{\psi^m} \ar{r}{\mathrm{tr}} & \Fil^\star_\sub{HKR}\HC^-(-/\bb Q) \ar{d}{\psi^m}\\
\Fil^\star_\sub{mot}\K[\tfrac{1}{m}] \ar[swap]{r}{\mathrm{tr}}  & \Fil^\star_\sub{HKR}\HC^-(-/\bb Q).
\end{tikzcd}
\]
in $\PShv(\Sch^\sub{qcqs}_\bb Q,\CAlg(\Fil\Sp))$ commutes, where the $\psi^m$ on the right vertical arrow is Raksit's Adams operator on filtered $\HC^-(-/\bb Q)$ (see Appendix \ref{ss_Raksit}), and the horizontal arrows are the top horizontal arrow of \eqref{eq:motfilt}.

\item Given any qcqs $\bb Q$-scheme $X$, the induced endomorphism $\psi^m$ on $\K[\tfrac1m]|_{\Sm_X}$ (by passing to $\Fil^0_\sub{mot}$ and restricting to smooth $X$-schemes) agrees with the Adams operator for smooth $X$-schemes arising from the Annala--Iwasa theorem (see Construction \ref{constr:adams-k}).
\end{enumerate}
\end{theorem}
\begin{proof}
We begin by constructing the Adams operator $\psi^m$ on $\Fil^\star_\sub{mot}\K(X)[\tfrac1m]$, then explain how to verify the claimed compatibilities.

We restrict Bachmann--Hopkins' Adams operator \eqref{eqn:BH_adams} to smooth $\bb Q$-schemes to obtain $\psi^m: \Fil_{\bb A}^{\star}\K[\tfrac{1}{m}] \to \Fil_{\bb A}^{\star}\K[\tfrac{1}{m}]$. Similarly we restrict Raksit's Adams operator \eqref{eq:filtered-hc-} to smooth $\bb Q$-schemes to obtain $\psi^m:\Fil^\star_\sub{HKR}\HC^-(-/\bb Q)\to\Fil^\star_\sub{HKR}\HC^-(-/\bb Q)$. These are both filtered multiplicative maps on $\Sm_\bb Q$.

Recalling from Proposition~\ref{prop:compat} that the trace map promotes uniquely to a filtered multiplicative map $\mathrm{tr}:\Fil_{\bb A}^{\star}\K[\tfrac1m]\to \Fil_{\bb A}^{\star}\K[\tfrac1m]$ on $\Sm_\bb Q$, we claim that the diagram 
\begin{equation}\label{eq:chw-filt}
\begin{tikzcd}
\Fil_{\bb A}^{\star}\K[\tfrac1m] \ar[swap]{d}{\psi^m} \ar{r}{\mathrm{tr}} &\Fil_\sub{HKR}^{\star}\HC^-(-/\bb Q) \ar{d}{\psi^m}\\
\Fil_{\bb A}^{\star}\K[\tfrac1m] \ar{r}[swap]{\mathrm{tr}} &\Fil_\sub{HKR}^{\star}\HC^-(-/\bb Q)
\end{tikzcd}
\end{equation}
in $\PShv(\Sm^\sub{qcqs}_\bb Q,\CAlg(\Fil\Sp))$ commutes. Indeed, the circuits $\tr\circ\psi^m$ and $\psi^m\circ\tr$ are homotopic if we forget the filtrations, by Corollary \ref{corol:key-q}, and therefore are also homotopic if we keep the filtrations by Remark \ref{rem:promote-compat}; more precisely, the commutative diagram~\eqref{corol:key-q} (with $k = \bb Q$) uniquely promotes to a to a commutative diagram in $\PShv(\Sm^\sub{qcqs}_\bb Q,\CAlg(\Fil\Sp))$.

Left Kan extending and cdh sheafifying, as in the proof of Corollary \ref{corol_cdh_filtered_trace}, then shows that the filtered multiplicative upgrade of the cdh-local trace map $L_\sub{cdh}\mathrm{tr}:\Fil_\bb A^\star\KH\to \Fil^{\star}_\sub{HKR}L_\sub{cdh}\HC^-(-/\bb Q)$ on qcqs $\bb Q$-schemes is also compatible with Adams operators on each side; here the left side is equipped with Bachmann--Hopkins' Adams operator \eqref{eqn:BH_adams}, while the right side is equipped with the cdh sheafification of Raksit's \eqref{eq:filtered-hc-}.

Working in the category of presheaves of filtered $\bb E_\infty$-algebras equipped with an endomorphism on qcqs $\bb Q$-schemes, we now pull back the diagram
\begin{equation}
\begin{tikzcd}
   & \mathrm{Fil}^{\star}_\sub{HKR}\HC^-(-/\bb Q)\ar{d} \\
\mathrm{Fil}^\star_{\bb A}\KH[\tfrac1m] \ar{r} & \Fil^{\star}_\sub{HKR}L_\sub{cdh}\HC^-(-/\bb Q)
\end{tikzcd}
\end{equation}
to define our desired Adams operator $\psi^m$ on $\Fil^{\star}_{\mot}\K[\tfrac{1}{m}]$.

Next we prove that the induced endomorphism on $\gr^\star_\sub{mot}\K[\tfrac1m]$ is homotopic to multiplication-by-$m^\star$. Informally this holds because it holds for the other Adams operators appearing in the construction, but care is required because we need to know that the various homotopies appearing are compatible. For each of the above Adams operators on a filtered object, we will continue to denote by $\psi^m$ the induced operator on the associated graded. More precisely, restricting to smooth $\bb Q$-schemes, the operator $\psi^m$ on $\gr^{\star}_{\bb A}\K(X)[\tfrac{1}{m}]$ (resp.~on $\gr_\sub{HKR}^{\star}\HC^{-}(X/\bb Q)$) is homotopic to multiplication-by-$m^\star$ by Lemma \ref{lemma_adams_on_smooth} (resp.~\cite[Proposition 6.4.12]{raksit-hkr}). We need a homotopy between these homotopies along the map $\mathrm{tr}: \gr^{\star}_{\bb A}\K[\tfrac{1}{m}] \rightarrow \gr^{\star}_\sub{HKR}\HC^{-}(-/\bb Q)$, namely $\gr^{\star}$ of the filtered enhancement of the trace map on $\Sm_{\bb Q}$; this homotopy between the homotopies will even turn out to be unique. We consider the following diagram of spaces:
\[
\begin{tikzcd}
 & \Omega_{\psi^m}\Map(\gr_\sub{HKR}^{\star}\HC^{-}(-/\bb Q),\gr_\sub{HKR}^{\star}\HC^{-}(-/\bb Q)) \ar{d}{-\circ\mathrm{tr}}\\
 & \Omega_{\psi^m\circ\mathrm{tr}}\Map( \gr^{\star}_{\bb A}\K[\tfrac{1}{m}], \gr^{\star}_\sub{HKR}\HC^{-}(-/\bb Q))\ar{d}{\simeq}\\
\Omega_{\psi^m}\Map(\gr_{\bb A}^{\star}\K[\tfrac{1}{m}],\gr_{\bb A}^{\star}\K[\tfrac{1}{m}])  \ar{r}{\mathrm{tr}\circ-} & \Omega_{\mathrm{tr} \circ \psi^m}\Map( \gr^{\star}_{\bb A}\K[\tfrac{1}{m}], \gr^{\star}_\sub{HKR}\HC^{-}(-/\bb Q)), 
\end{tikzcd}
\]
\begin{enumerate}[(a)]
\item The mapping spaces are taken in the $\Shv_{\Zar}(\Sm_{\bb Q}, \CAlg(\Gr\Spt))$. We have indicated the base points of the loop spaces, and the equivalence is induced by the homotopy $\psi^m\circ\mathrm{tr}\simeq \mathrm{tr} \circ \psi^m$ obtained by passing to associated gradeds in \eqref{eq:chw-filt}.

\item The other maps are pre/postcomposition with the trace map.

\item The two homotopies $\psi^m\simeq m^\star$ from the paragraph above the diagram correspond to points in the bottom left and top right. The ``homotopy between homotopies'' which we seek is any point in the pullback of the diagram which sits over these existing points.

\item The bottom right corner is a single point. Indeed, the mapping space $\Map( \gr^{\star}_{\bb A}\K[\tfrac{1}{m}], \gr^{\star}_\sub{HKR}\HC^{-}(-/\bb Q))$ is discrete by the graded analog of the proof of Proposition~\ref{prop:compat}, where the analogous discreteness was proved at the filtered level; therefore the loop space is a point.
\end{enumerate}
The pullback of the diagram is therefore equivalent to the product
\[
 \Omega_{\psi^m}\Map(\gr_\sub{HKR}^{\star}\HC^{-}(-/\bb Q),\gr_\sub{HKR}^{\star}\HC^{-}(-/\bb Q)) \times \Omega_{\psi^m}\Map(\gr_{\bb A}^{\star}\K[\tfrac{1}{m}],\gr_{\bb A}^{\star}\K[\tfrac{1}{m}]).
 \] 
and so the desired homotopy between the two homotopies $\psi^m\simeq m^\star$ is just the cartesian product of these two points. Left Kan extending, cdh sheafifying, and forming the pull back, one obtains an induced homotopy $\psi^m \simeq m^{\star}$ on $\gr^{\star}_\sub{mot}\K[\tfrac{1}{m}]$, as desired.

It remains to check the compatibilities (1)--(3). Parts (1) and (2) are immediate from the construction. To prove part (3), will use the techniques of Appendices \ref{sec:ai-thm} \& \ref{sec:adj}. Thanks to the comparison between Bachmann--Hopkins' Adams operator and that of Construction~\ref{constr:adams-k} on smooth $\bb Q$-schemes, there is a commutative diagram
\[
\begin{tikzcd}
\bb S_\bb Q[\Pic] \ar[swap]{d}{\psi^m} \ar{r} &  \K[\tfrac{1}{m}] \ar{d}{\psi^m}\\
\bb S_\bb Q[\Pic] \ar{r} &   \K[\tfrac{1}{m}],
\end{tikzcd}
\]
on smooth $\bb Q$-schemes, where the left vertical arrow is by definition induced by the $m$-power map $\roi^\times\to\roi^\times$. Left Kan extending to qcqs $\bb Q$-schemes, Zariski sheafifying, and composing with the counit back to $K[\tfrac1m]$, we obtain a commutative diagram
\begin{equation}
\begin{tikzcd}
\bb S[\Pic] \ar[swap]{d}{\psi^m} \ar{r} &  \K[\tfrac{1}{m}] \ar{d}{\psi^m}\\
\bb S[\Pic] \ar{r} &   \K[\tfrac{1}{m}]
\end{tikzcd}
\end{equation}
on all qcqs $\bb Q$-schemes, where the $\psi^m$ on the right is the Adams operator constructed in this proof. Here $\bb S[\Pic]$ is defined to be the Zariski sheafification of $Y\mapsto\mathbb{S}[\scr O(Y)^{\times}[1]]$, which we have used is Zariski locally left Kan extended from smooth $\bb Q$-schemes (since $\Pic$ is).

Finally, fixing any qcqs $\bb Q$-scheme $X$, we may restrict the previous diagram to smooth $X$-schemes to obtain a commutative diagram
\[
\begin{tikzcd}
\bb S_X[\Pic] \ar[swap]{d}{\psi^m} \ar{r} &  \K[\tfrac{1}{m}] \ar{d}{\psi^m}\\
\bb S_X[\Pic] \ar{r} &   \K[\tfrac{1}{m}]
\end{tikzcd}
\]
of presheaves on $\Sm_X$. But according to the characterising property of Construction~\ref{constr:adams-k}, that means that our $\psi^m$ agrees with that arising on smooth $X$-schemes from Annala--Iwasa's theorem.
\end{proof}

\begin{remark}[Adams operators on $K$-theory and $\HC^-$ -- historical comments]\label{rem:chw}
Forgetting filtrations in Theorem \ref{rem:discuss-adams}(1) says that, for any qcqs $\bb Q$-scheme $X$, the Adams operators on $\K(X)$ and $\HC^-(X/\bb Q)$ are naturally compatible at the level of spectra, even as $\bb E_\infty$-algebras. The question of this compatibility has a long history which we briefly review here. The reader who consults any of the earlier work on the subject should be aware of Remark~\ref{rem:discuss-adams}, especially its final paragraph on previous incompatible terminology surrounding Adams operations.

Cathelineau in 1990 \cite{Cathelineau1990} proved the compatibility for the relative $K$-/cyclic homology groups of a nilpotent ideal, modulo some later corrections by Corti\~nas--Haesemeyer--Weibel \cite[Appendix~B]{Cortinas2009}. The latter also extended the compatibility in this case to the level of spectra, having defined the Adams operator on the relative $K$-theory via the stack $\rm BGL$ (which is allowed since the relative $K$-theory of a nilpotent ideal is connective).

Geller--Weibel in 1994 \cite[Remark 0.4.1]{GellerWeibel1994} explicitly expressed the expectation that the Adams operators on $K$-/negative cyclic homology groups should be compatible. Kantorovitz in 1999 \cite{Kantorovitz1999} proved the compatibility at the level of the $K$-/Hochschild homology groups.

Corti\~nas--Haesemeyer--Weibel in 2009 proved the compatibility at the level of $K$-/negative cyclic homology groups in the case of finite type schemes over any characteristic zero field. They commented that the classical Adams operators on $K$-groups and negative cyclic homology were defined in ``very different ways,'' which posed a challenge in establishing their compatibility. The key point of our approach is that, given the results of Annala--Iwasa \cite{AnnalaIwasa2023}, we can reduce the problem to checking compatibility on the sheaf of units $\roi^\times$. Bearing in mind the tight analogy between the endomorphisms $\roi^{\times} \rightarrow \roi^{\times}, z \mapsto z^m$ in algebra and $S^1 \rightarrow S^1, z \mapsto z^m$ in topology, we see that the two Adams operators are not so different after all. In particular, our proof is very different from the one of \cite{Cortinas2009}. 
\end{remark}

\begin{example}[Affine cones]\label{example_cones_0}
Let $Y$ be a smooth, geometrically connected, projective variety over a field $k$ of characteristic zero, equipped with a fixed embedding $Y\into \bb P^N_k$ into some projective space, and let $R$ be the associated homogenous coordinate ring. The $K$-groups of $R$ were calculated in \cite{CortinasHaesemeyerWalkerWeibel2013}, and here we explain how to view their calculations in terms of our motivic cohomology group.

We recall the geometry of the situation: the blow-up $X$ of $R$ at its ``irrelevant'' maximal ideal is a line bundle over $Y$, equipped with $0$-section $i:Y\into X$ and fitting into an abstract blow-up square
\[\xymatrix{
Y\ar[r] \ar[d]\ar[r]^{i}& X\ar[d]\\
\Spec(k)\ar[r]&\Spec(R)
}\]
(see, for example, \cite[Lemma~2.2]{CortinasHaesemeyerWalkerWeibel2013}). Since $\bb Z(j)^{\bb A}$ is both $\bb A^1$-invariant and a cdh sheaf by Theorem \ref{thm:cdh}, it follows that $\bb Z(j)^{\bb A}(R)\quis \bb Z(j)^{\bb A}(k)$, which in turn is the same as $\bb Z(j)^\sub{mot}(k)$ since fields are points for the cdh topology (so that the right vertical arrow in Theorem \ref{thm:graded-pieces} is an equivalence, hence also the left arrow). The fundamental fibre sequence of Theorem \ref{thm:graded-pieces}(3) may therefore be rewritten as a calculation of the relative motivic cohomology, i.e., \begin{equation}\opp{fib}(\bb Z(j)^\sub{mot}(R)\to \bb Z(j)^\sub{mot}(k))\simeq C(j)[-2]\label{eqn_cones_0}\end{equation} where $C(j):=\opp{cofib}(L\Omega^{<j}_{R/\bb Q}\to R\Gamma_\sub{cdh}(R,\Omega^{<j}_{-/\bb Q}))$.

Writing $\tilde K(R):=\opp{fib}(K(R)\to K(k))$ for the relative $K$-theory, it now follows from the rational degeneration of the Atiyah--Hirzebruch spectral sequence (Theorem \ref{theorem_AH_SS_0}) that $\tilde K(R)$ is already a rational spectrum, with homotopy groups given by \[\tilde \K_n(R):=\bigoplus_{j\ge0}H^{2j-2-n}(C(j))\] for $n\in\bb Z$. Calculating cdh cohomologies of differential forms as in \cite[\S2]{CortinasHaesemeyerWalkerWeibel2013}, much of the previous line can be calculated explicitly, in particular all $\tilde \K_n(R)$ for $n\le 0$. We note that this is just a slight repackaging of the main results of \cite{CortinasHaesemeyerWalkerWeibel2013}, where the authors instead start with the observation that $\tilde \K(R)\simeq {\rm cofib}\big(\HC(X)\to L_\sub{cdh}\HC(X)\big)$ and then decompose cyclic homology into Adams summands.

In particular, writing $d:=\dim R=\dim Y+1$, which we assume is $\ge 2$, consider the group $H^{2d}_\sub{mot}(R,\bb Z(d))$ of ``zero cycles'' on $\Spec(R)$ (see Remark \ref{rem:zcyc}); one can show that this is given by
\[H^{2d}_\sub{mot}(R,\bb Z(d))\cong\opp{ker}\left(H^{2d-2}_\sub{Zar}(X,\Omega^{<d}_{-/\bb Q})\to H^{2d-2}_\sub{Zar}(Y,\Omega^{<d}_{-/\bb Q})\right)\cong\bigoplus_{i\ge1}H^{d-1}_\sub{Zar}(Y,\Omega^{d-1}_{-/\bb Q}(i)),\] where the first isomorphism is obtained from (\ref{eqn_cones_0}) and the second as in \cite[Proposition~2.10]{CortinasHaesemeyerWalkerWeibel2013}. If $k$ is algebraic over $\bb Q$ then the right side of the previous line vanishes by Serre duality and we deduce that $H^{2d}_\sub{mot}(R,\bb Z(d))=0$.
\end{example}

\subsection{Characteristic $p > 0$}\label{sec:charp}
We first give a with a quick overview of syntomic cohomology in characteristic $p$, in the sense of \cite{BhattMorrowScholze2}. The reader should refer to \cite[\S 8]{BhattMorrowScholze2}, \cite[\S6.2]{AntieauMathewMorrowNikolaus}, and \cite[\S5.1]{BachmannElmantoMorrow} for more details.

For any $\bb F_p$-algebra $A$, let $W_r\Omega^j_{A,\sub{log}}$ denote the global sections of the subsheaf $W_r\Omega^j_\sub{log}$ of the de Rham--Witt sheaf $W_r\Omega^j_{\Spec(A)}$ which is generated \'etale locally (or, equivalently, Zariski locally \cite[Corollary~4.2(i)]{Morrow_pro_GL2}) by $\tfrac{d[f_1]}{f_1}\wedge\cdots\wedge \tfrac{d[f_j]}{f_j}$ for units $f_1,\dots,f_j$. Alternatively \cite[Corollary~4.2(iii)]{Morrow_pro_GL2}, $W_r\Omega^j_{A,\sub{log}}$ is the kernel of the Artin--Schreier map \begin{equation}C^{-1}-1:W_r\Omega^j_A\To W_r\Omega^j_A/dV^{r-1}\Omega^j_A.\label{eqn_WOmegalog}\end{equation} Since the de Rham-Witt sheaves have no higher cohomology on affines and the Artin--Schreier map is \'etale locally surjective, the previous observations may alternatively be expressed as a fibre sequence \[R\Gamma_\sub{\'et}(A,W_r\Omega^j_\sub{log})\To W_r\Omega^j_A\stackrel{C^{-1}-1}\To W_r\Omega^j_A/dV^{r-1}\Omega^j_A;\]
in particular the cohomology of $R\Gamma_\sub{\'et}(A,W_r\Omega^j_\sub{log})$ is concentrated in degrees zero and one. 

\begin{definition}\label{def:syn}
For any $r\ge1$ and $j\ge0$, the {\em mod-$p^r$, weight-$j$ syntomic cohomology} of $\bb F_p$-algebras \[\bb Z_p(j)^\sub{syn}(-)/p^r:\text{CAlg}_{\bb F_p}\to \rm{D}(\bb Z)\] is defined to be the left Kan extension of $R\Gamma_\sub{\'et}(-,W_r\Omega^j_\sub{log})[-j]$ along the inclusion $\text{CAlg}_{\bb F_p}^\Sigma\subseteq \text{CAlg}_{\bb F_p}$; here $\text{CAlg}^{\Sigma}_{\bb F_p}$ denotes the category of finitely generated polynomial $\bb F_p$-algebras. When $r=1$ we will often write $\bb F_p(j)^\sub{syn}$ to simplify notation. For $j<0$ we set $\bb Z_p(j)^\sub{syn}(-)/p^r:=0$.

Taking the inverse limit over $r$, the weight-$j$ syntomic cohomology of an $\bb F_p$-algebra $A$ is defined by 
\[
\bb Z_p(j)^\sub{syn}(A):=\lim_r\bb Z_p(j)^\sub{syn}(A)/p^r.
\]
Multiplication of differential forms induces a natural $\bb E_\infty$-algebra structure on $\bb Z_p(\star)^\sub{syn}(A)\in\Gr\rm D(\bb Z)$. By convention we declare that $\bb Z_p(j)^\sub{syn}(A)=0$ when $j<0$.
\end{definition}

\begin{remark}\label{remark_WOmegalog}
\begin{enumerate}
\item Taking $\bb Z_p(j)^\sub{syn}(A)$ modulo $p^r$ does recover $\bb Z_p(j)^\sub{syn}(A)/p^r$ as it was initially defined, thanks to Illusie's short exact sequence of \'etale sheaves $0\to W_{s}\Omega^j_\sub{log}\xto{p^r}W_{r+s}\Omega^j_\sub{log}\to W_r\Omega^j_\sub{log}\to 0$ on smooth $\bb F_p$-schemes \cite[\S I.5.7]{Illusie1979}.
\item For any $\bb F_p$-algebra $A$ and $r\ge 1$ there is, by construction, a natural comparison map \[\bb Z_p(j)^\sub{syn}(A)/p^r\To R\Gamma_\sub{\'et}(A,W_r\Omega^j_\sub{log})[-j].\] It is an equivalence whenever $A$ is regular Noetherian, or more generally Cartier smooth \cite[Prop.~5.1]{KellyMorrow2021}.
\end{enumerate}
\end{remark}

Syntomic cohomology can be loosely controlled via the cotangent complex through the following lemma:

\begin{lemma}\label{lem_fin_fil_on_syn}
For any $\bb F_p$-algebra $A$, the complex $\bb F_p(j)^\sub{syn}(A)$ admits a natural finite increasing filtration in $\textrm{D}(\bb F_p)$, of length $2(j+1)$, with graded pieces given in increasing order by
\begin{align*}
&L_{A/\bb F_p}^j[-j-1], L_{A/\bb F_p}^{j-1}[-j],L_{A/\bb F_p}^{j-2}[-j+1],\dots, L_{A/\bb F_p}^0[-1],\\
& L_{A/\bb F_p}^0[0], L_{A/\bb F_p}^1[-1], L_{A/\bb F_p}^2[-2],\dots, L_{A/\bb F_p}^j[-j].
\end{align*}
\end{lemma}
\begin{proof}
The key is to show the following claim: for $R$ any smooth $\bb F_p$-algebra, then $\Omega_R^j/d\Omega^{j-1}_R$ admits a natural finite increasing filtration (in $\mathrm D(\bb F_p)$, not as submodules) of length $2j+1$ with graded pieces in increasing order
\begin{align*}
&\Omega_R^j,\Omega_R^{j-1}[1],\Omega_R^{j-2}[2],\dots,\Omega_R^0[j],\\
&\Omega_R^0[j+1],\Omega_R^1[j],\Omega_R^2[j-1],\dots,\Omega_R^{j-1}[2].
\end{align*}
The case $j=0$ (when the bottom row of the listed graded pieces is empty) is trivial; we proceed by induction to treat the case $j>0$, so assume that we already have the filtration on $\Omega_R^{j-1}/d\Omega^{j-2}_R$, i.e., \[\Fil_0(\Omega_R^{j-1}/d\Omega^{j-2}_R)\to \Fil_1(\Omega_R^{j-1}/d\Omega^{j-2}_R)\to\cdots\to \Fil_{2j-1}(\Omega_R^{j-1}/d\Omega^{j-2}_R)=\Omega_R^{j-1}/d\Omega^{j-2}_R,\] with the desired graded pieces. Then we define, for $i=1,\dots,2j$, the filtered step $\Fil_i(\Omega_R^j/d\Omega^{j-1}_R)$ to be the pullback
\[\xymatrix{
\Fil_{i-1}(\Omega_R^{j-1}/d\Omega^{j-2}_R)[1]\ar[r] & \Omega_R^{j-1}/d\Omega^{j-2}_R[1]\ar[r]^{\pi[1]} & \Omega_R^{j-1}/\ker d[1]\\
\Fil_i(\Omega_R^j/d\Omega^{j-1}_R)\ar@{-->}[u]\ar@{-->}[rr]&&\ar[u]_{\delta}\Omega_R^j/d\Omega^{j-1}_R
}\]
Here $\pi$ is the canonical quotient map with kernel $H^{j-1}_\sub{dR}(R)$, and $\delta$ is the connecting map associated to the short exact sequence 
\begin{equation}0\To\Omega_R^{j-1}/\ker d\xto{d}\Omega_R^j\To\Omega_R^j/d\Omega_R^{j-1}\To 0\label{eqn_fil0}.\end{equation}

Since pulling back a filtration does not change the graded pieces, we see at once that $\gr_i(\Omega_R^j/d\Omega^{j-1}_R)$ is as desired for $i=1,\dots,2j-1$.

We now set $\Fil_0(\Omega_R^j/d\Omega^{j-1}_R):=0$ and $\Fil_{2j+1}(\Omega_R^j/d\Omega^{j-1}_R)=\Omega_R^j/d\Omega^{j-1}_R$; we must show that $\gr_0(\Omega_R^j/d\Omega^{j-1}_R)$ and $\gr_{2j}(\Omega_R^j/d\Omega^{j-1}_R)$ are as desired. Firstly, $\gr_0(\Omega_R^j/d\Omega^{j-1}_R)=\Fil_1(\Omega_R^j/d\Omega^{j-1}_R)$, which was defined to be the pullback of $\Fil_0(\Omega_R^{j-1}/d\Omega^{j-2}_R)=0\to\Omega^{j-1}/\ker d[1]$ along $\delta$; that is, it is given by $\opp{fib}(\delta)$, which is indeed $\Omega^j_R$ thanks to (\ref{eqn_fil0}). Secondly, $\gr_{2j}(\Omega_R^j/d\Omega^{j-1}_R)$ is precisely $\opp{cofib}(\pi)[1]=H^{j-1}_\sub{dR}(R)[2]$, which identifies with $\Omega_R^{j-1}[2]$ via the Cartier isomorphism.

This completes the proof of the existence of the filtration on $\Omega_R^j/d\Omega^{j-1}_R$ when $R$ is smooth. We then obtain the desired filtration on $\bb F_p(j)^\sub{syn}(R)$ by recalling that $\bb F_p(j)^\sub{syn}(R)=\text{fib}(\Omega_R^j\xto{C^{-1}-1}\Omega_R^j/d\Omega_R^{j-1})[-j]$. Finally the desired filtration on $\bb F_p(j)^\sub{syn}(A)$, for arbitrary $\bb F_p$-algebras $A$, is obtained by left Kan extension from the smooth case.
\end{proof}

As a consequence of the previous lemma and fpqc descent for the cotangent complex, we deduce the following:

\begin{corollary} The presheaves $\bb Z_p(j)^\sub{syn}$ satisfy fpqc descent on the category of $\bb F_p$-algebras.
\end{corollary}
Therefore, by right Kan extension, they extend uniquely to fpqc sheaves
\[
\bb Z_p(j)^\sub{syn}:\text{Sch}_{\bb F_p}^\sub{qcqs,op}\To\rm D(\bb Z),
\] thereby defining syntomic cohomology in the non-affine case; they again assemble into a presheaf valued in $\bb E_\infty$-algebras in graded complexes. Just as derived de Rham cohomology appeared in characteristic zero through the HKR filtration on negative cyclic homology, syntomic cohomology similarly appears through topological cyclic homology:

\begin{theorem}[BMS filtration \cite{BhattMorrowScholze2}]\label{thm_BMS2+}
For any qcqs $\bb F_p$-scheme $X$, its topological cyclic homology $\TC(X)$ admits a natural, multiplicative, complete, $\bb N$-indexed filtration $\Fil^\star_\sub{BMS} \TC(X)$ with graded pieces naturally and multiplicatively given by \[\gr^j_\sub{BMS}\TC(X)\simeq\bb Z_p(j)^\sub{syn}(X)[2j]\] for $j\ge0$. Moreover,
\begin{enumerate}
\item The filtration is bounded, i.e., there exists $d\ge0$ (depending on $X$) such that, for any $j\ge0$, the filtered step $\Fil^j_\sub{BMS}\TC(X)$ is supported in homological degrees $\ge j-d$ (and so the syntomic cohomology $\bb Z_p(j)^\sub{syn}(X)$ is supported in cohomological degrees $\le j+d$).
\item The induced filtration on $\TC(X)[\tfrac1p]$ is naturally split, so that there is a natural multiplicative equivalence of filtered spectra
\[
\Fil^\star_\sub{BMS}\TC(X)[\tfrac1p]\simeq\bigoplus_{j\ge \star}\bb Q_p(j)^\sub{syn}(X)[2j]
\] where $\bb Q_p(j)^\sub{syn}(X):=\bb Z_p(j)^\sub{syn}(X)[\tfrac1p]$.
\end{enumerate}
\end{theorem}
\begin{proof}
We first recall that $\TC$ and its $p$-completion $\TC^\comp_p \simeq \TC(-;\bb Z_p)$ agree on qcqs $\bb F_p$-schemes; this will used throughout the rest of the paper without further mention.

The existence of a filtration on $\TC$ with graded pieces given by shifts of syntomic cohomology is one of the main theorems of \cite{BhattMorrowScholze2}, in the case of quasisyntomic $\bb F_p$-algebras. It was extended, by $p$-completed left Kan extension, to all $\bb F_p$-algebras in \cite{AntieauMathewMorrowNikolaus}. It is then obtained for arbitrary qcqs $\bb F_p$-schemes by right Kan extension. It remains to explain (1) and (2).

The proof of part (1) proceeds via several cases. Firstly, for any quasisyntomic $\bb F_p$-algebra $R$, the BMS filtration on $\TC(R)$ is defined by descent from quasiregular semiperfect rings of the two-speed Postnikov filtration; the latter is manifestly complete, which is preserved by the descent. In particular, for smooth $\bb F_p$-algebras $R$, the BMS filtration on $\TC(R)$ is complete and each of its graded pieces $\gr_\sub{BMS}^j\TC(R)\simeq \text{lim}_rR\Gamma_\sub{et}(R,W_r\Omega_\sub{log}^r)[j]$ is supported in cohomological degrees $[-j,-(j+1)]$; using the completeness it follows in this case that $\Fil_\sub{BMS}^j\TC(R)$ is supported in homological degrees $\ge j-1$. By left Kan extending we see that $\Fil_\sub{BMS}^j\TC(A)$ is supported in homological degrees $\ge j-1$ for all $\bb F_p$-algebras $A$; so the filtration is bounded on affines. Finally, right Kan extending preserves completeness, so we deduce that the BMS filtration on $\TC(X)$, for any qcqs $\bb F_p$-scheme $X$, is at least complete. Therefore it is enough to check boundedness of the filtration on graded pieces, as mentioned in \S\ref{ss_filtrations}; we do this next.

Since $X$ is qcqs, it has finite cohomological dimension for quasi-coherent sheaves (even if it does not have finite Krull dimension), and we take $d$ to be one plus this dimension. Then the Zariski sheafification of each graded piece of Lemma \ref{lem_fin_fil_on_syn} has global sections supported in cohomological degrees $\le j+d$, as required.

For part (2) note first that the absolute Frobenius $\phi:X\to X$ induces a natural endomorphism of $\TC(X)$, compatible (by functoriality) with the BMS filtration; its action on the graded piece $\bb Z_p(j)^\sub{syn}[2j]$ is as multiplication by $p^j$ (this follows by left Kan extending the same statement for $R\Gamma_\sub{\'et}(-,W_r\Omega^j_\sub{log})$ of finitely generated polynomial algebras, for all $r\ge0$). Next observe that the BMS filtration on $\TC(X)[\tfrac1p]$ is bounded thanks to part (1), therefore complete. Since $\phi-p^j$ acts invertibly on $\bb Q_p(i)^\sub{syn}(X)$for $i\ge j+1$, we deduce that it acts invertibly on $\Fil^{j+1}_\sub{BMS}\TC(X)[\tfrac1p]$. Similarly it acts invertibly on $\TC(X)[\tfrac1p]/\Fil^{j}\TC(X)[\tfrac1p]$. Taking $\phi-p^j$-fixed points, we have shown that the maps \[\bb Z_p(j)^\sub{syn}(X)[2j]= \gr^j_\sub{BMS}\TC(X)\longleftarrow (\Fil^{j}_\sub{BMS}\TC(X))^{\phi=p^j}\To \TC(X)^{\phi=p^j}\] are equivalences after inverting $p$. Repeating the same arguments in the last three paragraphs of the proof of Theorem~\ref{theorem_AH_SS_0} gives us the equivalence stated in (2).
\end{proof}

\begin{remark}[Variant: cdh-local BMS filtration] \label{rem:cdh-local}
Cdh sheafifying the BMS filtration levelwise we see that, for any qcqs $\bb F_p$-scheme $X$, there exists a functorial, multiplicative, $\bb N$-indexed filtration 
\[
\Fil^\star_\sub{BMS}L_\sub{cdh}\TC(X):=L_{\cdh}\Fil^\star_\sub{BMS}\TC(X),
\] on $L_\sub{cdh}\TC(X)$ whose graded pieces are naturally and multiplicatively given by \[\gr^j_\sub{BMS}L_\sub{cdh}\TC(X)\simeq L_\sub{cdh}\bb Z_p(j)^\sub{syn}(X)[2j]\]  for $j\ge0$. Here $L_\sub{cdh}\bb Z_p(j)^\sub{syn}$ is the cdh sheafification of the presheaf $\bb Z_p(j)^\sub{syn}:\Sch^\sub{qcqs,op}_{\bb F_p}\to\text D(\bb Z)$.

We record the following vanishing bound which will be needed frequently: for a qcqs $\bb F_p$-scheme $X$ of finite valuative dimension $\le d$, each $\Fil^j_\sub{BMS}L_\sub{cdh}\TC(X)$ is supported in homological degrees $\ge j-d-1$, and the cdh sheafified syntomic cohomology $L_\sub{cdh}\bb Z_p(j)^\sub{syn}(X)$ is supported in cohomological degrees $\le j+d+1$). Indeed, on any affine the vanishing bounds appearing in Theorem \ref{thm_BMS2+}(1) may be taken to be $\ge j-1$ and $\le j+1$ (see the proof: affines have cohomological dimension $0$ for quasi-coherent sheaves), whence the desired bound follows from cdh cohomological dimension being bounded by valuative dimension \cite[Proposition 2.4.3]{ElmantoHoyoisIwasaKelly2021}.
\end{remark}

\begin{remark}[éh motivic cohomology]\label{remark_eh}
Cdh sheafified syntomic cohomology appears in Geisser's theory of éh motivic cohomology and arithmetic cohomology \cite{Geisser2006}. More precisely, we claim that there are multiplicative equivalences
\begin{equation}L_\sub{cdh}\bb Z_p(j)^\sub{syn}/p^r\xrightarrow{\simeq} L_\sub{cdh}R\Gamma_\sub{\'et}(-,W_r\Omega^j_\sub{log})[-j]\xrightarrow{\simeq} R\Gamma_\sub{\'eh}(-,W_r\Omega^j_\sub{log})[-j]\label{eqn_Geisser_eh}\end{equation} of presheaves on qcqs $\bb F_p$-schemes. Indeed, globalising Remark \ref{remark_WOmegalog}(2) defines comparison maps \[\bb Z_p(j)^\sub{syn}/p^r\To R\Gamma_\sub{\'et}(-,W_r\Omega^j_\sub{log})[-j]\] for all $j\ge0$, which are isomorphisms on any valuation ring since they are Cartier smooth \cite[\S2]{KellyMorrow2021} \cite[\S5.1]{LuedersMorrow2023}; cdh sheafifying therefore defines the first equivalence in \eqref{eqn_Geisser_eh}. The second equivalence follows from Theorem~\ref{theorem:eh}.
\end{remark}

We warn the read that, since sheafification does not commute with cofiltered limits in general, there is no reason that $L_\sub{cdh}\bb Z_p(j)^\sub{syn}$ should land in $p$-complete complexes.\footnote{For example, letting $A=\bb F_p[t^{1/p^\infty}]/(t-1)$, one can show that $H^1$ of $\opp{fib}(\bb Z_p(1)^\sub{syn}(A)\to L_\sub{cdh}\bb Z_p(1)^\sub{syn}(A))$ is the principal units $\opp{ker}(A^\times\to\bb F_p^\times)$, which is not bounded $p$-power torsion. This also shows that the cdh sheaf $L_\sub{cdh}\bb Z_p(1)^\sub{syn}$ is not invariant for the nil (but not nilpotent) ideal $\opp{ker}(A\to\bb F_p)$.

In fact, such an issue already occurred in characteristic zero: $R\Gamma_\sub{cdh}(-,\hat{L\Omega}_{-/\bb Q})=R\Gamma(-,\hat{L\Omega}_{-/\bb Q})$ was not necessarily complete with respect to the cdh sheafification of the Hodge filtration.
} It is therefore not always sufficient to analyse it modulo $p$. Fortunately, after inverting $p$, the cdh sheafification disappears; this will be required to control our motivic cohomology after inverting $p$:

\begin{lemma}\label{lemma_Qpsyn}
The presheaf $\bb Q_p(j)^\sub{syn} :\mathrm{Sch}^\sub{qcqs}_{\bb F_p}\to \mathrm D(\bb Z)$ is a cdh sheaf, i.e., the map $\bb Z_p(j)^\sub{syn}\to L_\sub{cdh}\bb Z_p(j)^\sub{syn}$ is an equivalence after inverting $p$.
\end{lemma}
\begin{proof}
Since $\bb Q_p(j)^\sub{syn}$ is a direct summand of $\TC[\tfrac1p]$ by Theorem \ref{thm_BMS2+}(2), it is sufficient to check that the latter is a cdh sheaf. In other words, since cdh sheafification commutes with inverting $p$, we must show that the map $\TC(X)[\tfrac1p]\to (L_\sub{cdh}\TC(X))[\tfrac1p]$ is an equivalence for any qcqs $\bb F_p$-scheme $X$. But this follows from Theorem \ref{thm:mainsq} and the result of Weibel that $\K(X)[\tfrac1p]\quis\KH(X)[\tfrac1p]$ \cite{Weibel1989a}.
\end{proof}

Occasionally Lemma \ref{lemma_Qpsyn} will be insufficient and we will need to know that the failure of $\bb Z_p(j)^\sub{syn}$ to satisfy cdh descent is controlled by a bounded power of $p$ in each degree; this happens precisely when $L_\sub{cdh}\bb Z_p(j)^\sub{syn}$ happens to take a $p$-complete value:

\begin{lemma}\label{lemma_derived_p_complete}
For any fixed qcqs $\bb F_p$-scheme $X$ and $j\ge0$ the following are equivalent:
\begin{enumerate}[(a)]
\item Each cohomology group of $\opp{fib}(\bb Z_p(j)^\sub{syn}(X)\to L_\sub{cdh}\bb Z_p(j)^\sub{syn}(X))$ is bounded $p$-power torsion.
\item $L_\sub{cdh}\bb Z_p(j)^\sub{syn}(X)$ is derived $p$-complete.
\end{enumerate}
These equivalent conditions hold for all $j\ge0$ in each of the following cases:
\begin{enumerate}
\item $X$ a quasi-excellent Noetherian $\bb F_p$-scheme of finite Krull dimension;
\item $X$ a qcqs $\bb F_p$-scheme of finite valuative dimension which admits a smooth map to the spectrum of some valuation ring;
\item $X=\Spec(R)$ where $R$ is an $\bb F_p$-algebra of finite valuative dimension which is weakly regular and stably coherent;
\item $X$ a qcqs $\bb F_p$-scheme of finite valuative dimension for which the fibre of $\K(X)\to\KH(X)$ is $p$-complete.
\end{enumerate}
\end{lemma}
\begin{proof}
The equivalence of (a) and (b) follows from the following fact, where the non-trivial implication is a result of Bhatt \cite{Bhatt2019}: an abelian group which vanishes after inverting $p$ is derived $p$-complete if and only if each of its cohomology groups is bounded $p$-power torsion.

We claim that condition 4 on $X$ implies (b). Indeed, Theorem \ref{thm_BMS2+} and its cdh-local analogue in Remark \ref{rem:cdh-local} provide us with a bounded spectral sequence
\[E_2^{ij}=H^{i-j}(\fib(\bb Z_p(-j)^\sub{syn}(X)\to L_\sub{cdh}\bb Z_p(-j)^\sub{syn}(X))\Longrightarrow \pi_{-i-j}\fib(\TC(X)\to L_\sub{cdh}\TC(X)).\] The homotopy groups appearing in the abutment are all bounded $p$-power torsion: indeed, we may replace the fibre by that of $\K(X)\to\KH(X)$ by Theorem \ref{thm:mainsq}, which vanishes both after inverting $p$ (by a classical result of Weibel) and is derived $p$-complete (by hypothesis), so are bounded $p$-power torsion by the previous paragraph. But the action of the absolute Frobenius forces the spectral sequence to degenerate up to bounded $p$-power torsion (see the second paragraph of the proof of Theorem \ref{thm:p-ahss} for details of such an argument), so all terms on the $E_2$ page are also bounded $p$-power torsion, thereby establishing (b).

Meanwhile, we have 2 $\implies$ 3 by \cite[Corollary 2.3]{AntieauMathewMorrow2022} and 3 $\implies$ $\K(X)\quis\KH(X)$ by \cite[Proposition 2.4]{AntieauMathewMorrow2022}, which clearly implies 4.

To complete the proof it remains to show that 1$\implies 4$. The key non-trivial observation is that $\K(-;\bb Q_p):=\K^{\comp}_p[\tfrac1p]$ is an h-sheaf on qcqs $\bb F_p$-scheme. To prove this we recall from Theorem \ref{thm:mainsq} that the fibre of $\K\to\TC$ is a cdh sheaf, and so $p$-completing shows that the fibre of $\K^{\comp}_p \to \TC$ is also a cdh sheaf, and therefore also the fibre of $\K(-;\bb Q_p)\to \TC[\tfrac1p]$. But, as used in the proof of Lemma~\ref{lemma_Qpsyn}, we already know that $\TC[\tfrac1p]$ is a cdh sheaf. This shows that $\K(-;\bb Q_p)$ is a cdh sheaf; but it is also rational and has transfers along finite flat maps, therefore satisfies finite flat descent and is an h-sheaf \cite[Remark 2.11]{BachmannElmantoMorrow}.

More easily, $\KH(-;\bb Q_p)$ is a cdh sheaf (since $\KH$ is) with transfers, so is also an h-sheaf. Therefore $\fib(\K(-;\bb Q_p)\to\KH(-;\bb Q_p))$ is an h-sheaf which vanishes on regular Noetherian schemes, and so vanishes on all quasi-excellent Noetherian $\bb F_p$-schemes by the existence of alterations. 

Writing $F:=\fib(\K\to\KH)$ for simplicity we now know, for for any quasi-excellent Noetherian $\bb F_p$-scheme $X$, that both $F(X)[\tfrac1p]$ and $F(X;\bb Q_p)$ vanish. Therefore $F(X)\to F(X)^{\comp}_p$ is an equivalence, as required.
\end{proof}

We next establish an analogue of Proposition \ref{prop:compat}, namely that the trace map in characteristic $p$ is compatible with the motivic and BMS filtrations:

\begin{proposition}\label{prop:mot-v-bms}
Let $k$ be a field of characteristic $p$. Then the trace map $\mathrm{tr}:\K\to\TC$, viewed as a map between spectra-valued presheaves on $\Sm_k$, admits a unique, multiplicative extension to a map of filtered presheaves $\Fil_\bb A^{\star}\K\to \Fil_\sub{BMS}^{\star}\TC$.
\end{proposition}
\begin{proof}
We apply the same $t$-structure argument as Propositions \ref{prop:compat}. Step 1 of that proof was independent of the characteristic, and so shows in the present context that $\Fil^j_\bb A\K$ is $j$-connective for each $j\ge0$. It remains to check that $\Fil^{<j}_\sub{BMS}\TC(R)$ vanishes in cohomological degrees $\le -j$ for each smooth $k$-algebra $R$. But for each $i=0,\dots,j-1$ the $i^\sub{th}$ graded piece is $\text{gr}_\sub{BMS}^i\TC(R)\simeq \bb Z_p(i)^\sub{syn}(R)[2i]$, where $\bb Z_p(i)^\sub{syn}(R)$ is supported in cohomological degrees $[i,i+1]$; the desired vanishing bound follows by a trivial induction. Multiplicativity follows from the same argument as in characteristic zero using the Postnikov $t$-structure.
\end{proof}

As in characteristic zero, we need a cdh-local analogue of the previous proposition. It would follow formally from the previous proposition, exactly as in characteristic zero, if we were to know that $\bb Z_p(j)^\sub{syn}$ and $L_\sub{cdh}\bb Z_p(j)^\sub{syn}$ coincided on smooth $\bb F_p$-schemes; but we will not know this until Corollary \ref{corol_e_vs_eh} so an alternative argument is required. 

\begin{proposition}\label{prop_cdh_filtered_trace_p}
The cdh-local trace map $L_\sub{cdh}\mathrm{tr}:\KH\to L_\sub{cdh}\TC$, viewed as a map between spectra-valued presheaves on $\Sch_{\bb F_p}^\sub{qcqs}$, admits a unique extension to a multiplicative map of filtered presheaves \[\Fil^{\star}_{\bb A}\KH\To \Fil^{\star}_\sub{BMS}L_\sub{cdh}\TC\] (the filtration on the left being the $\bb A^1$-motivic filtration of Theorem \ref{thm:cdh}(1); the filtration on the right is the cdh-local BMS filtration of Remark~\ref{rem:cdh-local}).
\end{proposition}
\begin{proof}
We begin with an argument which is essentially the same as the second half of the proof of Corollary~\ref{corol_cdh_filtered_trace}. Namely, since $\Fil^{\star}_{\bb A}\KH$ is cdh-locally left Kan extended from smooth $\bb F$-schemes by Theorem~\ref{thm:cdh}, it suffices to prove that  the map of spectra-valued presheaves $\K\to (L_\sub{cdh}\TC)|_{\sub{Sm}_{\bb F_p}}$ on $\text{Sm}_{\bb F_p}$ admits a unique extension to a multiplicative map of filtered presheaves $\Fil^\star_\bb A\K\to (\Fil^\star_\sub{BMS}L_\sub{cdh}\TC)|_{\sub{Sm}_{\bb F_p}}$.

To prove the claim we apply the same $t$-structure argument as in Propositions \ref{prop:compat} and \ref{prop:mot-v-bms}. We have already noted in the proof of Proposition \ref{prop:mot-v-bms} that $\Fil^j_\bb A\K$ is $j$-connective for any $j\ge0$, so it remains only to show that  $L_\sub{cdh}\Fil^{<j}_\sub{BMS}\TC(X)$ vanishes in cohomological degrees $\le -j$ for each smooth $\bb F_p$-scheme $X$; by induction it is enough to check that $L_\sub{cdh}\bb Z_p(i)^\sub{syn}(X)$ is supported in cohomological degrees $\ge i$ for each $i\ge0$. But $L_\sub{cdh}\bb Z_p(i)^\sub{syn}(X)$ is derived $p$-complete by Lemma \ref{lemma_derived_p_complete}, so it is enough to check this coconnectivity modulo $p$, where it follows from \eqref{eqn_Geisser_eh}.
\end{proof}

\begin{definition}\label{def:charp}
For a qcqs $\bb F_p$-scheme $X$, let $\Fil^\star_\sub{mot}\K(X)$ be the pullback in $\bb E_\infty$-algebras in filtered spectra of the diagram 
\begin{equation}\label{eq:motfilt_charp}
\begin{tikzcd}
\mathrm{Fil}^{\star}_{\mot}\K(X) \ar[dotted]{d} \ar[dotted]{r}    & \mathrm{Fil}^{\star}_\sub{BMS}\TC(X)\ar{d} \\
\mathrm{Fil}^\star_{\bb A}\KH(X) \ar{r} & \Fil^{\star}_\sub{BMS}L_\sub{cdh}\TC(X),
\end{tikzcd}
\end{equation}
where the bottom map is given by Proposition~\ref{prop_cdh_filtered_trace_p}. 

For $j\in\bb Z$, define the {\em weight-$j$ motivic cohomology} of $X$ to be \[\bb Z(j)^\sub{mot}(X):=(\gr^j_\sub{mot}\K(X))[-2j],\] which we will see in Theorem \ref{thm:graded-pieces_charp} naturally lifts to $\rm D(\bb Z)$ and vanishes for $j<0$; in particular $\bb Z(\star)^\sub{mot}(X)$ is an $\bb E_\infty$-algebra in $\rm D(\bb Z)$. The associated motivic cohomology groups, for $i\in\bb Z$, are $H^i_\sub{mot}(X,\bb Z(j)):=H^i(\bb Z(j)^\sub{mot}(X))$.
\end{definition}

Here are some fundamental properties of our motivic cohomology in characteristic $p$:

\begin{theorem}\label{thm:graded-pieces_charp}
Let $j\in \bb Z$. For any qcqs $\bb F_p$-scheme $X$, the weight-$j$ motivic cohomology $\bb Z(j)^\sub{mot}(X)$ has the following properties:
\begin{enumerate}
\item $\bb Z(j)^\sub{mot}(X)=0$ for $j<0$.
\item There is a natural pullback square
\[
\begin{tikzcd}
\bb Z(j)^\sub{mot}(X) \ar{r} \ar{d} & \bb Z_p(j)^\sub{syn}(X)\ar{d}\\
\bb Z(j)^{\bb A}(X) \ar{r} & L_\sub{cdh}\bb Z_p(j)^\sub{syn}(X).
\end{tikzcd}
\]
(Varying $j$, this is a pullback of $\bb E_\infty$-algebras in graded complexes; see Example \ref{example_0p}.)
\item The canonical map $\bb Z(j)^\sub{mot}(X)[\tfrac1p]\to\bb Z(j)^{\bb A}(X)[\tfrac1p]$ is an equivalence. In particular, $\bb Z(j)^\sub{mot}[\tfrac1p]$ is a cdh sheaf on $\Sch^\sub{qcqs,op}_{\bb F_p}$.

\item The presheaf $\bb Z(j)^\sub{mot}:\Sch_{\bb F_p}^\sub{qcqs,op}\to\mathrm{D}(\bb Z)$ is a finitary Nisnevich sheaf.

\end{enumerate}
\end{theorem}
\begin{proof}
(1): The three corners in (\ref{eq:motfilt_charp}) used to define the pullback are $\bb N$-indexed, whence the same is true of the pullback $\Fil^\star_\sub{mot}\K(X)$, i.e., the graded pieces vanish in negative weights. Part (2) is obtained by taking graded pieces in the pullback square (\ref{eq:motfilt_charp}).

(3): Lemma \ref{lemma_Qpsyn} states that $\text{fib}(\bb Z_p(j)^\sub{syn}\to L_\sub{cdh}\bb Z_p(j)^\sub{syn})[\tfrac1p] \simeq 0$, whence the result follows from the cartesian square in (2). Recall that $\bb Z(j)^\bb A$ is a cdh sheaf by Theorem \ref{thm:cdh}(2).

(4): It suffices to prove that $\bb Z(j)^\sub{mot}[\tfrac1p]$ and $\bb Z(j)^\sub{mot}/p$ are finitary. The first follows from part (3) and Theorem \ref{thm:cdh}(2). The second reduces, via the pullback square of part (2), to finitariness of $\bb Z(j)^{\bb A}/p$ (which is indeed finitary by another application of Theorem \ref{thm:cdh}(2)), of $\bb F_p(j)^\sub{syn}$ (finitary since, on affines, it is left Kan extended from finitely generated polynomial algebras), and of $L_\sub{cdh}\bb F_p(j)^\sub{syn}$ (finitary since it is the cdh sheafification of a finitary presheaf).

\end{proof}

\begin{remark}\label{rem:fibre-finitary} Neither the presheaf $\bb Z_p(j)^\sub{syn}$ nor $L_\sub{cdh}\bb Z_p(j)^\sub{syn}$ is finitary. However, the proof of Theorem~\ref{thm:graded-pieces_charp}(4) shows that their difference, i.e, the fibre
\[
X \mapsto \opp{fib}\left(\bb Z_p(j)^{\syn}(X) \rightarrow L_\sub{cdh}\bb Z_p(j)^\sub{syn}(X)\right),
\]
is a finitary presheaf. 
\end{remark}

\begin{example}[Weight $0$ and $\bb Z$-algebra structure]\label{example_0p}
As in Example \ref{example_00} in characteristic $0$, the map \[\bb Z(0)^\sub{mot}(X)\To \bb Z(0)^{\bb A}(X)\simeq R\Gamma_\sub{cdh}(X,\bb Z)\] is an equivalence of $\bb E_\infty$-algebras in spectra for any qcqs $\bb F_p$-scheme $X$. Indeed, from the pullback square Theorem \ref{thm:graded-pieces_charp} it is enough to show that $\bb Z_p(0)^\sub{syn}$, or equivalently $\bb F_p(0)^\sub{syn}$, satisfies cdh descent. But it is easily checked from the definitions that $\bb F_p(0)^\sub{syn}\simeq R\Gamma_\sub{\'et}(-,\bb Z/p\bb Z)$, which even satisfies cdh descent by Deligne (or even arc descent \cite{BhattMathew2021}).

As in Example \ref{example_00}, this naturally upgrades $\bb Z(\star)^\sub{mot}(X)$ to be an $\bb E_\infty$-algebra in $\Gr\rm D(\bb Z)$. 
\end{example}

\begin{remark}[Definition of $\bb Z(j)^\sub{mot}$ without trace map]\label{rem:horizontal_map_charp}
Similarly to Remark \ref{rem:horizontal_map_char0}, the bottom horizontal arrow in Theorem \ref{thm:graded-pieces_charp}(2) is a priori defined as the $j^\sub{th}$-graded piece of the filtered multiplicative cdh-local trace map from Proposition \ref{prop_cdh_filtered_trace_p}, but in fact it admits an explicit description as a cdh-local syntomic cycle class map. Namely, on smooth $\bb F_p$-schemes, we begin with the composition \[\bb Z(j)^\bb A\To R\Gamma_\sub{Zar}(-,K_j^M)[-j]\xto{\sub{dlog}}\lim_rR\Gamma_\sub{Zar}(-,W_r\Omega^j_\sub{log})[-j]\To \lim_rR\Gamma_\sub{\'et}(-,W_r\Omega^j_\sub{log})[-j]\simeq \bb Z_p(j)^\sub{syn},\] where the final equivalence is Remark \ref{remark_WOmegalog}(2) for smooth $\bb F_p$-schemes.

Similarly to Remark \ref{rem:horizontal_map_char0}, left Kan extending, cdh sheafifying, and composing along a counit then defines a map of presheaves on qcqs $\bb F_p$-schemes $\bb Z(j)^\bb A\to L_\sub{cdh}\bb Z_p(j)^\sub{syn}$, which we claim is precisely the bottom horizontal arrow in Theorem \ref{thm:graded-pieces_charp}(2). By multiplicativity and running through the proofs of Propositions \ref{prop:mot-v-bms} and \ref{prop_cdh_filtered_trace_p}, this reduces to the fact that the composition
\[\K_1(\bb F_p[t^{\pm1}])\xto{\sub{tr}}\TC_1(\bb F_p[t^{\pm1}])\stackrel\simeq\leftarrow\pi_1(\Fil^1_\sub{BMS}\TC(\bb F_p[t^{\pm1}]))\to \pi_1(\gr^1_\sub{BMS}\TC(\bb F_p[t^{\pm1}]))=\lim_rW_r\Omega^1_{\bb F_p[t^{\pm1}],\sub{log}}\]
sends $t$ to $\tfrac{dt}t$.\footnote{Unlike the characteristic zero case, this calculation of the trace map is slightly involved and so we simply cite Geisser--Hesselholt: see the commutative square above \cite[Lemma 4.2.3]{GeisserHesselholt1999} (note that on smooth schemes the BMS filtration is the \'etale Postnikov filtration, as captured by Geisser--Hesselholt's map $\rho$).}
\end{remark}

Here is the analogue, in characteristic $p$, of Theorem \ref{theorem_AH_SS_0} about the existence of the Atiyah--Hirzebruch spectral sequence:

\begin{theorem}\label{thm:p-ahss}
Let $X$ be a qcqs $\bb F_p$-scheme. Then the $\bb E_\infty$-algebra in filtered spectra $\Fil^\star_\sub{mot}\K(X)$ is $\bb N$-indexed, and satisfies $\Fil^0_\sub{mot}\K(X)=\K(X)$. Moreover, 
\begin{enumerate}
\item If $X$ has finite valuative dimension then the filtration is bounded and it induces a bounded multiplicative Atiyah--Hirzebruch spectral sequence \[E_2^{ij}=H^{i-j}_\sub{mot}(X,\bb Z(-j))\implies \K_{-i-j}(X)\] which degenerates up to bounded denominators (i.e., every differential $\delta$ appearing in the spectral sequence satisfies $N\delta=0$ for some $N\ge1$ depending on the differential).
\item The filtration naturally splits rationally, i.e., there is a natural, multiplicative equivalence of filtered spectra \[\Fil^\star_\sub{mot}\K(X)_\bb Q\quis \bigoplus_{j\ge\star}\bb Q(j)^\sub{mot}(X)[2j]\]
\end{enumerate}
\end{theorem}
\begin{proof}
We have already seen in Theorem \ref{thm:graded-pieces_charp}(1) that the filtration is $\bb N$-indexed; it satisfies $\Fil^0_\sub{mot}\K(X)=\K(X)$ thanks to Theorem \ref{thm:mainsq}. The rest of the proof mimics that of Theorem \ref{theorem_AH_SS_0}, with the following two modifications.

We replace the Adams operators in characteristic zero with the action of the absolute Frobenius, using Theorem \ref{thm:frobenius} below. This implies degeneration of the Atiyah--Hirzebruch spectral sequence up to bounded denominators because the action is defined integrally, not only after inverting some integer. More precisely, the differential $\delta:E_m^{i,j}\to E_m^{i+m,j+1-m}$, where $j\le 0$, is compatible with the Frobenius, which acts as $p^{-j}$ on the domain and $p^{-j-1+m}$ on the codomain, so we have $p^{-j}(p^{m-1}-1)\delta=0$. 

Secondly, when $X$ has finite valuative dimension $\le d$ (hence also finite Krull dimension $\le d$), we know from Theorem \ref{thm:cdh}(1) that $\Fil^j_{\bb A}\KH(X)$ is supported in homological degrees $\ge j-d$, from the proof of Theorem \ref{thm_BMS2+}(1) that  $\Fil^j_\sub{BMS}\TC(X)$ supported in homological degrees $\ge j-d-1$, and from Remark~\ref{rem:cdh-local} that $\Fil^j_\sub{BMS}L_\sub{cdh}\TC(X)$ is supported in homological degrees $\ge j-d-1$. From the defining pullback square (\ref{eq:motfilt_charp}), we then see that $\Fil^j_\sub{mot}\K(X)$ is supported in homological degrees $\ge j-d-2$, which is good enough to prove the desired boundedness (but not the optimal bound: see~\S\ref{section_Weibel}).
\end{proof}

The following theorem serves as an analogue over $\bb F_p$ to the existence of Adams operators in characteristic zero (Theorem~\ref{thm:adams}):

\begin{theorem}[Action of Frobenius on motivic cohomology]\label{thm:frobenius}
For any qcqs $\bb F_p$-scheme $X$, the multiplicative graded endomorphism of $\bb Z(\star)^\sub{mot}(X)$ induced by the absolute Frobenius $\phi:X\to X$ is naturally homotopic to the endomorphism multiplication-by-$p^{\star}$.
\end{theorem}
\begin{proof} 
We begin by sketching a proof of the analogous claim for the classical motivic cohomology of smooth $\bb F_p$-schemes. Firstly, on smooth $\bb F_p$-schemes, the Geisser--Levine theorem states that $\bb Z_p(\star)^\bb A\simeq \lim_r R\Gamma_\sub{Zar}(-,W_r\Omega^\star_\sub{log})[-\star]$ \cite{GeisserLevine2000} (see also \cite[Theorem~6.14]{BachmannElmantoMorrow}), and it is clear from the description of $W_r\Omega^\star_\sub{log}$ that the Frobenius acts as $p^\star$. Secondly, still restricting to smooth $\bb F_p$-schemes, we claim that the Frobenius action on $\bb Z(\star)^\bb A[\tfrac1p]$ is also multiplication-by-$p^\star$. To prove this one notes that the absolute Frobenius induces an endomorphism $\phi_{\KGL}$ of the $\bb E_\infty$-algebra $\KGL_{\bb F_p}[\tfrac1p]\in\SH(\bb F_p)$, uniquely characterised by the following properties: the induced action on $\omega^\infty\KGL_{\bb F_p}[\tfrac1p]=\K[\tfrac1p]|_{\Sm_{\bb F_p}}$ is that induced by the absolute Frobenius; and the composition \[\bb T_{\bb F_p}\xto{\sub{Bott}}\KGL_{\bb F_p}[\tfrac1p]\xto{\phi_{\KGL}}\KGL_{\bb F_p}[\tfrac1p]\] makes $\KGL_{\bb F_p}[\tfrac1p]$ into a Bott periodic motivic ring spectrum. Indeed, the existence of $\phi_{\KGL}$ follows from the standard theory of Bott period motivic ring spectra \cite[Proposition 3.2]{Hoyois2020} and fact that the absolute Frobenius acts as multiplication-by-$p$ on the element $1-[\roi(-1)]\in \K_0(\bb P_{\bb F_p}^1)$. Equipped with $\phi_{\KGL}$, we claim that its induced action on $s^\star\KGL_{\bb F_p}[\tfrac1p]\in\CAlg(\Gr\SH(\bb F_p))$ is homotopic to multiplication-by-$p^\star$; as in the proof of Lemma~\ref{lemma_adams_on_smooth} this may be checked after rationalising (even after forgetting multiplicative structures, but we do not do this); then by Bott periodicity $s^*(\KGL_{\bb F_p})_\bb Q$ is the free $s^0(\KGL_{\bb F_p})_\bb Q$-algebra on the Tate motive $\bb T_{\bb F_p}$ in degree $0$, and the desired description of the action of $\phi_{\KGL}$ follows from the aforementioned behaviour of the absolute Frobenius on $1-[\roi(-1)]$.

We now consider the usual arithmetic pullback square of presheaves of graded $\bb E_\infty$-algebras of complexes on smooth $\bb F_p$-schemes:
\[\xymatrix{
\bb Z(\star)^\bb A\ar[r]\ar[d] & \bb Z(\star)^\bb A[\tfrac1p]\ar[d] \\
\bb Z_p(\star)^\bb A\ar[r] & \bb Z_p(\star)^\bb A[\tfrac1p].
}\]
All maps in the diagram are, by naturality, compatible with the endomorphisms induced by the absolute Frobenius and with multiplication-by-$p^\star$. We have moreover shown that the absolute Frobenius and multiplication-by-$p^\star$ are homotopic on both the bottom left (therefore also on the bottom right, by inverting $p$) and on top right. As in the proof of Theorem \ref{thm:adams} to obtain such a homotopy on the top left we must produce a homotopy between homotopies, and it is enough to show that the mapping space
\[\Map(\bb Z(\star)^\bb A[\tfrac1p],\bb Z_p(\star)^\bb A[\tfrac1p])\] in $\Shv_\sub{Zar}(\Sm_{\bb F_p},\CAlg(\Gr\mathrm{D}(\bb Z)))$ is discrete. But this follows from the same $t$-structure argument as in the proof of Proposition \ref{prop:compat}: each $\bb Z(j)^\bb A[\tfrac1p]$ vanishes in cohomological degrees $>j$ by the Gersten vanishing bound in motivic cohomology, while $\bb Z_p(j)^\bb A[\tfrac1p]$ is supported in cohomological degrees $\ge j$ by another use of Geisser--Levine.

The completes the proof of our initial claim, namely, working in $\Shv_\sub{Zar}(\Sm_{\bb F_p},\CAlg(\Gr\mathrm{D}(\bb Z)))$, the endomorphism of $\bb Z(\star)^\bb A$ induced by the absolute Frobenius is multiplication-by-$p^\star$.

Next, as presheaves of graded $\bb E_\infty$-algebras of complexes on qcqs $\bb F_p$-schemes, we have the following formulae:
\begin{itemize}
\item $\bb Z(\star)^\bb A$ is the cdh-local left Kan extension of $\bb Z(\star)^\bb A|_{\Sm_{\bb F_p}}$, by Theorem \ref{thm:cdh}(10).
\item $\bb Z_p(\star)^\sub{syn}$ is the $p$-completion of the left Kan extension of the \'etale sheafification of $\bb Z(\star)^\bb A|_{\Sm_{\bb F_p}}$, by the coherent Geisser--Levine theorem \cite[Theorem 6.14]{BachmannElmantoMorrow} and fact that syntomic cohomology is $p$-completely left Kan extended from polynomial algebras (so a fortiori from smooth algebras).
\end{itemize}
We thus obtain compatible homotopies between the absolute Frobenius and multiplication-by-$p^\star$ on the bottom left, top right, and bottom right of the cartesian (by Theorem \ref{thm:graded-pieces_charp}) diagram
\[
\begin{tikzcd}
\bb Z(\star)^\sub{mot}(X) \ar{r} \ar{d} & \bb Z_p(\star)^\sub{syn}(X)\ar{d}\\
\bb Z(\star)^{\bb A}(X) \ar{r} & L_\sub{cdh}\bb Z_p(\star)^\sub{syn}(X).
\end{tikzcd}
\]
 Pulling back yields the desired homotopy at the top left and so completes the proof.
\end{proof}

\begin{example}[Affine cones]\label{example_cones_p}
Let $Y$ be a smooth, geometrically connected, projective variety over a field $k$, equipped with a fixed embedding $Y\into \bb P^N_k$ into some projective space, and let $R$ be the associated homogenous coordinate ring. Write $d:=\dim R=\dim Y+1$, which we assume is $\ge2$. The focus of this example is the group $H^{2d}_\sub{mot}(R,\bb Z(d))$, which we hope can be thought of as a group of zero cycles on $\Spec(R)$ (see Remark \ref{rem:zcyc}). We already saw at the end of Remark \ref{example_cones_0} that this group vanishes if $k$ is an algebraic extension of $\bb Q$, and the goal of this remark is to show (for the moment only when $d>2$; see the end of the remark) that it also vanishes if $k$ is an algebraically closed field of characteristic $p>0$, i.e.,
\begin{quote}
(\dag) $H^{2d}_\sub{mot}(R,\bb Z(d))=0$ if $k$ is an algebraically closed field of characteristic $p>0$.
\end{quote}
A consequence\footnote{One argues as follows. Since $\bb Z(d)^\sub{mot}(R)$ vanishes in degrees $>2d$ by Soul\'e--Weibel vanishing, there is a natural edge map $H^{2d}_\sub{mot}(R,\bb Z(d))\to \K_0(X)$ from the Atiyah--Hirzebruch spectral sequence; the image of this edge map contains the image of the canonical map $H^d_\sub{Nis}(X,\hat \K_d^M)\to \K_0(X)$ by Remark \ref{rem:zcyc}; finally, since $\Spec(R)$ is ``nice'' in the terminology of Kato--Saito \cite[Definition 2.2]{KatoSaito1986}, the image of the latter map contains $F_0\K_0(R)$ by \cite[Theorem 2.5]{KatoSaito1986}. So the vanishing of $H^{2d}_\sub{mot}(R,\bb Z(d))$ implies the vanishing of $F_0\K_0(R)$.} of (\dag) is the vanishing of $F_0\K_0(R)$, which denotes the subgroup of $\K_0(R)$ generated by classes of smooth closed points of $\Spec(R)$. To put this into historical perspective, a famous result of Srinivas \cite[Thm.~1]{Srinivas1982} established the vanishing of $F_0\K_0(R)$ assuming $k$ algebraically closed of characteristic $p$ and $R$ normal and Cohen--Macaulay; we remove the latter two conditions on $R$. On the other hand, Srinivas proved his vanishing for more general graded rings than cones; we expect that our methods can be adapted.

So henceforth assume $k$ is an algebraically closed field of characteristic $p>0$. Following the same notation as Example \ref{example_cones_0}, let $X$ be the blow-up of $R$ at the vertex of the cone, and recall that $X$ is a line bundle over $Y$ equipped with $0$-section $i:Y\into X$ fitting into an abstract blow-up square
\[\xymatrix{
Y\ar[r] \ar[d]\ar[r]^{i}& X\ar[d]\\
\Spec(k)\ar[r]&\Spec(R).
}\]
As in Example \ref{example_cones_0} we see that $\bb Z(d)^{\bb A}(R)\simeq\bb Z(d)^\sub{mot}(k)$, which is supported in degrees $\le d$. So from Theorem \ref{thm:graded-pieces_charp} we read off a short exact sequence \begin{equation}H^{2d-1}_\sub{syn}(R,\bb Z_p(d))\To H^{2d-1}(L_\sub{cdh}\bb Z_p(d)^\sub{syn}(R))\To H^{2d}_\sub{mot}(R,\bb Z(d))\To 0\label{eqn_srinivas_ses}\end{equation} (where the left term is zero if $d>2$, since $\bb Z_p(d)^\sub{syn}(R)$ is supported in degrees $\le d+1$). To calculate the middle term of this sequence we apply $L_\sub{cdh}\bb Z_p(d)^\sub{syn}$ to the blow-up square to get a cartesian square
\[\xymatrix{
L_\sub{cdh}\bb Z_p(d)^\sub{syn}(R)\ar[r] \ar[d]\ar[r]& L_\sub{cdh}\bb Z_p(d)^\sub{syn}(k)\ar[d]\\
L_\sub{cdh}\bb Z_p(d)^\sub{syn}(X)\ar[r]&L_\sub{cdh}\bb Z_p(d)^\sub{syn}(Y),
}\]
which we simplify as follows. Firstly, the map $\bb Z_p(d)^\sub{syn}\to L_\sub{cdh}\bb Z_p(d)^\sub{syn}$ is an equivalence when evaluated on $k$, $Y$, and $X$: in the case of $k$ this is because it is a point for the cdh topology; for $Y$ and $X$ it is by the forthcoming Corollary \ref{corol_e_vs_eh}.\footnote{While this corollary is deduced using arguments involving motivic cohomology (namely Corollary~\ref{corol_smooth_comparison}), the appearance of motivic cohomology is illusory: the core of the proof of Corollary \ref{corol_e_vs_eh} is really contained in Theorem  \ref{thm:cdh-syn-pbf}, whose proof does not require the new motivic cohomology.} Also $\bb Z_p(d)^\sub{syn}(k)=0$ since $k$ is a perfect field. But we claim that also $\bb Z_p(d)^\sub{syn}(Y)=0$: indeed it is enough to check this mod $p$, namely that $R\Gamma_\sub{\'et}(Y,\Omega_{Y,\sub{log}}^d)[-d]=0$, which is true simply because the sheaf $\Omega^d_{Y,\sub{log}}$ on $Y_\sub{\'et}$ vanishes as $Y$ is a smooth $d-1$-dimensional variety over a perfect field of characteristic $p$. In conclusion, the previous cartesian square simplifies to an equivalence \[L_\sub{cdh}\bb Z_p(d)^\sub{syn}(R)\simeq \bb Z_p(d)^\sub{syn}(X).\] Combined with (\ref{eqn_srinivas_ses}) we have shown the following:
\begin{itemize}
\item If $d>2$ then $H^{2d}_\sub{mot}(R,\bb Z(d))\cong H^{2d-1}_\sub{syn}(X,\bb Z_p(d))$;
\item if $d=2$ then $H^4_\sub{mot}(R,\bb Z(2))\cong \opp{coker}(H^3_\sub{syn}(R,\bb Z_p(2))\to H^3_\sub{syn}(X,\bb Z_p(2)))$.
\end{itemize}
Since the syntomic cohomology groups are $p$-complete and, in case $d=2$ that $\bb Z_p(2)$ vanishes in degree $4$ (see Remark \ref{remarks_tildenu}), it is therefore enough to check the following claims to prove (\dag):
\begin{itemize}
\item If $d>2$ then $H^{2d-1}_\sub{syn}(X,\bb F_p(d))=0$;
\item if $d=2$ then the map $H^3_\sub{syn}(R,\bb F_p(2))\to H^3_\sub{syn}(X,\bb F_p(2))$ is surjective.
\end{itemize}

Here we treat the case $d>2$, where the syntomic cohomology group in question is $H^{d-1}_\sub{\'et}(X,\Omega^d_{X,\sub{log}})$. On the \'etale site of $X$ we have $\Omega^d_{X/k}=Z\Omega^d_{X/k}$ since $\Omega^{d+1}_{X/k}=0$ (as $X$ is a $d$-dimensional smooth variety over a perfect field), and so the Cartier isomorphism may be rewritten as $C:\Omega^d_{X/k}/d\Omega^{d-1}_{X/k}\isoto \Omega^d_{X/k}$; this lets us express the \'etale sheaf $\Omega^d_{X,\sub{log}}$ as the kernel of the surjection of \'etale sheaves $1-C:\Omega^d_{X/k}\to \Omega^d_{X/k}$, from which we obtain an exact sequence \[H^{d-2}_\sub{\'et}(X,\Omega^d_{X/k})\xto{1-C}H^{d-2}_\sub{\'et}(X,\Omega^d_{X/k})\To H^{d-1}_\sub{\'et}(X,\Omega^d_{X,\sub{log}})\To H^{d-1}_\sub{\'et}(X,\Omega^d_{X/k})\xto{1-C}H^{d-1}_\sub{\'et}(X,\Omega^d_{X/k}).\] Recall our goal is to show the vanishing of the middle term in this exact sequence.

The group $H^{d-1}_\sub{\'et}(X,\Omega^d_{X/k})$ vanishes for purely geometric reasons as in characteristic zero. Indeed, writing $\pi:X\to Y$ for the structure map of the line bundle, there are isomorphisms of $\roi_Y$-modules \begin{equation}\pi_*\Omega^d_{X/k}\cong \Omega_{Y/k}^{d-1}(1)\otimes_{\roi_Y}\pi_X\roi_X\cong \Omega_{Y/k}^{d-1}(1)\otimes_{\roi_Y}\bigoplus_{m\ge0}\roi_Y(m),\label{eqn_srinivas_Omega}\end{equation} whence $H^{d-1}_\sub{\'et}(X,\Omega^d_{X/k})\cong\bigoplus_{m\ge0}H^{d-1}(Y,\Omega_{Y/k}^{d-1}(1+m))$. But Serre duality identifies the $k$-vector space $H^{d-1}(Y,\Omega_{Y/k}^{d-1}(1+m))$ with the dual of $H^{0}(Y,\roi_Y(-1-m))$, which vanishes since any negative multiple of an ample line bundle has no global sections.

Next, appealing to the decomposition (\ref{eqn_srinivas_Omega}) a second time, there is an isomorphism of $k$-vector spaces $H^{d-2}_\sub{\'et}(X,\Omega^d_{X/k})\cong \bigoplus_{m\ge0}H^{d-2}(Y,\Omega_{Y/k}^{d-1}(1+m))$. The ampleness of $\roi_Y(1)$ implies that only finitely many terms in the direct sum are non-zero (this is the moment we use that $d>2$), and so in particular $H^{d-2}_\sub{\'et}(X,\Omega^d_{X/k})$ is finite dimensional. As the endomorphism $C$ of this $k$-vector space is $\phi^{-1}$-semilinear (i.e., $C(f\omega)=f^{1/p}C(\omega)$ for all $f\in k$ and $\omega\in H^{d-2}$), Artin--Schreier theory implies that $1-C$ is surjective (this is the moment we use that $k$ is not merely perfect, but algebraically closed\footnote{In fact, this shows that in (\dag) we can weaken ``algebraically closed'' to just assuming that $k$ has no finite field extension of degree $p$.}). This completes the proof in the case $d>2$.

We plan to treat the remaining case, i.e., $d=2$, in forthcoming work on Chow groups of singular schemes.
\end{example}

In the remainder of this section we explicitly describe $p$-adic motivic cohomology in characteristic $p$, analogously to Geisser--Levine's identification \cite{GeisserLevine2000} of classical mod-$p^r$ motivic cohomology $\bb Z(j)^\bb A/p^r$ as $R\Gamma_\sub{Zar}(-,W_r\Omega^j_\sub{log})[-j]$. More precisely, we show that $\bb Z(j)^\sub{mot}/p^r$ can be obtained by glueing syntomic cohomology to cdh and $\eh$ cohomologies of $W_r\Omega^j_\sub{log}$:

\begin{proposition}\label{prop_GL}
For any qcqs $\bb F_p$-scheme $X$ and $j,r\ge0$, there is a natural pullback square in $\rm{D}(\bb Z)$:
\[
\begin{tikzcd}
\bb Z(j)^\sub{mot}(X)/p^r \ar{r} \ar{d} & \bb Z_p(j)^\sub{syn}(X)/p^r\ar{d}\\
R\Gamma_\sub{cdh}(X,W_r\Omega^j_\sub{log})[-j] \ar{r} & R\Gamma_\sub{\'eh}(X,W_r\Omega^j_\sub{log})[-j]
\end{tikzcd}
\]
(Moreover, as $j$ varies, the maps are naturally ones of $\bb E_\infty$-algebras in graded complexes.)
\end{proposition}
\begin{proof}
More precisely, we will obtain the square as the mod-$p^r$ reduction of the square of Theorem \ref{thm:graded-pieces_charp}(2). The bottom left corner of the square is indeed $R\Gamma_\sub{cdh}(X,W_r\Omega^j_\sub{log})[-j]$ by Theorem \ref{thm:cdh}(4), or rather by an analogue for mod-$p^r$ rather than mod-$p$, while the bottom right corner is the target of the cdh-sheafification map from the top right corner and therefore naturally identifies with $R\Gamma_\sub{\'eh}(X,W_r\Omega^j_\sub{log})[-j]$ by~\eqref{eqn_Geisser_eh}.
\end{proof}

In practice we often use Theorem \ref{prop_GL} in the form of a fibre sequence rather than a pullback square. To formulate the statement we need the following invariant:

\begin{definition}
The {\em mod-$p^r$, weight-$j$ Artin--Schreier obstruction} of an $\bb F_p$-algebra $A$ is the cokernel of the Artin--Schreier map from (\ref{eqn_WOmegalog}) \[\widetilde\nu_r(j)(A):=\text{coker}(C^{-1}-1:W_r\Omega^j_A\To W_r\Omega^j_A/dV^{r-1}\Omega^j_A).\] Given a topology $\tau$ on qcqs $\bb F_p$-schemes (notably Zariski, Nisnevich, or cdh), then we write $R\Gamma_\tau(X,\tilde\nu_r(j))$ for the cohomology of the sheafification of $\tilde\nu_r(j)$ in the topology $\tau$, and similarly $H^i_\tau(X,\tilde\nu_r(j))$ for the individual cohomology groups. We warn the reader that $\tilde\nu_r(j)$ is not even a Zariski sheaf on affines, so that in general the map $\tilde\nu_r(j)(A)\to H^0_\sub{Zar}(A,\tilde\nu_r(j))$ is not an isomorphism. 
\end{definition}

\begin{remark}\label{remarks_tildenu}
Here are several alternative descriptions of the groups $\tilde\nu_r(j)(A)$:
\begin{enumerate}
\item (Cohomological) Since $C^{-1}-1$ is \'etale locally surjective and the sheaves $W_r\Omega^j$, $W_r\Omega^j/dV^{r-1}\Omega^{j-1}$ have no higher cohomology on affines, we see that there is a natural isomorphism \[\tilde\nu_r(j)(A)\cong H^1_\sub{\'et}(A,W_r\Omega^j_\sub{log}).\]
\item (Syntomic) There is a natural isomorphism \[\tilde\nu(j)(A)\cong H^{j+1}(\bb Z_p(j)^\sub{syn}(A)/p^r),\] and moreover this is the top degree of $\bb Z_p(j)^\sub{syn}(A)/p^r$, by \cite[Corol.~5.43]{AntieauMathewMorrowNikolaus}; more precisely, the comparison map of Remark \ref{remark_WOmegalog}(2) is an isomorphism in degrees $>j$.
\item ($K$-theoretic) For $A$ local, there are natural isomorphisms\[\tilde\nu_r(j)(A)\cong\pi_{j-1}\text{cofib}(\K^\sub{cn}(A)/p^r\to\TC(A)/p^r)\] by \cite[Theorem~6.11]{clausen2018k}.
\end{enumerate}
At least in the case in which $A=k$ is a field, these invariants have also appeared notably in work of Kato \cite{Kato1982a}, denoted by $H_{p^r}^{j+1}(k)$, and are related to class field theory.
\end{remark}

\begin{remark}[Rigidity]\label{rem_rigidity_of_nutilde}
A key property of $\tilde\nu_r(j)$ is its {\em rigidity}, namely whenever $R\to A$ is a Henselian surjection of $\bb F_p$-algebras, then the induced map $\tilde\nu_r(j)(R)\to\tilde\nu_r(j)(A)$ is an isomorphism. This can be deduced directly from the definition and Hensel's lemma: see the proof of \cite[Proposition~6.12]{clausen2018k}.
\end{remark}

For any qcqs $\bb F_p$-scheme $X$ there is a natural map \begin{equation}\bb Z_p(j)^\sub{syn}(X)/p^r\To R\Gamma_\sub{cdh}(X,\tilde\nu_r(j))[-j-1],\label{eqn_syn_to_nutilde}\end{equation} defined by Zariski sheafifiying the following composition on affines:
\[ \bb Z_p(j)^\sub{syn}(A)/p^r\xto{\sub{Rem.~\ref{remarks_tildenu}(2)}}\tilde\nu_r(j)(A)[-j-1]\xto{\sub{can.~map}} R\Gamma_\sub{cdh}(A,\tilde\nu_r(j))[-j-1].\] Our mod-$p^r$ motivic cohomology identifies with the fibre of the map (\ref{eqn_syn_to_nutilde}):

\begin{corollary}[Fundamental fibre sequence in characteristic $p$]\label{corol_fundamental_p}
For any qcqs $\bb F_p$-scheme $X$ and $j,r\ge0$, there is a natural fibre sequence
\[\bb Z(j)^\sub{mot}(X)/p^r\To \bb Z_p(j)^\sub{syn}(X)/p^r\stackrel{\sub{(\ref{eqn_syn_to_nutilde})}}{\To} R\Gamma_\sub{cdh}(X,\tilde\nu_r(j))[-j-1].\]
\end{corollary}
\begin{proof}
In terms of the pullback square of Theorem \ref{prop_GL}, the map (\ref{eqn_syn_to_nutilde}) is the dotted composition:
\[\hspace{-26mm}
\xymatrix@C=3mm{
\bb Z(j)^\sub{mot}/p^r \ar[r] \ar[d] & \bb Z_p(j)^\sub{syn}/p^r\ar[d]\ar@{-->}[drrr]&&\\
R\Gamma_\sub{cdh}(-,W_r\Omega^j_\sub{log})[-j] \ar[r] & R\Gamma_\sub{eh}(-,W_r\Omega^j_\sub{log})[-j]\ar@{=}[r]&L_\sub{cdh}R\Gamma_\sub{et}(-,W_r\Omega^j_\sub{log})[-j]\ar[r] &(L_\sub{cdh}\tau^{\ge1}R\Gamma_\sub{et}(-,W_r\Omega^j_\sub{log}))[-j]\ar@{=}[r] & R\Gamma_\sub{cdh}(-,\tilde\nu_r(j))[-j-1]
}
\]
(the middle bottom equality having been explained at the end of the proof of Theorem \ref{prop_GL}). The claim to be proved is therefore that the bottom row is a fibre sequence; but this follows from exactness of cdh sheafification and the fibre sequence $W_r\Omega^j_{A,\sub{log}}\to R\Gamma_\sub{\'et}(A,W_r\Omega^j_\sub{log})\to \tilde\nu_r(j)(A)[-1]$ on affines.
\end{proof}

\subsection{A Beilinson--Lichtenbaum equivalence}
The classical Beilinson--Lichtenbaum conjecture states that motivic cohomology with finite coefficients is given by \'etale cohomology, in the range where cohomological degree is less than or equal to the weight. We refer to \cite[\S2]{HaesemeyerWeibel2019} for a discussion of the conjecture in the smooth case and exactly how it relates to the other main conjectures, such as Bloch--Kato. Here we record that such a Beilinson--Lichtenbaum equivalence holds for our motivic cohomology, including at the characteristic (where the correct replacement for \'etale cohomology is syntomic cohomology):

\begin{theorem}\label{thm_BL}
Let $\bb F$ be a prime field, $X$ a qcqs $\bb F$-scheme, and $j\ge0$.
\begin{enumerate}
\item For any integer $\ell>0$ prime to the characteristic of $\bb F$, there is a natural map \[\bb Z(j)^\sub{mot}(X)/\ell\To R\Gamma_\sub{\'et}(X,\mu_\ell^{\otimes j}),\] whose cofibre is supported in degrees $>j$.
\item If $\bb F=\bb F_p$ then for any $r\ge0$ there is a natural map \[\bb Z(j)^\sub{mot}(X)/p^r \To \bb Z_p(j)^\sub{syn}(X)/p^r,\] whose cofibre is supported in degrees $>j$.
\end{enumerate}
\end{theorem}
\begin{proof}
(1): Recall that $\bb Z(j)^\sub{mot}/\ell\to \bb Z(j)^\sub{cdh}/\ell$ is an equivalence, by Theorems \ref{thm:graded-pieces}(4) and \ref{thm:graded-pieces_charp}(3), and the latter is given by $L_\sub{cdh}\tau^{\le j}R\Gamma_\sub{\'et}(-,\mu_\ell^{\otimes j})$ by Theorem \ref{thm:cdh}(3). It remains only to use that the fiber of the canonical map $L_\sub{cdh}\tau^{\le j}R\Gamma_\sub{\'et}(-,\mu_\ell^{\otimes j})\to R\Gamma_{\et}(-,\mu_{\ell}^{\otimes j})$ is supported in degrees $>j$, since $R\Gamma_{\et}(-,\mu_{\ell}^{\otimes j})$ satisfies cdh descent as recalled in Example \ref{example_0p}.

(2): This is clear from Corollary \ref{corol_fundamental_p}.
\end{proof}

\begin{remark}[Uniqueness and multiplicative structure]
The proof of Theorem \ref{thm_BL} and \cite[Theorem~7.16]{BachmannElmantoMorrow} really establishes the existence of a unique map of $\bb E_\infty$-algebras in graded presheaves of complexes \[\bb Z(\star)^\sub{mot}(X)/\ell\To R\Gamma_\sub{\'et}(X,\mu_\ell^{\otimes \star})\] which is compatible with the first Chern class map (\ref{eq:c1-mot}). Similarly at the characteristic (in which case one argues by viewing $\bb Z(\star)^\sub{mot}/p^r$ on finitely presented $\bb F$-schemes as the pro cdh-local left Kan extension of its restriction to smooth $\bb F$-schemes, as in Theorem \ref{theorem_KS}).
\end{remark}

\subsection{Comparison maps}
Let $\bb F$ be a prime field. In this subsection we explicitly record the canonical comparison maps between the classical motivic cohomology of \S\ref{ss_classical}, the lisse motivic cohomology of \S\ref{subsec_lke}, the $\bb A^1$-motivic cohomology of \S\ref{ss_cdh_local}, and our new motivic cohomology. These comparisons are induced by various filtered maps between $K$-theory, $\KH$-theory, and connective $K$-theory.

\begin{construction}[Classical vs new motivic cohomology of smooth varieties]\label{cons:cla-vs-new}
We claim, for any smooth $\bb F$-scheme $X$, that there is a natural multiplicative comparison map of filtered spectra \begin{equation}\Fil_\bb A^\star \K(X)\To\Fil^\star_\sub{mot}\K(X)\label{eqn_cla_to_mot}\end{equation} given on $\Fil^0$ by $\Fil^0_\bb A\K(X)=\K(X)=\Fil^0_\sub{mot}\K(X)$. On shifted graded pieces this induces natural multiplicative comparison maps \begin{equation}\bb Z(j)^\bb A(X)\To\bb Z(j)^\sub{mot}(X)\label{eqn_cla_to_mot_j}\end{equation} for $j\ge0$.

We define (\ref{eqn_cla_to_mot}) as follows. When $\bb F=\bb Q$ (resp.~$\bb F_p$), the filtered cdh-local trace map of Corollary \ref{corol_cdh_filtered_trace} (resp.~Proposition \ref{prop_cdh_filtered_trace_p}) was constructed so as to fit into a commutative diagram
\[
\begin{tikzcd}
\mathrm{Fil}^{\star}_\bb A\K(X) \ar[swap]{d}{} \ar{r}    & \mathrm{Fil}^{\star}_\sub{HKR}\HC^-(X/\bb Q)\ar{d} \\
\mathrm{Fil}^\star_{\bb A}\KH(X) \ar{r} & \Fil^{\star}_\sub{HKR}L_\sub{cdh}\HC^-(X/\bb Q).
\end{tikzcd}
\qquad\qquad\mathrm{resp. }
\begin{tikzcd}
\mathrm{Fil}^{\star}_\bb A\K(X) \ar[swap]{d}{} \ar{r}    & \mathrm{Fil}^{\star}_\sub{BMS}\TC(X)\ar{d} \\
\mathrm{Fil}^\star_{\bb A}\KH(X) \ar{r} & \Fil^{\star}_\sub{BMS}L_\sub{cdh}\TC(X),
\end{tikzcd}
\]
where the top horizontal arrow is the filtered trace map of Proposition \ref{prop:compat} (resp.~\ref{prop:mot-v-bms}) for the smooth $\bb F$-scheme $X$. From the pullback Definition \ref{eq:char0} (resp.~\ref{def:charp}) of $\Fil_\sub{mot}^\star \K(X)$, there is therefore a natural induced map (\ref{eqn_cla_to_mot}) as desired.

In Corollary \ref{corol_smooth_comparison} we will prove that (\ref{eqn_cla_to_mot}) is an equivalence for every smooth $\bb F$-scheme $X$.
\end{construction}

\begin{construction}[LKE vs new motivic cohomology of affines]\label{cons_lke_to_mot}
Restricting (\ref{eqn_cla_to_mot}) to smooth $\bb F$-algebras and then left Kan extending to all $\bb F$-algebras defines, for any $\bb F$-algebra $A$, a natural multiplicative comparison map of filtered spectra
\begin{equation}\Fil_\sub{lse}^\star \K^\sub{cn}(A)\To\Fil^\star_\sub{mot}\K(A)\label{eqn_lke_to_mot}\end{equation}
given on $\Fil^0$ by the canonical map $\Fil^0_\sub{lse}\K^\sub{cn}(A)=\K^\sub{cn}(A)\to \K(A)=\Fil^0_\sub{mot}\K(A)$. On shifted graded pieces this induces natural multiplicative comparison maps \[\bb Z(j)^\sub{lse}(A)\To\bb Z(j)^\sub{mot}(A)\] for $j\ge0$. We will study these comparison maps in detail in Section \ref{section_lke}.
\end{construction}

\begin{construction}[New vs $\bb A^1$-motivic cohomology]\label{cons_mot_to_cdh}
Tautologically from the pullback definition of $\Fil_\sub{mot}^\star \K(X)$, there is a natural comparison map of filtered spectra \[\Fil_\sub{mot}^\star \K(X)\To \Fil_{\bb A}^\star \KH(X)\] for any qcqs $\bb F$-scheme $X$, given on $\Fil^0$ by the canonical map $\Fil_\sub{mot}^0 \K(X)=\K(X)\to \KH(X)=\Fil_{\bb A}^\star \KH(X)$. On shifted graded pieces this induces the natural comparison maps $\bb Z(j)^\sub{mot}(X)\to\bb Z(j)^{\bb A}(X)$, for $j\ge0$, which have already appeared in Theorems \ref{thm:graded-pieces}(2) and \ref{thm:graded-pieces_charp}(2). We will study these maps further in Section \ref{section_smooth}.
\end{construction}

\begin{remark}[(\ref{eqn_cla_to_mot}) is split]\label{remark_split}
While the various comparison maps are displayed, we point out the following: for any smooth $\bb F$-scheme $X$ and $j\ge0$, the composition \[\bb Z(j)^\bb A(X)\stackrel{\sub{Cons.~\ref{cons_lke_to_mot}}}{\To} \bb Z(j)^\sub{mot}(X)\stackrel{\sub{Cons.~\ref{cons_mot_to_cdh}}}{\To}\bb Z(j)^{\bb A}(X)\] (in other words, the the left vertical arrow in the commutative squares of Construction \ref{cons:cla-vs-new})  is an equivalence. This is the content of Theorem \ref{thm:cdh}(9), namely that $\bb A^1$-motivic cohomology of smooth $\bb F$-schemes is classical motivic cohomology. We will see in Corollary \ref{corol_smooth_comparison} that, not only is the composition an equivalence, but in fact each map is an equivalence.
\end{remark}

\subsection{Derived schemes and beyond}\label{sec:derived}
We finish this section by briefly explaining that our motivic cohomology extends to derived schemes, though we do not require the theory in such generality in the present article.

We write $\CAlg^\sub{ani}_\bb F$ for the $\infty$-category of \emph{animated $\bb F$-algebras}, i.e., the  subcategory of $\Fun(\CAlg^{\Sigma}_\bb F, \Spc)$ which preserves finite products, where $\CAlg^{\Sigma}_\bb F$ denotes the category of finitely generated polynomial $\bb F$-algebras. Animated $\bb F$-algebras are derived affine schemes, out of which we build the $\infty$-category of  \emph{derived $\bb F$-schemes} $\rm dSch_{\bb F}$; see \cite[Chapter 25]{LurieSAG} for more details. 

\begin{construction}\label{constr:anim-mot}
Let $\bb F$ be a prime field and $X$ a qcqs derived $\bb F$-scheme. Note that the HKR filtration of Theorem \ref{thm:hkr} and the BMS filtration of Theorem \ref{thm_BMS2+} extend to the generality of derived $\bb F$-schemes. Indeed, in the case of the HKR filtration the references \cite{antieau-fil,raksit-hkr,mrt-hkr} work in this degree of generality; for the BMS filtration one $p$-completely left Kan extends the filtration from discrete algebras, as in \cite[Construction~5.33]{AntieauMathewMorrowNikolaus}

By naturality of these filtrations, there are natural multiplicative comparison maps of filtered spectra $\Fil^\star_{\rm HKR}\HC^-(X/\Q) \rightarrow \Fil^\star_{\rm HKR}\HC^-(X^\sub{cla}/\Q)$ if $\bb F=\bb Q$, and $\Fil^\star_{\BMS}\TC(X) \rightarrow \Fil^\star_{\BMS}\TC(X^\sub{cla})$ and if $\bb F=\bb F_p$, where $X^\sub{cla} \hookrightarrow X$ is the classical locus of $X$.

The motivic filtration on the $K$-theory of $X$ is then defined by the following cartesian square 
\[
\begin{tikzcd}
\mathrm{Fil}^{\star}_{\mot}\K(X) \ar[dotted]{d} \ar[dotted]{r}    & \mathrm{Fil}^{\star}_\sub{HKR}\HC^-(X)\ar{d} \\
\mathrm{Fil}^{\star}_{\mot}\K(X^\sub{cla}) \ar{r} & \mathrm{Fil}^{\star}_\sub{HKR}\HC^-(X^\sub{cla}),
\end{tikzcd}
\qquad\text{resp.~}
\begin{tikzcd}
\mathrm{Fil}^{\star}_{\mot}\K(X) \ar[dotted]{d} \ar[dotted]{r}    & \mathrm{Fil}^{\star}_\sub{BMS}\TC(X)\ar{d} \\
\mathrm{Fil}^{\star}_{\mot}\K(X^\sub{cla}) \ar{r} & \mathrm{Fil}^{\star}_\sub{BMS}\TC(X^\sub{cla}),
\end{tikzcd}
\]
(the first if $\bb F = \bb Q$; the second if $\bb F = \bb F_p$). In both cases $\mathrm{Fil}^{\star}_{\mot}\K(X^\sub{cla})$ refers to the motivic filtration which we have defined earlier on the $K$-theory of the classical qcqs $\bb F$-scheme $X^\sub{cla}$.

The \emph{weight-$j$ motivic cohomology} of $X$ is then defined to be
\[
\bb Z(j)^{\mot}(X) := \gr^j_{\mot}\K(X)[-2j],\qquad j\in\bb Z.
\]
\end{construction}

It is not our intention to present here an exhaustive account of the motivic filtration on derived schemes; we content ourselves with stating the following summary of the main properties:

\begin{theorem}[Motivic filtration for derived schemes] 
For any qcqs derived $\bb F$-scheme, there exists a natural $\bb N$-indexed, multiplicative filtration $\Fil^\star_{\mot}\K(X)$ on $\K(X)$ satisfying $\Fil^0_{\mot}\K(X) \simeq \K(X)$. If $X$ is classical then this filtration agrees with the earlier motivic filtration of Definitions \ref{eq:char0} and \ref{def:charp}.
\end{theorem}
\begin{proof}
We just explain the claim that $\Fil^0_\sub{mot}\K(X)=\K(X)$, the other statements being clear. By Zariski descent for derived $\bb F$-schemes, it suffices to prove the result for $X = \Spec(A)$ when $A$ is an animated $\bb F$-algebra. In this case, the result follows from the same arguments as in the classical case and part of the Dundas--Goodwillie--McCarthy theorem \cite{Dundas2013}, stating that for a simplicial ring $A$ the square of spectra
\[
\begin{tikzcd}
\K(A) \ar{r} \ar{d}    & \TC(A) \ar{d} \\
\K(\pi_0A) \ar{r} & \TC(\pi_0A)
\end{tikzcd}
\]
is cartesian.
\end{proof}

\begin{remark}[Motivic cohomology of ring spectra]
The Dundas--Goodwillie--McCarthy theorem also applies to the $K$-theory of connective $\bb E_{\infty}$-ring spectra $R$. Assuming in addition that $R$ is chromatically quasisyntomic in the sense of Hahn--Raksit--Wilson \cite{HahnRaksitWilson2025}, they equip its topological cyclic homology with the so-called even filtration which serves as a good candidate for a motivic filtration in this setting. Equipping $\K(\pi_0(R))$ with its motivic filtration (of this paper if $\pi_0(R)$ is equicharacteristic, or using Bouis' work \cite{Bouis2025, Bouis2025a, Bouis2025b} in general), one may then mimic the above pullback recipe to equip $\K(R)$ with a motivic filtration, as envisaged for arbitrary $\bb E_\infty$-algebras by Rognes in his ICM talk \cite{Rognes2014}.
\end{remark}

\section{The projective bundle formula and regular blowup squares}\label{section_pbf}
As usual $\bb F$ continues to be a prime field. The goal of this section is to prove that motivic cohomology of qcqs $\bb F$-schemes has two good geometric properties, namely that it satisfies the projective bundle and regular blow-up formulae. The statement of the former uses the first Chern class for motivic cohomology, which we now define.

\begin{lemma}\label{lem:c1-map}
There exists a unique map of presheaves $R\Gamma_\sub{Nis}(-,\bb G_m)[-1]\to\bb Z(1)^\sub{mot}$ on $\Sch_\bb F^\sub{qcqs}$ which is given on smooth $\bb F$-schemes by the composition 
\[R\Gamma_\sub{Zar}(-,\bb G_m)[-1]\stackrel{c_1^\bb A\,\sub{\eqref{eqn_1st_Chern_A}}}\To \bb Z(1)^\bb A\stackrel{\sub{\eqref{eqn_cla_to_mot_j}}}\To\bb Z(1)^\sub{mot}.\]
\end{lemma}
\begin{proof}
The left Kan extension of the indicated composition from smooth $\bb F$-algebras gives us a map $(\tau^{\le 1}R\Gamma_\sub{Zar}(-,\bb G_m))[-1]\to \bb Z(1)^\sub{mot}$ on affine $\bb F$-schemes. Nisnevich sheafifying this defines the desired first Chern class map. Uniqueness holds since $R\Gamma_\sub{Nis}(-,\bb G_m)[-1]$ is Nisnevich locally left Kan extended from smooth schemes.
\end{proof}

\begin{definition}\label{def:c1-mot}
For any qcqs $\bb F$-scheme $X$, we refer to the map of the previous lemma 
\begin{equation} \label{eq:c1-mot}
c_1:R\Gamma_\sub{Nis}(X,\bb G_m)[-1]\To\bb Z(1)^\sub{mot}(X)
\end{equation} as the {\em first Chern class map} for motivic cohomology. We will often refer to the induced natural map on $H^2$, namely \[c_1:\text{Pic}(X)\To H^2_\sub{mot}(X,\bb Z(1)),\] in the same way.
\end{definition}

Now, for any $X \in \Sch_{\bb F}^\sub{qcqs}$, thanks to the multiplicative structure of the graded presheaf $\bb Z(\star)^\sub{mot}$, multiplication by the first Chern class of $\roi(1)\in\mathrm{Pic}(\bb P_X^r)$ defines maps
\[
\Z(j)^{\mot}(\P^r_X) \xrightarrow{c_1(\scr O(1))} \Z(j+1)^{\mot}(\P^r_X)[2]
\]
for $j\in\bb Z$.

We may now state the main goal of this section:

\begin{theorem}\label{thm:pbf-blowup}
Let $X$ be a qcqs $\bb F$-scheme and $j\in\bb Z$.
\begin{enumerate}
\item Projective bundle formula: for any $r \ge 0$, the map 
\begin{equation}\label{eq:pbf-r}
\sum_{i=0}^r c_1(\scr O(1))^i\pi^{\ast}\colon\bigoplus_{i=0}^r\bb Z(j-i)^\sub{mot}(X)[-2i]\To \bb Z(j)^\sub{mot}(\bb P_X^r)
\end{equation} is an equivalence, were $\pi:\bb P^r_X\to X$ is the projection map.
\item More generally, if $\scr E$ is a locally free sheaf on $X$ of rank $r+1$, then the map
\[\sum_{i=0}^r c_1(\scr O(1))^i\pi^{\ast}\colon\bigoplus_{i=0}^r\bb Z(j-i)^\sub{mot}(X)[-2i]\To \bb Z(j)^\sub{mot}(\bb P_X(\scr E)),\]
is an equivalence.
\item Regular blow-up formula: Let $Y\to X$ be a regular closed immersion (i.e., $X$ admits an open affine cover such that, on each such affine, $Y$ is defined by a regular sequence); then applying $\bb Z(j)^\sub{mot}$ to the square of schemes
\[
\begin{tikzcd}
Y\times_X\mathrm{Bl}_Y(X) \ar{r} \ar{d} &\mathrm{Bl}_Y(X) \ar{d}\\
Y \ar{r}& X\\
\end{tikzcd}
\]
yields a cartesian square in $\mathrm{D}(\bb Z)$.
\end{enumerate}
\end{theorem}

Theorem \ref{thm:pbf-blowup} fits within recent developments in the theory of non-$\A^1$-invariant motives as developed in \cite{AnnalaHoyoisIwasa2023, AnnalaIwasa2023}. In this theory, the projective bundle and elementary blow-up formulae are isolated as the key properties of cohomology theories, in lieu of $\A^1$-invariance. In particular, our results imply that the motivic cohomology $\Z(\star)^{\mot}$ assembles into a motivic spectrum in the sense of \cite{AnnalaIwasa2023}; see \cite[Appendix C]{AnnalaShin2025}.

\subsection{Generalities on $\bb P^1$-bundle formulae for cohomologies and $\K^\sub{cn}$-modules}\label{sec:p1}
In this subsection we set up some general machinery which allows us to formulate, and sometimes prove, $\bb P^1$-bundle formulae in two general contexts: modules over Zariski sheafified lisse motivic cohomology and modules over connective $K$-theory. We will interpolate between them using modules over the Zariski sheafification of connective $K$-theory equipped with its lisse motivic filtration. As usual $\bb F$ denotes a prime field.\footnote{Though a lot of the machinery of this subsection works verbatim over $\bb Z$.}

This subsection is required only to prove the crucial Proposition \ref{proposition_payoff} below. So we suggest the following reading plan: the honest reader should first look at the statement of Proposition \ref{proposition_payoff}, then read the final paragraph of this subsection to see the proof strategy, and only then look at the details of this subsection. The reckless cheat should skip this subsection. The thrill seeker should take Theorem~\ref{thm:pbf-blowup} for granted and jump ahead to \S\ref{section_smooth}.

\begin{remark}[Variant]
A more attractive and conceptual variant of the machinery of this subsection, which the reader might wish to keep in mind, is obtained by replacing Zariski sheafified lisse motivic cohomology $L_\sub{Zar}\bb Z(\star)^\sub{lse}$ by the new motivic cohomology $\bb Z(\star)^\sub{mot}$ of \S\ref{s_motivic_def}, and the filtered Zariski sheafified connective $K$-theory $L_\sub{Zar}\Fil^\star_\sub{lse}\K^\sub{cn}$ by motivically filtered $K$-theory $\Fil_\sub{mot}^\star\K$. One reason we do not work with this variant is because it is logically important to keep the proof of Theorem \ref{thm:cdh-syn-pbf} independent of the new motivic cohomology: see Remark \ref{rem_not_circular}.
\end{remark}

First we formulate the $\bb P^1$-bundle formula for modules over Zariski sheafified lisse motivic cohomology, in the following sense:

\begin{definition}[$L_\sub{Zar}\bb Z(\star)^\sub{mot}$-modules]\label{def:P1bf_cohomology}
As in Remark \ref{rem_Zariski_lisse}, let $L_\sub{Zar}\bb Z(\star)^\sub{lse}$ be the presheaf of $\rm{D}(\bb Z)$-valued $\bb E_\infty$-algebras on qcqs $\bb F$-schemes obtained by Zariski sheafifiying the lisse motivic cohomology. Similarly to Example \ref{example_lke_1}, the first Chern class in classical motivic cohomology induces an equivalence 
\[c_1:R\Gamma_\sub{Zar}(-,\bb G_m))[-1]\quis L_\sub{Zar}\bb Z(1)^\sub{lse},\] which we will also refer to as a first Chern class.

The $\infty$-category of {\em modules over Zariski sheafified lisse motivic cohomology}, or {\em $L_\sub{Zar}\bb Z(\star)^\sub{mot}$-modules}, on qcqs $\bb F$-schemes is 
\[L_\sub{Zar}\bb Z(\star)^\sub{mot}\opp{-Mod}:=\mathrm{Mod}_{L_\sub{Zar}\bb Z(\star)^\sub{mot}}(\PShv(\Sch^\sub{qcqs}_\bb F,\Gr\mathrm{D}(\bb Z))).\]
A $L_\sub{Zar}\bb Z(\star)^\sub{lse}$-module $F(\star)$ is thus a graded presheaf of complexes on qcqs $\bb F$-schemes equipped with coherent compatible action maps $L_\sub{Zar}\bb Z(j)^\sub{lse}\otimes F(i)\to F(j+i)$ for all $i,j\in\bb Z$; in particular, the first Chern class of $\roi(1)$ induces a map $c_1(\roi(1)):F(j-1)(\bb P^1_X)[-2]\to F(j+1)(\bb P^1_X)$ for any qcqs $\bb F$-scheme $X$. We say that $F(\star)$ satisfies the {\em $\bb P^1$-bundle formula} if the map
\begin{equation}\label{eqn:P1_for_cohomologies}
F(j)(X)\oplus F(j-1)(X)[-2] \xrightarrow{\pi^* \oplus c_1(\scr O(1))\circ \pi^* } F(j)(\P^1_X)
\end{equation}
is an equivalence for all $X\in\Sch^\sub{qcqs}_\bb F$ and $j\in\bb Z$. Denote by
\[L_\sub{Zar}\bb Z(\star)^\sub{mot}\opp{-Mod}_\sub{pbf}\subseteq L_\sub{Zar}\bb Z(\star)^\sub{mot}\opp{-Mod}\] the full subcategory of $L_\sub{Zar}\bb Z(\star)^\mot$-modules satisfying the $\bb P^1$-bundle formula; it is clearly closed under all limits and colimits.
\end{definition}

\begin{remark}[$\bb P^1$-bundle formula via reduced cohomology]\label{rem_reduced_P1}
Let $F(\star)$ be a $L_\sub{Zar}\bb Z(\star)^\sub{lse}$-module and $X$ a qcqs $\bb F$-scheme. The map $c_1(\roi(1))\circ\pi^*$ in \eqref{eqn:P1_for_cohomologies} naturally factors through the fibre of $\infty^*$ (by functoriality, since $\infty^*\roi(1)$ is trivial in $\mathrm{Pic}(X)$), and the $\bb P^1$-bundle formula is equivalent to asking that
\begin{equation}\label{eqn:P1_for_cohomologies2}
c_1(\roi(1))\circ\pi^*:F(j-1)(X)[-2]\To \fib\big(F(j)(\bb P_X^1)\xto{\infty^\star} F(j)(X)\big)
\end{equation}
be an equivalence for all $j\in\bb Z$.
\end{remark}

\begin{example}[Derived de Rham cohomology]\label{exam:dr_pbf}
In this example we take $\bb F=\bb Q$. Arguing as in the second paragraph of Remark \ref{rem:horizontal_map_char0}, the dlog map and Zariski-local left Kan extension induce a map of presheaves of graded $\bb E_\infty$-algebras of complexes $L_\sub{Zar}\bb Z(\star)^\sub{lse}\To R\Gamma(-,\widehat{L\Omega}_{-/\Q}^{\geq \star})$ on qcqs $\bb Q$-schemes. Forgetting the multiplicative structure on the target we thus obtain \[R\Gamma(-,\widehat{L\Omega}_{-/\Q}^{\geq \star})\in L_\sub{Zar}\bb Z(\star)^\sub{mot}\opp{-Mod}\] which we claim satisfies the $\bb P^1$-bundle formula.

Precomposing the aforementioned map of algebras in weight one with the first Chern class, we note that the composition
\[R\Gamma_\sub{Zar}(X,\bb G_m)[-1]\stackrel\sim\To L_\sub{Zar}\bb Z(1)^\sub{lse}\To R\Gamma(-,\widehat{L\Omega}_{-/\Q}^{\geq 1})\]
is by construction the usual first Chern class map for derived de Rham cohomology, defined by the dlog map. %Indeed, this reduces to the case of smooth $\bb Q$-schemes and then follows from the final paragraph of Remark \ref{rem:horizontal_map_char0}.
 The claimed $\bb P^1$-bundle formula therefore reduces to the usual $\bb P^1$-bundle formulae for $R\Gamma(-,\widehat{L\Omega}_{-/\Q}^{\geq \star})$, which is proved by reducing to wedge powers of the cotangent complex over $\bb Q$, then using K\"unneth to reduce to wedge powers of the cotangent complex over $X$, and there it is the standard calculation of the cohomology of $\Omega^\star_{\bb P^1_X/X}$; see \cite[Lemma 9.1.3]{BhattLurie2022} for details. 
\end{example}

\begin{example}[Syntomic cohomology]\label{exam:syn_pbf}
In this example we take $\bb F=\bb F_p$. Arguing as in Remark \ref{rem:horizontal_map_charp}, the dlog map and Zariski-local left Kan extension again induce a map of presheaves of graded $\bb E_\infty$-algebras of complexes $L_\sub{Zar}\bb Z(\star)^\sub{lse}\to \bb Z_p(\star)^\sub{syn}$ on qcqs $\bb F_p$-schemes. Forgetting the multiplicative structure on the target we thus obtain \[\bb Z_p(\star)^\sub{syn}\in L_\sub{Zar}\bb Z(\star)^\sub{lse}\opp{-Mod}\] which we claim satisfies the $\bb P^1$-bundle formula.

Similarly to the previous example, the Chern classes in Zariski sheafified lisse motivic and syntomic cohomology are compatible by definition, so we reduce to the usual $\bb P^1$-bundle formula in syntomic cohomology. That is proved by reducing to Nygaard filtered prismatic cohomology, then as in the previous example to wedge powers of the cotangent complex and Hodge cohomology; a proof (for arbitrary schemes, not only over $\bb F_p$) may be found in \cite[Theorem 9.1.1]{BhattLurie2022}.
\end{example}

\begin{remark}[Derived category of motives]
The category $L_\sub{Zar}\bb Z(\star)^\sub{mot}\opp{-Mod}_\sub{pbf}$ should vaguely be thought of as first order approximation to a big derived category of non-$\bb A^1$-invariant motives over $\bb F$. A precise manifestation of this philosophy (which we do not need in what follows) may be found in \cite[Remark 9.11]{BachmannElmantoMorrow}: on any qcqs scheme $X$, the category of $\bb A^1$-invariant Nisnevich sheaves of modules over $\bb Z(\star)^\bb A$ on $\Sm_X$ is equivalent to $\mathrm{DM}(X):=$~modules over the motivic Eilenberg--Maclane spectrum $H\bb Z_X^\bb A$ in $\SH(X)$.
\end{remark}

We now turn to the second context for the $\bb P^1$-bundle formula, namely $\K^\sub{cn}$-modules. Some of the discussion here overlaps with that in our joint work with Bachmann \cite[\S 9.3.5]{BachmannElmantoMorrow} and another variant is discussed in \S\ref{sec:ai-thm}. Here we work over $\bb F$, rather than more generally over $\bb Z$, just because that is our current context of interest.

\begin{definition}[$\K^\sub{cn}$-modules]\label{def:kcn}
The $\infty$-category of {\em modules over connective $K$-theory}, or {\em $\K^\sub{cn}$-modules}, on qcqs $\bb F$-schemes is \[\K^\sub{cn}\opp{-Mod}:=\opp{Mod}_{\K^\sub{cn}}(\opp{PSh}(\Sch^\qcqs_\bb F,\Spt)).\] Given a $\K^\sub{cn}$-module $E$, the class $[\roi(1)]\in K_0(\bb P^1_\bb F)$ defines a natural map ``multiplication by $[\roi(-1)]$'' $[\roi(-1)]:E(\bb P^1_X)\to E(\bb P^1_X)$ for any qcqs $\bb F$-scheme $X$. We say that $E$ satisfies the {\em $\bb P^1$-bundle formula} if the map of spectra \begin{equation}%\label{eq:p1-loc} 
E(X)\oplus E(X)\xto{\pi^* \oplus [\roi(-1)] \circ \pi^*} E(\bb P_X^1)\end{equation} is an equivalence for all $X\in\Sch^\sub{qcqs}_\bb F$, where $\pi: \bb P^1_X \rightarrow X$ denotes the projection map. Denote by \[\K^\sub{cn}\opp{-Mod}_\sub{pbf}\subseteq \K^\sub{cn}\opp{-Mod}\] the full subcategory of $\K^\sub{cn}$-modules satisfying the $\bb P^1$-bundle formula; it is clearly closed under all limits and colimits.
\end{definition}

\begin{remark}[Reduced formulation of $\bb P^1$-bundle formula]
Similarly to Remark \ref{rem_reduced_P1}, the $\bb P^1$-bundle formula for a $\K^\sub{cn}$-module $E$ is equivalent to the difference of the two maps, as a map to the fibre of $\infty^*$, i.e.,
\begin{equation}\label{eqn:P1_for_Kcnmods2}
\pi^*- [\roi(-1)] \circ \pi^*:E(X)\To\fib\big(E(\bb P^1_X)\xto{\infty^*}E(X)\big),
\end{equation}
being an equivalence for all $X\in\Sch^\sub{qcqs}_\bb F$.
\end{remark}

\begin{construction}[Additive invariants]\label{cons_add_invariant}
Similarly to how $L_\sub{Zar}\bb Z(\star)^\sub{lse}\opp{-Mod}_\sub{pbf}$ is to be thought of as an approximation to a derived category of motives, here we explain how to construct objects of $\K^\sub{cn}\opp{-Mod}_\sub{pbf}$ from the theory of {\em non-commutative} motives.

Let $\opp{Cat}^\sub{perf}_\bb F$ denote the $\infty$-category of idempotent-complete $\bb F$-linear stable $\infty$-categories, and let $\opp{Fun}_\sub{add}(\opp{Cat}^\sub{perf}_\bb F,\Spt)$ be the $\infty$-category of $\Spt$-valued additive invariants on them; we call such functors {\em $\bb F$-linear additive invariants}. Unlike our main reference \cite[5.4]{HoyoisScherotzkeSibilla2017} for this formalism, we do not require our additive invariants to be finitary. 

The functor category $\opp{Fun}(\opp{Cat}^\sub{perf}_\bb F,\Spt)$ is symmetric monoidal under Day convolution, with $\bb E_{\infty}$-algebras given by lax monoidal functors \cite{Glasman2015}. Its full subcategory $\opp{Fun}_\sub{add}(\opp{Cat}^\sub{perf}_\bb F,\Spt)$ is closed under Day convolution and contains the lax monoidal functor $\K^\sub{cn}$. In conclusion $\opp{Fun}_\sub{add}(\opp{Cat}^\sub{perf}_\bb F,\Spt)$ is thus a symmetric monoidal category with $\K^\sub{cn}$ being an $\bb E_\infty$-algebra,\footnote{In fact it is expected that $\K^\sub{cn}$ is the unit of $\opp{Fun}_\sub{add}(\opp{Cat}^\sub{perf}_\bb F,\Spt)$. This is known to be true if one further requires that additive invariants be finitary, by \cite[Theorem 5.24]{HoyoisScherotzkeSibilla2017}, though we do not need this result.} and so we can form the symmetric monoidal $\infty$-category \[\opp{Mod}_{\K^\sub{cn}}(\opp{Fun}_\sub{add}(\opp{Cat}^\sub{perf}_\bb F,\Spt))\] of $\bb F$-linear additive invariants which are modules over $\K^\sub{cn}$.

Precomposing along the functor $\Perf:\Sch_\bb F^\sub{qcqs,op}\to \opp{Cat}^\sub{perf}_\bb F$ defines a symmetric monoidal functor
\begin{equation}
 \opp{Mod}_{\K^\sub{cn}}(\opp{Fun}_\sub{add}(\opp{Cat}^\sub{perf}_\bb F,\Spt))\To \K^\sub{cn}\opp{-Mod}_\sub{pbf}\subseteq \K^\sub{cn}\opp{-Mod},\qquad E\mapsto E\circ\Perf,
\label{eq:restrict_perf}\end{equation}
where we are using that the $\K^\sub{cn}$-module $E\circ\Perf$ satisfies the $\bb P^1$-bundle formula in the sense of Definition~\ref{def:kcn}; this holds because of the additivity of $E$ and the semiorthogonal decomposition \[\Perf(\bb P_X^1) \simeq \Perf(\bb P^1_\bb F) \otimes \Perf(X) \simeq \langle \scr O_X(-1), \scr O_X \rangle\] (see, for example, \cite[Theorem~4.25]{CisinskiKhan2020} for a more general result). 

In what follows, we will often speak of a $\K^\sub{cn}$-module ``extending to'' an $\bb F$-linear additive (or, even better, a localising or truncating) invariant. By this we mean that it is in the essential image of \eqref{eq:restrict_perf}, and in practice we will have explicitly given an additive/localising/truncating invariant lifting it.
\end{construction}

The follow lemma is an important example of such phenomena:

\begin{lemma}\label{lem:lcdh-tc} The $\K^\sub{cn}$-module $L_{\cdh}\TC\in\K^\sub{cn}\opp{-Mod}$ extends to an $\bb F$-linear additive invariant.
\end{lemma}
\begin{proof}
First we remark that $\KH$ extends to an additive invariant of $\bb F$-linear categories by the procedure in \cite[Definition~3.13]{LandTamme2019} which is stated for $\bb Z$ but works for more general commutative rings. We define $L_{\cdh}\TC$ by taking the following pushout in $\opp{Mod}_{\K^\sub{cn}}(\opp{Fun}_\sub{add}(\opp{Cat}^\sub{perf}_{\bb F},\Spt))$:
\begin{equation}\label{eq:lax-mon}
\begin{tikzcd}
\K   \ar{r} \ar{d} & \TC \ar[dotted]{d}  \\
\KH  \ar[dotted]{r} & L_{\cdh}\TC.
\end{tikzcd}
\end{equation}
Since the forgetful functor to $\opp{Fun}_\sub{add}(\opp{Cat}^\sub{perf}_\bb F,\Spt)$ preserves finite colimits, the above diagram is also a pushout in $\bb F$-linear additive invariants, whence the fact that the pushout agrees on qcqs $\bb F$-schemes with $L_{\cdh}\TC$ follows from the bicartesian square of Theorem~\ref{thm:mainsq}. 
\end{proof}

\begin{remark}
The pushout in~\eqref{eq:lax-mon} has no reason to promote to a pushout of lax monoidal functors. However it does so after restricting via $\Perf$ to qcqs $\bb F$-schemes, because cdh sheafification is lax monoidal.
\end{remark}

Combining Construction \ref{cons_add_invariant} and Lemma~\ref{lem:lcdh-tc}, we obtain the following non-obvious $\bb P^1$-bundle formulae, which for $\bb F_p$-schemes we do not know how to prove without appealing to the above machinery (for $\bb Q$-schemes one could instead use strong resolution of singularities):

\begin{corollary}\label{cor:lcdh-pbf} %When $\bb F = \bb Q$, the $\K^\sub{cn}$-module structure on $L_{\cdh}\HC^-(-/\bb Q)$ induces natural equivalences for any $X \in \Sch^{\qcqs}_{\bb Q}$:
The $\K^\sub{cn}$-module $L_{\cdh}\TC\in\K^\sub{cn}\opp{-Mod}$ satisfies the $\bb P^1$-bundle formula. In other words, for any qcqs $\bb Q$-scheme $X$ the natural map
\[
L_{\cdh}\HC^-(X/\bb Q)\oplus L_{\cdh}\HC^-(X/\bb Q) \xrightarrow{\pi^* \oplus [\scr O(-1)]\circ\pi^*}L_{\cdh}\HC^-(\bb P^1_X/\bb Q)
\]
is an equivalence, and %When $\bb F = \bb F_p$, the $\K^\sub{cn}$-module structure on $L_{\cdh}\TC$ induces natural equivalences for any $X \in \Sch^{\qcqs}_{\bb Q}$:
for any qcqs $\bb F_p$-scheme $X$, the natural map
\[
L_{\cdh}\TC(X) \oplus L_{\cdh}\TC(X)  \xrightarrow{\pi^* \oplus [\scr O(-1)]\circ\pi^*}L_{\cdh}\TC(\bb P^1_X)
\]
is an equivalence.
\end{corollary}

At this point we have formulated the $\bb P^1$-bundle formula for $L_\sub{Zar}\bb Z(\star)^\sub{lse}$-modules and for $\K^\sub{cn}$-modules. We now interpolate between these two contexts by explaining how in our cases of interest the map \eqref{eqn:P1_for_Kcnmods2} appearing in the $\bb P^1$-bundle formula for $\K^\sub{cn}$-modules respects filtrations on each side, with the induced map on associated gradeds being the map \eqref{eqn:P1_for_cohomologies2} appearing in the $\bb P^1$-bundle formula for $L_\sub{Zar}\bb Z(\star)^\sub{lse}$-modules. This interpolation uses $L_\sub{Zar}\Fil^\star_\sub{lse}\K^\sub{cn}$, namely the presheaf of filtered $\bb E_\infty$-algebras on qcqs $\bb F$-schemes obtained by Zariski sheafifying the filtration of Proposition \ref{prop:mot-filt}; as in Remark \ref{rem_Zariski_lisse}, its underlying presheaf is the Zariski sheafification $L_\sub{Zar}K^\sub{cn}$ of connective $K$-theory and its associated graded is $L_\sub{Zar}\bb Z(\star)^\sub{lse}[2\star]$.

The first step is to show that the first Chern class of Definition \ref{def:P1bf_cohomology} lifts to an ``orientation'' of $L_\sub{Zar}K^\sub{cn}$ landing in $L_\sub{Zar}\Fil^1_\sub{lse}\K^\sub{cn}$; we summarise the situation in the following diagram of presheaves on qcqs $\bb F$-schemes, after which we will explain the notation and constructions:
\begin{equation}
\label{eqn_orientation}
\xymatrix{
L_\sub{Zar}\bb Z(1)^\sub{lse}[2]&L_\sub{Zar}\Fil^1_\sub{lse}\K^\sub{cn}\ar[r] \ar[l]_-{\sub{edge}}& L_\sub{Zar}\K^\sub{cn}\\
(\tau^{\le 1}R\Gamma_\sub{Zar}(-,\bb G_m))[1]\ar[u]^{c_1[2]}&\Sigma^\infty\Omega^\infty\Pic\ar[l]\ar[r]_-{1-(\cdot)^\vee}\ar@{..>}[u]_{\exists!\,1-(\cdot)^\vee} & \K^\sub{cn}\ar[u]
}
\end{equation}
Here, for any qcqs $\bb F$-scheme $X$ we write $\Pic(X)$ for the spectrum corresponding to the symmetric monoidal $1$-groupoid of line bundle on $X$ and isomorphisms between them; in other words $\Pic(X)=(\tau^{\le 1}R\Gamma_\sub{Zar}(X,\bb G_m))[1]$ \cite[Remark 3.4]{BachmannElmantoMorrow}. Then $\Omega^\infty\Pic(X)$ is its underlying pointed space, and  the counit map $\Sigma^\infty\Omega^\infty\Pic(X)\to\Pic(X)=(\tau^{\le 1}R\Gamma_\sub{Zar}(X,\bb G_m))[1]$ appears as the bottom left of \eqref{eqn_orientation}.

There is a unique map $1-(\cdot)^\vee:\Sigma^\infty\Omega^\infty\Pic\to \K^\sub{cn}$ of presheaves on qcqs $\bb F$-schmes which classifies $1-[\roi(-1)]\in \K_0(\bb P^m_\bb F)$ for all $m\ge1$. See \cite[Definition 4.4]{BachmannElmantoMorrow}.

The key point is to establish the existence of the dotted arrow in \eqref{eqn_orientation}, which we will continue to denote by $1-(\cdot)^\vee$:

\begin{lemma}
The composition $\Sigma^\infty\Omega^\infty\Pic \xto{1-(\cdot)^\vee}\K^\sub{cn}\to L_\sub{Zar}\K^\sub{cn}$ of presheaves on qcqs $\bb F$-schemes factors uniquely through $L_\sub{Zar}\Fil^1_\sub{lse}\K^\sub{cn}$. Moreover, the resulting left square in \eqref{eqn_orientation} commutes.
\end{lemma}
\begin{proof}
We reduce to the case of smooth $\bb F$-schemes as follows. The top row of the diagram consists of Zariski sheaves and the presheaf $L_\sub{Zar}\Sigma^\infty\Omega^\infty\Pic$ is Zariski locally left Kan extended from smooth $\bb F$-schemes (use \cite[Proposition 2.23]{BachmannElmantoMorrow} to reduce to the case of local rings, then commutation of $\Sigma^\infty$ and $\Omega^\infty$ with sifted colimits to reduce to the fact that $\bb G_m$ of $\bb F$-algebras is left Kan extended from smooth $\bb F$-algebras).

So by the adjunction of restriction and Zariski-local left Kan extension, the lemma becomes equivalent to the following: the map $1-(\cdot)^\vee:\Sigma^\infty\Omega^\infty\Pic \to \K$ of presheaves on smooth $\bb F$-schemes factors uniquely through $\Fil^1_\bb A\K$, and the resulting square
\begin{equation}
\xymatrix{
\bb Z(1)^\bb A[2]&\Fil^1_\bb A\K \ar[l]_-{\sub{edge}}\\
R\Gamma_\sub{Zar}(-,\bb G_m)[1]\ar[u]^{c_1^{\bb A}[2]}&\Sigma^\infty\Omega^\infty\Pic\ar[l]\ar[u]_{1-(\cdot)^\vee} 
}
\end{equation}
on $\Sm_\bb F$ commutes. We may now quote our work with Bachmann: the unique factoring was proved in \cite[Lemma 4.33]{BachmannElmantoMorrow}, and the resulting square commutes by \cite[Construction 4.35]{BachmannElmantoMorrow}.
\end{proof}

This finishes our discussion of the diagram \eqref{eqn_orientation}. We now define the map of\ presheaves of spectra \[1-[\roi(-1)]:\Sigma^\infty(\bb P_\bb F^1,\infty)\To L_\sub{Zar}\Fil^1_\sub{lse}\K^\sub{cn}\] to be the precomposition of $1-(\cdot)^\vee:\Sigma^\infty\Omega^\infty\Pic\to L_\sub{Zar}\Fil^1_\sub{lse}\K^\sub{cn}$ along $\Sigma^\infty$ of the map of pointed spaces $[\roi(1)]:(\bb P_\bb F^1,\infty)\to \Omega^\infty\Pic$.

\begin{definition}[$L_\sub{Zar}\Fil^\star_\sub{lse}\K^\sub{cn}$-modules]\label{def_LzarKlse}
Recalling from above that $L_\sub{Zar}\Fil^\star_\sub{lse}\K^\sub{cn}$ is a presheaf of $\bb E_\infty$-algebras in filtered spectra on qcqs $\bb F$-schemes, we may define the $\infty$-category of {\em $L_\sub{Zar}\Fil^\star_\sub{lse}\K^\sub{cn}$-modules} 
\[L_\sub{Zar}\Fil^\star_\sub{lse}\K^\sub{cn}\opp{-Mod}:=\mathrm{Mod}_{L_\sub{Zar}\Fil^\star_\sub{lse}\K^\sub{cn}}(\PShv(\Sch^\sub{qcqs}_\bb F,\Fil\Spt)).\] Given $E\in L_\sub{Zar}\Fil^\star_\sub{lse}\K^\sub{cn}\opp{-Mod}$, the adjoint to
\[
\Sigma^{\infty}(\bb P^1_\bb F, \infty) \otimes \Fil^{\star-1}E \xrightarrow{(1 - [\scr O(-1)]) \otimes \id} L_\sub{Zar}\Fil^1_\sub{lse}\K^\sub{cn}\otimes \Fil^{\star-1}E \xrightarrow{\rm act} \Fil^{\star}E
\]
defines a morphism ``pullback and multiply by $1-[\roi(-1)]$''
\begin{equation}
\mathrm{``}(1-[\roi(-1)])\pi^*\mathrm{"}:\Fil^{\star-1} E\to\fib\big(\Fil^{\star}E(\bb P^1_-)\xto{\infty^*}\Fil^\star E\big)
\label{eqn_mult_by1-O}\end{equation}
in $L_\sub{Zar}\Fil^\star_\sub{lse}\K^\sub{cn}\opp{-Mod}$.
\end{definition}

\begin{example}\label{example_LcdhTC_as_Kcn}
Zariski locally left Kan extending the filtered trace map of Propositions \ref{prop:compat} and \ref{prop:mot-v-bms} defines maps of presheaves of filtered $\bb E_\infty$-algebras $L_\sub{Zar}\Fil^\star_\sub{lse}\K^\sub{cn}\to \Fil^\star_\sub{HKR}\HC^-(-/\bb Q)$ (when $\bb F=\bb Q$) and $L_\sub{Zar}\Fil^\star_\sub{lse}\K^\sub{cn}\to \Fil^\star_\sub{BMS}\TC$ (when $\bb F=\bb F_p)$. These may also be composed with the canonical maps to the cdh sheafifications of the targets. Forgetting algebra structure we thus obtain $L_\sub{Zar}\Fil^\star_\sub{lse}\K^\sub{cn}$-modules \[\Fil^\star_\sub{HKR}\HC^-(-/\bb Q),\quad\Fil^\star_\sub{HKR}L_\sub{cdh}\HC^-(-/\bb Q),\quad\Fil^\star_\sub{BMS}\TC,\quad\Fil^\star_\sub{BMS}L_\sub{cdh}\TC,\] of which the last one will be particularly relevant to us.
\end{example}

We now note that $L_\sub{Zar}\Fil^\star_\sub{lse}\K^\sub{cn}$-modules interpolate between $L_\sub{Zar}\bb Z(\star)^\sub{lse}$-modules and $\K^\sub{cn}$-modules, in the sense that there are lax symmetric monoidal functors
\begin{equation}
\xymatrix@R=3mm{
{L_\sub{Zar}\bb Z(\star)^\sub{lse}}\opp{-Mod} & {L_\sub{Zar}\Fil^\star_\sub{lse}\K^\sub{cn}}\opp{-Mod}\ar[l]\ar[r] & {\K^\sub{cn}}\opp{-Mod}\\
\gr^\star E[2\star] &\Fil^\star E\ar@{|->}[r] \ar@{|->}[l] & \colim_{j\to-\infty}\Fil^j E
}
\label{eqn:interpolation}
\end{equation}
Here the leftwards functor takes graded pieces and then shears the indexing; the rightwards functor takes the underlying $L_\sub{Zar}\K^\sub{cn}$-module $\colim_{j\to-\infty}\Fil^j E$ and then restricts along $\K^\sub{cn}\to L_\sub{Zar}\K^\sub{cn}$. These functors are moreover compatible with $\bb P^1$-bundle formulae as follows:

\begin{lemma}\label{lemma:interpolation}
Let $\Fil^\star E\in {L_\sub{Zar}\Fil^\star_\sub{lse}\K^\sub{cn}}\opp{-Mod}$, with associated $L_\sub{Zar}\bb Z(\star)^\sub{lse}$-module $\gr^\star E[2\star]\in L_\sub{Zar}\bb Z(\star)^\sub{lse}\opp{-Mod}$ and associated $\K^\sub{cn}$-module $\colim_{j\to-\infty}\Fil^j E\in \K^\sub{cn}\opp{-Mod}$ given by the functors \eqref{eqn:interpolation}
\begin{enumerate}
\item The leftwards functor of \eqref{eqn:interpolation} sends \eqref{eqn_mult_by1-O} to the map \[c_1(\roi(1))\circ\pi^*:\gr^{\star-1} E[2\star-2]\to \fib\big(\gr^\star E(\bb P^1_-)[2\star]\xto{\infty^*}\gr^\star E[2\star]\big),\] i.e., \eqref{eqn:P1_for_cohomologies2} for the  $L_\sub{Zar}\bb Z(\star)^\sub{lse}$-module $\gr^\star E[2\star]$.
\item The rightwards functor of \eqref{eqn:interpolation} sends \eqref{eqn_mult_by1-O} to the map \[\pi^*- [\roi(-1)] \circ \pi^*: \colim_{j\to-\infty}\Fil^jE\to \fib\big(\colim_{j\to-\infty}\Fil^jE(\bb P^1_-)\xto{\infty^*}\colim_{j\to-\infty}\Fil^jE\big),\] i.e., \eqref{eqn:P1_for_Kcnmods2} for the $\K^\sub{cn}$-module $\colim_{j\to-\infty}\Fil^jE$.
\end{enumerate}
\end{lemma}
\begin{proof}
Part (1) follows from the left commutative square in \eqref{eqn_orientation}, while part (2) follows from the right commutative square.
\end{proof}

\begin{remark}
We record some consequences of Lemma \ref{lemma:interpolation} which are not necessary for what follows, but whose absence would perhaps cause confusion. Let us say that a $L_\sub{Zar}\Fil^\star_\sub{lse}\K^\sub{cn}$-module $\Fil^\star E$ satisfies the $\bb P^1$-bundle formula if the map of filtered presheaves of spectra \eqref{eqn_mult_by1-O} is an equivalence. Then Lemma \ref{lemma:interpolation} shows that the functors \eqref{eqn:interpolation} restrict to the subcategories of modules satisfying $\bb P^1$-bundle formulae:
\[\xymatrix@R=3mm{
L_\sub{Zar}\bb Z(\star)^\sub{lse}\opp{-Mod}_\sub{pbf} &L_\sub{Zar}\Fil^\star_\sub{lse}\K^\sub{cn}\opp{-Mod}_\sub{pbf}\ar[l]\ar[r] & \K^\sub{cn}\opp{-Mod}_\sub{pbf}
}
\]
Furthermore, suppose $\Fil^\star E$ is a $L_\sub{Zar}\Fil^\star_\sub{lse}\K^\sub{cn}$-module and $X$ is a qcqs $\bb F$-scheme for which the following hold: the filtrations $\Fil^\star E(X)$ and $\Fil^\star E(\bb P^1_X)$ are $\bb N$-indexed and complete, and the underlying $\K^\sub{cn}$-module satisfies the $\bb P^1$-bundle formula on $X$. Then $\Fil^\star$ satisfies the $\bb P^1$-bundle formula on $X$ if and only if the same is true of the associated $L_\sub{Zar}\bb Z(\star)^\sub{lse}$-module $\gr^\star E[2\star]$. (It is exactly for this reason that we will not use again the terminology of a $L_\sub{Zar}\Fil^\star_\sub{lse}\K^\sub{cn}$-module satisfying the $\bb P^1$-bundle formula.)
\end{remark}

We now have the necessary tools to explain a general strategy which can sometimes be used to prove the $\bb P^1$-bundle formula for $L_\sub{Zar}\bb Z(\star)^\sub{lse}$-modules $F(\star)$. Suppose that $F(\star)=\gr^\star E[2\star]$ occurs as the sheared graded pieces of some $L_\sub{Zar}\Fil^\star_\sub{lse}\K^\sub{cn}$-module $\Fil^\star E$ such that the underlying $\K^\sub{cn}$-module $\colim_{j\to-\infty}\Fil^jE$ comes from an $\bb F$-linear additive invariant. Then, evaluating on some fixed qcqs $\bb F$-scheme $X$, the map \eqref{eqn_mult_by1-O} yields a map of filtered spectra
\[\mathrm{``}(1-[\roi(-1)])\pi^*\mathrm{"}:\Fil^{\star-1} E(X)\to\fib\big(\Fil^{\star}E(\bb P^1_X)\xto{\infty^*}\Fil^\star E(X)\big)\]
which is an equivalence on underlying $\K^\sub{cn}$-modules (by the compatibility of Lemma \ref{lemma:interpolation}(2) and Construction \ref{cons_add_invariant}). Depending on $X$ and on the extensions occurring in the filtered objects, it will sometimes be possible to deduce from the equivalence on underlying spectra that the map is also an equivalence on associated gradeds. (The most extreme example to have in mind is if we knew that the filtrations on both sides were the Postnikov filtration) By the compatibility of Lemma \ref{lemma:interpolation}(1), this would establish the $\bb P^1$-bundle formula for $F(\star)$ on~$X$.

We will use exactly this strategy in the next subsection to prove Theorem \ref{thm:cdh-syn-pbf}.

\subsection{$\P^1$-bundle formula for cdh sheafified syntomic cohomology}\label{sec:syn-p1}
As will become clear in \S\ref{ss_proof_of_Pn_bundle}, the main new part of Theorem~\ref{thm:pbf} is establishing it $p$-adically in characteristic $p > 0$. For this subsection we fix $\bb F = \F_p$. Although syntomic cohomology satisfies the $\bb P^1$-bundle formula as explained in Example~\ref{exam:syn_pbf}, there is no a priori reason that this property should be inherited by its cdh sheafification; nevertheless, in this subsection we will show that it is true. More precisely, we are interested in the presheaf of $\bb E_\infty$-algebras in graded complexes $L_{\cdh}\bb Z_p(\star)^\sub{syn}$ on $\Sch^\sub{qcqs}_{\bb F_p}$, together with its first Chern class map \[c_1:R\Gamma_\sub{Nis}(-,\bb G_m)[-1]\To L_\sub{cdh}\bb Z_p(1)^\sub{syn}\] (defined either by cdh sheafifying the usual syntomic first Chern class, or composing the first Chern class of Definition \ref{def:P1bf_cohomology} along $L_\sub{Zar}\bb Z(1)^\sub{lse}\to \bb Z_p(1)^\sub{syn}\to L_{\cdh}\bb Z_p(1)^\sub{syn}$; these coincide by the argument of Example~\ref{exam:syn_pbf}), and our goal is to establish the following $\bb P^1$-bundle formula:

\begin{theorem}\label{thm:cdh-syn-pbf}
For any qcqs $\bb F_p$-scheme $X$ and $j\ge0$, the map
\begin{equation}\label{eq:lcdh-pbf}
L_{\cdh}\Z_p(j)^{\syn}(X)\oplus L_{\cdh}\Z_p(j-1)^{\syn}(X)[-2] \xrightarrow{\pi^* \oplus c_1(\scr O(1))\circ\pi^* } L_{\cdh}\Z_p(j)^{\syn}(\P^1_X)
\end{equation}
is an equivalence.
\end{theorem}

\begin{remark}[No circular reasoning]\label{rem_not_circular}
As already explained in \S\ref{ss_cdh_local}, the motivic cohomology developed in this paper depends on some of our joint work with Bachmann \cite{BachmannElmantoMorrow}. On the other hand, {\em op.~cit.} makes use of Theorem \ref{thm:cdh-syn-pbf} when establishing the $\bb P^1$-bundle formula for mod-$p$ cdh-motivic cohomology on $\bb F_p$-schemes (see the proof of \cite[Lemma 9.25]{BachmannElmantoMorrow}).

We assure the reader that there is no circular reasoning, and have tried to make this clear in the exposition. The proof of Theorem \ref{thm:cdh-syn-pbf} in this subsection depends only on the techniques of \S\ref{sec:p1} (which makes use of classical motivic cohomology of smooth $\bb F$-schemes -- where admittedly we sometimes quote early sections of \cite{BachmannElmantoMorrow} for foundational results, but these are totally independent of Theorem \ref{thm:cdh-syn-pbf}), and isolated results of \S\ref{sec:charp} about syntomic cohomology.
\end{remark}

We will prove the theorem by implementing the strategy outlined at the end of the previous subsection; the techniques overlap with \cite[\S 8]{BachmannElmantoMorrow} but differ considerably in details. For each qcqs $\bb F_p$-scheme $X$ and $j\ge0$, let $C(j)(X)$ denote the cofiber of \eqref{eq:lcdh-pbf}, or equivalently the cofiber of \[c_1(\scr O(1))\circ\pi^* :L_{\cdh}\Z_p(j-1)^{\syn}(X)[-2]\To\fib\big(L_{\cdh}\Z_p(j)^{\syn}(\P^1_X)\xto{\infty^*}L_{\cdh}\Z_p(j)^{\syn}(X)\big)\]
Our goal is to establish the vanishing of the $C(j)(X)$. The following proposition is the main payoff of the previous subsection, showing that the $C(j)(X)$ appear up to shifting as the graded pieces of a filtration on the zero spectrum:

\begin{proposition}\label{proposition_payoff}
For any qcqs $\bb F_p$-scheme $X$ of finite valuative dimension, the spectrum $0$ admits a natural, bounded, $\bb N$-indexed filtration with graded pieces $C(j)(X)[2j]$ for $j\ge0$.
\end{proposition}
\begin{proof}
By cdh sheafifying in Example \ref{exam:syn_pbf}, we view $L_{\cdh}\bb Z_p(\star)^\sub{syn}$ as a $L_\sub{Zar}\bb Z(\star)^\sub{lse}$-module, so that Theorem~\ref{thm:cdh-syn-pbf} is the assertion that it satisfies the $\bb P^1$-bundle formula in the sense of Definition \ref{def:P1bf_cohomology}. We also view $\Fil^{\star}_\sub{BMS}L_{\cdh}\TC$ as a $L_\sub{Zar}\Fil^{\star}_\sub{lse}\K^\sub{cn}$-module as in Example \ref{example_LcdhTC_as_Kcn}. These are compatible via the functors \eqref{eqn:interpolation}, i.e.,
\[
\xymatrix{
L_{\cdh}\bb Z_p(\star)^\sub{syn}[2\star] &\Fil_\sub{BMS}^\star L_\sub{cdh}\TC \ar@{|->}[r] \ar@{|->}[l] & \colim_{j\to-\infty}\Fil_\sub{BMS}^\star L_\sub{cdh}\TC \simeq L_\sub{cdh}\TC
}
\]
(for the leftwards functor we are using the compatibility of the trace map and dlog maps, as in Remark \ref{rem:horizontal_map_charp}) and so from the compatibilities of Lemma \ref{lemma:interpolation} we obtain the following conclusion for any qcqs $\bb F_p$-scheme $X$: The natural $\bb N$-indexed filtered spectrum given by the cofiber of
\begin{equation}\label{eq:filtered-cdh-tc}
\mathrm{``}(1-[\roi(-1)])\pi^*\mathrm{"}:\Fil^{\star-1}_\sub{BMS}L_{\cdh}\TC (X)\to\fib\big(\Fil^{\star}_\sub{BMS}L_{\cdh}\TC(\bb P^1_X)\xto{\infty^*}\Fil^\star_\sub{BMS}L_{\cdh}\TC(X)\big )
\end{equation}
(i.e., \eqref{eqn_mult_by1-O} for $\Fil^{\star}_\sub{BMS}L_{\cdh}\TC$) has associated graded $C(j)(X)[2\star]$; moreover its underlying spectrum is the cofiber of
\[
\pi^*- [\roi(-1)] \circ \pi^*:L_{\cdh}\TC(X)  \to \fib\big(L_{\cdh}\TC(\bb P^1_X)\xto{\infty^*}L_{\cdh}\TC(X)\big)
\]
(i.e., \eqref{eqn:P1_for_Kcnmods2} for $L_\sub{cdh}\TC$), which vanishes by Corollary \ref{cor:lcdh-pbf}. 

Ignoring the boundedness claim this establishes the statement of the proposition for arbitrary qcqs $\bb F_p$-schemes. In the case when $X$ has finite valuative dimension, the boundedness claim follows from the boundedness of the filtration $\Fil^{\star}_\sub{BMS}L_{\cdh}\TC(X)$, which was explained in Remark \ref{rem:cdh-local}.
\end{proof}

In order to show that the presheaves $C(j)$ vanish by exploiting the previous proposition, we first need to know that they satisfy a number of favourable properties:

\begin{proposition}\label{prop:C-properties}
The following hold for all $j\ge0$:
\begin{enumerate}
\item $C(j)$ is a finitary cdh sheaf on qcqs $\bb F_p$-schemes, and $C(j)[\tfrac1p]\simeq 0$.
\item Let $X$ be an $\bb F_p$-scheme which is smooth over a valuation ring, and assume that $X$ has finite valuative dimension $\le d$. Then $C(j)(X)$ is derived $p$-complete and supported in cohomological degrees $[j+1,j+d+2]$.
\item $C(j)$ satisfies henselian v-excision. (We will recall this terminology in the course of the proof.)
\item $C(j)(k)\simeq 0$ for any field $k$ of characteristic zero.
\end{enumerate}
\end{proposition}
\begin{proof} 
(1): As the cofiber of a map of cdh sheaves, $C(j)$ is a cdh sheaf. Furthermore, $C(j)[\tfrac1p]$ vanishes since $\bb Q(j)^\sub{syn}\quis L_\sub{cdh}\bb Q_p(j)^\sub{syn}$ by Lemma \ref{lemma_Qpsyn}, which satisfies the $\bb P^1$-bundle formula by inverting $p$ in Example \ref{exam:syn_pbf}. To prove the finitariness claim it is therefore enough to prove it modulo $p$, which follows from the same property of $L_\sub{cdh}\bb F_p(j)^\sub{syn}$, already explained at the end of the proof of Theorem \ref{thm:graded-pieces_charp}.

(2): Now let $X$ be an $\bb F_p$-scheme of finite valuative dimension $\le d$ which is smooth over a valuation ring. Then Lemma \ref{lemma_derived_p_complete}(2) implies that $L_\sub{cdh}\bb Z_p(j)^\sub{syn}(X)$ and $L_\sub{cdh}\bb Z_p(j)^\sub{syn}(\bb P_X^1)$ are derived $p$-complete, whence the same is true of $C(j)(X)$. To establish the range of support it is therefore enough to show that $C(j)(X)/p$ is supported in cohomological degrees $[j,j+d+2]$. From Remark \ref{remark_eh} we know that $L_\sub{cdh}\bb F_p(j)^\sub{syn}$ is equivalent to the cdh sheafification of $R\Gamma_\sub{\'et}(-,\Omega^j_\sub{log})[-j]$. On any affine the latter is supported in cohomological degrees $[j,j+1]$ (c.f., Remark \ref{remarks_tildenu}), so after cdh sheafifiying and evaluating on $X$ we deduce that $L_\sub{cdh}\bb F_p(j)^\sub{syn}(X)$ is supported in $[j,j+d+1]$, and that $L_\sub{cdh}\bb F_p(j)^\sub{syn}(\bb P^1_X)$ is supported in $[j,j+d+2]$; recall here that the $\cdh$ cohomological dimension of a qcqs scheme is bounded by its valuative dimension \cite[Proposition 2.4.3]{ElmantoHoyoisIwasaKelly2021}. We may now read off from the definition of $C(j)(X)$ that it is supported in degrees $[j-1,j+d+2]$ and that \[H^{j-1}(C(j)(X))=\ker(H^j(L_\sub{cdh}\bb F_p(j)^\sub{syn}(X))\to H^j(L_\sub{cdh}\bb F_p(j)^\sub{syn}(\bb P^1_X));\] but this map is split by $\infty^*$, so injective, thereby completing the proof of the range of support.

(3): Recall from \cite[Definition 3.3.2]{ElmantoHoyoisIwasaKelly2021} that a functor $F$ on $\bb F_p$-algebras, valued in $\mathrm{Sp}$ or $\rm{D}(\bb Z)$, is said to satisfy henselian v-excision if, for every valuation ring $V$ of characteristic $p$ and prime ideal $\frak p\subseteq V$, the functor $F$ carries the Milnor square
\begin{equation}\label{eq:hv}
\begin{tikzcd}
V \ar{d} \ar{r} & V_{\mathfrak{p}} \ar{d}\\
V/\mathfrak{p} \ar{r} & \kappa(\mathfrak{p}) 
\end{tikzcd}
\end{equation}
to a cartesian square. This holds for $\bb F_p(j)^\syn$ by \cite[Proposition 7.25]{BachmannElmantoMorrow} (it reduces to henselian v-excision for the cotangent complex by the increasing filtration in Lemma~\ref{lem_fin_fil_on_syn}), hence also for $L_\sub{cdh}\bb F_p(j)^\sub{syn}$ since cdh sheafification does not change the value of a presheaf on henselian valuation rings. It now follows from \cite[Lemma~3.3.7]{ElmantoHoyoisIwasaKelly2021} that $L_\sub{cdh}\bb F_p(j)^\sub{syn}(\bb P^1_-)$ also satisfies henselian v-excision; to apply that lemma we note that $L_\sub{cdh}\bb F_p(j)^\sub{syn}$ is finitary by (1) and lands in $\rm{D}(\bb Z)$, where compact objects are cotruncated. Taking a cofiber we deduce that $C(j)/p$ satisfies henselian v-excision, which suffices to complete the proof of (3) since $C(j)[\tfrac1p]$ vanishes by (1).

(4): The comparison maps $\bb F_p(j)^\sub{syn}(k)\to L_\sub{cdh}\bb F_p(j)^\sub{syn}(k)$ and $\bb F_p(j)^\sub{syn}(\bb P^1_k)\to L_\sub{cdh}\bb F_p(j)^\sub{syn}(\bb P^1_k)$ are equivalences for all $j$: the first because fields are points for the cdh topology; the second because the presheaves are both sides are Nisnevich sheaves and the Nisnevich local rings of $\bb P_k^1$ are henselian discrete valuation rings, which are again points for the cdh topology. Combined with (1), this reduces part (4) to the $\bb P^1$-bundle formula for mod-$p$ syntomic cohomology, which we already saw holds in Example \ref{exam:syn_pbf}.
\end{proof}

We now have the necessary tools to prove Theorem~\ref{thm:cdh-syn-pbf}:

\begin{proof}[Proof of Theorem~\ref{thm:cdh-syn-pbf}]
Our goal is to establish the vanishing of $C(j)$ of arbitrary qcqs $\bb F_p$-schemes.

{\bf Step 1.} Proposition \ref{prop:C-properties} allows us to carry out some initial reductions. Indeed, by part (1) it is enough to prove vanishing of $C(j)/p$; but this is a finitary cdh sheaf satisfying henselian v-excision by parts (1) and (3), so by \cite[Corollary 2.4.19]{ElmantoHoyoisIwasaKelly2021} it is enough to check vanishing on henselian valuation rings of rank $\le 1$. For the rest of the proof we fix a henselian valuation ring $V\supseteq \bb F_p$ of rank $\le 1$, and we will show that $C(j)(V)$ vanishes for all $j\ge0$.

{\bf Step 2.} We will show that $H^j(C(j)(V)/p)=0$ for all $j\ge0$. We remind the reader that, as in the proof of Proposition \ref{prop:C-properties}, $L_\cdh\bb F_p(j)^\sub{syn}$ is the cdh sheafification of $R\Gamma_\sub{\'et}(-,\Omega^j_\sub{log})[-j]$, and in particular its value on any qcqs $\bb F_p$ scheme of valuation dimension $\le d$ is supported in $[j,j+d+1]$. We will sometimes use this without mention.

Denote by 
\[
\tilde{H}^n(L_{\cdh}\bb F_p(j)^\sub{syn}(\bb P^1_V)):=\ker( H^n(L_{\cdh}\bb F_p(j)^\sub{syn}(\bb P^1_V)) \xto{\infty^*} H^n(\bb F_p(j)^\sub{syn}(V)) ))
\]
(which is isomorphic to the cokernel of $H^n(\bb F_p(j)^\sub{syn}(V))\to H^n(L_{\cdh}\bb F_p(j)^\sub{syn}(\bb P^1_V))$) the reduced cdh syntomic cohomology of $\bb P^1_V$. From the definition of $C(j)$ there is an exact sequence
\begin{equation}\label{eq:h0-eh}
0 \rightarrow \tilde{H}^j(L_{\cdh}\bb F_p(j)^\sub{syn}(\bb P^1_V)) \rightarrow H^j(C(j)(V)/p) \rightarrow H^{j-1}(\bb F_p(j-1)^\sub{syn}(V)) \xrightarrow{c_1(\roi(1))\circ\pi^*} \tilde{H}^{j+1}(L_{\cdh}\bb F_p(j)^\sub{syn}(\bb P^1_V))).
\end{equation}
From the exact sequence~\eqref{eq:h0-eh}, to complete step 2 it is enough to prove the following two claims:
\begin{enumerate}
\item[(1)] $\tilde{H}^j(L_{\cdh}\bb F_p(j)^\sub{syn}(\bb P^1_V)) = 0$, and
\item[(2)] the map $c_1(\roi(1))\circ\pi^*$ in~\eqref{eq:h0-eh} is injective. 
\end{enumerate}

Let us first prove (1); we may assume $V$ has rank exactly $1$, as otherwise we may appeal to Proposition~\ref{prop:C-properties}(4). Let $\kappa$ be the residue field of $V$, and $F$ the field of fractions, and consider the commutative diagram
\begin{equation}\label{eq:part-exact}
\begin{tikzcd}
H^j(L_{\cdh}\bb F_p(j)^{\syn}(\P_V^1)) \ar{r}{\infty^{\ast}} \ar{d} &  H^j(\bb F_p(j)^{\syn}(V)) \ar{d}\\
H^j(L_{\cdh}\bb F_p(j)^{\syn}(\P_{\kappa}^1)) \times H^j(L_{\cdh}\bb F_p(j)^{\syn}(\P_{F}^1)) \ar{r}{\infty^{\ast}} & H^j(\bb F_p(j)^{\syn}(\kappa)) \times H^j(\bb F_p(j)^{\syn}(F)) 
\end{tikzcd}
\end{equation}
We want to prove that the top horizontal is injective. By Proposition \ref{prop:C-properties}(4), we know that $C(j)$ vanishes on $\kappa$ and $F$, whence the left term of sequence~\eqref{eq:h0-eh} vanishes if we replace $V$ by $F$ or $\kappa$; in other words, we have shown that the bottom horizontal arrow of \eqref{eq:part-exact} is injective. It now suffices to prove that the left vertical map is injective. But the bound of Proposition \ref{prop:C-properties} implies that the abelian presheaf $H^j(L_{\cdh}\bb F_p(j)(-))$ on qcqs $\bb F_p$-schemes is separated for the cdh topology, and so the first map in the composite
\begin{equation}\label{eq:separated}
H^j(L_{\cdh}\bb F_p(j)^\sub{syn}(\bb P^1_V)) \To \prod_{\Spec(W) \to \bb P^1_V} H^j(\bb F_p(j)^\sub{syn}(W))  \To \prod_{\Spec(W)\to \bb P^1_V} H^j(\bb F_p(j)^\sub{syn}(\mathrm{Frac}(W)))
\end{equation}
is injective; here the products run across all henselian valuation rings $W$ mapping to $\bb P^1_V$. Furthermore, the second map is also injective: each factor identifies with the map $\Omega^j_{W, \log} \rightarrow \Omega^j_{\mathrm{Frac}(W), \log}$, which is injective by \cite[Corollary 6.5.21]{GabberRamero2003}. We may now conclude the injectivity of the left vertical map as follows: suppose $\alpha \in H^j(L_{\cdh}\bb F_p(j)^{\syn}(\P_V^1))$ vanishes after pullback to both $\bb P^1_{\kappa}$ and $\bb P^1_{F}$. Then for any valuation ring $W$ and a map $\Spec(W) \rightarrow \bb P^1_V$, the induced map $\Spec(\mathrm{Frac}(W)) \rightarrow \bb P^1_V$ factors through either $\bb P^1_{\kappa}$ or $\bb P^1_{F}$; in either case we conclude that $\alpha$ vanishes after pullback to $\mathrm{Frac}(W)$. The injectivity of~\eqref{eq:separated} now implies $\alpha=0$, completing the proof of claim (1).

We now prove claim (2). We compare the exact sequences~\eqref{eq:h0-eh} for $V$ and $F$ 
\[
\begin{tikzcd}
0 \ar{r} & H^j(C(j)(V)/p) \ar{r} \ar{d} & H^{j-1}(\bb F_p(j-1)^\sub{syn}(V)) \ar{rr}{c_1(\roi(1))\circ\pi^*}  \ar{d} & & \tilde{H}^{j+1}(L_{\cdh}\bb F_p(j)^\sub{syn}(\bb P^1_V)))\ar{d}\\
0 \ar{r} & H^j(C(j)(F)/p) =0 \ar{r} & H^{j-1}(\bb F_p(j-1)^\sub{syn}(F)) \ar{rr}{c_1(\roi(1))\circ\pi^*}  & & \tilde{H}^{j+1}(L_{\cdh}\bb F_p(j)^\sub{syn}(\bb P^1_{F})))
\end{tikzcd}
\]
where we use Proposition \ref{prop:C-properties}(4) at the bottom left. The middle map identifies with he map $\Omega^{j-1}_{V/\bb F_p} \rightarrow \Omega^{j-1}_{\mathrm{Frac}(V)/\bb F_p}$, which is injective again by \cite[Corollary 6.5.21]{GabberRamero2003}. Therefore the top right horizontal map is injective, as required to establish (2) and complete the step 2.

{\bf Step 3.} A priori $C(j)(V)$ is supported in $[j+1,j+3]$ by Proposition~\ref{prop:C-properties}(2), but step 2 shows that $H^{j+1}(C(j)(V))$ is $p$-torsion-free and so vanishes by Proposition \ref{prop:C-properties}(1). Therefore the spectral sequence $E_2^{ij}=H^{i-j}(C(-j)(V))\Rightarrow 0$ arising from Proposition \ref{proposition_payoff} is supported in two adjacent columns: there is no room for differentials and so all the terms of the $E_2$ page vanish.
\end{proof}

\subsection{Proof of Theorem \ref{thm:pbf-blowup}}\label{ss_proof_of_Pn_bundle}
Again let $\bb F$ be an arbitrary prime field. We may now prove Theorem \ref{thm:pbf-blowup}; we begin with the projective bundle formula in the special case $r=1$:

\begin{theorem}[$\bb P^1$-bundle formula for motivic cohomology]\label{thm:pbf} 
For any qcqs $\bb F$-scheme $X$ and $j \in\bb Z $, the map
\begin{equation}\label{eq:pbf}
\Z(j)^{\mot}(X)\oplus \Z(j-1)^{\mot}(X)[-2] \xrightarrow{\pi^* \oplus c_1(\scr O(1))\circ \pi^* } \Z(j)^{\mot}(\P^1_X)
\end{equation}
is an equivalence. 
\end{theorem}
\begin{proof}
First suppose that $\bb F=\bb Q$. Thanks to the pullback square Theorem~\ref{thm:graded-pieces}(2) it suffices to prove that the $\bb P^1$-bundle formula holds for the cohomologies
\[
R\Gamma(-,\widehat{L\Omega}_{-/\Q}^{\geq \star}), \qquad  R\Gamma_\sub{cdh}(-,\widehat{L\Omega}_{-/\Q}^{\geq \star}), \qquad \bb Z(\star)^\bb A.
\] 
It is moreover enough to treat $\bb Q$-schemes of finite type, by the finitariness of motivic cohomology (Theorem \ref{thm:graded-pieces}). For the $\bb P^1$-bundle formula of $\bb A^1$-motivic cohomology see Theorem \ref{thm:cdh}(7), while the $\bb P^1$-bundle formula for $R\Gamma(-,\widehat{L\Omega}_{-/\Q}^{\geq \star})$ has been recorded in Example~\ref{exam:dr_pbf}. Finally, for $R\Gamma_\sub{cdh}(-,\widehat{L\Omega}_{-/\Q}^{\geq \star})$ we use resolution of singularities to reduce to smooth $\bb Q$-schemes (here is where we benefit from our restriction to finite type $\bb Q$-schemes), where $R\Gamma_\sub{cdh}(-,\widehat{L\Omega}_{-/\Q}^{\geq \star})$ agrees\footnote{For any smooth $\bb Q$-scheme $X$, the arrows
\[\xymatrix{
R\Gamma(X,\hat{L\Omega}^{\ge j}_{-/\bb Q}) \ar[r]\ar[d] & R\Gamma_\sub{cdh}(X,\hat{L\Omega}^{\ge j}_{-/\bb Q})\ar[d]\\
R\Gamma(X,\Omega^{\ge j}_{-/\bb Q})\ar[r] & R\Gamma_\sub{cdh}(X,\Omega^{\ge j}_{-/\bb Q})
}\]
are all equivalences. The case of the left arrow follows from the equivalences $L_{R/\bb Q}^i\quis\Omega^i_{R/\bb Q}$ for smooth $\bb Q$-algebras $R$. For the right vertical arrow use cdh descent of derived de Rham cohomology (Lemma \ref{lemma_cdh_descent_HP}) to replace $\ge j$ by $<j$, and then that $L_{-/\bb Q}^i\to\Omega^i_{-/\bb Q}$ is an equivalence cdh locally (as in the proof of Theorem \ref{thm:graded-pieces}). The bottom horizontal arrow reduces to showing that $R\Gamma(X,\Omega^{i}_{-/\bb Q})\quis R\Gamma_\sub{cdh}(X,\Omega^{i}_{-/\bb Q})$ for all $i$, which is a well-known consequence of strong resolution of singularities and $R\Gamma(-,\Omega^{i}_{-/\bb Q})$ satisfying the regular blow-up formula on smooth $\bb Q$-schemes; see the references in the first paragraph of the proof of Corollary~\ref{corol_cdh_filtered_trace}. Finally the top horizontal arrow is an equivalence by the commutativity of the diagram.
}
with $R\Gamma(-,\widehat{L\Omega}_{-/\Q}^{\geq \star})$ and so we may again appeal to Example~\ref{exam:dr_pbf}.

Next we treat the case $\bb F = \bb F_p$. Then, similarly to the previous characteristic zero context, Theorem~\ref{thm:graded-pieces_charp}(2) reduces the problem to checking the $\bb P^1$-bundle formulae for 
\[
\bb Z_p(\star)^\sub{syn}, \qquad  L_\sub{cdh}\bb Z_p(\star)^\sub{syn}, \qquad \bb Z(\star)^\bb A.
\] 
For $\bb A^1$-motivic cohomology again see Theorem \ref{thm:cdh}(7), while the $\bb P^1$-bundle formula for syntomic cohomology has been recorded in Example~\ref{exam:syn_pbf}. Finally, Theorem~\ref{thm:cdh-syn-pbf} established the $\bb P^1$-bundle formula for $L_\sub{cdh}\bb Z_p(\star)^\sub{syn}$.
\end{proof}

\begin{proof}[Proof of Theorem \ref{thm:pbf-blowup}]
We first establish part (3) of the theorem, i.e., the regular blow-up formula. Since this holds for cdh sheaves (indeed, cdh sheaves carry arbitrary abstract blow-up squares to cartesian squares), the squares of Theorem \ref{thm:graded-pieces}(2) and \ref{thm:graded-pieces_charp}(3) reduce the problem to checking the regular blow-up formulae for $R\Gamma(-,L\Omega^{<j}_{-/\bb Q})$ on qcqs $\bb Q$-schemes and for $\bb Z_p(j)^\sub{syn}$ on qcqs $\bb F_p$-schemes. In both cases that reduces to the regular blow-up formula for $R\Gamma(-,L^i_{-/\bb F})$ for all $i\ge 0$ (here we use Lemma \ref{lem_fin_fil_on_syn} in characteristic $p$) which can be proved directly \cite[Lemma~9.4.3]{BhattLurie2022}.

We now know that $\bb Z(\star)^\sub{mot}$ satisfies both the $\bb P^1$-bundle formula (by Theorem \ref{thm:pbf}) and regular blow-up formula. The $\bb P^r$-bundle formula for arbitrary $r\ge1$ now follows formally from an argument in \cite[Lemma~3.3.5]{AnnalaIwasa2023} (which in fact only requires the elementary blowup formula to hold), and the projective bundle formula for an arbitrary locally free sheaf follows by Zariski descent.
\end{proof}

\begin{remark}[Variants on the regular blow-up formula]
\label{rem:blowup}
Given a regular closed immersion $Y\into X$ of qcqs $\bb F$-schemes of pure codimension $r$, the blow-up formula for motivic cohomology may instead be written as a natural equivalence \[
\Z(j)^{\mot}(\mathrm{Bl}_Y(X))  \simeq \Z(j)^{\mot}(X) \oplus \bigoplus_{i=1}^{r-1} \Z(j-i)^{\mot}(Y)[-2i]
\]
(for example this is how Thomason expressed the blow-up formula for $K$-theory \cite{Thomason1993}). Indeed, this follows by applying $\bb Z(j)^\sub{mot}$ to the blow-up square to get a cartesian square of spectra and then using the projective bundle formula to write $\bb Z(j)^\sub{mot}(Y\times_X\mathrm{Bl}_Y(X))\simeq \bigoplus_{i=0}^{r-1}\bb Z(j-i)^\sub{mot}(X)[-2i]$.

Furthermore, we mention that the extension of motivic cohomology to derived $\bb F$-schemes, as in \S\ref{sec:derived}, satisfies the derived blow-up formula in that it converts derived blow-up squares in the sense of \cite{KhanRydh2019} to cartesian squares. Indeed, the $\cdh$ parts of the theory do not depend on the derived structure and thus satisfy the derived blow-up formula. On the other hand, we reduce the claim for syntomic and filtered derived de Rham cohomology to the case of the cotangent complex, which satisfies the derived blow-up formula since derived blow-up squares are ``pulled back'' from blow-ups along regular immersions  (see \cite[Definition 3.1]{KerzStrunkTamme2018}).

On the other hand we stress that for a closed immersion $Y \rightarrow X$ which is not regular then motivic cohomology need not carry the corresponding blowup square to a cartesian square; indeed, $\bb Z(j)^\sub{mot}$ is not a cdh sheaf. One must instead taking into account infinitesimal thickenings, which leads to the pro cdh descent of Theorem~\ref{theorem_pro_cdh_descent}.
\end{remark}

\section{Comparison to $\bb A^1$-motivic cohomology}\label{section_smooth}
We begin by repeating Construction \ref{cons_mot_to_cdh} for the sake of clarity: on the category of qcqs $\bb F$-schemes, there are natural multiplicative comparison maps of $\mathrm D(\bb Z)$-valued presheaves \begin{equation}\bb Z(j)^\sub{mot}\To\bb Z(j)^{\bb A}\label{equation_mot_to_cdh}\end{equation} for $j\ge0$, arising as the shifted graded pieces of a multiplicative comparison map of filtered presheaves of spectra $\Fil_\sub{mot}^\star \K\to \Fil_{\bb A}^\star \KH$. These comparison maps are tautological from the pullback definition of our motivic cohomology, which should be seen a modification of $\bb A^1$-motivic cohomology. 

The goal of this section is to prove the following equivalences related to the maps (\ref{equation_mot_to_cdh}):

\begin{theorem}\label{thm_mot_vs_cdh}
Let $\bb F$ be a prime field and $j\ge0$.
\begin{enumerate}
\item The map (\ref{equation_mot_to_cdh}) induces equivalences of $\mathrm D(\bb Z)$-valued presheaves on $\mathrm{Sch}^\sub{qcqs}_\bb F$ \[L_\sub{cdh}\bb Z(j)^\sub{mot}\quis \bb Z(j)^{\bb A}\qquad\text{and}\qquad L_{\bb A^1}\bb Z(j)^\sub{mot}\quis \bb Z(j)^{\bb A}.\]
\item For any regular Noetherian $\bb F$-scheme $X$, the map (\ref{equation_mot_to_cdh}) induces an equivalence \[\bb Z(j)^\sub{mot}(X)\quis \bb Z(j)^{\bb A}(X).\] (Equivalently, using part (1), the maps $\bb Z(j)^\sub{mot}(X)\to \bb Z(j)^\sub{mot}(\bb A_X^m)$ are equivalences for all $m\ge0$.)
\end{enumerate}
\end{theorem}

The maps in part (1) of the theorem are induced by (\ref{equation_mot_to_cdh}), recalling that the $\bb A^1$-motivic cohomology presheaves $\bb Z(j)^{\bb A}:\mathrm{Sch}^\sub{qcqs,op}_\bb F\to\mathrm D(\bb Z)$ are cdh sheaves and $\bb A^1$-invariant by Theorem \ref{thm:cdh}(2)\&(6). In case of confusion, here $L_{\A^1}$ denotes the endofunctor of presheaves of complexes (or of spectra) on $\Sch_{\bb F}^{\qcqs}$ reflecting onto $\A^1$-invariant presheaves; we will recall the explicit formula for $L_{\A^1}$ in the proof of Lemma~\ref{lem:a10} below.

\begin{remark}
Informally, Theorem \ref{thm_mot_vs_cdh}(1) says that our motivic cohomology may be viewed as a ``de-cdh-sheafification'' or ``de-$\bb A^1$-localisation'' of $\bb A^1$-motivic cohomology. More precisely, it states that on equicharacteristic schemes the comparison equivalences \[L_\sub{cdh}\K\quis \KH\qquad\text{and}\qquad L_{\bb A^1}\K\quis \KH\] (the first being part of Theorem~\ref{thm:mainsq}, the second being the definition of $\KH$) upgrade to filtered equivalences, where we equip $\KH$ with the filtration $\Fil_{\bb A}^\star$ of Theorem \ref{thm:cdh}, and we equip the left sides with $L_\sub{cdh}$, resp.~$L_{\bb A^1}$, of our motivic filtration $\Fil^\star_\sub{mot}$. That is, the $\bb A^1$-motivic filtration on $\KH$-theory can be recovered by cdh sheafifying or $\bb A^1$-localising our motivic filtration on $\K$-theory.
\end{remark}

\begin{remark}
Theorem \ref{thm_mot_vs_cdh}(2) is a motivic upgrade of the equivalence $\K(X)\quis\KH(X)$ for regular Noetherian $\bb F$-schemes $X$. Indeed, combined with this equivalence, it states that the map of filtered spectra $\Fil_\sub{mot}^\star \K(X)\to \Fil_{\bb A}^\star \KH(X)$ is an equivalence.
\end{remark}

In particular, we deduce that the new motivic cohomology coincides with the old on smooth varieties:

\begin{corollary}\label{corol_smooth_comparison}
Restricted to the category of smooth schemes over any field, the motivic cohomology $\bb Z(\star)^\sub{mot}$ agrees with the classical theory of \S\ref{ss_classical} and therefore recovers Bloch's cycle complexes. 
\end{corollary}
\begin{proof}
On smooth schemes over a field Theorem \ref{thm_mot_vs_cdh}(2) shows that $\bb Z(\star)^\sub{mot}\quis \bb Z(\star)^\bb A$; the latter is classical motivic cohomology, as recorded in Theorem \ref{thm:cdh}(9).
\end{proof}

Assuming resolution of singularities, one sometimes show that cohomologies are unchanged on smooth varieties by cdh sheafification \cite[\S5]{SuslinVoevodsky2000} \cite[\S4]{Voevodsky2010}; we can eliminate the hypothesis on resolution of singularities in the case of syntomic cohomology:

\begin{corollary}\label{corol_e_vs_eh}
For any regular Noetherian $\bb F_p$-scheme $X$ and $j\ge0$, the canonical maps \[\bb Z_p(j)^\sub{syn}(X)\To L_\sub{cdh}\bb Z_p(j)^\sub{syn}(X)\qquad\text{and}\qquad R\Gamma_\sub{\'et}(X,\Omega^j_\sub{log})\To R\Gamma_\sub{\'eh}(X,\Omega^j_\sub{log})\] are equivalences.
\end{corollary}
\begin{proof}
The first follows from Theorem \ref{thm_mot_vs_cdh}(2) and the cartesian square Theorem \ref{thm:graded-pieces_charp}. The second equivalence follows by taking the first equivalence modulo $p$.
\end{proof}

The core of the proof of Theorem \ref{thm_mot_vs_cdh}(1) is the fact that derived de Rham and syntomic cohomology are very far from being homotopy invariant \cite{Elmanto2021, GellerWeibel1989}:

\begin{lemma}\label{lem:a10} In the category of $\rm D(\bb Z)$-valued presheaves on $\Sch^{\qcqs}_{\bb F}$, the following hold for all $j\ge0$:
\begin{enumerate}
\item $L_{\A^1}R\Gamma(-,L^j_{-/\bb F}) \simeq 0$;
\item if $\bb F= \bb Q$ then the map $L_{\A^1}R\Gamma(-,\widehat{L\Omega}^{\geq j}_{-/\bb Q})\rightarrow L_{\A^1}R\Gamma_{\cdh}(-,\widehat{L\Omega}^{\geq j}_{-/\bb Q})$ is an equivalence.
\item if $\bb F= \bb F_p$, then $L_{\A^1}\Z_p(j)^{\syn} \simeq 0$ and $L_{\A^1}L_{\cdh}\Z_p(j)^{\syn} \simeq 0$.
\end{enumerate}
\end{lemma}
\begin{proof}
We will use the following explicit formula for the endofunctor $L_{\bb A^1}$ of presheaves on $\textrm{Sch}^\sub{qcqs}_{\bb F}$: given a presheaf $\cal F$ then
\begin{equation}\label{eq:la1}
L_{\A^1}\scr F (X) = \colim_{\Delta^{\op}} \scr F(X \times \Delta^{\bullet})
\end{equation}
where $\Delta^{\bullet}$ is the cosimplicial $\bb F$-scheme built from algebraic $m$-simplices:
\[
\Delta^m = \Spec(\bb F[T_0, \cdots, T_m]/(\sum^m_{i=0} T_i = 1). 
\]
Note that $L_{\A^1}$ preserves Nisnevich and cdh sheaves; indeed, this follows from the description of Nisnevich and cdh descent in terms of cd structures and~\eqref{eq:la1}. 

(1): Let $A$ be an $\bb F$-algebra. By the K\"unneth formula for the cotangent complex, there is a natural equivalence
\[
L^j_{A[\Delta^m]/\bb F} \simeq \bigoplus_{a+b = j} L^a_{A/\bb F} \otimes \Omega^b_{\bb F[\Delta^m]/\bb F}.
\]
Therefore $(L_{\A^1}L^j_{-/\bb F})(A) \simeq  \bigoplus_{a+b = j} L^a_{A/\bb F} \otimes (L_{\A^1}\Omega^b_{-/\bb F})(\bb F)$, which reduces the problem to showing that $(L_{\A^1}\Omega^b_{-/\bb F})(\bb F) \simeq 0$ for all $b \geq 0$. The latter vanishing is due to Geller--Weibel \cite{GellerWeibel1989}.

(2): Since $L_{\A^1}$ preserves $\cdh$ sheaves, it suffices to prove that $L_{\A^1}R\Gamma(-,\widehat{L\Omega}^{\geq j}_{-/\bb Q})$ is a cdh sheaf. Since $L_{\A^1}$ preserves fibre sequences, we have a fibre sequence for all $j \geq 0$
\[
L_{\A^1}R\Gamma(-,L\Omega^{< j}_{-/\bb Q})[-1] \rightarrow  L_{\A^1}R\Gamma(-,\widehat{L\Omega}^{\geq j}_{-/\bb Q}) \rightarrow L_{\A^1}R\Gamma(-,\widehat{L\Omega}_{-/\bb Q}).
\]
The presheaf $R\Gamma(-,L\Omega^{< j}_{-/\bb Q})$ is killed by $L_{\A^1}$, thanks to part (1) and induction on $j$. On the other hand, by Lemma~\ref{lemma_cdh_descent_HP}, the last term is a $\cdh$ sheaf since $L_{\A^1}$ preserves $\cdh$ sheaves. In particular, the middle term is a $\cdh$ sheaf, completing the proof.

(3): There are maps of presheaves of $\bb E_\infty$-algebras $\bb Z_p(0)^\sub{syn}\to \bigoplus_{j \geq 0} \Z_p(j)^{\syn} \to \bigoplus_{j \geq 0} L_{\cdh}\Z_p(j)^{\syn}$, and $L_{\A^1}$ is lax symmetric monoidal, so it is enough to show that $L_{\bb A^1}\bb Z_p(0)^\sub{syn}$ vanishes; indeed, any algebra over $0$ is automatically $0$. It is enough to check this vanishing on any affine $\bb F_p$-scheme $\Spec(A)$ and, as observed in \cite[Lemma~3.0.3]{Elmanto2021}, the complex $L_{\bb A^1}\bb Z_p(0)^\sub{syn}(A)$ is $p$-complete since we have the universal bound that $\bb Z_p(0)^\sub{syn}(A[\Delta^m])$ is supported in degrees $\le 1$ for any $m$; so it is finally enough to prove the vanishing of $L_{\bb A^1}\bb F_p(0)^\sub{syn}(A)$. But that follows from part (1) and Lemma~\ref{lem_fin_fil_on_syn}. 
\end{proof}

\begin{proof}[Proof of Theorem \ref{thm_mot_vs_cdh}(1)]
Recall that $\bb Z(j)^\bb A$ is an $\bb A^1$-invariant cdh sheaf by Theorem \ref{thm:cdh}(2)\&(6). So cdh sheafifying the pullback squares of Theorems \ref{thm:graded-pieces}(2) or \ref{thm:graded-pieces_charp}(2) shows that $L_\sub{cdh}\bb Z(j)^\sub{mot}\quis \bb Z(j)^{\bb A}$. Similarly, $\bb A^1$-localising the pullback squares and using the previous lemma yields $L_{\bb A^1}\bb Z(j)^\sub{mot}\quis \bb Z(j)^{\bb A}$.
\end{proof}

The remainder of the section is devoted to the proof of Theorem \ref{thm_mot_vs_cdh}(2). The key inputs are the $\P^1$-bundle formula for motivic cohomology (Theorem \ref{thm:pbf}), the already proved Theorem \ref{thm_mot_vs_cdh}(1), and an argument of Gabber used to prove Gersten injectivity statements \cite{Gabber1994, GrosSuwa1988}. Gabber's argument has been axiomatized by Colliot-Th\'el\`ene--Hoobler--Kahn \cite{ColliotThelene-Hoobler-Kahn1997}, and we now review their formalism in a more modern language.

Let $k$ be any field and suppose that we have a presheaf $\scr F: \Sm^{\op}_{k} \rightarrow \Spt$; for $X \in \Sm_{k}$, we will write $\scr F^X$ for the presheaf $U \mapsto \scr F(U \times_k X)$. There are two morphisms of presheaves
\[
j^*, \pi^*\infty^*: \scr F^{\bb P^1} \rightarrow \scr F^{\bb A^1},
\]
where:
\begin{enumerate}
\item $\pi$ is induced the projection $\A^1 \times_k X \rightarrow X$,
\item $\infty$ is the closed immersion $\Spec(k)\to\P^1$ of the point at $\infty$,
\item $j: \A^1 \hookrightarrow \P^1$ is the open immersion complementary to the point at $\infty$.
\end{enumerate}
In general, there is no reason for the maps $j^*$ and $\pi^*\infty^*$ to be homotopic. This leads to the next definition:
\begin{definition}\label{def:good} We say that a presheaf $\scr F: \Sm^{\op}_{k} \rightarrow \Spt$ is a \emph{deflatable}\footnote{We wish to invoke the picture of deflating a balloon: the Riemann sphere is thought of as a balloon and a presheaf is deflatable if ``after puncturing at $\infty$'' the sphere deflates onto a point.} if the maps $j^*$ and $\pi^*\infty^*$ are homotopic.
\end{definition}

\begin{remark} More precisely, to ask that two morphisms of presheaves are homotopic means that they are identified in the homotopy category of presheaves. This means that there is an $2$-morphism, functorial in smooth $k$-schemes, between these two morphisms of presheaves. We do not keep track of this $2$-morphism (we only require its existence), but we note that functoriality in smooth schemes is a substantial amount of extra compatibilities. In fact, calling these $2$-morphisms \emph{deflations}, the space of deflations can be parametrized as following: it is the space of sections $s: \scr F^{\bb P^1} \rightarrow \cal E$ of the canonical map $\cal E \rightarrow \scr F^{\bb P^1}$, were $\cal E$ is the equaliser of the two maps $j^*, \pi^*\infty^*: \scr F^{\bb P^1} \rightrightarrows \scr F^{\bb A^1}$.
\end{remark}

\begin{remark} Definition~\ref{def:good} implies the validity of axiom ``SUB 2'' of \cite{ColliotThelene-Hoobler-Kahn1997}, which is much weaker than deflatability and instead asks only for scheme-wise homotopy commutativity of a relative variant of this axiom.
\end{remark}

\begin{example} If a presheaf $\scr F:\Sm^{\op}_{k} \rightarrow \Spt$ is $\bb A^1$-invariant, then it is deflatable. Indeed, the map $\pi^*$ is an equivalence and there is a natural $\bb A^1$-homotopy between $j^*$ and $\infty^*$.
\end{example}

The following lemma is a variant of one of the main results of \cite{ColliotThelene-Hoobler-Kahn1997}, stated in a convenient language for our use. It proves Gersten injectivity for good cohomology theories satisfying Nisnevich descent. We denote by $\mathrm{Reg}_{k}$ the category of regular Noetherian $k$-schemes.

\begin{lemma}[{\cite{ColliotThelene-Hoobler-Kahn1997}}] \label{lem:ColliotThelene-Hoobler-Kahn1997}
Let $k$ be a perfect field and $\scr F: \mathrm{Reg}^{\op}_k \rightarrow \Sp$ be a finitary, Nisnevich sheaf such that $\scr F|_{\Sm^{\op}_k}$ is deflatable. Then for any $n \in \Z$ and any regular local $k$-algebra $R$ with fraction field $F$, the canonical map
\[
\pi_n(\scr F(R)) \To \pi_n(\scr F(F))
\]
is injective.
\end{lemma}

\begin{proof}
For a regular $k$-scheme $X$ and closed immersion $Z\into X$ we will write
\[
\scr F_Z(X):= \mathrm{fib}(\scr F(X) \rightarrow \scr F(X \setminus Z)),
\]
so that we have a long exact sequence functorial in $X$ and $Z$
\[
\cdots \rightarrow \pi_n(\scr F(X)) \rightarrow \pi_n(\scr F(X \setminus Z)) \rightarrow \pi_{n-1}(\scr F_Z(X)) \rightarrow \cdots.
\] 
By N\'eron--Popescu it suffices to prove the result when $R=\roi_{X,x}$ is the local ring of a closed point $x \in X$ where $X$ is a smooth affine $k$-scheme.

Let $s \in \ker(\pi_n(\scr F(R))\to\pi_n(\scr F(F))$; by possibly shrinking $X$, we may assume that $s$ is defined on $X$ and vanishes away from some closed subscheme $Z \hookrightarrow X$ of positive codimension, i.e., $s$ lifts to an element $\widetilde{s} \in \pi_n(\scr F_Z(X))$. To prove the result, it suffices to produce an open neighborhood $U\subseteq X$ of $x$ and a closed subscheme $Z' \hookrightarrow U$ with $Z \cap U \subset Z'$ such that $\widetilde{s}$ vanishes on $\pi_n(\scr F_{Z'}(U))$. 

Gabber's presentation lemma \cite[Theorem 3.1.1]{ColliotThelene-Hoobler-Kahn1997} (see \cite{gabber-finite} for the case in which $k$ is a finite field) furnishes an open neighborhood $U \subseteq X$ of $x$, a smooth affine $k$-scheme $V$, a morphism $\phi = (\psi, v): U \rightarrow V \times \A^1$ such that $\psi|_{Z \cap U}$ is finite, and a Nisnevich square
\[
\begin{tikzcd}
 U \setminus (Z \cap U) \ar{r} \ar{d} &  U \ar{d}{\phi}\\
 \A^1_V \setminus (\phi(Z \cap U)) \ar{r} &   \A^1_V.\\
\end{tikzcd}
\]
In particular $\phi(Z \cap U) \hookrightarrow \A^1_V$ is a closed immersion. By Nisnevich excision, we have that $\pi_n(\scr F_{Z \cap U}(U)) \cong \pi_n(\scr F_{\A^1_V \cap \phi(Z \cap U)}(\A^1_V))$. Now set $F:= \psi(Z \cap U)$ so that $Z \cap U \subset \psi^{-1}(F)=:Z'$. So we have a commutative diagram
\[
\begin{tikzcd}
\pi_n(\scr F_{Z\cap U}(U)) \ar{r}  &  \pi_n(\scr F_{Z'}(U)) \\
\pi_n(\scr F_{\phi(Z\cap U)}(\A^1_V) ) \ar{r} \ar{u}{\cong}&   \pi_n(\scr F_{\A^1_F}(\A^1_V))\ar{u}\\
\end{tikzcd}
\] 
and, to finish the proof, we need only show that the bottom map is zero. The map of interest is the top horizontal map of the following commutative diagram
\[
\begin{tikzcd}
\pi_n(\scr F_{\phi(Z\cap U)}(\A^1_V)) \ar{r} & \pi_n(\scr F_{\A^1_F}(\A^1_V))  & \\
 & & \pi_n(\scr F_{F}(V)) \ar{ul}\\
\pi_n(\scr F_{\phi(Z\cap U)}(\P^1_V)) \ar{r} \ar{uu}{\simeq} &  \pi_n(\scr F_{\P^1_F}(\P^1_V))\ar{uu} \ar{ur}&\\
\end{tikzcd}
\]
where the triangle commutes exactly because $\scr F$ is a deflateable. However, the bottom composite is zero, since $\phi(Z \cap U)$ does not meet the $\infty$-section of $\P^1_V$, and thus the top map is also zero as desired. 
\end{proof}

The projective bundle formula implies that our motivic cohomology is deflatable:

\begin{lemma}\label{lem_mot_deflat}
For any $j\ge0$, the presheaf
\[
\bb Z(j)^\sub{mot}|_{\Sm_{\bb F}}:\Sm^{\op}_\bb F\To \rm D(\Z)
\]
is deflatable. 
\end{lemma}
\begin{proof}
This is a standard consequence of the $\bb P^1$-bundle formula. Theorem~\ref{thm:pbf} furnishes us with an equivalence
\[
\Z(j)^{\mot} \oplus \Z(j-1)^{\mot}[-2] \xrightarrow{\pi^* \oplus c_1(\scr O(1))\pi^* } (\Z(j)^{\mot})^{\P^1},
\]
whence it suffices to explain why the diagram 
\[
\begin{tikzcd}
\Z(j)^{\mot} \oplus \Z(j-1)^{\mot}[-2] \ar{rr} \ar{dr} & & (\Z(j)^{\mot})^{\A^1}\\
 &  \Z(j)^{\mot} \ar{ur} & 
\end{tikzcd}
\]
commutes in the homotopy category of presheaves. On the $\Z(j)^{\mot}$ component, the diagram commutes already at the level of schemes. On the $\Z(j-1)^{\mot}[-2]$ component, the diagram commutes because on $\A^1_\bb F$ there are natural identifications $\pi^*\infty^*\scr O(1) \cong \scr O \cong  j^*\scr O(1)$.
\end{proof}

We now have all the necessary ingredients to prove Theorem \ref{thm_mot_vs_cdh}(2):

\begin{proof}[Proof of Theorem \ref{thm_mot_vs_cdh}(2)]
The goal is to prove, for any regular Noetherian $\bb F$-scheme $X$, that the map $\bb Z(j)^\sub{mot}(X)\to \bb Z(j)^{\bb A}(X)$ is an equivalence. Since this factors as \begin{equation}\bb Z(j)^\sub{mot}(X)\to L_{\bb A^1}\bb Z(j)^\sub{mot}(X)\quis \bb Z(j)^{\bb A}(X),\label{eqn_mot_vs_cdh_proof}\end{equation} where the equivalence is Theorem \ref{thm_mot_vs_cdh}(1), it is equivalent to show that $\cal N(X)\simeq0$ where $\cal N:=\fib(\bb Z(j)^\sub{mot}\to L_{\bb A^1}\bb Z(j)^\sub{mot})$. Furthermore, since $\cal N$ is finitary and satisfies Zariski descent, it is enough to show that $\cal N(R)\simeq 0$ for every regular, Noetherian, local $\bb F$-algebra $R$. But we will show in the next paragraph that $\cal N|_{\Sm_{\bb F}}$ is deflatable, whence Lemma \ref{lem:ColliotThelene-Hoobler-Kahn1997} implies that $H^n(\cal N(R))\to H^n(\cal N(F))$ is injective for all $n$, where $F$ is the fraction field of $R$. This therefore reduces the problem to showing that $\cal N(F)\simeq0$ for every field extension $F$ of $\bb F$; appealing again to (\ref{eqn_mot_vs_cdh_proof}), this time for $X=\Spec(F)$, it is equivalent to show that $\bb Z(j)^\sub{mot}(F)\quis \bb Z(j)^{\bb A}(F)$. But this follows from the part of Theorem \ref{thm_mot_vs_cdh}(1) stating that $L_\sub{cdh}\bb Z(j)^\sub{mot}\quis \bb Z(j)^{\bb A}$, as fields are local for the cdh topology.

It remains to prove that $\cal N|_{\Sm_{\bb F}}$ is deflatable. Firstly, we know from Lemma \ref{lem_mot_deflat} that $\bb Z(j)^\sub{mot}|_{\Sm_{\bb F}}$ is deflatable. Fixing any choice of deflation for it, this deflation induces (using the explicit formula (\ref{eq:la1})) a deflation for $L_{\bb A^1}\bb Z(j)^\sub{mot}|_{\Sm_{\bb F}}$ which is compatible with the canonical map $\bb Z(j)^\sub{mot}|_{\Sm_{\bb F}}\to L_{\bb A^1}\bb Z(j)^\sub{mot}|_{\Sm_{\bb F}}$. Passing to the fibre induces a deflation for $\cal N|_{\Sm_{\bb F}}$, as desired. 
\end{proof}

\section{Comparison to lisse motivic cohomology}\label{section_lke}
We continue to fix a prime field $\bb F$. The goal of this section is to study the comparison map from lisse motivic cohomology to our new motivic cohomology, as discussed in Construction \ref{cons_lke_to_mot}. However, we may now adopt a cleaner point of view on this comparison map. Indeed, we now know from Corollary \ref{corol_smooth_comparison} that the restriction of $\bb Z(j)^\sub{mot}$ to smooth $\bb F$-algebras coincides with classical motivic cohomology $\bb Z(j)^\bb A$. Therefore we will henceforth view $\bb Z(j)^\sub{lse}$ as the left Kan extension of $\bb Z(j)^\sub{mot}$, restricted to smooth $\bb F_p$-algebras, back along the inclusion $\text{CAlg}_\bb F^\sub{sm}\subseteq \text{CAlg}_\bb F$. For any $\bb F_p$-algebra, this formally induces the same comparison map \begin{equation}\bb Z(j)^\sub{lse}(A)\To \bb Z(j)^\sub{mot}(A)\label{eqn:lke_vs_mot}\end{equation} as Construction \ref{cons_lke_to_mot}. This map is certainly not an equivalence in general: the left side is supported in cohomological degree $\le 2j$ (see Proposition \ref{prop:mot-filt}) but this bound cannot always be true for the right side: otherwise the Atiyah--Hirzebruch spectral sequence would then imply that the spectrum $\K(A)$ was always connective.

In general it seems to be a deep question to what extent the right side of (\ref{eqn:lke_vs_mot}) of is controlled by the left side. In other words, how much of motivic cohomology can be recovered from that of smooth algebras? In this section we provide some partial answers to this question. In particular we will show that, for $A$ local, (\ref{eqn:lke_vs_mot}) induces an equivalence \[\bb Z(j)^\sub{lse}(A)\quis \tau^{\le j}\bb Z(j)^\sub{mot}(A).\] Note that, in light of Remark \ref{remark_lke_as_cycles}, this provides a description of $\tau^{\le j}\bb Z(j)^\sub{mot}(A)$ purely in terms of algebraic cycles. We will return to the link between motivic cohomology and algebraic cycles in Section \ref{section_cf_cycles}.

In this section we also establish some vanishing theorems and prove a Nesterenko--Suslin isomorphism.

\subsection{Behaviour of motivic cohomology in degrees $\le 2j$}

We begin with the following rational statement, writing \[\bb Q(j)^\sub{lse}(A):=\bb Z(j)^\sub{lse}\otimes_{\bb Z}\bb Q, \qquad \bb Q(j)^\sub{mot}(A):=\bb Z(j)^\sub{mot}(A)\otimes_{\bb Z}\bb Q\] for the rationalisations of our motivic cohomologies. (Note that $\bb Q(j)^\sub{lse}$ is the left Kan extension of the restriction of $\bb Q(j)^\sub{mot}$ to smooth $\bb F$-algebras, since left Kan extension commutes with filtered colimits of functors.)

\begin{lemma}
For any $\bb F$-algebra $A$ and $j\ge0$, the map (\ref{eqn:lke_vs_mot}) induces an equivalence \[\bb Q(j)^\sub{lse}(A)\quis\tau^{\le 2j}\bb Q(j)^\sub{mot}(A).\] In other words, the functor $\tau^{\le 2j}\bb Q(j)^\sub{mot}:\CAlg_\bb F\to\text \rm D(\bb Z)$ is left Kan extended from smooth $\bb F$-algebras.
\end{lemma}
\begin{proof}
Rationally, by Theorem~\ref{theorem_AH_SS_0}(2) (characteristic zero) and Theorem~\ref{thm:p-ahss}(2) (characteristic $p > 0$), there is a natural isomorphism of filtered spectra $\Fil_\sub{mot}^\star \K(A)_\bb Q\cong \bigoplus_{j\ge\star}\bb Q(j)^\sub{mot}(A)[2j]$ for any $\bb F$-algebra $A$. Restricting to smooth $\bb F$-algebras and left Kan extending back identifies the map \[\bigoplus_{j\ge 0}\bb Q(j)^\sub{lse}(A)[2j]\To \bigoplus_{j\ge 0}\bb Q(j)^\sub{mot}(A)[2j]\] (obtained by rationalising the direct sum of (\ref{eqn:lke_vs_mot}) over all weights) with the canonical map $\K^\sub{cn}(A)_\bb Q\to \K(A)_\bb Q$. Since the latter map is the connective cover, we deduce the same for the former map, i.e., $\bb Q(j)^\sub{lse}(A)\quis\tau^{\le 2j}\bb Q(j)^\sub{mot}(A)$.
\end{proof}

\begin{corollary}\label{corollary_lej_rational}
For any local $\bb F$-algebra $A$ and $j\ge0$, the map (\ref{eqn:lke_vs_mot}) induces an equivalence \[\bb Q(j)^\sub{lse}(A)\quis\tau^{\le j}\bb Q(j)^\sub{mot}(A).\] In other words, the functor $\tau^{\le j}\bb Q(j)^\sub{mot}:\CAlg_\bb F^\sub{loc}\to \rm D(\bb Z)$ is left Kan extended from essentially smooth, local $\bb F$-algebras.
\end{corollary}
\begin{proof}
This follows from the previous lemma since, for $A$ local, we noted in Proposition \ref{prop:mot-filt} that the lisse motivic cohomology $\bb Z(j)^\sub{lse}(A)$ is supported in cohomological degrees $\le j$.
\end{proof}

The proof of the previous corollary also implies the following rational vanishing result:

\begin{corollary}\label{corol_rat_vanishing}
For any local $\bb F$-algebra $A$ and $0\le j<i\le 2j$, we have $H^i_\sub{mot}(A,\bb Q(j))=0$.
\end{corollary}

\begin{remark}[Rational Drinfeld vanishing]\label{remark_drinfeld_1}
If $A$ is a Henselian local $\bb F$-algebra, then we can improve the vanishing bound of the previous corollary by $1$; namely we also have $H^{2j+1}_\sub{mot}(A,\bb Q(j))=0$ for all $j\ge0$. Indeed, this follows from the theorem of Drinfeld that $\K_{-1}(A)=0$ and the decomposition $\K(A)_\bb Q\simeq\bigoplus_{j\ge0} \bb Q(j)^\sub{mot}(A)[2j]$. We also expect this result to hold integrally: see Conjecture~\ref{conj:vanishing}.
\end{remark}

We now prove an integral version of Corollary \ref{corollary_lej_rational}. By taking $H^1$ of the map~\eqref{eq:c1-mot}, we get a natural map \[A^\times\To H^1_\sub{mot}(A,\bb Z(1))\] for any $\bb F$-algebra $A$; by multiplicativity this induces {\em symbol maps} \begin{equation}(A^\times)^{\otimes j}=A^\times\otimes_{\bb Z}\cdots\otimes_{\bb Z} A^\times \To H^j_\sub{mot}(A,\bb Z(j)\label{eqn_symbol}\end{equation} for all $j\ge1$.

\begin{lemma}\label{lemma_factors_through_Mil}
For any local $\bb F$-algebra $A$ and $j\ge1$, the map (\ref{eqn_symbol}) factors through the Milnor $K$-group $\K_j^M(A)$
\end{lemma}
\begin{proof}
We must show that the map respects the Steinberg relation, so may assume $j=2$. Now let $a\in A^\times$ be a unit such that $1-a$ is also a unit; let $\bb F[t]\to A$, $t\mapsto a$ be the induced map, and $\frak p\subseteq \bb F[t]$ the pullback of the maximal ideal of $A$. There is a commutative diagram by naturality
\[\xymatrix{
\bb F[t]_\frak p^\times\otimes_{\bb Z}\bb F[t]_\frak p^\times\ar[d]\ar[r] & H^2_\sub{mot}(\bb F[t]_\frak p,\bb Z(2))\ar[d]\\
A^\times\otimes_{\bb Z}A^\times\ar[r] & H^2_\sub{mot}(A,\bb Z(2))
}\]
in which the left vertical arrow sends $t\otimes 1-t$ to $a\otimes 1-a$. The problem therefore reduces to the case of the local ring $\bb F[t]_\frak p$. But, setting $F:=\text{Frac}(\bb F[t])$, we have a second commutative diagram by naturality
\[\xymatrix{
F^\times\otimes_{\bb Z}F^\times\ar[r] & H^2_\sub{mot}(F,\bb Z(2))\\
\bb F[t]_\frak p^\times\otimes_{\bb Z}\bb F[t]_\frak p^\times\ar[u]\ar[r] & H^2_\sub{mot}(\bb F[t]_\frak p,\bb Z(2))\ar[u]
}\]
in which the right vertical arrow is injective by Gersten injectivity in motivic cohomology (which we know from Lemmas \ref{lem:ColliotThelene-Hoobler-Kahn1997} and \ref{lem_mot_deflat}). So the problem finally reduces to the case of the field $F$, in which case it is a theorem of Nesterenko--Suslin and Totaro \cite{Suslin1989, Totaro1992} that the symbol map indeed respects the Steinberg relation; of course we use here that the new motivic cohomology of $F$ coincides with the classical theory, but that is a special case of Corollary \ref{corol_smooth_comparison}.
\end{proof}

We will repeatedly use the following general observation about functors in what follows. Recall that a functor $F: \CAlg^{\sub{loc}}_{\bb F} \rightarrow \Spt$ is said to be \emph{rigid} if for any local $\bb F$-algebra $A$ and henselian ideal $I\subseteq A$, the canonical map is an equivalence $F(A) \xrightarrow{\simeq} F(A/I)$. The following is also recorded as the conjuction of \cite[Proposition 2.23 \& Lemma 2.25]{BachmannElmantoMorrow}:

\begin{lemma}\label{lem:rigid-lke} Let $F:\CAlg^{\sub{loc}}_{\bb F} \rightarrow \Spt$ be a rigid functor. Then $F$ is left Kan extended from the subcategory of essentially smooth local $\bb F$-algebras.
\end{lemma}

\begin{proof}
As observed in Remark \ref{remark_lke_as_cycles}, we can build for any $B \in \CAlg^{\sub{loc}}_{\bb F}$ a simplicial resolution $P_\bullet\to B$ where each term $P_m$ is an ind-smooth, local $\bb F$-algebra and each face map $P_{m+1}\to P_m$ is a henselian surjection. Since $F$ is rigid the simplicial spectrum $m \mapsto F(P_{m})$ is equivalent to the constant simplicial diagram at $F(B)$, and so $|F(P_{\bullet})| \quis |F(B)|$. Furthermore, since $\Delta^{\op}$ is contractible diagram, we have that  $|F(B)|\quis F(B)$. 
\end{proof}

The following is the first main theorem of the section, showing that Zariski locally our weight-$j$ motivic cohomology is left Kan extended from smooth algebras in degrees $\le j$:

\begin{theorem}\label{thm_lke_lej}
For any local $\bb F$-algebra $A$ and $j\ge0$, the map (\ref{eqn:lke_vs_mot}) induces an equivalence \[\bb Z(j)^\sub{lse}(A)\quis\tau^{\le j}\bb Z(j)^\sub{mot}(A).\] In other words, the functor $\tau^{\le j}\bb Z(j)^\sub{mot}:\CAlg_\bb F^\sub{loc}\to  \rm D(\bb Z)$ is left Kan extended from essentially smooth, local $\bb F$-algebras.
\end{theorem}
\begin{proof}
The result is true rationally by Corollary \ref{corollary_lej_rational}, so it suffices to prove the result for $\tau^{\le j}(\bb Z(j)^\sub{mot}(-))/\ell$ for all primes $\ell$. We first claim that the result is true for the functor \[\tau^{\le j}(\bb Z(j)^\sub{mot}(-)/\ell):\text{CAlg}^\sub{loc}_\bb F\to\text D(\bb Z).\] Indeed, if $\ell$ is invertible in $\bb F$ then we have an equivalence $\tau^{\le j}(\bb Z(j)^\sub{mot}/\ell)=\tau^{\le j}R\Gamma_\sub{\'et}(-,\bb \mu_\ell^{\otimes j})$ by Theorem \ref{thm_BL}(1). Since \'etale cohomology is even rigid, it is left Kan extended from smooth $\bb F$-algebras by Lemma~\ref{lem:rigid-lke}. On the other hand, if $\ell=p=\text{char}(\bb F)$ then $\tau^{\le j}(\bb Z(j)^\sub{mot}/p)=\tau^{\le j}(\bb F_p(j)^\sub{syn})$ by Corollary~\ref{corol_fundamental_p}, which is also left Kan extended from smooth $\bb F_p$-algebras: indeed, $\bb F_p(j)^\sub{syn}$ is even left Kan extended from finitely generated polynomial $\bb F_p$-algebras by definition, and $\tau^{>j}\bb F_p(j)^\sub{syn}$ identifies with $\tilde\nu(j)[-j-1]$ by Remark \ref{remarks_tildenu}(2), which is rigid by Remark \ref{rem_rigidity_of_nutilde} and so left Kan extended from smooth algebras by Lemma \ref{lem:rigid-lke}.

We now claim, for all local $\bb F$-algebras $A$, that the canonical map \[\tau^{\le j}(\bb Z(j)^\sub{mot}(A)/\ell)\to (\tau^{\le j}(\bb Z(j)^\sub{mot}(A)))/\ell\] is an equivalence for all primes $\ell$; this will complete the proof. For this, it suffices to prove that the map  $H^j_\sub{mot}(A,\bb Z(j))\to H^j_\sub{mot}(A,\bb Z(j)/\ell)$ is surjective for all local $\bb F$-algebras $A$. To see this, we pick  a Henselian surjection $P\to A$ where $P$ is an ind-smooth (necessarily local) $\bb F$-algebra. Since $\tau^{\le j}(\bb Z(j)^\sub{mot}(-)/\ell)$ is left Kan extended from essentially smooth $\bb F$-algebras, the induced map in top degree $H^j_\sub{mot}(P,\bb Z/\ell(j))\to H^j_\sub{mot}(A,\bb Z/\ell(j))$ is surjective. By naturality and Lemma \ref{lemma_factors_through_Mil}, this surjective map moreover fits into a commutative diagram 
\[
\begin{tikzcd}
\K_j^M(P)\ar[->>]{r} \ar{d} & H^j_{\mot}(P,\bb Z(j))\ar{d}\ar[->>]{r} & H^j_{\mot}(P,\bb Z/\ell(j))\ar[->>]{d}\\
\K_j^M(A) \ar{r} & H^j_{\mot}(A,\bb Z(j))\ar{r} & H^j_{\mot}(A,\bb Z/\ell(j)).
\end{tikzcd}
\]
As indicated, the arrows on the top row are also surjective. Indeed, by taking filtered colimits it is enough to prove such surjectivities for an essentially smooth, local $\bb F$-algebra in place of $P$: then for the first top arrow it is a theorem of Kerz \cite[Theorem B]{Kerz2010}, and for the second arrow it follows from the usual bound that $\bb Z(j)^\sub{mot}$ is Zariski locally supported in degrees $\le j$ on smooth $\bb F$-algebras. From this it follows that the map $H^j_\sub{mot}(A,\bb Z(j))\to H^j_\sub{mot}(A,\bb Z/\ell(j))$ is surjective as required.
\end{proof}

The proof of the previous result yields the following vanishing theorem:

\begin{corollary}[Weak Gersten vanishing]\label{corollary_Hilb_90}
For any local $\bb F$-algebra $A$ and $j\ge1$, we have $H^{j+1}_\sub{mot}(A,\bb Z(j))=0$; if $A$ is henselian then also $H^{1}_\sub{mot}(A,\bb Z(0))=0$.
\end{corollary}
\begin{proof}
The surjectivity of the map $H^j_\sub{mot}(A,\bb Z(j))\to H^j_\sub{mot}(A,\bb Z/\ell(j))$ from the end of the previous proof means that that $H^{j+1}_\sub{mot}(A,\bb Z(j))$ is torsion-free. But it also a torsion group since it vanishes rationally by Corollary \ref{corol_rat_vanishing} and Remark \ref{remark_drinfeld_1}.
\end{proof}

We can also prove a Nisnevich local vanishing result in one additional degree:

\begin{proposition}[More weak Gersten vanishing]\label{proposition_91}
For any henselian local $\bb F$-algebra $A$ and $j\ge1$, we have $H^{j+2}_\sub{mot}(A,\bb Z(j))=~0$.
\end{proposition}
\begin{proof}
The vanishing holds rationally by Corollary \ref{corol_rat_vanishing}, or Remark \ref{remark_drinfeld_1} if $j=1$. So it remains to show that $H^{j+2}_\sub{mot}(A,\bb Z(j))$ is torsion-free; we will prove the stronger (actually equivalent, since we already have Corollary \ref{corollary_Hilb_90}) result that  $H^{j+1}_\sub{mot}(A,\bb Z/\ell(j))=0$ for all prime numbers $\ell$.

Let us first suppose that $\ell$ is invertible in $\bb F$. Consider the fibre sequence \[\tau^{\le j}R\Gamma_\sub{\'et}(-,\mu_{\ell}^{\otimes j})\To \tau^{\le j+1}R\Gamma_\sub{\'et}(-,\mu_{\ell}^{\otimes j})\To H^{j+1}_\sub{\'et}(-,\mu_\ell^{\otimes j})[-j-1]=:\cal F[-j-1]\] on qcqs $\bb F$-schemes. Sheafifying this sequence with respect to the cdh topology and using the identification of Theorem \ref{thm_BL}(1), we get a fibre sequence:
\[\bb Z^\sub{mot}(j)/\ell\To L_\sub{cdh}\tau^{\le j+1}R\Gamma_\sub{\'et}(-,\mu_{\ell}^{\otimes j})\To (L_\sub{cdh}\cal F)[-j-1].\]
Since \'etale cohomology satisfies cdh descent, the middle term agrees with $R\Gamma_\sub{\'et}(-,\mu_{\ell}^{\otimes j})$ up to degrees $\leq j+1$, whence we deduce that \[H^{j+1}_\sub{mot}(A,\bb Z/\ell(j))=\ker\big(H^{j+1}_\sub{\'et}(A,\mu_\ell^{\otimes j})\to H^0(L_\sub{cdh}\cal F(A))\big).\] 
But now we claim that the map appearing on the right side is injective. Indeed, letting $k$ be the residue field of $A$, by functoriality it fits into a commutative diagram
\[\xymatrix{
H^{j+1}_\sub{\'et}(A,\mu_\ell^{\otimes j})\ar[r]\ar[d] & H^0(L_\sub{cdh}\cal F(A))\ar[d] \\
H^{j+1}_\sub{\'et}(k,\mu_\ell^{\otimes j})\ar[r] & H^0(L_\sub{cdh}\cal F(k))
}\]
where the left vertical arrow is an isomorphism (by rigidity of \'etale cohomology), and the bottom horizontal arrow is also an isomorphism (since fields are cdh points). This completes the proof that $H^{j+1}_\sub{mot}(A,\bb Z/\ell(j))=0$.

Next suppose that $\bb F=\bb F_p$ and $\ell=p$. Then from the fundamental fibre sequence of Corollary \ref{corol_fundamental_p} we see that there is a natural identification \[H^{j+1}_\sub{mot}(A,\bb Z/p(j))=\ker\big(\tilde\nu(j)(A)\to H^0_\sub{cdh}(A,\tilde\nu(j))\big).\] Exactly as in the previous case, this vanishes by comparison with the residue field: namely $\tilde\nu(j)(A)\isoto\tilde\nu(j)(k)\isoto H^0_\sub{cdh}(k,\tilde\nu(j))$, the first isomorphism being rigidity of $\tilde\nu(j)$ and the second being the fact that $k$ is a point for the cdh topology
\end{proof}

For the next corollary, we let $\CAlg_\bb F^\sub{h.loc}$ be the category of henselian local $\bb F$-algebras. 

\begin{corollary}\label{cor:lke-hloc}
For $j\ge1$, the functor $\tau^{\le j+2}\bb Z(j)^\sub{mot}:\CAlg_\bb F^\sub{h.loc}\to\text D(\bb Z)$ is left Kan extended from henselisations of essentially smooth, local $\bb F$-algebras.
\end{corollary}
\begin{proof}
Corollary \ref{corollary_Hilb_90} and Proposition \ref{proposition_91} imply that $\tau^{\le j+2}\bb Z(j)^\sub{mot}=\tau^{\le j}\bb Z(j)^\sub{mot}$ on Henselian local $\bb F$-algebras, so the claim reduces to Theorem \ref{thm_lke_lej}.
\end{proof}

An important consequence of the above results is a partial description of weight one motivic cohomology. 

\begin{corollary}[Weight one motivic cohomology]\label{corollary_1}
For any qcqs $\bb F$-scheme $X$, the cofibre of the first Chern class map $R\Gamma_\sub{Nis}(X,\bb G_m)[-1]\to \bb Z(1)^\sub{mot}(X)$ from \eqref{eq:c1-mot} is supported in degrees $>3$. In particular there are natural isomorphisms \[
H^i_{\mot}(X, \Z(1)) \cong \begin{cases}
0 & i \le 0\\
\scr O(X)^{\times}  & i = 1\\
\mathrm{Pic}(X) & i = 2\\
H^2_{\Nis}(X,\bb G_m) & i=3.
\end{cases}
\]
\end{corollary}
\begin{proof}
It is enough to prove the claim Nisnevich locally, i.e., that the map $A^\times[-1]\to \tau^{\le 3}\bb Z(1)^\sub{mot}(A)$ is an equivalence for any Henselian local $\bb F$-algebra $A$. But that is exactly what Corollary~\ref{cor:lke-hloc} states in the case $j=1$.
\end{proof}

\begin{remark}[Weight one motivic cohomology versus Nisnevich cohomology of $\bb G_m$]
The first Chern class map $R\Gamma_\sub{Nis}(X,\bb G_m)[-1]\to \bb Z(1)^\sub{mot}(X)$ is an equivalence whenever $X$ is a Noetherian $\bb F$-scheme of dimension $\le 2$ (we will prove this in Corollary \ref{corol_surfaces}), but not in general: here we present a three dimensional example.

Following the notation of Example \ref{example_cones_0}, let $Y$ be a two-dimensional, smooth, geometrically connected, projective variety over a field $k$, equipped with a fixed embedding $Y\into \bb P_k^N$, and let $R$ be the associated homogenous coordinate ring; also let $A=R_\frak m^h$ be the henselian local ring of the affine cone $\Spec(R)$ at its irrelevant maximal ideal $\frak m$, and 
\[\xymatrix{
Y\ar[r] \ar[d]\ar[r]^{i}& X\times_RA\ar[d]\\
\Spec(k)\ar[r]&\Spec(A)
}\]
the abstract blow-up square corresponding to the desingularisation $X\times_RA$ of $\Spec(A)$. The forthcoming pro cdh descent Theorem \ref{theorem_Weibel_vanishing} therefore yields a cartesian square of pro complexes
\[\xymatrix{
\bb Z(1)^\sub{mot}(A) \ar[r]\ar[d] & \bb Z(1)^\sub{mot}(X\times_RA)\ar[d]\\
\{\bb Z(1)^\sub{mot}(R/\frak m^r)\}_r \ar[r] & \{\bb Z(1)^\sub{mot}(X\times_RR/\frak m^r)\}_r.
}\]
Since the first Chern class map is an equivalence for $R/\frak m^r$ and $X\times_RR/\frak m^r$ (as they have dimension $\le 2$) and for $X\times_RA$ (as it is regular), the long exact Mayer--Vietoris sequence of pro abelian groups arising from this cartesian square yields an isomorphism \[\{H^2_\sub{Nis}(X\times_RR/\frak m^r, \bb G_m)\}_r\isoto H^4_\sub{mot}(A,\bb Z(1)).\] If $H^2_\sub{Zar}(Y,\roi_Y(1))\neq 0$ then the left side of the previous line does not vanish (indeed, the transition maps in the pro abelian group are surjective, so it vanishes if and only if each term vanishes, but $\ker(H^2_\sub{Nis}(X\times_RR/\frak m^2,\bb G_m)\to H^2_\sub{Nis}(X\times_RR/\frak m,\bb G_m))$ identifies with $H^2_\sub{Zar}(Y,\roi_Y(1))\neq 0$), and so we see that the first Chern class map \[0=H^3_\sub{Nis}(A,\bb G_m)\To H^4_\sub{mot}(A,\bb Z(1))\neq 0\] is not an isomorphism.
\end{remark}

We finish this subsection by stating and discussing a conjecture on the local vanishing of motivic cohomology:

\begin{conjecture}[Local vanishing conjecture] \label{conj:vanishing}
Let $A$ be a local $\bb F$-algebra and $j\ge0$. Then
\[
H^i_{\mot}(A, \bb Z(j)) = 0
\]
for all $i$ in the range $j+1\le i\le 2j$. If $A$ is henselian then the vanishing should extend to $i=2j+1$.
\end{conjecture}

We have proved that the conjecture holds rationally (Corollary \ref{corol_rat_vanishing} and Remark \ref{remark_drinfeld_1}), if $i=j+1$ (Corollary \ref{corollary_Hilb_90}), and if $i=j+2$ and $A$ is henselian (Proposition \ref{proposition_91}). Another way to state the conjecture is that the functor $\tau^{\le 2j}\bb Z(j)^\sub{mot}:\CAlg_\bb F^\sub{loc}\to  \rm D(\bb Z)$ is left Kan extended from essentially smooth, local $\bb F$-algebras (and that we can improve $\tau^{\le 2j}$ by $\tau^{\le 2j+1}$ in the henselian case).

\begin{remark}[Atiyah--Hirzebruch spectral sequence of a local ring]\label{rem:local-vanishing}
Let $A$ be a local Noetherian $\bb F$-algebra of dimension $d$, and write $H^i(j):= H^i_{\mot}(A, \bb Z(j))$ for the sake of space. Using the Soul\'e--Weibel vanishing of Theorem~\ref{theorem_Weibel_vanishing}, and the descriptions of low weight motivic cohomology groups (Examples~\ref{example_00} and \ref{example_0p} and Corollary~\ref{corollary_1}), the local vanishing conjecture predicts that the $E_2$ page of the Atiyah--Hirzebruch spectral sequence for $A$ takes the following peculiar shape:
\[
\begin{tikzcd}[sep=small]
     \cdots & 0 & H^0_\sub{cdh}(A, \bb Z) & H^1_\sub{cdh}(A, \bb Z) & H^2_\sub{cdh}(A, \bb Z) & \cdots &\cdots & H^d_\sub{cdh}(A, \bb Z)\\
     \cdots & 0 & A^\times & 0 & H^2_\sub{Nis}(A, \bb G_m) & \cdots &\cdots & H^{1+d}(1)\\
       \cdots & H^1(2) & H^2(2) & 0 & 0 & \ddots &\ddots & H^{2+d}(2)\\
       \cdots &      \vdots & \vdots & \vdots  & \vdots  & \ddots & \ddots & \vdots \\
        \cdots &      H^{d-2}(d-1) & H^{d-1}(d-1) & 0 & 0 & \cdots & 0 & H^{2d-1}(d-1)\\
         \cdots &     H^{d-1}(d) & H^d(d) & 0 & 0 & \cdots & \cdots &0\\
          &\vdots &\vdots &\vdots &\vdots &\vdots &\vdots &\vdots
\end{tikzcd}.
\]
(where the spectral sequence vanishes both to the right and above of what we have drawn); i.e., the spectral sequence would be supported in the third quadrant and in a finite triangle in the fourth quadrant with vertices at $(1,0)$, $(d,0)$, and $(d,d-1)$.

It might seem heretical to believe in such a spectral sequence, but we already known that the third quadrant of the Atiyah--Hirzebruch spectral sequence for $A$ coincides with lisse motivic cohomology by Theorem \ref{thm_lke_lej}, which converges to connective $K$-theory by Proposition \ref{prop:mot-filt}. So the local vanishing conjecture is merely predicting that the fourth quadrant spectral sequence
\[E_2^{ij}=H^{i-j}(\mathrm{cofib}(\bb Z(-j)^\sub{lse}\to \bb Z(-j)^\sub{mot}))\Longrightarrow\pi_{-i-j}\tau_{<0}\K(A)\]
(i.e., ``Atiyah--Hirzebruch spectral sequence for negative $K$-theory''), whose $E_\infty$-page is supported in the aforementioned triangle by consideration of the abutment, is already supported in the triangle on the $E_2$-page. In other words, the only way in which the local vanishing conjecture could fail would be if there were lots of non-zero differentials in the Atiyah--Hirzebruch spectral sequence of~$A$.
\end{remark}

\subsection{Singular Nesterenko--Suslin isomorphism}\label{sec:singular-nst}
In the previous subsection we constructed the symbol map $\K_j^M(A)\to H^j_\sub{mot}(A,\bb Z(j))$ and used it in the course of the proof of Theorem \ref{thm_lke_lej}. We now establish an analogue of the theorem of Nesterenko--Suslin and Totaro, namely the symbol map is essentially an isomorphism; we just need to take care to replace Milnor $K$-theory by the improved variant $\hat \K_j^M(A)$ of Gabber and Kerz \cite{Kerz2010}.

\begin{theorem}[Singular Nesterenko--Suslin isomorphism]\label{theorem_NS} For any local $\bb F$-algebra $A$ and $j\ge0$, the symbol map $\K^M_j(A) \rightarrow H^j_{\mot}(A,\Z(j))$ descends to an isomorphism
\[
\widehat{\K}^M_j(A) \isoto H^j_{\mot}(A,\Z(j)).
\]
\end{theorem}
\begin{proof}
Let $P_\bullet\to A$ be a simplicial resolution as in Remark \ref{remark_lke_as_cycles}, so that the totalisation of the simplicial complex $m\mapsto \tau^{\le j}\bb Z(j)^\sub{mot}(P_m)$ calculates the evaluation on $A$ of the left Kan extension of $\tau^{\le j}\bb Z(j)^\sub{mot}$ from essentially smooth local $\bb F$-algebras. In light of Theorem \ref{thm_lke_lej}, the totalisation is equivalent to $\tau^{\le j}\bb Z(j)^\sub{mot}(A)$. Calculating the top degree $H^j$ as a coequaliser, this means that the canonical map \[\text{coeq}\big(H^j_\sub{mot}(P_1,\bb Z(j))\rightrightarrows H^j_\sub{mot}(P_0,\bb Z(j))\big)\To H^j_\sub{mot}(A,\bb Z(j))\] is an equivalence.

The canonical map \[\text{coeq}\big(\K_j^M(P_1)\rightrightarrows \K_j^M(P_0)\big)\To \K_j^M(A)\] is also an equivalence; this is the content of \cite[Proposition 1.17]{LuedersMorrow2023}.

Comparing the two coequaliser diagrams via the natural symbol maps we obtain two immediate conclusions. 
\begin{enumerate}
\item The symbol map $\K_j^M(A)\to H^j_\sub{mot}(A,\bb Z(j))$ is surjective. Indeed, as already used in the proof of Theorem \ref{thm_lke_lej}, the symbol map $\K_j^M(P_0)\to H^j_\sub{mot}(P_0,\bb Z(j))$ is surjective by Kerz.
\item Secondly, we may complete the proof in the case that $A$ has big residue field, i.e., its residue field has more than $M_j$ elements in the sense of \cite[Proposition~10(5)]{Kerz2010}. Indeed, in that case the ind-smooth local rings $P_i$, $i=0,1$, also have big residue field and so the symbol maps $\K_j^M(P_i)=\hat \K_j^M(P_i)\to H^j_\sub{mot}(P_i,\bb Z(j))$ are isomorphisms by Kerz \cite[Proposition~10(11)]{Kerz2010}. Comparing the two coequaliser diagrams we deduce that the symbol map $\K_j^M(A)=\hat \K_j^M(A)\to H^j_\sub{mot}(A,\bb Z(j))$ is also an isomorphism.
\end{enumerate}

It remains to treat the case that $A$ has small (in particular, finite) residue field $\bb F_q$, which we do by constructing some ad-hoc transfer maps on $H^j_\sub{mot}(-,\bb Z(j))$. Let $\ell>0$ be an integer prime to $|\bb F_q:\bb F_p|$, so that $\bb F_{q^\ell}=\bb F_q\otimes_{\bb F_p}\bb F_{p^\ell}$ (this identity holds because the right side is a tensor product of Galois extensions of coprime degree, therefore a field); this also implies that the semi-local ring $A\otimes_{\bb F_p}\bb F_{p^\ell}$ is in fact local, as its quotient by its Jacobson radical is a field. Finally observe that $P_\bullet\otimes_{\bb F_p}\bb F_{p^\ell}\to A\otimes_{\bb F_p}\bb F_{p^\ell}$ is a simplicial resolution satisfying the conditions of Remark \ref{remark_lke_as_cycles}, and so (replacing $A$ by $A\otimes_{\bb F_p}\bb F_{p^\ell}$ above), we have a coequaliser diagram
\[\text{coeq}\big(H^j_\sub{mot}(P_1\otimes_{\bb F_p}\bb F_{p^\ell},\bb Z(j))\rightrightarrows H^j_\sub{mot}(P_0\otimes_{\bb F_p}\bb F_{p^\ell},\bb Z(j))\big)\quis H^j_\sub{mot}(A\otimes_{\bb F_p}\bb F_{p^\ell},\bb Z(j))\] Since classical motivic cohomology of smooth schemes admits functorial transfer maps along finite flat morphisms, this diagram induces a transfer map $N:H^j_\sub{mot}(A\otimes_{\bb F_p}\bb F_{p^\ell},\bb Z(j))\to H^j_\sub{mot}(A,\bb Z(j))$ such that the pre-composition with the canonical map $H^j_\sub{mot}(A,\bb Z(j))\to H^j_\sub{mot}(A\otimes_{\bb F_p}\bb F_{p^\ell},\bb Z(j))$ is multiplication by $\ell$. We make no claims that this transfer map is natural, independent of the simplicial resolution, compatible with any transfers on Milnor $K$-theory, etc.; in fact, we only care about the resulting fact that therefore $\text{ker}\big(H^j_\sub{mot}(A,\bb Z(j))\to H^j_\sub{mot}(A\otimes_{\bb F_p}\bb F_{p^\ell},\bb Z(j))\big)$ is annihilated by $\ell$.

It now follows formally that the symbol map factors through $\hat \K_j^M(A)$: indeed, given $x\in\text{ker}(\K_j^M(A)\to\hat \K_j^M(A))$ and any $\ell$ as in the previous paragraph such that $p^\ell>M_j$ then, by functoriality of the symbol map, and the established isomorphism for the local ring $A\otimes_{\bb F_p}\bb F_{p^\ell}$, we deduce that $\ell x$ is annihilated by the symbol map. Picking a different value of $\ell$, prime to the first value, shows that $x$ is annihilated by the symbol map, i.e., the latter factors through $\hat \K_j^M(A)$. The new symbol map $\hat \K_j^M(A) \rightarrow H^j_\sub{mot}(A,\bb Z(j))$ is surjective since the original one was.

To prove that the new symbol map $\hat \K_j^M(A)\to H^j_\sub{mot}(A,\bb Z(j))$ is injective, we again use a transfer argument; let $x$ be in the kernel. Then the transfer map for improved Milnor $K$-theory, and the established isomorphism in case of big residue field, shows that $\ell x=0$ for any $\ell$ as above. Again picking coprime values of $\ell$ shows that $x=0$ and so completes the proof.
\end{proof}

\begin{example}\label{example_Milnor_nil_Q}
Let $A$ be a local $\bb Q$-algebra, and $I\subseteq A$ a nilpotent ideal such that the quotient map $A\to A/I$ is split. Comparing the fundamental fibre sequences of Theorem \ref{thm:graded-pieces}(3) for $A$ and $A/I$, we get an equivalence of relative cohomologies \[\bb Z(j)^\sub{mot}(A,I)\simeq L\Omega^{<j}_{A,I/\bb Q}[-1].\] The right side is supported in degrees $\le j$ and a simple calculation shows that in degree $j$ it is given by $\Omega^{j-1}_{A,I/\bb Q}/d\Omega^{j-2}_{A,I/\bb Q}$. Meanwhile Theorem \ref{theorem_NS} implies that the left side of previous line is given by the relative Milnor $K$-group $K_j^M(A,I)$ in degree $j$. Taking $H^j(-)$ of the previous line thus obtains a natural isomorphism \begin{equation}\K_j^M(A,I)\cong \Omega^{j-1}_{A,I/\bb Q}/d\Omega^{j-2}_{A,I/\bb Q}.\label{eqn_Milnor_Goodwillie}\end{equation}
When $n=2$ the isomorphism \eqref{eqn_Milnor_Goodwillie} goes back to Bloch \cite[Thm.~0.1]{Bloch1974} and is famous precursor to Goodwillie's isomorphism between relative $K$-theory and cyclic homology of $\bb Q$-algebras.

We remark that such an isomorphism \eqref{eqn_Milnor_Goodwillie} has been shown to hold by Gorchinskiy--Tyurin \cite{GorchinskiyTyurin2018} under very mild hypotheses: instead of requiring $A$ to be a local $\bb Q$-algebra, they only require it to be a weakly $5$-fold stable $\bb Z[\tfrac1N]$-algebra where $N$ is large enough so that $I^N=0$. It would be interesting to try to reprove their result via the motivic cohomology of this article and \cite{Bouis2025, Bouis2025a, Bouis2025b}.
\end{example}

\section{Motivic Soul\'e--Weibel vanishing and pro cdh descent}\label{section_Weibel}
One of the most influential conjectures concerning the algebraic $K$-theory of singular schemes has been Weibel's conjecture \cite{Weibel1980}, now a theorem of Kerz--Strunk--Tamme \cite{KerzStrunkTamme2018}. It states, in particular, that for a Noetherian scheme $X$ of dimension $\le d$, the negative $K$-groups $\K_{-n}(X)$ vanish for $n>d$. Kerz--Strunk--Tamme's proof proceeds by first establishing pro cdh descent for $K$-theory of Noetherian schemes. For earlier work on special cases on Weibel's conjecture and pro cdh descent, see for example \cite{Cortinas2008, GeisserHesselholt2010, Krishna2006, Krishna2009a, Krishna2010, pro-cdh, Weibel2001}.

Our goal in this section is to prove the following motivic refinement of part of Weibel's conjecture in the equicharacteristic case and pro cdh descent for new motivic cohomology; as usual, let $\bb F$ denote a prime field.

\begin{theorem}[Motivic Soul\'e--Weibel vanishing]\label{theorem_Weibel_vanishing}
Let $j\ge0$ and let $X$ be a Noetherian $\bb F$-scheme of finite dimension. Then $H^i_\sub{mot}(X,\bb Z(j))=0$ for all $i> j+\dim X$.
\end{theorem}

\begin{theorem}[Pro cdh descent]\label{theorem_pro_cdh_descent}
On the category of Noetherian $\bb F$-schemes, the presheaf $\bb Z(j)^\sub{mot}$ satisfies pro cdh descent for each $j\ge0$. That is, given any abstract blow-up square of Noetherian $\bb F$-schemes
\begin{equation}
\xymatrix{
Y'\ar[r]\ar[d] & X'\ar[d] \\
Y\ar[r] & X
},\label{eqn_blowup}\end{equation}
the associated square of pro complexes 
\[\xymatrix{
\bb Z(j)^\sub{mot}(X) \ar[r]\ar[d] & \bb Z(j)^\sub{mot}(X')\ar[d]\\
\{\bb Z(j)^\sub{mot}(rY)\})_r \ar[r] & \{\bb Z(j)^\sub{mot}(rY')\})_r
}\]
is cartesian.\footnote{Here $rY$ denotes the $r-1^\sub{st}$ infinitesimal thickening of $Y$ inside $X$, and similarly for $rY'$. By ``cartesian'' we simply mean that all pro cohomology groups of the birelative term are zero as pro abelian groups; since $\bb Z(j)^\sub{mot}$ of a Noetherian scheme is bounded above depending only on the dimension (this does not require Theorem \ref{theorem_Weibel_vanishing}, but only the descriptions given in the proof Corollary~\ref{cor:surjections}), this is equivalent to being cartesian in the $\infty$-category of pro complexes.}
\end{theorem}

\begin{remark}[Relation to Weibel's $K$-theoretic vanishing conjecture]\label{remark_Weibel} 
Let $X$ be a Noetherian $\bb F$-scheme of dimension $\le d$. Theorem \ref{theorem_Weibel_vanishing} states that the Atiyah--Hirzebruch spectral sequence $E_2^{ij}=H^{i-j}_\sub{mot}(X,\bb Z(-j))\Rightarrow \K_{-i-j}(X)$ is supported in the left half plane $x\le d$. From this one immediately reads off parts of Weibel's package of conjectures about lower $K$-groups: both the vanishing $\K_{-n}(X)=0$ for $n>d$, and the usual description of $\K_{-d}(X)$ via an edge map isomorphism \[H^{d}_\sub{cdh}(X,\bb Z)=H^{d}_{\sub{mot}}(X,\bb Z(0))\cong \K_{-d}(X).\] Theorem \ref{theorem_Weibel_vanishing} can therefore be seen as a motivic refinement of Weibel's vanishing conjecture. This is moreover reflected in the proof of the theorem, which is based on both the arguments of \cite{Cortinas2008, Cortinas2008a}, where Corti\~nas--Haesemyer--Schlichting--Weibel proved Weibel's vanishing conjecture and Vorst's conjecture for varieties over a characteristic zero field, and of \cite{KerzStrunkTamme2018}, where Kerz--Strunk--Tamme used pro cdh descent to prove Weibel's conjecture in general.

We stress however that Theorem \ref{theorem_Weibel_vanishing} is certainly not a new proof of Weibel's vanishing conjecture in $K$-theory, since our theory uses the fundamental square Theorem \ref{thm:mainsq} which itself relies on the work of Kerz--Strunk--Tamme. We refer the reader to Remark~\ref{sec:usual} for more details on this point.
\end{remark}

\begin{remark}[Applications to Adams eignenspaces]\label{remark_Soule_Weibel}
To use the Atiyah--Hirzebruch spectral sequence to deduce the usual $K$-theoretic Weibel conjecture from a vanishing result in motivic cohomology, it would have been sufficient to establish the following weaker diagonal vanishing line: $H^i_\sub{mot}(X,\bb Z(j))=0$ for all $i> 2j+\dim X$. The stronger vertical vanishing line of Theorem \ref{theorem_Weibel_vanishing} is related to a vanishing theorem of Soul\'e as follows. By rationalising the Atiyah--Hirzebruch spectral sequence and rewriting in terms of Adams eigenspaces, Theorem \ref{theorem_Weibel_vanishing} implies that for any Noetherian $\bb F$-scheme $X$ of dimension $\le d$ we have the following vanishing for each $n\in\bb Z$: the Adams eigenspace $\K_n(X)_\bb Q^{(j)}$ vanishes whenever $j>n+d$. This vanishing is due to SGA6 \cite[Expos\'e~VI, Theorem~6.9]{SGA_VI} in the case of $\K_0$ of Noetherian schemes with an ample line bundle, and when $n>0$ to Soul\'e \cite[Corollary~1]{Soule1985} for the higher algebraic $K$-groups of Noetherian rings; when $n<0$ this vanishing of Adams eigenspaces of negative $K$-groups is new as far as we are aware.

In short, Theorem \ref{theorem_Weibel_vanishing} provides an integral refinement of Soul\'e's result, as well as an extension beyond the affine case and to negative $K$-groups.

\end{remark}

\subsection{Proof of pro cdh descent}\label{sec:pro-cdh}
Here we prove Theorem \ref{theorem_pro_cdh_descent}. We begin by noting a similar pro cdh descent property for the Nisnevich cohomology of wedge powers of the cotangent complex $R\Gamma(-,L^i_{-/\bb F}): \Sch_\bb F^{\qcqs,\op} \to \mathrm{D}(\bb Z)$. The following is a slight generalization of \cite[Theorem~2.10]{pro-cdh}.

\begin{lemma}\label{lem_pro_cdh_for_cotangent}
For any abstract blow-up square of Noetherian $\bb F$-schemes (\ref{eqn_blowup}) and $i\ge0$,
the square of pro complexes
\[
\begin{tikzcd}
R\Gamma(X,L^i_{-/\bb F}) \ar{r} \ar{d} & R\Gamma(X',L^i_{-/\bb F}) \ar{d}\\
\{ R\Gamma(rY,L^i_{-/\bb F})\}_r \ar{r} &\{ R\Gamma(rY',L^i_{-/\bb F})\}_r
\end{tikzcd}
\]
is cartesian.
\end{lemma}

\begin{proof}
The proof works in the exact same way as in \cite[Theorem~2.10]{pro-cdh}, except that we need to justify why \cite[Theorem~2.4]{pro-cdh} does not require the stated finite dimensionality hypothesis; but this follows from the general formal functions theorem of \cite[Lemma~8.5.1.1]{LurieSAG}.
\end{proof}

Next we establish pro $\cdh$ descent for syntomic cohomology; remarkably, the proof uses algebraic $K$-theory:

\begin{proposition}\label{prop_pro_cdh_syn}
For any abstract blow-up square of Noetherian $\bb F_p$-schemes (\ref{eqn_blowup}) and $j\ge0$, the square of pro complexes
\[
\begin{tikzcd}
\Z_p(j)^{\syn}(X) \ar{r} \ar{d} & \Z_p(j)^{\syn}(X') \ar{d}\\
\{ \Z_p(j)^{\syn}(rY)\}_r \ar{r} &\{ \Z_p(j)^{\syn}(rY')\}_r
\end{tikzcd}
\]
is cartesian.
\end{proposition}
\begin{proof}
Since mod-$p$ syntomic cohomology $\bb F_p(j)^\sub{syn}$ admits a finite filtration with graded pieces given by shifts of $R\Gamma(-, L_{-/\bb F}^i)$ for various $i$ (by sheafifying Lemma \ref{lem_fin_fil_on_syn}), it satisfies pro cdh descent thanks to Lemma \ref{lem_pro_cdh_for_cotangent}. The remaining difficulty is to extend the result from $\bb F_p(j)^\sub{syn}$ to $\bb Z_p(j)^\sub{syn}$.

Fix $n\in\bb Z$ and set $A_r:=H^n(\bb Z_p(j)^\sub{syn}(X,X',rY))$ for each $r\ge0$; the goal is to show that the pro abelian group $\{A_r\}_r$ vanishes. We claim that each group $A_r$ is bounded $p$-power torsion. Granting this claim, we may complete the proof as follows. Given $s\ge1$, pick $c>$ such that $p^cA_s=0$. By the previous paragraph and an induction, we see that $\{A_r/p^c\}_r=0$; so there exists $s'> s$ such that the transition map $A_{s'}/p^c\to A_s/p^c$ is zero. But $A_s/p^c=A_s$, so this shows that the transition map $A_{s'}\to A_s$ is zero, as required.
 
It remains to prove that $A_r$ is bounded $p$-power torsion. But it is both derived $p$-complete (since it is $H^n$ of a $p$-complete complex) and satisfies $A_r[\tfrac1p]=0$ (since $\bb Q_p(j)^\sub{syn}$ satisfies cdh descent by Lemma \ref{lemma_Qpsyn}, which we proved using $K$-theory), so it is killed by a power of $p$ by \cite[Theorem~1.1]{Bhatt2019}.
\end{proof}

\begin{proof}[Proof of Theorem \ref{theorem_pro_cdh_descent}]
If $\bb F=\bb F_p$ then Theorem \ref{thm:graded-pieces_charp}(2) shows that $\bb Z(j)^\sub{mot}$ differs from $\bb Z_p(j)^\sub{syn}$ by a cdh sheaf; since syntomic cohomology satisfies pro cdh descent by Proposition \ref{prop_pro_cdh_syn}, the same is true for motivic cohomology.

If instead $\bb F=\bb Q$ then the third term in the fundamental fibre sequence Theorem \ref{thm:graded-pieces}(3) satisfies pro cdh descent by Lemma \ref{lem_pro_cdh_for_cotangent}; since the middle term in the fibre sequence is a cdh sheaf, it follows that motivic cohomology also satisfies pro cdh descent.
\end{proof}

\begin{remark}[Extension to arbitrary qcqs (derived) schemes]
Kelly--Saito--Tamme \cite[Theorem 5.4]{KellySaitoTamme2026} have extended Theorem~\ref{theorem_pro_cdh_descent} to qcqs derived schemes, replacing $rY$ and $rY'$ by suitable derived thickenings. The key is their extension of pro cdh descent for the cotangent complex (Lemma~\ref{lem_pro_cdh_for_cotangent}) to such generality.
\end{remark}

\subsection{Proof of motivic Soul\'e--Weibel vanishing}
Fix a weight $j\ge0$. Here we prove Theorem \ref{theorem_Weibel_vanishing}. In fact, we prove the following stronger statement:

\begin{theorem}\label{theorem_Weibel_vanishing2}
Let $X$ be a Noetherian $\bb F$-scheme of dimension $\le d$. Then the fibre \[W(j)(X):=\opp{fib}\big(\bb Z(j)^\sub{mot}(X)\To \bb Z(j)^{\bb A}(X)\big)\] vanishes in degrees $>j+d$.
\end{theorem}

\begin{remark}\label{rem:wj}
The cohomology theory $W(j)(X)$ are the shifts of graded pieces of a filtration on the fibre of $\K(X) \rightarrow \KH(X)$. In turn each $W(j)(X)$ admits a filtration whose graded pieces are the ``$N^r$ of motivic cohomology,'' i.e., the fibres of the maps $\Z(j)^{\mot}(X) \rightarrow \Z(j)^{\mot}(\bb A^r \times X)$. These groups refine Bass' $N^rK$-groups which measure the failure of algebraic $K$-theory to be $\bb A^r$-invariant. We intend to explore questions surrounding these groups using motivic methods in the future. 
\end{remark}

Note that Theorem \ref{theorem_Weibel_vanishing2} implies Theorem \ref{theorem_Weibel_vanishing}, as we already know from Theorem~\ref{thm:cdh}(1) that the $\bb A^1$-motivic cohomology $\bb Z(j)^\bb A(X)$ is supported in degrees $\le j+d$; but the stronger statement also tells us that the map $\bb Z(j)^\sub{mot}(X)\to \bb Z(j)^{\bb A}(X)$ is surjective in degree $j+d$.
First we quote the following result, whose Zariski version has often appeared in previous work on the subject:

\begin{lemma}[Nisnevich vanishing lemma]\label{lem:nis-vanish}
Let $\cal F$ be a Nisnevich sheaf of abelian groups on a Noetherian scheme $X$, and $d\ge0$ such that the stalk $\cal F_x$ vanishes for all $x\in X$ satisfying $\dim\res{\{x\}}>d$. Then $H^i_\sub{Nis}(X,\cal F)=0$ for all $i>d$.
\end{lemma}
\begin{proof}
This is proved by induction using the coniveau spectral sequence, as stated in \cite[(1.2.5)]{KatoSaito1986} (see also the proof of \cite[Theorem~1.32]{Nisnevich1989}).
\end{proof}

When proving Weibel and Vorst's conjecture for finite type schemes over characteristic zero fields \cite{Cortinas2008, Cortinas2008a}, Corti\~nas--Haesemeyer(--Schlichting)--Weibel analysed the relation between of the top degree Nisnevich and cdh cohomologies of sheaves of differential forms; although pro cdh descent did not appear explicitly, it was implicitly encoded in their use of the formal functions theorem. When proving Weibel's conjecture for $K$-theory \cite{KerzStrunkTamme2018}, Kerz--Strunk--Tamme used pro cdh descent to show that the desired vanishing was of a birational nature. The following proposition may be seen as an axiomatisation of the aforementioned arguments.

\begin{proposition}\label{prop_axiomatic_Weibel}
Let $k$ be a base ring and $W:\Sch_k^\sub{qcqs,op}\to \rm D(\bb Z)$ a finitary Nisnevish sheaf with the following properties:
\begin{enumerate}
\item $L_\sub{cdh}W\simeq0$.
\item $W$ satisfies pro cdh descent on Noetherian $k$-schemes.
\item For any Noetherian, local, henselian $k$-algebra $A$ and nilpotent ideal $I\subseteq A$,  the fibre $W(A,I)=\fib(W(A)\to W(A/I))$ is supported in cohomological degrees $\le 0$.
\end{enumerate}
Then, for any Noetherian $k$-scheme $X$ of finite dimension, $W(X)$ is supported in cohomological degrees $\le \dim X$.
\end{proposition}
\begin{proof}
We begin by globalising hypothesis (3) by noting the following 
\begin{quote}
(3'): for any Noetherian $k$-scheme $X$ of finite dimension and nil immersion $X_0\to X$, the fibre $W(X,X_0)$ is supported in degrees $\le \dim X$ (and so $W(X)\to W(X_0)$ is an equivalence in degrees $>\dim X$).
\end{quote}
This follows from Nisnevich descent and Nisnevich exactness of the closed embedding $X_0\into X$, more precisely using that $X_\sub{Nis}$ has cohomological dimension $\le\dim X$ and the sheaf $W(-,-\times_XX_0)$ on $X_\sub{Nis}$ has stalks supported in degrees $\le 0$ by (3).

Now let $X$ be a Noetherian $k$-scheme. We must show that $W(X)$ is supported in cohomological degrees $\le\dim X$. Using (3') we may assume that $X$ is reduced.

If $\dim X=0$ then $X$ is a finite disjoint union of the spectra of fields. Since spectra of fields are points for the cdh topology we have $W(X)\quis L_\sub{cdh}W(X)$, which vanishes by hypothesis (1).

We now proceed by induction on $\dim X$, so assume that $d:=\dim X>0$ and that the desired vanishing has been proved for Noetherian $k$-schemes of dimension $<d$. We examine the bounded Nisnevich descent spectral sequence \[E_2^{ab}=H^a_\sub{Nis}(X,\cal H^b(W))\implies H^{a+b}(W(X)),\] where $\cal H^b(W)$ is the Nisnevich sheafifcation of the abelian presheaf $Y\mapsto H^b(W(Y))$. The $E_2$ page of this spectral sequence enjoys various vanishings:
\begin{enumerate}
\item $E_2^{ab}=0$ if $a>d$ (or if $a<0$), since $X$ has Nisnevich cohomological dimension $\le d$.
\item $E_2^{ab}=0$ if $a>0$ and $b>d$. Indeed, for such $b$ and any $x\in X$ such that $\dim\overline{\{x\}}>0$, then $\dim\roi_{X,x}^h<d$ and so the stalk $\cal H^b(W)_x=H^b(W(\roi_{X_x}^h))$ (the equality is a consequence of $W$ being finitary) vanishes by the inductive hypothesis. Lemma \ref{lem:nis-vanish} now implies that $\cal H^b(W(j))$ has no higher cohomology.
\item $E_2^{ab}=0$ if $b\le d$ and $a+b>d$. The proof will be clearest if we start by fixing $b\le d$. Then, for any $x\in X$ such that $\dim\overline{\{x\}}>d-b$, we have that $\dim\roi_{X,x}^h<b\le d$, i.e., $b>\dim\roi_{X,x}^h$ and $\dim\roi_{X,x}^h<d$; so $\cal H^b(W(j))_x=0$ by the inductive hypothesis (and again finitariness to compute the stalk in terms of $\roi_{X,x}^h$). Lemma \ref{lem:nis-vanish} now implies that $\cal H^b(W)$ has no cohomology in degrees $>d-b$, or in other words $H^a(\cal H^b(W))=0$ whenever $a+b>d$.
\end{enumerate}
Thanks to vanishings (1)--(4), we can read off from the Nisnevich descent spectral sequence edge map isomorphisms \[H^{n}(W(X))\isoto H^0_\sub{Nis}(X,\cal H^{n}(W))\] for all $n>d$. For the rest of the proof fix $n>d$. Allowing $X$ to vary, the previous isomorphism may be rephrased as follows:
\begin{quote}
(\dag) On the category of Noetherian $k$-schemes of dimension $\le d$, the abelian presheaf $H^{n}(W(-))$ is a Nisnevich sheaf.
\end{quote}
In fact, we will only need to know that it is Nisnevich separated.

We now return to our fixed $X$ of dimension $\le d$, and pick a cohomology class $\al\in H^{n}(W(X))$; we must show that $\al=0$. We claim that there exists a modification $f:X'\to X$ (i.e., a proper morphism where $X'$ is also reduced and such there there exists a dense open $U\subseteq X$ satisfying $f^{-1}(U)\isoto U$) such that $f^*\al=0$ in $H^{n}(W(X'))$. To prove the claim we first use hypothesis (1) to see that, for any $a\in\bb Z$, the presheaf $H^{a}(W(-))$ vanishes on valuation rings, therefore vanishes after cdh sheafification. In particular there exists a cdh cover $U\to X$ such that $\al$ vanishes in $H^{n}(W(U))$; we can then refine $U$ to a cdh cover of the form $X_2\to X_1\xto{g} X$ where $X_1\to X$ is a proper cdh (often called a cdp) cover and $X_2\to X_1$ is a Nisnevich cover \cite[Proposition~5.9]{SuslinVoevodsky2000}. Next note that there exists a modification $f:X'\to X$ which factors through $X_1$: for example pick a dense open $U\subseteq X$ such that $f^{-1}(U)\isoto U$, and define $X'$ to be $(-)_\sub{red}$ of the closure of $U$ in $X'$. By construction $\al$ vanishes when we pull back to $X_2$, hence also to $X'\times_X X_2$; but $X'\times_X X_2\to X'$ is a Nisnevich cover of schemes of dimension $\le d$ (since modifications and \'etale morphisms do not increase dimension), so (\dag) implies that $H^{n}(W(X'))\to H^{n}(W(X'\times_XX_2))$ is injective and therefore $\al$ already vanished when pulled back to $X'$. This completes the proof of the claim.

Our modification $X'\to X$ fits into an abstract blowup square (\ref{eqn_blowup}) in which $Y'$ and $Y$ have dimension $<d$. From hypothesis (2) and the inductive hypothesis applied to the infinitesimal thickenings of $Y$ and $Y'$, we see that $H^{n}(W(X))\to H^{n}(W(X'))$ is an isomorphism. But this map was constructed so as to kill $\al$. Therefore $\al$ was already zero in $H^{n}(W(X))$, completing the proof.
\end{proof}

\begin{example}[Usual Weibel vanishing]\label{sec:usual}
Here we present a revisionist version of Kerz--Strunk--Tamme's proof of Weibel vanishing \cite{KerzStrunkTamme2018}. First note that Proposition \ref{prop_axiomatic_Weibel} applies verbatim to finitary Nisnevich presheaves of spectra $W:\textrm{Sch}_k^\sub{qcqs,op}\to\textrm{Sp}$ satisfying the same hypotheses; we stated it for presheaves of complexes only for simplicity.

In particular, the proposition applies when $W:=\fib(\K\to\KH)$ and $k=\bb Z$. Indeed, hypothesis (1) follows from the fact that $K(V)\quis \KH(V)$ for any valuation ring $V$ \cite[Theorem~6.3]{KerzStrunkTamme2018} \cite[Theorem.~3.4]{KellyMorrow2021}; hypothesis (2) follow fom pro cdh descent of $K$-theory and cdh descent of $\KH$-theory \cite{KerzStrunkTamme2018}; hypothesis (3) follows from nil invariance of negative $K$-theory. We therefore deduce, for any Noetherian scheme $X$, that $\fib(\K(X)\to\KH(X))$ is supported in homological degrees $\ge-\dim X$. Since there are various ways to show that $\KH(X)$ is supported in homological degrees $\ge-\dim X$ \cite{KerzStrunk2017} \cite[Remark~3.5(a)]{KellyMorrow2021}, we deduce that $\K(X)$ is also supported in homological degrees $\ge-\dim X$ as required.
\end{example}

We now verify that the previous proposition may also be applied in our motivic situation of interest, at least up to a harmless shift:

\begin{proposition}\label{prop_of_W(j)}
The presheaf $W(j)=\fib(\bb Z(j)^\sub{mot}\to\bb Z(j)^{\bb A}):\Sch_\bb F^\sub{qcqs,op}\to\rm D(\bb Z)$ is a finitary Nisnevich sheaf with the following properties:
\begin{enumerate}
\item $L_\sub{cdh}W(j)=0$.
\item $W(j)$ satisfies pro cdh descent on Noetherian $\bb F$-schemes.
\item For any $\bb F$-algebra $A$ and finitely generated nilpotent ideal $I\subseteq A$, the fibre $W(j)(A,I)$ is supported in cohomological degrees $\le j$.
\end{enumerate}
\end{proposition}
\begin{proof}
The presheaf $W(j)$ is a finitary Nisnevich sheaf since $\bb Z(j)^\sub{mot}$ and $\bb Z(j)^{\bb A}$ are, by Theorems \ref{thm:cdh}(2), \ref{thm:graded-pieces}(5), and  \ref{thm:graded-pieces_charp}(4).

(1): $W(j)$ vanishes after cdh sheafification since $L_\sub{cdh}\bb Z(j)^\sub{mot}\quis \bb Z(j)^{\bb A}$ by Theorem~\ref{thm_mot_vs_cdh}(1).

(2): $W(j)$ satisfies pro cdh descent on Noetherian schemes, since the same is true of $\bb Z(j)^\sub{mot}$ by Theorem~\ref{theorem_pro_cdh_descent}.

(3): We first treat the case that $\bb F=\bb Q$. From the fundamental fibre sequence of Theorem \ref{thm:graded-pieces}(3), for both $A$ and $A/I$, we have a description of the relative term $W(j)(A,I)$ as \[W(j)(A,I)={\rm fib}\left(L\Omega_{A/\Q}^{< j} \to  L\Omega_{(A/I)/\Q}^{< j}\right)[-1].\] This is supported in degrees $<j$ since $\Omega_{A/\bb Q}^{j-1}\to\Omega_{(A/I)/\bb Q}^{j-1}$ is surjective.

In the case that $\bb F=\bb F_p$, the pullback square of Theorem \ref{thm:graded-pieces_charp}(2) shows that $W(j)(A,I)\simeq\fib(\bb Z_p(j)^\sub{syn}(A)\to\bb Z_p(j)^\sub{syn}(A/I))$, which is derived $p$-complete; so it is sufficient to prove the claim modulo $p$, namely that $\fib(\bb F_p(j)^\sub{syn}(A)\to\bb F_p(j)^\sub{syn}(A/I))$ is supported in degree $<j$. Since nilpotent ideals are henselian, this is a special case of \cite[Theorem~5.2]{AntieauMathewMorrowNikolaus}.
\end{proof}

\begin{proof}[Proof of Theorem \ref{lem:nis-vanish}]
Apply Proposition \ref{prop_axiomatic_Weibel} to the presheaf $W:=W(j)[j]$. The hypotheses of the proposition are satisfied thanks to Proposition \ref{prop_of_W(j)}.
\end{proof}

Our arguments implicitly reprove the results of \cite{Cortinas2008, Cortinas2008a} concerning Nisnevich and cdh cohomology of differential forms, as well as a similar style of result in finite characteristic:

\begin{corollary}\label{cor:surjections}
\begin{enumerate}
\item For any Noetherian $\bb F$-scheme of dimension $d$ and $j \geq 0$, the canonical map $H^{j+d}_{\mot}(X, \Z(d)) \rightarrow H^{j+d}_{\bb A}(X, \Z(d))$ is surjective.
\item For $j\ge1$ and any Noetherian $\bb Q$-scheme $X$ of dimension $\le d$, the canonical map \[H^d_\sub{Nis}(X,\Omega^{j-1}_{-/\bb Q})\to H^d_\sub{cdh}(X,\Omega^{j-1}_{-/\bb Q})\] is surjective.
\item For $j\ge0$ and any Noetherian $\bb F_p$-scheme $X$ of dimension $\le d$, the canonical maps \[H^d_\sub{Nis}(X,\hat \K_j^M/p)\to  H^d_\sub{cdh}(X,\hat \K_j^M/p)\quad\mathrm{and}\quad H^{d-1}_\sub{Nis}(X,\tilde\nu(j))\to H^{d-1}_\sub{cdh}(X,\tilde\nu(j))\] are surjective, and the canonical map \[H^{d}_\sub{Nis}(X,\tilde\nu(j))\To H^{d}_\sub{cdh}(X,\tilde\nu(j))\] is an isomorphism. Here $\hat \K_j^M/p$ denotes improved Milnor $K$-theory mod $p$, as an abelian Nisnevich or cdh sheaf.
\end{enumerate}
\end{corollary}
\begin{proof}
The first claim was explained after Remark \ref{rem:wj}. The rest are related to $W(j)(X)$ via the following descriptions, which we state in the generality of qcqs schemes for the sake of possible future reference:
\begin{enumerate}
\item For any qcqs $\bb Q$-scheme $X$ of valuation dimension $\le d$, then $W(j)(X)$ vanishes in cohomological degrees $>j+d+1$ and there is a natural isomorphism $H^{j+d+1}(W(j)(X)) \cong \opp{coker}(H^d_\sub{Nis}(X,\Omega^{j-1}_{-/\bb Q})\to H^d_\sub{cdh}(X,\Omega^{j-1}_{-/\bb Q}))$.
\item For any qcqs $\bb F_p$-scheme $X$ of valuative dimension $\le d$, then $W(j)(X)/p$ vanishes in cohomological degrees $>j+d+2$ and there is a natural diagram in which the row and column are exact:
\begin{equation}\xymatrix@=5mm{
&&& H^{d-1}_\sub{Nis}(X,\tilde\nu(j))\ar[d]&\\
&&& H^{d-1}_\sub{cdh}(X,\tilde\nu(j))\ar[d]&\\
H^d_\sub{Nis}(X,\hat \K_j^M/p)\ar[r] & H^d_\sub{cdh}(X,\hat \K_j^M/p)\ar[r]^-{\delta} & H^{j+d+1}(W(j)(X)/p)\ar[r] & \opp{coker}\delta\ar[r] \ar[d]& 0\\
&&& H^{d}_\sub{Nis}(X,\tilde\nu(j))\ar[d]&\\
&&& H^{d}_\sub{cdh}(X,\tilde\nu(j))\ar[d]&\\
&&& H^{j+d+2}(W(j)(X)/p)\ar[d]&\\
&&&0&
}\end{equation}
\end{enumerate}
These two claims are clearly sufficient to deduce the corollary, since Theorem \ref{theorem_Weibel_vanishing2} tells us that $W(j)(X)$ (and so also $W(j)(X)/p$) is supported in cohomological degrees $\le j+d$ whenever $X$ is a Noetherian $\bb F$-scheme of dimension $\le d$.

It remains to prove the claims. We first treat the case that $\bb F=\bb Q$. From the fundamental fibre sequence Theorem \ref{thm:graded-pieces}(3) we have a description of $W(j)(X)$, for any qcqs $\bb Q$-scheme $X$, as \[W(j)(X)={\rm fib}\left(R\Gamma(X,L\Omega_{-/\Q}^{< j})\to  R\Gamma_\sub{cdh}(X,L\Omega_{-/\Q}^{< j})\right)[-1].\] By cohomological vanishing bounds in the Nisnevich and cdh topologies, this fibre is supported in cohomological degrees $\le j+d+1$ if $X$ has valuative dimension $\le d$, with its $H^{j+d+1}$ being exactly the desired cokernel.

Next suppose $\bb F=\bb F_p$. From the pullback square of Theorem \ref{thm:graded-pieces_charp}(2) we see that $W(j)/p=\opp{fib}(\bb F_p(j)^\sub{syn}\to L_\sub{cdh}\bb F_p(j)^\sub{syn})$, which we compute as follows: on the category of qcqs $\bb F_p$-schemes, Nisnevich sheafifying Remark \ref{remarks_tildenu} provides us with a fibre sequence \[L_\sub{Nis}\tau^{\le j}\bb F_p(j)^\sub{syn}\To \bb F_p(j)^\sub{syn}\To R\Gamma_\sub{Nis}(-,\tilde\nu(j))[-j-1],\] which may be compared to its cdh sheafification to get the following fibre sequence of presheaves on qcqs $\bb F_p$-schemes: \begin{equation}\opp{fib}\big(L_\sub{Nis}\tau^{\le j}\bb F_p(j)^\sub{syn}\to L_\sub{cdh}\tau^{\le j}\bb F_p(j)^\sub{syn}\big)\To W(j)/p\To \opp{fib}\big(R\Gamma_\sub{Nis}(-,\tilde\nu(j))\to R\Gamma_\sub{cdh}(-,\tilde\nu(j)))\big[-j-1].\label{eqn_nis_vs_cdh}\end{equation} Moreover, $H^j(\bb F_p(j)^\sub{syn}(-))$ is Nisnevich locally given by $\hat \K_j^M/p$; this follows from the isomorphisms $\hat \K_j^M(A)/p\isoto H^j_\sub{mot}(A,\bb F_p(j))\isoto H^j_\sub{syn}(A,\bb F_p(j))$ for local $\bb F_p$-algebras $A$, the first being the Nesterenko--Suslin isomorphism of Theorem \ref{theorem_NS} (or rather, the mod-$p$ version obtained using Corollary \ref{corollary_Hilb_90}) and the second isomorphism coming from the fundamental fibre sequence of Corollary \ref{corol_fundamental_p}. Since the Nisnevich and cdh sites of $X$ have cohomological dimension $\le d$ when $X$ has valuative dimension $\le d$, the claimed vanishing and diagram in (2) can now be read off by by evaluating (\ref{eqn_nis_vs_cdh}) on $X$.
\end{proof}

\section{Some comparisons to algebraic cycles}\label{section_cf_cycles}
We present in this section a variety of contexts in which our motivic cohomology admits a description in terms of algebraic cycles. We do not know what to expect in general.

\begin{definition}
For $j\ge0$ and $A$ a local $\bb F$-algebra, we say that $A$ has {\em geometric weight-$j$ motivic cohomology} if $\bb Z(j)^\sub{mot}(A)$ is supported in cohomological degrees $\le j$.
\end{definition}

If $A$ has geometric weight-$j$ motivic cohomology in the sense of the definition, then Theorem \ref{thm_lke_lej} implies that the canonical map $\bb Z(j)^\sub{lse}(A)\to \bb Z(j)^\sub{mot}(A)$ is an equivalence. But, as explained in Remark~\ref{remark_lke_as_cycles}, the complex $\bb Z(j)^\sub{lse}(A)$ admits a description purely in terms of algebraic cycles; so in this case we obtain a cycle theoretic description of the whole motivic cohomology $\bb Z(j)^\sub{mot}(A)$. This applies to a surprisingly large class of rings:

\begin{theorem}\label{theorem_geometric_cohomology}
Let $j\ge0$ and let $A$ be a local $\bb F$-algebra.
\begin{enumerate}
\item If $A$ is regular Noetherian then it has geometric weight-$j$ motivic cohomology.
\item If $A$ is a valuation ring then it has geometric weight-$j$ motivic cohomology.
\item If there exists a nilpotent ideal $I\subseteq A$ such that $A/I$ has geometric weight-$j$ motivic cohomology, then so does $A$.
\item If $A$ is Noetherian, henselian of dimension $1$, then it has geometric weight-$j$ motivic cohomology.
\item If $j\ge1$ and $A$ is Noetherian, henselian of dimension $\le 2$, then it has geometric weight-$j$ motivic cohomology.
\end{enumerate}
\end{theorem}
\begin{proof}
(1): Using N\'eron--Popescu we reduce, by taking a filtered colimit, to the case that $A$ is essentially smooth over $\bb F$; then we apply the usual Gersten vanishing bound for classical motivic cohomology after knowing the classical comparison result, Corollary~\ref{corol_smooth_comparison}.

(2): If $A$ is a valuation ring then the canonical map $\bb Z(j)^\sub{mot}(A)\to \bb Z(j)^{\bb A}(A)$ is an equivalence since the right vertical maps of the fundamental squares of Theorems~\ref{thm:graded-pieces} and~\ref{thm:graded-pieces_charp} are equivalences. But now, $\bb Z(j)^{\bb A}$ is cdh locally left Kan extended from smooth $\bb F$-schemes by Theorem~\ref{thm:cdh}(10) and therefore $\bb Z(j)^\sub{lse}(A)\quis \bb Z(j)^{\bb A}(A)$. Finally recall once again from the Gersten bound that $\bb Z(j)^\sub{lse}$ is Zariski locally supported in cohomological degrees $\le j$.

(3): We must show that the relative motivic cohomology $\bb Z(j)^\sub{mot}(A,I)$ is supported in degrees $\le j$; by finitariness we may assume that $I$ is finitary generated, hence nilpotent. First we treat the case that $\bb F=\bb Q$. Since $\bb Z(j)^{\bb A}$ and $R\Gamma_\sub{cdh}(-,L\Omega^{<j}_{-/\bb Q})$ are cdh sheaves, they are invariant for the ideal $I$; by taking the horizontal fibres of the fundamental fibre sequence Theorem \ref{thm:graded-pieces}(3) we therefore obtain an equivalence of relative terms \[\bb Z(j)^\sub{mot}(A,I)\simeq L\Omega^{<j}_{A,I/\bb Q}[-1].\] The right side is clearly supported in degrees $\le j$, which completes the proof. Next we assume $\bb F=\bb F_p$. Then again $\bb Z(j)^\sub{cdh}$ and $L_\sub{cdh}\bb Z_p(j)^\sub{syn}$ are invariant for $I$, and so we obtain an equivalence \[\bb Z(j)^\sub{mot}(A,I)\simeq\bb Z_p(j)^\sub{syn}(A,I).\] In this case it is a non-trivial result that the right side is supported in degrees $\le j$ \cite[Theorem~5.2]{AntieauMathewMorrowNikolaus}.

(4): Combine the Soul\'e--Weibel vanishing bound Theorem~\ref{theorem_Weibel_vanishing} with Corollary \ref{corollary_Hilb_90}. Part (5) is proved in the same way, but also using Proposition \ref{proposition_91}.
\end{proof}

\begin{example}[Dimension $0$]\label{example_dim_0}
Let $j\ge0$ and let $A$ be a local $\bb F$-algebra with nilpotent maximal ideal $\frak m$. Then, as we saw in the proof of part (3) of the theorem, in characteristic zero there is an equivalence \begin{equation}\bb Z(j)^\sub{mot}(A,\frak m)\simeq L\Omega^{<j}_{A,\frak m/\bb Q}[-1]\label{eqn_cycles=differential}\end{equation} for the relative motivic cohomology, and in characteristic $p$ there is an equivalence \begin{equation}\bb Z(j)^\sub{mot}(A,\frak m)\simeq \bb Z_p(j)^\sub{syn}(A,\frak m).\label{eqn_cycles=differential2}\end{equation} These are remarkable equivalences. Indeed, $A$ has geometric weight-$j$ motivic cohomology (combine cases (1) and (3) of the theorem) so, resolving $A$ as in Remark \ref{remark_lke_as_cycles}, the left sides of (\ref{eqn_cycles=differential}) and (\ref{eqn_cycles=differential2}) admit presentations purely in terms of complexes of algebraic cycles. But the right sides are linear invariants ultimately built from differential forms. Isomorphisms between algebraic cycles and differential forms also appear in the theory of Chow groups with modulus \cite{Bloch2003, Rulling2007, RullingSaito2018}.

For example, if $A=k[x]/x^e$ where $k$ is a perfect field of characteristic $p$, then (\ref{eqn_cycles=differential2}) states that \[\bb Z(j)^\sub{mot}(A,\frak m)[1]\simeq \bb W_{ej}(k)/V^e\bb W_j(k)\] (using \cite[Theorem~1.1]{Sulyma2023} to describe the syntomic cohomology), thereby offering a cycle theoretic description of the group $\bb W_{ej}(k)/V^e\bb W_j(k)$.
\end{example}

\subsection{Zero cycles on surfaces}
In Theorem \ref{theorem_geometric_cohomology} we worked in a local context, but the main idea globalizes. Suppose that $X$ is a qcqs $\bb F$-scheme such that, for any $x\in X$, the local ring $\roi_{X,x}$ has geometric weight-$j$ motivic cohomology. Then, by checking on stalks, we see that $\bb Z(j)^\sub{mot}(X)$ is given by Zariski sheafifying $\tau^{\le j}\bb Z(j)^\sub{mot}$ on $X_\sub{zar}$; in other words, appealing to Theorem \ref{thm_lke_lej}, the canonical map \[(L_\sub{Zar}\bb Z(j)^\sub{lse})(X)\To \bb Z(j)^\sub{mot}(X)\] is an equivalence, where $\bb Z(j)^\sub{lse}$ is the left Kan extension of classical motivic cohomology from smooth $\bb F$-schemes as at the start of \S\ref{ss_cdh_local}. Similar conclusions holds with Zariski replaced by Nisnevich, if we had instead assumed that each henselian local ring $\roi_{X,x}^h$ had geometric weight-$j$ motivic cohomology. These arguments allow us to calculate the motivic cohomology of surfaces in low weights:

\begin{corollary}\label{corol_surfaces}
Let $X$ be a Noetherian $\bb F$-scheme of dimension $\le 2$ (e.g., a curve or surface, with arbitrarily bad singularities, over a field extension of $\bb F$). Then there are natural equivalences
\[\bb Z(j)^\sub{mot}(X)\simeq
\begin{cases}R\Gamma_\sub{cdh}(X,\bb Z) & j=0, \\
R\Gamma_\sub{Nis}(X,\bb G_m)[-1] & j=1,\\
(L_\sub{Nis}\tau^{\le j}\bb Z(j)^\sub{mot})(X) & j\ge2,
\end{cases}\]
and an isomorphism
\[H^4_\sub{mot}(X,\bb Z(2))\cong H^2_\sub{Nis}(X, \K_2).\] (The right side denotes Nisnevich cohomology with coefficients in the Nisnevich sheafification of $\K_2$.)
\end{corollary}
\begin{proof}
The description of the weight zero motivic cohomology does not depend on the hypotheses on $X$ and may be founded in Examples \ref{example_00} and \ref{example_0p}.

For weight one, we appeal to Corollary \ref{corollary_1} and note that $R\Gamma_\sub{Nis}(X,\bb G_m)$ and $\bb Z(1)^\sub{mot}(X)$ are supported in degrees $\le 2$ and $\le 3$ respectively; for the cohomology of $\bb G_m$ this is because $X$ has Krull dimension $\le 2$, and for the motivic cohomology we appeal to the Soul\'e--Weibel vanishing bound, Theorem~\ref{theorem_Weibel_vanishing}.

Now let $j\ge 2$ (in fact, the following argument equally works when $j=1$). Then the canonical map 
\begin{equation}\label{eq:tau-nis}
L_\sub{Nis}\tau^{\le j}\bb Z(j)^\sub{mot}\to\bb Z(j)^\sub{mot}
\end{equation} of Nisnevich sheaves is an equivalence at all points of $X_\sub{Nis}$ by Theorem \ref{theorem_geometric_cohomology}(5) (which, notably, uses the Soul\'e--Weibel vanishing bound), hence is an equivalence when evaluated on $X$.

Let us now prove the last statement. We write $\cal H^j(\bb Z(j)^\sub{mot}_X)$ for the sheafification on $X_\sub{Nis}$ of $X_\sub{Nis}\ni U\mapsto H^j_\sub{mot}(U,\bb Z(j)e)$. By the equivalence of~\eqref{eq:tau-nis} and the fact that $X$ has Nisnevich cohomological dimension $\le 2$, there is a natural edge map isomorphism \[H^{j+2}_\sub{mot}(X,\bb Z(j))\cong H^2_\sub{Nis}(X, \cal H^j(\bb Z(j)^\sub{mot}_X)).\] But the singular Nesterenko--Suslin Theorem~\ref{theorem_NS} defines a symbol isomorphism $\hat{K}_j^M\isoto \cal H^j(\bb Z(j)^\sub{mot}_X))$ where $\hat{K}_j^M$ is the Nisnevich sheaf of improved Milnor $K$-groups on $X$; in case $j=2$ we moreover have $\hat{K}_2^M\isoto K_2$ \cite[Proposition~10(3)]{Kerz2010}, completing the proof.
\end{proof}

We are very grateful to F.~Binda for help with the following proof:

\begin{theorem}\label{thm:lw-comparison}
Let $X$ be a reduced, equi-dimensional, quasi-projective surface over a field $k$; then there is a natural isomorphism \[H^4_\sub{mot}(X,\bb Z(2))\cong \mathrm{CH}_0(X),\] where $\mathrm{CH}_0(X)$ denotes the lci Chow group of zero cycles\footnote{If $k$ is infinite then this is isomorphic to the older Levine--Weibel Chow group of zero cycles $\mathrm{CH}_0^\sub{LW}(X)$.} of \cite{BindaKrishna2018}.
\end{theorem}
\begin{proof}
In light of the final isomorphism of Corollary \ref{corol_surfaces}, we must produce a natural isomorphism $\mathrm{CH}_0(X)\cong H^2_\sub{Nis}(X,\K_2)$. Such Bloch--Quillen formulae for singular surfaces are due to Levine \cite{Levine1985} and Binda--Krishna--Saito \cite{BindaKrishnaSaito2023}; since the precise form we need does not quite explicitly appear in the papers, we provide the extra details.

Firstly, citing from \cite[Lemma~3.4]{BindaKrishnaSaito2023} and the reference therein to \cite[Lemma~3.2]{GuptaKrishna2020}, there is a commutative diagram
\[\xymatrix{
H^2_\sub{Zar}(X, \K_2)\ar[r]^{\lambda_X'} & H^2_\sub{Nis}(X, \K_2)\ar@/^20mm/[dd]^{\sub{edge}}\\
H^2_\sub{Zar}(X, \K_2^M)\ar[u]\ar[r]^{\lambda_X} & H^2_\sub{Nis}(X, \K_2^M)\ar[d]^{\gamma_X}\ar[u]\ar[d]\\
\mathrm{CH}_0(X)\ar[u]^{\rho_X}\ar[r]_{\sub{cyc}_X} &\K_0(X)
}\]
where
\begin{itemize}
\item $\text{cyc}_X$ and $\rho_X$ are cycle class maps from the lci Chow group; 
\item $\lambda_X$ and $\lambda_X'$ are change of topology maps; 
\item the two vertical maps at the top are induced by the canonical map $\K_2^M\to \K_2$; 
\item the edge map is the edge map in the Nisnevich descent spectral sequence; and $\gamma_X$ is defined to make the curvy triangle commute. 
\end{itemize}

We will explain that the cycle class map $\xi_X:\mathrm{CH}_0(X)\to H^2_\sub{Nis}(X, \K_2)$, defined to be the composite from the bottom left to the top right of the diagram, is an isomorphism.

We first treat the case that $k$ is finite (or more generally perfect). According to \cite[Theorem~8.1]{BindaKrishnaSaito2023}, the cycle class maps $\text{cyc}_X$ is injective (this does not require the hypothesis on $k$); so $\xi_X$ is also injective. Furthermore, according to \cite[Theorem~2.5]{KatoSaito1986}, the Nisnevich cycle class map $\lambda_X\rho_X$ is surjective (this does require the hypothesis on $k$, as it means that the regular locus $X_\sub{reg}$ is ``nice'' in the terminology of [op.~cit.]). Finally note that $H^2_\sub{Nis}(X, \K_2^M)\to H^2_\sub{Nis}(X, \K_2)$ is surjective, because $X_\sub{Nis}$ has cohomological dimension $2$ and the map $ \K_2^M\to  \K_2$ is Nisnevich-locally surjective. The last two sentences show $\xi_X$ is surjective, completing the proof in this case.

Next we treat the case that $k$ is infinite. Then the cycle class map $\mathrm{CH}_0(X)\to H^2_\sub{Nis}(X, \K_2)$ (i.e., bottom left to top left of the diagram) is an isomorphism by \cite[Corollary~7.8]{BindaKrishnaSaito2023}. It remains to show that $\lambda_X'$ is an isomorphism (which does not require the hypothesis on $k$); this is well-known to experts but we could not find a reference. This isomorphism is proved by comparing the Zariski descent spectral sequence $E_2^{ij}=H^i_\sub{Zar}(X, \K_{-j})\Rightarrow \K_{-i-j}(X)$ to the analogous Nisnevich descent spectral sequence, as follows. Both spectral sequences are supported in columns $i=0,1,2$ since $X_\sub{Zar}$ and $X_\sub{Nis}$ have cohomological dimension $2$; therefore the only non-zero differentials $\partial$ are on the first page, from the $0^\sub{th}$ column to the $2^\sub{ed}$ column, and so the abutement filtrations on $\K_0(X)$ are described via a commutative diagram
\[\xymatrix{
&&&H^0_\sub{Zar}(X,\bb Z)&\\
H^0_\sub{Zar}(X,\bb G_m)\ar[r]^{\partial} & H^2_\sub{Zar}(X, \K_2)\ar[r]^{-\sub{edge}} & \K_0(X) \ar[r] & \K_0(X)/\text{edge}(H^2_\sub{Zar}(X, \K_2))\ar[r]\ar[u]&0\\
&&&H^1_\sub{Zar}(X,\bb G_m)\ar[u]&\\
&&&0\ar[u]&\\
}\]
and similarly replacing Zar by Nis everywhere. The Zariski diagram maps to the Nisnevich one, involving in particular the map $\lambda_X'$, and one sees from a diagram chase that $\lambda_X'$ being an isomorphism follows from the following standard facts:
\begin{enumerate}
\item $H^0_\sub{Zar}(X,\bb Z)\to H^0_\sub{Nis}(X,\bb Z)$ is injective (it is even an isomorphism);
\item $H^1_\sub{Zar}(X,\bb G_m)\to H^1_\sub{Nis}(X,\bb G_m)$ is injective (it is even an isomorphism);
\item the boundary maps $\partial$ in both the Zariski and Nisnevich diagrams is zero because $\bb G_m$ is representable by a one-dimensional scheme: more precisely, given $f\in H^0_\sub{Zar}(X,\bb G_m)$ (resp.~Nisnevich), let $X\to \Spec(k[t^{\pm1}])$ be the induced map; then the analogous boundary map in the Zariski (resp.~Nisnevich) descent spectral sequence for $\Spec(k[t^{\pm1}])$ is zero, simply because $H^2_\sub{Zar}(\Spec(k[t^{\pm1}]),\K_2)=0$ (resp.~Nisnevich) for dimensional reasons; so by functoriality we deduce $\partial(f)=0$, as desired.
\end{enumerate}
This completes the proof.
\end{proof}

\begin{remark}[Zero cycles]\label{rem:zcyc}
The argument at the end of the proof of Corollary \ref{corol_surfaces} shows, for any $j\ge0$ and any qcqs $\bb F$-scheme $X$ of Krull dimension $\le d$, that there is a natural map \[H^d_\sub{Nis}(X,\hat{\K}_j^M) \cong H^{j+d}(L_{\Nis}\tau^{\leq j}\Z(j)^{\mot}(X)) \To H^{j+d}_\sub{mot}(X,\bb Z(j)).\]  Taking $j=d=\text{dim}(X)$, we hope that $H^{2d}_\sub{mot}(X,\bb Z(d))$ provides a ``good'' group of zero cycles on $X$. This point of view will be explored further elsewhere.
\end{remark}

\appendix

\section{The $\cdh$-sheafification of an \'etale sheaf}\label{app:cdh}
In this appendix we prove a technical result about \'eh sheafification. It states that, under certain hypotheses, the cdh sheafification of an \'etale sheaf is still an \'etale sheaf and therefore  in fact an \'eh sheaf. It is a crucial input into controlling our mod-$p$ motivic cohomology in characteristic $p$.

We begin by recalling the definition:

\begin{definition}\label{def:eh} 
Letting $B$ be a commutative ring, the \emph{$\eh$ topology} on $\textrm{Sch}_B^\sub{qcqs}$ is the Grothdendieck topology generated by abstract blowup squares and \'etale covers. (As in \S\ref{ss_cdh_local}, our convention for abstract blowup squares is that the proper map and closed embedding are assumed to be finitely presented.)
\end{definition}

It is thus relatively formal that a presheaf $\cal F:\textrm{Sch}_B^\sub{qcqs,op}\to \textrm{Sp}$ is an \'eh sheaf if and only if it is both a cdh and \'etale sheaf; in particular, assuming $\cal F$ is an \'etale sheaf, then it is an \'eh sheaf if and only if it sends abstract blowup squares to cartesian squares. See \cite[Proposition 2.8]{BachmannElmantoMorrow} for further details of the argument and references.

\begin{remark}\label{rem:eh}
The $\eh$ topology is finer than the $\cdh$ topology but coarser than the $\h$ topology; the latter can be defined as the topology generated by abstract blowup squares and fppf covers \cite[\S2.2.1]{BachmannElmantoMorrow}. It is an insight of Geisser, when defining his arithmetic cohomology of finite type separated schemes over finite fields, that the $\eh$ topology is better suited to capturing mod-$p$ information than the $h$ topology \cite[Page 30, Remark]{Geisser2006}.

The points of the $\eh$ topology are given by strictly henselian valuation rings.
\end{remark}

We will need the following structural result about algebras over henselian valuation rings:

\begin{lemma}[Nagata's Hensel lemma]\label{lem:nagata}
Let $V$ be a henselian valuation ring and $A$ a finite \'etale $V$-algebra. Then $A$ is a finite product of henselian valuation rings.
\end{lemma}
\begin{proof}
We recall a modern proof which we learned from \cite[Corollary~2.15]{HuberKelly2018}. A ring $R$ is said to have weak dimension $\le 1$ if every submodule of every flat $R$-module is also flat over $R$; equivalently, $R_\frak p$ is a valuation ring for every prime ideal $\frak p\subset R$ \cite[Tag 092A]{Stacks}. The \'etale map $V\to A$ is in particular weakly \'etale (i.e., $V\to A$ and $A\otimes_VA\to A$ are both flat), so that $V$ having weak dimension $\le 1$ implies the same about $A$; this shows that $A_\frak p$ is a valuation ring for every prime ideal $\frak p\subset A$.

But $A$ is finite over a henselian local ring, so it is a finite product of henselian local rings \cite[Tag 04GH]{Stacks}, each of which is the localisation at $A$ at some prime ideal. These localisations are valuation rings by the previous paragraph, completing the proof.
\end{proof}

The following is the main theorem of the appendix:

\begin{theorem} \label{theorem:eh}
Let $B$ be a commutative ring and $\cal F$ a finitary \'etale sheaf on $\Sch_B^{\qcqs}$ valued in $\Sp$ or $\mathrm{D}(\bb Z)$; assume that its cdh sheafification $L_\sub{cdh}\cal F$ takes coconnective values. Then $L_{\cdh}\scr F$ is an \'eh sheaf and the canonical map $L_\sub{cdh}\cal F\to L_\sub{\'eh}\cal F$ is an equivalence.
\end{theorem}
\begin{proof}
It suffices to prove that $L_{\cdh}\scr F$ satisfies \'etale descent (indeed, we have already noted above that being an \'eh sheaf is equivalent to being a sheaf for both the cdh and \'etale topologies, and we will then have $L_{\cdh}\scr F \simeq L_\sub{\'eh}L_{\cdh} \scr F\simeq L_\sub{\'eh}\cal F$, where the final equivalence follows from the fact that the \'eh topology is finer than the cdh topology). Since $\cdh$ sheaves are in particular Nisnevich sheaves, this reduces to the special case of finite \'etale descent  by \cite[Theorem~B.6.4.1]{LurieSAG} \cite[Remark 4.34]{ClausenMathew2021}. So let $Y \to X$ be a finite \'etale cover of qcqs $B$-schemes, and let $C^{\bullet}_X(Y)$ be the corresponding \v{C}ech nerve; our goal is to prove that the map $L_{\cdh}\scr F(X) \rightarrow \lim_{\Delta} L_{\cdh}\scr F(C^{\bullet}_X(Y))$ in $\cal C$ is an equivalence. We define a presheaf
\[
\scr G: \Sch_X^\sub{qcqs,op} \rightarrow \cal C \qquad U \mapsto \mathrm{fib}( L_{\cdh}\scr F(U) \rightarrow \lim_{\Delta} L_{\cdh}\scr F(C^{\bullet}_X(Y)\times_X U)).
\]
and will show that $\scr G \simeq 0$; evaluating at $X$ will then complete the proof. Note that $C^{\bullet}_X(Y)\times_X U$ is the \v{C}ech nerve of the cover $Y\times_X U\to U$, and that $\cal G$ is a cdh sheaf since sheaves are closed under limits.

We note that $\cal G$ is finitary: indeed, $L_\sub{cdh}\cal F$ is finitary by \cite[Proposition 2.15(2)]{BachmannElmantoMorrow} and takes values in coconnective spectra/complexes, where  totalisations commute with filtered colimits (for a proof of this commutation in the generality of $\infty$-categories compactly generated by cotruncated objects, see \cite[Lemma~3.1.7]{ElmantoHoyoisIwasaKelly2021}). The latter citation also proves that $\cal G$ is hypercomplete, since it takes values in $1$-coconnective spectra/complexes.

To prove that $\cal G$ vanishes, it is therefore enough to show that $\cal G(V)=0$ for all henselian valuation rings $V$ with a map $\Spec(V)\to X$. Then each term $C_X^n(Y)\times_XV$ appearing in the \v{C}ech diagram is the spectrum of a finite \'etale $V$-algebra, which is therefore a disjoint union of finitely many spectra of henselian valuation rings by Lemma \ref{lem:nagata}. Since all these henselian valuation rings are points for the cdh topology, we see that in the definition of $\cal G(V)$ we may remove all occurrances of $L_\sub{cdh}$, i.e.,
\[
\scr G(V)\simeq \mathrm{fib}(\scr F(V) \rightarrow \lim_{\Delta} \scr F(C^{\bullet}_X(Y)\times_X V)).
\]
The arrow is an equivalence since $\cal F$ was assumed to be an \'etale sheaf, and so indeed $\cal G(V)\simeq 0$ as required to complete the proof. 
\end{proof}

\begin{remark}
Let $V$ be a finite rank valuation ring under $B$ and $V_\sub{w\'et}$ its category of weakly \'etale $V$-schemes. The previous proofs imply the canonical map \[\scr F|_{V_\sub{w\'et}}\To (L_{\cdh}\scr F)|_{V_\sub{w\'et}}\] of presheaves on $V_\sub{w\'et}$ is an equivalence.

We also record a warning: the order of sheafifications in Theorem \ref{theorem:eh} is crucial. The canonical map $L_\sub{\'et}L_\sub{cdh}\cal F\to L_\sub{\'eh}\cal F$ need not be an equivalence, because the \'etale sheafification of a cdh sheaf need not be a cdh sheaf.
\end{remark}

\begin{example}
For each $j\ge0$, the presheaves \[R\Gamma_\sub{cdh}(-,L\Omega^{< j}_{-/\bb Q}),\quad R\Gamma_\sub{cdh}(-,\hat{L\Omega}^{\ge j}_{-/\bb Q}):\Sch_\bb Q^\sub{qcqs,op}\to\mathrm{D}(\bb Z)\] defined in Remark \ref{rem_cdh_local_HKR} are \'etale sheaves.

Indeed, we have a fibre sequence of \'etale sheaves \[R\Gamma(-,L\Omega^{< j}_{-/\bb Q})\To R\Gamma_\sub{cdh}(-,\hat{L\Omega}_{-/\bb Q})\To R\Gamma_\sub{cdh}(-,\hat{L\Omega}_{-/\bb Q}^{\ge j}),\] where the middle term is already a cdh sheaf by Lemma \ref{lemma_cdh_descent_HP}. So it is enough to prove that the cdh sheafification of $R\Gamma(-,L\Omega^{< j}_{-/\bb Q})$ is an \'etale sheaf. But that follows from Theorem \ref{theorem:eh}, since its sheafification $R\Gamma_\sub{cdh}(-,L\Omega^{< j}_{-/\bb Q})$ agrees with $R\Gamma(-,\Omega^{< j}_{-/\bb Q})$ (as in the proof of Theorem \ref{thm:graded-pieces}(3)), which is coconnective.
\end{example}

\begin{example}
For each $j\ge0$, the presheaf \[L_\sub{cdh}\bb Z_p(j)^\sub{syn}:\Sch_{\bb F_p}^\sub{qcqs,op}\to\mathrm{D}(\bb Z)\] of Remark \ref{rem:cdh-local} is an \'eh sheaf. Indeed firstly, after inverting $p$, the proof of Lemma~\ref{lemma_Qpsyn} showed that $L_\sub{cdh}\bb Q_p(j)^\sub{syn}\simeq \bb Q_p(j)^\sub{syn}$ is a direct summand of $\TC[\tfrac{1}{p}]$; the latter has finite flat descent thanks to transfers, whence $L_\sub{cdh}\bb Q_p(j)^\sub{syn}$ also has finite flat descent and so is an h-sheaf \cite[Remark 2.11]{BachmannElmantoMorrow}, in particular an \'eh sheaf. Secondly, modulo $p$, it follows from Theorem~\ref{theorem:eh} that $L_\sub{cdh}\bb F_p(j)^\sub{syn}$ is an \'eh sheaf, as already recorded in Remark \ref{remark_eh}.
\end{example}

\section{Adams operators}\label{app:chw}
In this appendix we discuss Adams operators $\psi^m$ on various flavours of $K$-theory and negative cyclic homology, and establish a certain compatibility between them in the case of smooth schemes over a fixed base of characteristic zero (Corollary~\ref{corol:key-q}). This machinary is used, and the aforementioned compatibility generalised to qcqs $\bb Q$-schemes, in Theorem \ref{thm:adams}, where we construct Adams operators on $K$-theory which rationally split our motivic filtration in characteristic zero.

For the remainder of the appendix we fix an integer $m\in\bb Z\setminus\{0\}$ and discuss the Adams operator $\psi^m$; we do not worry about commutation of Adams operators associated to different choices of $m$, as it is not necessary for our applications.

\begin{remark}[Multiplication-by-$m^\star$]\label{rem:mult_by_m}
On associated graded objects, the Adams operator in each context will be multiplicatively homotopic to the map {\em multiplication-by-$m^\star$}. For a highly structured definition of this endomorphism in the context of $\bb E_\infty$-algebras in any $\bb Z$-linear presentably symmetric monoidal $\infty$-category, we refer the reader to Raksit \cite[Construction 6.4.9 \& Remark 6.4.10]{raksit-hkr}.

We will also refer to the multiplication-by-$m^\star$ endomorphism in the case of presheaves of complexes on qcqs schemes, which is not presentable; but there we simply define it as being induced by multiplication-by-$m^\star$ pointwise on the presheaf.

Meanwhile, in the case of K-theory (or $\KH$) equipped with some sort of motivic/slice filtration $\Fil$, we will already know in each context that $\gr^0\K$ admits a natural (and often unique) $\bb E_\infty$-$\bb Z$-algebra structure. This structure defines the multiplication-by-$m^\star$ endomorphism on $\gr^\star K$.
\end{remark}

\subsection{$K$-theory of smooth varieties via motivic homotopy theory}\label{ss_Adams_on_KH}
For any qcqs scheme $X$, Bachmann--Hopkins \cite[\S3.3]{BachmannHopkins2020} define an Adams operator $\psi^m$ as an endomorphism of the motivic ring spectrum $\KGL_X[\tfrac1m]\in\SH(X)$. By functoriality of the slice filtration, $\psi^m$ respects the slice filtration or rather, more precisely, it naturally upgrades to an endomorphism $\psi^m$ of $\Fil^\star_\sub{slice}\KGL_X[\tfrac1m]$ as an $\bb E_\infty$-algebra in $\Fil\SH(X)$. This Adams operator is moreover natural in $X$ as Bachmann--Hopkins' construction really defines an endomorphism of $\KGL[\tfrac1m]$ as an absolute motivic ring spectrum. Applying $\omega^\infty$ we obtain an induced Adams operator
\begin{equation}\psi^m:\Fil^\star_\bb A\KH[\tfrac1m]\To \Fil^\star_\bb A\KH[\tfrac1m]\label{eqn:BH_adams}\end{equation}
of presheaves of filtered $\bb E_\infty$-algebras on qcqs schemes, where $\Fil^\star_\bb A$ refers to the $\bb A^1$-motivic filtration on $\KH$ as in Theorem \ref{ss_cdh_local}(1).

Ignoring multiplicative structure, Bachmann--Hopkins' Adams operator on $\KGL_X[\tfrac1m]$ is homotopic to the one defined by Riou \cite[Definition 5.3.2]{Riou2007} \cite[Proposition 3.12]{BachmannHopkins2020} (although Riou often restricts attention to regular Noetherian separated schemes, he also explicitly observes that his Adams operators are pulled back from $\SH(\bb Z)$; so one can pull them back to $\SH(X)$ even for the arbitrary qcqs scheme $X$). It follows that if $X$ is regular Noetherian and has an ample family of line bundles (so that Thomason--Trobaugh and Quillen K-theory coincide), then the Adams operator on $\K_n(X)[\tfrac1m]$ induced by \eqref{eqn:BH_adams} agrees with that of Hiller \cite{Hiller1981}, Kratzer \cite{Kratzer1980}, and Soul\'e \cite{Soule1985}; a proof of this agreement may be found in our joint work with Bachmann \cite[Theorem 4.48(4)]{BachmannElmantoMorrow} (where we worked rationally, but it would have been enough to invert $m$).

In particular, fixing a field $k$ and restricting attention to smooth $k$-schemes we obtain an endomorphism $\psi^m$ of $K$-theory with its motivic filtration $\Fil^\star_\bb A\K[\tfrac1m]$ as a presheaf of filtered $\bb E_\infty$-algebras on smooth $k$-schemes. We could not find a proof of the following result in the literature (since we care about multiplicative structure and only invert $m$ rather than rationalising):

\begin{lemma}\label{lemma_adams_on_smooth}
The induced endomorphism $\psi^m$ of $\gr^\star_\bb A\K[\tfrac1m]\simeq \bb Z(\star)^\bb A[2\star]$, as a presheaf of $\bb E_\infty$-algebras in graded complexes on smooth $k$-schemes, is homotopic to multiplication-by-$m^\star$.
\end{lemma}
\begin{proof}
In fact we prove the stronger result that the endomorphism $\psi^m$ of $s^*\KGL_k[\tfrac1m]:=\gr^\star_\sub{slice}\KGL_k[\tfrac1m]$, as a graded $\bb E_\infty$-algebra in $\SH(k)$, is uniquely homotopic to multiplication-by-$m^\star$.\footnote{Having the final paragraph of Remark \ref{rem:mult_by_m} in mind, the rank map from $K$-theory equips $\bb Z(0)^\bb A$ with the structure of a presheaf on $\Sm_k$ valued in $\bb E_\infty$-algebras in $\mathrm{D}(\bb Z)$ \cite[Construction 4.35 \& Remark 4.37]{BachmannElmantoMorrow}; applying the composition \[\Shv_{\sub{Nis},\bb A^1}(\Sm_k,\mathrm{D}(\bb Z))\to \Shv_{\sub{Nis},\bb A^1}(\Sm_k,\Sp)\xto{\sigma^\infty}\SH(k)\] then equips $\mathrm{Mod}_{s^0\KGL_k}(\SH(k))$ with a $\bb Z$-linear structure, so that the multiplication-by-$m^\star$ endomorphism of $s^\star\KGL_k$ is well-defined.} The desired result is then obtained by applying $\omega^\infty$. The idea of argument is to reduce to the analogous result rationally without multiplicative structure, where it is well-known.

We first use the result of Levine and Voevodsky \cite{Levine2006, Voevodsky2004} that the unit map $1_k\to\KGL_k$ in $\SH(k)$ induces an equivalence $s^0(1_k)\quis s^0(\KGL_k)$. Therefore $\bb E_\infty$-algebra endomorphisms of $s^0\KGL_k[\tfrac1m]$ are the same as $\bb E_\infty$-algebra maps $1_k\to s^0\KGL_k[\tfrac1m]$, of which there is only one since $1_k$ is the unit; more precisely, this shows that the space of endomorphisms $\mathrm{End}_{\CAlg(\SH(k))}(s^0\KGL_k[\tfrac1m])$ is contractible. So $\psi^m$ necessarily acts as the identity on $s^0\KGL_k[\tfrac1m]$ and therefore its action on $s^\star\KGL_k[\tfrac1m]$ is one of graded $\bb E_\infty$-$s^0\KGL_k$-algebras. 

We now compute and compare the spaces of endomorphisms of $s^\star\KGL_k[\tfrac1m]$, either as a graded $\bb E_\infty$-$s^0\KGL_k$-algebra or as a graded $s^0\KGL_k$-module:
\begin{equation}
\mathrm{End}_{\CAlg(\Gr\cal M)}(s^\star\KGL_k[\tfrac1m])\To \mathrm{End}_{\Gr\cal M}(s^\star\KGL_k[\tfrac1m])=\prod_{j\in\bb Z}\mathrm{End}_\cal M(s^j\KGL[\tfrac1m])
\label{eqn_end_compare}\end{equation}
where $\cal M$ denotes $s^0\KGL$-modules in $\SH(k)$. We claim that all the spaces of endomorphisms in \eqref{eqn_end_compare} are discrete and that the arrow is injective. Applying Lemma \ref{lem:algebra_vs_linear_maps} below (with $\cal C=\cal M$ and $A=B=s^\star\KGL_k[\tfrac1m]$) it is enough to show that the spaces $\Map_{\Gr\cal M}((s^\star\KGL_k[\tfrac1m])^{\otimes m},s^\star\KGL_k[\tfrac1m])$ are discrete for all $m\ge0$; expanding the $^{\otimes m}$, this space of maps can be rewritten as \[\prod_{n\in\bb Z}\prod_{\substack{i_1,\dots,i_m\in\bb Z \\ i_1+\cdots+i_m=n}}\Map_{\cal M}(s^{i_1}\KGL_k[\tfrac1m] \otimes \cdots \otimes s^{i_m}\KGL_k[\tfrac1m], s^{n}\KGL_k[\tfrac1m]).\] But by Bott periodicity and invertibility of the Tate motive, the mapping spaces in the previous line are all equivalent to $\mathrm{End}_\cal M(s^0\KGL_k[\tfrac1k])$, which in turn is equivalent to $\Map_{\SH(k)}(1_k,s^0\KGL_k[\tfrac1m])$, which is the underlying space of the spectrum $\map_{\SH(k)}(1_k,s^0\KGL_k[\tfrac1m])\simeq \bb Z(0)^\bb A(k)[\tfrac1m]$, which equals\footnote{The fact that weight-$0$ motivic cohomology of smooth $k$-schemes, defined in terms of $s^0\KGL_k$, is given by $R\Gamma_\sub{Zar}(-,\bb Z)$ is far from formal: any proof probably uses Levine's equivalence $s^0\KGL_k\simeq s^01_k$ and his description of the latter in terms of Bloch's cycle complex (and then it is formal that $z^0(\Spec k,\bullet)=\bb Z$).} $R\Gamma_\sub{Zar}(k,\bb Z)[\tfrac1m]=\bb Z[\tfrac1m]$. This completes the proof of the claim that both sides of \eqref{eqn_end_compare} are discrete and the map injective.

We can repeat the argument of the previous paragraph rationally instead of only inverting $m$. Since the canonical map $\bb Z[\tfrac1m]\to\bb Q$ is injective, the conclusion is that the map \[\mathrm{End}_{\CAlg(\Gr\cal M)}(s^\star\KGL_k[\tfrac1m])\To \mathrm{End}_{\Gr\cal M}(s^\star(\KGL_k)_\bb Q),\] by forgetting the algebra structure and rationalising, is an injection of discrete spaces.

In conclusion, to complete the proof it is enough to show, for each $j\in\bb Z$, that $\psi^m$ acts on $s^j(\KGL_k)_\bb Q$ as multiplication by $m^j$ for each $j\in\bb Z$. But then $s^\star(\KGL_k)_\bb Q$ is the free graded Laurent $\bb E_\infty$-$s^0(\KGL_k)_\bb Q$-algebra on the Tate motive in degree $1$, and $\psi^m$ acts on the Tate motive as multiplication by $m$ by definition.
\end{proof}

The following technical lemma (which is surely well-known to experts but which we could not find stated in the literature) was used and will be required again in the proof of Proposition \ref{prop:key-q}:

\begin{lemma}\label{lem:algebra_vs_linear_maps}
Let $\cal C$ be a presentably symmetric monoidal $\infty$-category, and let $A,B$ be $\bb E_\infty$-algebras in $\cal C$. Assume that the spaces of maps $\Map_{\cal C}(A^{\otimes m},B)$ are discrete for all $m\ge0$. Then the space of $\bb E_\infty$-algebra maps $\Map_{\CAlg(\cal C)}(A,B)$ is discrete and the forgetful map $\Map_{\CAlg(\cal C)}(A,B)\to \Map_{\cal C}(A,B)$ is injective.
\end{lemma}
\begin{proof}
The adjunction $F:\cal C\leftrightarrows\CAlg(\cal C):U $ induces a monad $FU:\CAlg(\cal C)\to \CAlg(\cal C)$ and an augmented simplicial object $(FU)^{\bullet+1}(A)\to A$ in $\CAlg(\cal C)$. This is a colimit diagram because $U$ preserves geometric realisations \cite[Corollary 3.2.3.2]{LurieHA}, is conservative, and the augmented simplicial object $U(FU)^{\bullet+1}(A)\to U(A)$ is split, so a colimit diagram. Taking mapping spaces therefore yields an equivalence \begin{equation}\Map_{\CAlg(\cal C)}(A,B)\quis\lim_{n\in\Delta}\Map_{\CAlg(\cal C)}((FU)^{n+1}(A),B).\label{eqn:mapping_space}\end{equation} We claim by induction on $n\ge0$ that the mapping spaces on the right side are all discrete; in fact we claim more, namely that $\Map_{\cal C}((U(FU)^{n}(A))^{\otimes m},U(B))$ is discrete for all $m\ge1$ and $n\ge 0$ (the case $m=1$ being the original weaker claim).

The case $n=0$ of the stronger claim is precisely our hypothesis, so it is enough to prove the following inductive step: given an $\bb E_\infty$-algebra $C\in \cal C$, if the spaces $\Map_{\cal C}(U(C)^{\otimes m},U(B))$ are discrete for all $m\ge0$ then the spaces $\Map_{\cal C}((UFU(C))^{\otimes m},U(B))$ are also discrete for all $m\ge0$. To prove this claim we first recall that, for any $X\in\cal C$, there is a natural equivalence $UF(X)\simeq \bigoplus_{i\ge0}X^{\otimes i}_{h\Sigma_i}$ \cite[Example 3.1.3.14]{LurieHA}. Applied to $X=U(C)$ we see that $(UFU(C))^{\otimes m}\simeq \bigoplus_{i_1,\dots,i_m\ge0}U(C)^{\otimes i_1}_{h\Sigma_{i_1}}\otimes\cdots\otimes U(C)^{\otimes i_m}_{h\Sigma_{i_m}}$, and then applying $\Map_\cal C(-,B)$ yields \[\Map_{\cal C}((UFU(C))^{\otimes m},B)\simeq\prod_{i_1,\dots,i_m\ge0}\Map_\cal C(U(C)^{\otimes i_1}\otimes\cdots\otimes U(C)^{\otimes i_m},B)^{h\Sigma_{i_1}\times\cdots\times h\Sigma_{i_m}}\] (here we have used the assumption that the tensor product on $\cal C$ commutes with colimits in each variable). But the mapping spaces on the right side are all discrete by assumption, whence the right side is discrete since limits of discrete spaces are discrete. This completes the proof of the inductive step, and so we deduce that all the mapping spaces on the right side of \eqref{eqn:mapping_space} are indeed discrete.

But a cosimplicial limit of discrete spaces is again discrete and just computed by the equaliser of sets at the beginning of the diagram; this shows that the forgetful map \[\Map_{\CAlg(\cal C)}(A,B)\To \Map_{\CAlg(\cal C)}(FU(A),B)=\Map_{\cal C}(U(A),U(B))\] is an inclusion of discrete spaces, as desired.
\end{proof}

\begin{remark}[General qcqs schemes]
In light of our results with Bachmann, the proof of Lemma \ref{lemma_adams_on_smooth} works with minor modifications for any qcqs scheme $X$ in place of $k$. The key inputs are that $s^0(1_X)\quis s^0\KGL_X$ in $\SH(X)$ \cite[Corollary 9.7]{BachmannElmantoMorrow} and that $\bb Z(0)^\bb A(X)$ is coconnective and torsion-free in degree $0$, which follow from the equivalence $R\Gamma_\sub{cdh}(X,\bb Z)\quis \bb Z(0)^\bb A(X)$ \cite[Theorem 9.3(2) and Corollary 9.12(1)]{BachmannElmantoMorrow}.

In particular it follows that, after passing to the associated presheaf of graded $\bb E_\infty$-algebras in complexes on qcqs schemes, the Adams operator \eqref{eqn:BH_adams} becomes homotopic to multiplication-by-$m^\star$.
\end{remark}

\subsection{Negative cyclic homology via Raksit's theorem}\label{ss_Raksit}
We now turn to Adams operators on negative cyclic homology. Although there exist classical definitions of these operators on the level of complexes (see Remark \ref{rem:discuss-adams}  for further discussion) we use Raksit's thesis \cite{raksit-hkr} as our main reference. It constructs functorial multiplicative Adams operators on negative cyclic homology at the level of filtered complexes. In this subsection we fix a discrete\footnote{With appropriate, minor, modifications the discussion of this section also applies to base animated rings as well, and even base derived schemes.} base commutative ring $k$ in which the integer $m$ is invertible.

For any qcqs $k$-scheme $X$, let $\HH(X/k)^{[n]}$ be the $S^1$-equivariant $\mathbb{E}_{\infty}$-algebra spectrum\footnote{The Hochschild complex of discrete commutative rings or, more genereally, animated rings, have a natural structure of a \emph{derived commutative algebra} as in \cite[Def.~4.2.22]{raksit-hkr} which we can forget to a $\mathbb{E}_{\infty}$-algebra. We will not make use of this richer structure explicitly, but it is used systematically in \cite{raksit-hkr} to formulate a universal property of Hochschild homology with the HKR filtration.} obtained by restricting the $S^1$-action on $\HH(X/k)$ along the $m$-power map $[m]:S^1 \rightarrow S^1$; there is a natural $\mathbb{E}_{\infty}$-algebra map $\psi^m:\HH(X/k) \rightarrow \HH(X/k)^{[m]}$ arising from the universal property of Hochschild homology \cite[Construction 6.4.3]{raksit-hkr}, and it promotes uniquely to a multiplicative, $S^1$-equivariant, filtered map
\[
\psi^n: \mathrm{Fil}^{\star}_\sub{HKR}\HH(X/k) \rightarrow \mathrm{Fil}^{\star}_\sub{HKR}\HH(X/k)^{[n]}. 
\]
by \cite[Proposition 6.4.4]{raksit-hkr}.
Passing to fixed points and applying \cite[Lemmas 6.4.5-6.4.6]{raksit-hkr}, as in \cite[Construction~6.4.7]{raksit-hkr}, yields the desired Adams operator
\begin{equation}
\psi^m: \HC^-(X/k) \To \HC^-(X/k)
\label{eqn_adams_on_HC-}\end{equation}
on negative cyclic homology, which similarly upgrades to a natural filtered multiplicative map
\begin{equation}\label{eq:filtered-hc-}
\psi^m: \mathrm{Fil}^{\star}_\sub{HKR}\HC^-(X/k) \To \mathrm{Fil}^{\star}_\sub{HKR}\HC^-(X/k)
\end{equation}
by \cite[Construction~6.4.8]{raksit-hkr}. The induced action of \eqref{eq:filtered-hc-} on $\gr^{\star}_\sub{HRK}\HC^-(X/k)$ is naturally homotopic to the map multiplication-by-$m^\star$ by \cite[Proposition~6.4.12]{raksit-hkr}.

This completes the recollection of what we need about Adams operators on Hochschild and negative cyclic homology as presheaves of $\bb E_\infty$-algebras on qcqs $k$-schemes, either with or without filtration.

\begin{remark}[Filtered circle]
A key tool behind Raksit's construction of filtered multiplicative Adams operators is his \emph{filtered circle} $S^1_{\mathrm{fil}}$, introduced in \cite[Section 6]{raksit-hkr}. The classical Adams operators are defined (Zariski locally) by tensoring the $S^1$-equivariant map $S^1 \rightarrow (S^1)^{[m]}, z \mapsto z^m$ with a commutative ring $R$, within the $\infty$-category of $\bb E_{\infty}$-$\bb Z$-algebras; this ultimately relies on the universal property of the Hochschild homology of a commutative ring as the free $S^1$-equivariant $\bb E_{\infty}$-$\bb Z$-algebra under $R$. Raksit refines this universal property by showing in \cite[Section 6.2]{raksit-hkr} that the HKR-filtered Hochschild homology is the initial $S^1_{\mathrm{fil}}$-equivariant filtered $\bb E_{\infty}$-ring under $R$. We also refer to the alternate approach to the same results by Moulinos--Robalo--Toen \cite{mrt-hkr}, who employ the language of derived algebraic geometry.
\end{remark}

\begin{remark}[Comparison with the operators of Loday \cite{Loday1992} and McCarthy \cite{McCarthy1993}] \label{rem:discuss-adams} 
We take this opportunity to clarify conflicting historical terminology surrounding operators on Hochschild homology and its variants.

Loday defined complex-level ``lambda'' and ``Adams'' operators on Hochschild homology and some of its variants, sometimes requiring that the base ring contained $\bb Q$; see \cite{Loday1989, Loday1992} and \cite[\S9]{Weibel1994} for his approach, often using the Eulerian idempotents. His operators are related by $\psi_{\sub{Loday}}^m:= (-1)^{m-1}m\lambda^m_\sub{Loday}$. Already at the level of Hochschild homology groups, Loday's so-called Adams operator is not multiplicative.

McCarthy \cite{McCarthy1993} clarified the necessary structures underlying Loday's approach and wrote down operators $\Phi^m$ (McCarthy's notation) on Hochschild homology over general base rings as a map of mixed complexes $\HH\to\HH^{[m]}$. His operators are related to Loday's via $\Phi^m=(-1)^{m-1}\lambda^m_\sub{Loday}$ and agree up to homotopy with Raksit's $\psi^m$ after forgetting any mutliplicative structure (indeed, Raksit's approach is a modern, highly structured, and filtered upgrade of McCarthy's).

In other words, Loday's ``lambda operators'' are in fact, up to sign, the modern Adams operators, which McCarthy denoted by $\Psi^m$; while Loday's so-called Adams operators are in fact a multiple of $m$ too much. Unfortunately Loday's $\psi^m_\sub{Loday}$ are often used in the older literature to define a ``Hodge decomposition'' of negative cyclic homology, which therefore differs from the modern Adams decomposition by a $\pm1$ shift in the indexing. The reader should particularly bear this in mind when looking at previous work comparing operators on $K$-theory and negative cyclic homology (see Remark \ref{rem:chw}).
\end{remark}

\subsection{$K$-theory of smooth $X$-schemes via the Annala-Iwasa theorem}\label{sec:ai-thm}
In this subsection we offer a construction of the Adams operator $\psi^m$ on $m$-inverted $K$-theory on $\Sm_X$, for any qcqs scheme $X$, using the technology of Annala--Iwasa \cite{AnnalaIwasa2023}.

\subsubsection{Spherical group rings-units adjunction}\label{sec:adj}
First we review the formalism of \cite{AndoBlumbergGepnerHopkinsRezk2014} regarding units of ring spectra.

To orient the reader who might not be topologically-minded, we review this adjunction in the context of classical algebra. To a monoid $M$ we can associate two groups: either its group completion $M^{\mathrm{gp}}$ or its subset of grouplike elements which we denote by $\mathrm{GL}_1(M)$ to suggest ``one-by-one invertible matrices.'' The formation of the former is the left adjoint of the inclusion of groups into monoids, while the latter is the right adjoint. The unit and counit of the respective adjunctions are recorded as the maps
\[
\mathrm{GL}_1(M) \xrightarrow{\rm include} M \xrightarrow{\rm canonical} M^{\mathrm{gp}}.
\]
For a commutative ring $R$, we denote by $\mathrm{GL}_1(R)=R^\times$ the group of grouplike elements of its multiplicative monoid, i.e., its group of units. This functor participates in an adjunction
\begin{equation}
\mathbb{Z}[-]: \mathrm{Ab} \rightleftarrows \mathrm{CAlg}: \mathrm{GL}_1.
\label{eqn_units_adjunction_honest}
\end{equation}
In particular the functor $\mathrm{GL}_1$ is representable, namely $\mathrm{GL}_1(R)=\Hom_{\mathrm{CAlg}}(\mathbb{Z}[t^{\pm 1}], R)$, and there is a natural map of commutative rings $\bb Z[\mathrm{GL}_1(R)] \rightarrow R$ which is the map from the ``free commutative ring on the units of $R$ back to itself'', produced as the counit of the adjunction. 

In \cite{AndoBlumbergGepnerHopkinsRezk2014}, a spectral enhancement of this adjunction is studied. To set up notation, we need the usual adjunction between spaces and spectra $\Sigma_+^{\infty}:\Spc \rightleftarrows \Spt: \Omega_{-}^{\infty}$, where we have decided to adopt the unconventional notation ``$\Omega_{-}^{\infty}$'' to indicate that we have forgotten the basepoint implicit in the infinite loop space of a spectrum; this notation should help avoid confusion later. The functor $\Sigma_+^{\infty}$ is strong monoidal (for the Cartesian product on the domain and the tensor product on the target), whence the right adjoint $\Omega_{-}^{\infty}$ preserves $\bb E_{\infty}$-algebras and there is therefore an induced adjunction $\Sigma_+^{\infty}:\mathrm{CMon}_{\bb E_{\infty}}(\Spc) \rightleftarrows \CAlg(\Spt): \Omega^{\infty}.$ On the other hand, by presentability reasons, the inclusion of grouplike $\bb E_\infty$-monoids $\iota:\mathrm{CMon}_{\bb E_{\infty}}^\sub{gp}(\Spc) \hookrightarrow \mathrm{CMon}_{\bb E_{\infty}}(\Spc)$ admits both left and right adjoints; we denote the latter by $\mathrm{GL}_1$. Composing yields an adjunction
\[
\Sigma^{\infty}_+ \circ \iota: \mathrm{CMon}^\sub{gp}_{\bb E_{\infty}}(\Spc) \rightleftarrows \CAlg(\Spt): \mathrm{GL}_1 \circ \Omega^{\infty}.
\]
Rewriting the left side via the equivalence $\Omega^{\infty}:\Spt_{\geq 0} \xrightarrow{\simeq} \mathrm{CMon}_{\bb E_{\infty}}^\sub{gp}(\Spc)$ we finally obtain the desired adjunction which provides a spectral version of \eqref{eqn_units_adjunction_honest}:
\begin{equation}
\mathbb{S}[-]:=\Sigma^{\infty}_+ \circ\iota\circ\Omega^{\infty}: \Spt_{\geq 0} \rightleftarrows \CAlg(\Spt): \mathrm{gl}_1.
\label{eqn_units_adjunction}
\end{equation} 
Here the notation $\mathbb{S}[-]$ reflects the fact that $\Sigma^{\infty}_+ \Omega^{\infty}$ is a spherical version of the group ring construction. In particular, the functor $\mathrm{gl}_1$ is also representable
\begin{equation}\label{eq:repr-gl1}
\mathrm{gl}_1(R) \simeq \Map_{\CAlg(\Spt)}(\mathbb{S}[ (\coprod_{n \in \mathbb{N}} \mathrm{B}\Sigma_n)^{\mathrm{gp}} ], R) \simeq \Map_{\CAlg(\Spt)}(\mathbb{S}\{ t^{\pm 1} \}, R),
\end{equation}
where $\mathbb{S}\{ t^{\pm 1} \}$ is the free $\mathbb{E}_{\infty}$-algebra on a single invertible variable $t$,\footnote{This is different from the spherical group ring on $\mathbb{Z}$, denoted by $\mathbb{S}[t^{\pm 1}]$ or $\mathbb{S}[\mathbb{Z}]$ whose $\Spec$ in the sense of spectral algebraic geometry is the ``flat'' version of the multiplicative group.} and just as in classical situation there is a natural $\bb E_{\infty}$-algebra map
\begin{equation}\label{eq:map}
\bb S[\mathrm{gl}_1(R)] \To R
\end{equation}
from the spherical group ring on the units back to the ring spectrum.

\begin{remark}[Homotopy groups of $\mathrm{gl}_1$]\label{rem:hpty} If $R$ is $\bb E_{\infty}$-ring spectra, then there are isomorphisms of homotopy groups $\pi_{n}\mathrm{gl}_1(R) \cong \pi_n(R)$ for each $n \geq 1$; see for example \cite[\S 1.5]{Rezk2006}. In general there is no map of spectra $\mathrm{gl}_1(R) \rightarrow R$ witnessing these isomorphisms.
\end{remark}

Tensoring the adjunction \eqref{eqn_units_adjunction} with a presentable $\infty$-category $\scr C$ obtains an adjunction
\[
\mathbb{S}_{\mathcal{C}}[-]: \scr C \otimes \Spt_{\geq 0} \rightleftarrows \scr C \otimes \CAlg(\Spt): \mathrm{gl}^{\mathcal{C}}_{1}.
\]
For us the relevant presentable $\infty$-category will be $\scr C:= \Shv_{\Zar}(\Sm_X, \Spt)$, for some qcqs scheme $X$; in this case the adjunction will be denoted by
\begin{equation}\label{eq:gl1}
\mathbb{S}_X[-]:  \Shv_{\Zar}(\Sm_X, \Spt)_{\geq 0} \rightleftarrows \Shv_{\Zar}(\Sm_X, \CAlg(\Spt)):\mathrm{gl}^X_1,
\end{equation}
where the left side denotes connective Zariski sheaves of spectra (see Remark \ref{rem:spherical} for clarification).

\begin{remark}[The spherical group ring on a connective sheaf of spectra]\label{rem:spherical} A connective Zariski sheaf of spectra, i.e. an object of $\Shv_{\Zar}(\Sm_X, \Spt)_{\geq 0}$, is not necessarily a sheaf of spectra whose sections are connective; rather, it is a sheaf of spectra in the image of the symmetric monoidal localisation $L_{\Zar}: \PShv(\Sm_X, \Spt_{\geq 0}) \rightarrow \Shv_\sub{Zar}(\Sm_X, \Spt)$ (see also the proof of Proposition~\ref{prop:compat} where this category is constructed as the non-negative part of a $t$-structure). Given $F \in \Shv_{\Zar}(\Sm_X, \Spt)_{\geq 0}$, the sheaf $\mathbb{S}_X[F]$ of $\bb E_\infty$-algebras is the Zariski sheafification of $U\mapsto\bb S[\tau_{\ge0}F(U)])$.
\end{remark}

\subsubsection{The Annala--Iwasa theorem}
We now turn to the work of Annala--Iwasa \cite{AnnalaIwasa2023} and use it to construct Adams operators on $m$-inverted $K$-theory over a fixed qcqs base scheme $X$.

Following the notation of \cite[\S 3.1.1]{BachmannElmantoMorrow} (rather than $\cite{AnnalaIwasa2023}$, where Zariski sheafifications are implicit in certain stabilisation functors), we denote by $\Pic: \Sch^\sub{qcqs,op} \rightarrow \Spt$ the Picard stack valued in connective spectra, and by $L_\sub{Zar}\Pic=R\Gamma_\sub{Zar}(-,\bb G_m)[1]$ its Zariski sheafification, which is a connective Zariski sheaf of spectra. Restricting $L_\sub{Zar}\Pic$ to the category $\Sm_X$ obtains a connective Zariski sheaf of spectra $L_\sub{Zar}\Pic \in \Shv_{\Zar}(\Sm_X, \Sp)_{\geq 0}$, and then applying the left adjoint of~\eqref{eq:gl1} defines $\mathbb{S}_X[L_\sub{Zar}\Pic] \in \Shv_{\Zar}(\Sm_X, \CAlg(\Spt))$. To both simplify notation and closer match that of \cite{AnnalaIwasa2023}, we will allow ourselves to write $\mathbb{S}_X[\Pic]$ in place of $\mathbb{S}_X[L_\sub{Zar}\Pic]$.

Unwinding definitions and applying Remark~\ref{rem:spherical}, we see that $\mathbb{S}_X[\Pic]$ is simply the Zariski sheafification of the presheaf of $\bb E_{\infty}$-algebras $Y\mapsto\mathbb{S}[\scr O(Y)^{\times}[1]]$ on $\Sm_X$.

The \emph{Bott element} is the map
\begin{equation}\label{eq:bott}
\beta:= 1 - [\scr O(-1)]: \Sigma^{\infty}_+ \bb P_X^1 \To \bb S_X[\Pic],
\end{equation}
in $\PShv(\Sm_X, \Spt)$, defined more precisely as follows. Via the canonical map of sets $\textrm{Pic}(\bb P_X^1)\to \pi_0\bb S[\Pic(\bb P_X^1)]$, any line bundle $L$ on $\bb P_X^1$ defines an element of $\pi_0\bb S[\Pic(\bb P_X^1)]$, which is uniquely classified by a map of presheaves from $\Sigma_+^\infty\bb P_X^1$ to the presheaf $U\mapsto \bb S[\Pic(U)]$; sheafifying the target then defines a map of sheaves $\Sigma_+^\infty\bb P_X^1\to \bb S_X[\Pic]$, which one denotes by $[L]$. Applying this construction to the trivial line bundle $L=\roi_{\bb P_X^1}$ yields $1:\Sigma_+^\infty\bb P_X^1\to\Sigma_+^\infty X=\bb S\xto{\sub{unit}} \bb S_X[\Pic]$; applying it $L=\roi(-1)$ defines $[\roi(-1)]:\Sigma_+^\infty\bb P_X^1\to \bb S_X[\Pic]$. Finally $\beta$ is by definition the difference of these two maps.

Since $\bb S_X[\Pic]$ is a Zariski sheaf, and the restrictions of $1$ and $[\roi(-1)]$ along the map $\Sigma^{\infty}_+X \xrightarrow{\infty} \Sigma^{\infty}_+ \bb P_X^1$ are homotopic, $\beta$ induces a map \[\beta: L_{\Zar}\Sigma^{\infty}(\bb P_X^1, \infty) \To \bb S_X[\Pic]\]  of sheaves of spectra on $\Sm_X$, for which we abusively adopt the same notation.

\begin{definition}\label{def:conditions} Let $X$ be qcqs scheme and $E \in \mathrm{Mod}_{\bb S_X[\Pic]}(\Shv_{\Zar}(\Sm_X, \Spt))$. 

\begin{enumerate}
\item We say that $E$ is \emph{$\bb P^1$-periodic} \cite[Definition 5.2.2]{AnnalaIwasa2023} if the map of Zariski sheaves of spectra
\[
\beta_E:E \To \mathrm{fib}(E(\bb P^1_{\ph})\xto{\infty^*}E)
\]
adjoint to the map of Zariski sheaves
\begin{equation}\label{eq:adjoint-bott}
L_{\Zar}\Sigma^{\infty}(\bb P_X^1, \infty) \otimes E \xrightarrow{\beta \otimes \id_E} \bb S_X[\Pic] \otimes E \xrightarrow{\rm act} E
\end{equation}
is an equivalence.
\item We say that $E$ satifies \emph{elementary blowup excision} \cite[Definition 3.3.1]{AnnalaIwasa2023} if, for each $Y \in \Sm_X$, the sheaf $E$ converts the square
\begin{equation}\label{eq:blowup}
\begin{tikzcd}
\bb P^{n-1}_Y\ar{r} \ar{d}& \mathrm{Bl}_0(\bb A^n_Y) \ar{d}\\
Y \ar{r}{0} & \bb A^n_Y\\
\end{tikzcd}
\end{equation}
to a cartesian square of spectra.
\item We say that $E$ satisfies the \emph{projective bundle formula} \cite[\S 5.2.8]{AnnalaIwasa2023} if the map of Zariski sheaves of spectra
\[
\sum_{i=1}^n \beta^i: \bigoplus^{n}_{i=1} E \rightarrow E(\bb P^n_{\ph})
\]
is an equivalence for all $n \geq 1$.
\end{enumerate}
We set
\begin{equation}\label{eq:pbf-picmod}
\mathrm{Mod}_{\bb S_X[\Pic]}(\Shv_{\Zar}(\Sm_X, \Spt))_\sub{pbf} \subset \mathrm{Mod}_{\bb S_X[\Pic]}(\Shv_{\Zar}(\Sm_X, \Spt))
\end{equation}
to be the subcategory of those $E$ which satisfy the projective bundle formula.
\end{definition}

\begin{remark}[Pbf-localisation] \label{rem:pbf}
For presentability reasons, the inclusion~\eqref{eq:pbf-picmod} admits a left adjoint and, following \cite[\S 5.2.8]{AnnalaIwasa2023}, we will denote the localisation endofunctor by 
\[
L_\sub{pbf}: \mathrm{Mod}_{\bb S_X[\Pic]}(\Shv_{\Zar}(\Sm_X, \Spt)) \rightarrow \mathrm{Mod}_{\bb S_X[\Pic]}(\Shv_{\Zar}(\Sm_X, \Spt)).
\]
This functor (thought of as landing onto its essential image) factors as a composite of two other functors
\[
\mathrm{Mod}_{\bb S_X[\Pic]}(\Shv_{\Zar}(\Sm_X, \Spt)) \xrightarrow{L_{\bb P^1}} \mathrm{Mod}_{\bb S_X[\Pic]}(\Shv_{\Zar}(\Sm_X, \Spt))_{\bb P^1} \xrightarrow{L_\sub{ex}} \mathrm{Mod}_{\bb S_X[\Pic]}(\Shv_{\Zar}(\Sm_X, \Spt))_\sub{pbf},
\]
as noted in \cite[5.2.8]{AnnalaIwasa2023}; the point here is that when $E$ is $\bb P^1$-periodic, it satisfies the projective bundle formula if and only if it satisfies elementary blowup excision \cite[Lemma 3.3.5]{AnnalaIwasa2023}.  In particular, the functor $L_\sub{pbf}$ naturally upgrades to a symmetric monoidal localisation functor since $L_\sub{ex}$ evidently does and $L_{\bb P^1}$ does as noted in \cite[Remark 5.2.3]{AnnalaIwasa2023}. We note that, the localisation functor $L_\sub{ex}$ indicates the functor of enforcing blowup excision only on $\bb P^1$-periodic objects and we will continue to use this notation only in this context. 
\end{remark}

\begin{example}[$K$-theory and localising invariants]\label{eq:k-theory}
For any qcqs scheme $Y$, we denote by $\Perf_Y^{\simeq} \subset \Perf_Y$ the maximal subgroupoid of the $\infty$-category of perfect complexes on $Y$. We can endow this space with a natural structure of an $\bb E_{\infty}$-monoid under the tensor product (not direct sum). The natural map\footnote{This map is obtained by applying $(-)^{\simeq} \rightarrow \Omega^{\infty}\K^\sub{cn} \xrightarrow{\simeq} \Omega^{\infty}\K$ to $\Perf_Y$, whose universal property is discussed extensively in \cite{Barwick2016}.} $\Perf^{\simeq}_Y \rightarrow \Omega^{\infty}\K(Y)$ is compatible with the tensor product structure on the domain and the multiplicative structure on the target; in order not to forget that this is the relevant monoidal structure we will write $(\Perf_Y^{\simeq}, \otimes) \rightarrow (\Omega^{\infty}\K(Y), \otimes)$. We moreover have a natural map of spaces $\Omega^{\infty}\Pic(Y) \rightarrow (\Perf_Y^{\simeq}, \otimes)$ by including line bundles into perfect complexes. Finally, applying the functor $\mathrm{GL}_1$ to these maps yields morphisms of grouplike $\bb E_{\infty}$-monoids
\[
\Omega^{\infty}\Pic(Y) \xleftarrow{\simeq} \mathrm{GL}_1(\Omega^{\infty}\Pic(Y)) \rightarrow \mathrm{GL}_1(\Perf_Y^{\simeq}, \otimes) \rightarrow \mathrm{GL}_1(\Omega^{\infty}\K(Y), \otimes),
\]
where the first equivalence follows from the fact that every line bundle is, by definition, invertible under $\otimes$. Overall the composition defines a natural map of grouplike $\bb E_{\infty}$-monoids $\Omega^{\infty}\Pic(Y)  \rightarrow \mathrm{GL}_1(\Omega^{\infty}\K(Y), \otimes)$, or in other words a natural map of connective spectra \[\Pic(Y)\To \mathrm{gl}_1\K(Y).\]

Varying $Y$ over smooth $X$-schemes and sheafifying, the preceding map defines a map
\[
\bb S_X[\Pic] \To \bb S_X[\mathrm{gl}_1^X(\K)]
\]
in $\Shv_{\Zar}(\Sm_X, \CAlg(\Spt))$ (recall from \S\ref{sec:adj} that the left side is given by sheafifying $Y\mapsto \bb S[\Pic(Y)]$, and similarly for the right side). We will be primarily interested in the composition
\begin{equation}\label{pic-to-k}
\bb S_X[\Pic] \To \bb S_X[\mathrm{gl}_1^X(\K)] \xrightarrow{\rm counit}  \K,
\end{equation}
 which is again a morphism of sheaves of $\bb E_{\infty}$-algebras on $\Sm_X$. This endows $K$-theory, restricted to $\Sm_X$, with the structure of $\bb E_{\infty}$-$\bb S_X[\Pic]$-algebra, whence we can regard it as a an object of $\mathrm{Mod}_{\bb S_X[\Pic]}(\Shv_\sub{Zar}(\Sm_X, \Spt))$.

Now the functor $\Perf: \Sm_X^{\op} \rightarrow \Cat_{\bb Z}$ and restriction of scalars along the just-constructed map $\bb S_X[\Pic] \rightarrow  \K$ induce functors \[\mathrm{Mod}_{\K}(\Fun(\Cat_{\bb Z}, \Spt)) \xrightarrow{E \mapsto E \circ \Perf} \mathrm{Mod}_{\K}(\mathrm{PShv}(\Sm_X, \Spt)) \xrightarrow{\rm restrict} \mathrm{Mod}_{\bb S_X[\Pic]}(\PShv(\Sm_X, \Spt)).\] When the input $E\in\mathrm{Mod}_{\K}(\Fun(\Cat_{\bb Z}, \Spt))$ is moreover a localising invariant then the underlying presheaf of $E$ satisfies Nisnevich descent and the resulting $\bb S_X[\Pic]$-module satisfies elementary blowup excision and the $\bb P^1$-periodicity. For the latter, we note that the map $\beta_E$ of Definition \ref{def:conditions}(1) is, by definition, the same as the map $(1-[\roi(-1)])\circ\pi^*: E\to \mathrm{fib}(E(\bb P^1_{\ph})\xto{\infty^*}E)$ discussed for $\K^\sub{cn}$-modules in~\eqref{eqn:P1_for_Kcnmods2},\footnote{While we formulated the map for $\K^\sub{cn}$-modules, we can restrict along the map $\K^\sub{cn} \rightarrow \K$ to discuss such a property for any $\K$-module. We also worked over a prime field throughout \S\ref{sec:p1}, as this was the relevant context, but everything held more generally over $\bb Z$.} which is indeed an equivalence by the discussion in Construction~\ref{cons_add_invariant} which applies since localising invariants are, in particular, additive invariants.

We thus obtain a functor
\[
\mathrm{Mod}_{\K}(\Fun_\sub{loc}(\Cat_{\bb Z}, \Spt)) \To \mathrm{Mod}_{\bb S_X[\Pic]}(\Shv_\sub{Zar}(\Sm_X, \Spt))_\sub{pbf}. 
\] 
and when we speak of an object of $\mathrm{Mod}_{\bb S_X[\Pic]}(\Shv_\sub{Zar}(\Sm_X, \Spt))$ which ``comes from a localising invariant'' we mean an object in the essential image of this functor. 
\end{example}

The Annala--Iwasa theorem concerns the map~\eqref{eq:pbf-picmod}. Notice that even though the target of this map satisfies the projective bundle formula, its source does not. We cite their theorem:

\begin{theorem}\cite[Theorem 5.3.3]{AnnalaIwasa2023} \label{thm:annala-iwasa}
The map~\eqref{pic-to-k} induces a equivalence of Zariski sheaves of $\bb E_\infty$-algebras on $\Sm_X$
\begin{equation}\label{eq:annala-iwasa}
L_\sub{pbf}\bb S_X[\Pic] \xrightarrow{\simeq} \K.
\end{equation}
\end{theorem}

\begin{remark}[Bott-inversion]\label{rem:bott-inversion}
In this remark we will make sense of the formula ``$\bb S_X[\Pic][\beta^{-1}] \xrightarrow{\simeq} \K$'' following \cite[Corollary 5.2.7]{AnnalaIwasa2023}; although this is essentially equivalent to Theorem~\ref{thm:annala-iwasa}, we will need this Bott-inverted point of view to define the Adams operators on $K$-theory.

The formalism of \cite[\S 1]{AnnalaIwasa2023} constructs a symmetric monoidal adjunction
\[
\Sigma^{\infty}_{\bb P^1}: \mathrm{Mod}_{\bb S_X[\Pic]}(\Shv_{\Zar}(\Sm_X, \Spt)) \rightleftarrows \Spt_{\bb P^1}(\mathrm{Mod}_{\bb S_X[\Pic]}(\Shv_{\Zar}(\Sm_X, \Spt)): \Omega^{\infty}_{\bb P^1}
\]
where, roughly speaking, the functor $\Sigma^{\infty}_{\bb P^1}$ witnesses its target as the formal inversion of $\bb P^1$ within the context of presentably symmetric monoidal $\infty$-categories; see \cite[Definition 1.3.8]{AnnalaIwasa2023} for what this exactly means. In particular, there is an autoequivalence \[(\bb P^1)^{-1} \otimes-: \Spt_{\bb P^1}(\mathrm{Mod}_{\bb S_X[\Pic]}(\Shv_{\Zar}(\Sm_X, \Spt)) \quis \Spt_{\bb P^1}(\mathrm{Mod}_{\bb S_X[\Pic]}(\Shv_{\Zar}(\Sm_X, \Spt))\] which is the inverse to tensoring\footnote{This tensoring is in the sense of presentable $\infty$-categories. Namely $\mathrm{Mod}_{\bb S_X[\Pic]}(\Shv_{\Zar}(\Sm_X, \Spt))$ is a module over $\Shv_{\Zar}(\Sm_X, \Spc_*)$ in $\mathrm{Pr}^L$ and so tensoring with $(\bb P^1,\infty)$ means tensoring with the pointed Zariski sheaf represented by $(\bb P^1_X,\infty)$.} by $(\bb P^1,\infty) \in \Shv_{\Zar}(\Sm_X, \Spc_{*})$. 

Now let $E \in \mathrm{Mod}_{\bb S_X[\Pic]}(\Shv_{\Zar}(\Sm_X, \Spt))$. Taking $\Sigma^{\infty}_{\bb P^1}$ of the map~\eqref{eq:adjoint-bott} and applying $(\bb P^1)^{-1} \otimes -$, we get a map
\begin{equation}\label{eq:bott-stable}
\beta_E: \Sigma_{\bb P^1}^{\infty}E \To (\bb P^1)^{-1} \otimes \Sigma_{\bb P^1}^{\infty}E
\end{equation}
which we abusively denote again by $\beta_E$. We may use this to express $L_\sub{pbf}E$ as
\begin{gather*}\label{eq:colimit-pbf}
L_\sub{pbf}\simeq L_\sub{ex}\colim\big(\Omega^{\infty}_{\bb P^1}\Sigma_{\bb P^1}^{\infty} E \xrightarrow{\Omega^{\infty}_{\bb P^1}\beta_E} \Omega^{\infty}_{\bb P^1}( (\bb P^1)^{-1} \otimes\Sigma_{\bb P^1}^{\infty} E)
 \xrightarrow{\Omega^{\infty}_{\bb P^1}(\id \otimes \beta_E)} \Omega_{\bb P^1}^{\infty}(((\bb P^1)^{-1})^{\otimes 2}\otimes\Sigma^{\infty}_{\bb P^1}E)  \to\cdots\big)
\end{gather*} 
 % \rightarrow \Omega_{\bb P^1}^{\infty}(((\bb P^1)^{-1})^{\otimes n}\otimes\Sigma^{\infty}_{\bb P^1}E\big)
by \cite[Corollary 5.2.7]{AnnalaIwasa2023}.
\end{remark}

We are now equipped to define the Adams operators $\psi^m$ on the $m$-inverted non-connective $K$-theory $\K[\tfrac1m]$ of smooth $X$-schemes. The $m$-power map of abelian sheaves $\roi^\times \rightarrow \roi^\times$, $f \mapsto f^m$, induces the multiplication-by-$m$ map of presheaves of spectra on qcqs schemes $\left(-\right)^m: \Pic \rightarrow \Pic$; sheafifying, restricting to $\Sm_X$ and applying the spherical group algebra functor \eqref{eq:gl1} then obtains a map of Zariski sheaves of $\bb E_{\infty}$-algebras on $\Sm_X$
\begin{equation}
\psi^m: \bb S_X[\Pic] \rightarrow  \bb S_X[\Pic].
\label{eqn_adams_on_spherical}\end{equation}

\begin{lemma}\label{lem:nbeta} The map in $\Shv_{\Zar}(\Sm_X,\Sp)$ given by
\[
\Sigma^{\infty}_+\P^1 \xrightarrow{\beta}  \bb S_X[\Pic] \xrightarrow{\psi^m}  \bb S_X[\Pic]\xrightarrow{\eqref{pic-to-k}} \K
\]
is homotopic to $m\beta$. Consequently, denoting by $\bb S_X[\Pic]^{[m]}$ the $\bb E_{\infty}$-$\bb S_X[\Pic]$-algebra obtained by scalar restricting $\bb S_X[\Pic]$ along $\psi^m$, we get an equivalence
\[
L_\sub{pbf}\bb S_X[\Pic]^{[m]} \xrightarrow{\simeq} \K[\tfrac{1}{m}]
\]
\end{lemma}

\begin{proof}
The first part follows from the exact same argument as in \cite[Lemma 3.11]{BachmannHopkins2020}. The first part then implies that the colimit formula calculating $L_\sub{pbf}\bb S_X[\Pic]^{[m]}$ from Remark~\ref{rem:bott-inversion}
is homotopic to:
\[
L_\sub{ex}\big(\colim \Omega^{\infty}_{\bb P^1}\Sigma_{\bb P^1}^{\infty} \bb S_X[\Pic] \xrightarrow{\Omega^{\infty}_{\bb P^1}m\beta} \Omega^{\infty}_{\bb P^1}( (\bb P^1)^{-1} \otimes\Sigma_{\bb P^1}^{\infty} \bb S_X[\Pic]) \xrightarrow{\Omega^{\infty}_{\bb P^1}(\id \otimes m\beta)} \Omega_{\bb P^1}^{\infty}((\bb P^1)^{-1})^{\otimes 2} \otimes\Sigma^{\infty}_{\bb P^1}\bb S_X[\Pic])\cdots\big)
\]
which, thanks to Theorem~\ref{thm:annala-iwasa}, is further homotopic to the colimit:
\[
\K \xrightarrow{ \cdot m} \K \xrightarrow{\cdot m} \K \cdots.
\] 
This latter colimit inverts $m$, whence we obtain $\K[\tfrac{1}{m}]$. 
\end{proof}

The next construction mimics Bachmann--Hopkins in the $\bb A^1$-invariant context \cite[\S 3.3.1]{BachmannHopkins2020}:

\begin{construction}[Adams operators on $K$-theory of smooth $X$-schemes]\label{constr:adams-k}
By Theorem~\ref{thm:annala-iwasa} and Lemma~\ref{lem:nbeta}, there is a unique map $\psi^m:\K\to \K[\tfrac1m]$ of Zariski sheaves of $\bb E_{\infty}$-algebras on $\Sm_X$ fitting (uniquely) into a commutative diagram 
\begin{equation}\label{eq:psin}
\begin{tikzcd}
\bb S_X[\Pic] \ar[swap]{d}{\psi^m} \ar{rr}{\eqref{pic-to-k}} &&  \K \ar[dashed]{d}{\exists! \psi^m}\\
\bb S_X[\Pic] \ar{r}{\eqref{pic-to-k}} &\K\ar{r}& \K[\tfrac{1}{m}].
\end{tikzcd}
\end{equation}
in $\Shv_\sub{Zar}(\Sm_X,\CAlg(\Sp))$. The dashed map factors uniquely through $K[\tfrac1m]$, and it is the resulting map
\begin{equation}\label{adams-final}
\psi^m: \K[\tfrac{1}{m}] \To \K[\tfrac{1}{m}].
\end{equation}
of Zariski sheaves of $\mathbb{E}_{\infty}$-algebras on $\Sm_X$ which we call the \emph{$m$-th Adams operator} on $\K$-theory of smooth $X$-schemes.
\end{construction}

\begin{remark}[Comparison with prior definitions in the regular case]\label{rem:comparison-regular}
Assume throughout this remark that $X$ is a regular Noetherian scheme. Then $\KH = \K$ on $\Sm_X$ and the construction of the Adams operators of Construction~\ref{constr:adams-k} agrees with the one by Bachmann--Hopkins \cite[\S3.3]{BachmannHopkins2020} in the $\infty$-category $\SH(X)$, after applying $\omega^{\infty}$ to the latter. More precisely, the construction above uniquely defines an $\bb E_{\infty}$-algebra map $\psi^m:\mathrm{KGL}_X[\tfrac{1}{m}] \rightarrow \mathrm{KGL}_X[\tfrac{1}{m}]$ in the $\infty$-category of motivic spectra $\mathrm{MS}_X$ in the sense of \cite{AnnalaHoyoisIwasa2023}, and this map agrees with that of \cite[\S3.3]{BachmannHopkins2020}; such a statement makes sense because $\SH(X)$ is a full subcategory of $\mathrm{MS}_X$ and $\mathrm{KGL}_X \in \SH(X)$ under our regularity assumption on $X$. Furthermore, as already mentioned in \S\ref{ss_Adams_on_KH}, Bachmann--Hopkins' operators coincide with those defined previously by Riou and, at the level of $K$-groups, with those of Hiller, Kratzer, and Soul\'e.
\end{remark}

\begin{remark}[Functoriality in $X$]\label{rem_adams_functorial}
The qcqs base scheme $X$ was fixed throughout this subsection and the Adams operators $\psi^m$ were defined only on $\K[\tfrac1m]$ of smooth $X$-schemes. Functoriality in $X$ was not addressed, and we are not sure whether the methods of this subsection can be used to define Adams operations on $\K[\tfrac1m]$ as a sheaf of $\bb E_\infty$-algebras on all qcqs schemes. But one can instead obtain such functorial operators by arguing as follows: taking $X=\Spec(\bb Z)$, we have constructed $\psi^m$ as an endomorphism of $\K[\tfrac1m]$ as a sheaf of $\bb E_\infty$-algebras on $\Sm_\bb Z$, characterised by diagram \eqref{eq:psin}. We left Kan extend this to finitely presented $\bb Z$-schemes, then pro cdh sheafify in the sense of Kelly--Saito (see \S\ref{ss_KellySaito}) on $\Sch_\bb Z^\sub{fp}$, and finally extend to arbitrary qcqs schemes by taking cofiltered limits: this defines an endomorphism $\psi^m$ on $\K[\tfrac1m]$ as a sheaf of $\bb E_\infty$-algebras on all qcqs schemes. Furthermore, this (a priori new) Adams operator fits into a commutative diagram
\begin{equation}
\begin{tikzcd}
\bb S[\Pic] \ar{d}{\psi^m} \ar{r} &  \K[\tfrac1m] \ar[dashed]{d}{\psi^m}\\
\bb S[\Pic] \ar{r} &   \K[\tfrac{1}{m}],
\end{tikzcd}
\end{equation}
of sheaves of $\bb E_\infty$-algebras on qcqs schemes, where $\bb S[\Pic]$ is defined to be the Zariski sheafification of $Y\mapsto\mathbb{S}[\scr O(Y)^{\times}[1]]$ and its endomorphism $\psi^m$ is induced by the $m$-power map $\roi^\times\to\roi^\times$. But upon restricting this diagram to $\Sm_X$, for any qcqs scheme $X$, we recover the characterising property of the $\psi^m$ from Construction \ref{constr:adams-k} and so see that this pro cdh-locally left Kan extended Adams operator agrees on $\Sm_X$ with that of Construction \ref{constr:adams-k}. 

We will provide a different construction of such functorial Adams operators in the case of qcqs $\bb Q$-schemes, which uses trace methods rather than the pro cdh topology: see Theorem \ref{thm:adams}, especially assertion (2).
\end{remark}

\subsection{Compatibility of the Adams operators}\label{sec:proof-chw}

This goal of this subsection is to prove Corollary~\ref{corol:key-q} which establishes a multiplicative compatibility result about Adams operators on $K$-theory and negative cyclic homology, at least when restricted to smooth $k$-schemes, for $k$ a $\bb Q$-algebra. The Annala--Iwasa theorem reduces this to Proposition~\ref{prop:key-q} which is, in turn, established by showing that a certain mapping space is discrete. Throughout this subsection we fix an integer $m\in\bb Z\setminus\{0\}$.

For any $\bb Q$-algebra $k$ recall from \S\ref{sec:ai-thm} the sheaf of $\bb E_\infty$-algebras on $\Sm_k$ \[\mathbb{S}_k[\Pic]:=\mathbb{S}_k[L_\sub{Zar}\Pic]=\text{Zariski sheafification of }Y\mapsto \bb S[\roi(Y)^\times[1]].\] We will consider the following sheaves of $\bb E_\infty$-algebras on $\Sm_k$ and maps between them:
\[
\mathbb{S}_k[\Pic] \xto{\sub{\eqref{pic-to-k}}} \K \stackrel{\mathrm{tr}}\To \HC^{-}(-/\bb Q)\To\HC^-(-/k)\To\HH(-/k).
\]
We will denote by $\mathrm{tr}$ any of these maps from $\mathbb{S}_k[\Pic]$ or $\K$ to $\HC^-$ or $\HH$; the domain and codomain will always be clear. Equipping $\mathbb{S}_k[\Pic]$, $\K$, and $\HC^-$ with the trivial $S^1$-action, and $\HH$ with either its usual $S^1$-action or its rescaled action from \S\ref{ss_Raksit}, we may view these sheaves as valued in $S^1$-equivariant $\bb E_\infty$-algebras and the trace maps as being $S^1$-equivariant (writing $\HH(-/k)^{[m]}$ in case of the rescaled action); in short, we may work in $\Shv_{\Zar}(\Sm_{k}, \CAlg(\mathrm{Sp})^{BS^1})$.

%, thereby establishing the compatibility between the Adams operators on the spherical group ring of the Picard stack \eqref{eqn_adams_on_spherical} and those on negative cyclic homology \eqref{eqn_adams_on_HC-}:

\begin{proposition}\label{prop:key-q}
Let $k$ be a $\bb Q$-algebra. Then the diagram
\begin{equation}\label{key-hc1}
\begin{tikzcd}
\mathbb{S}_k[\Pic] \ar{r}{\mathrm{tr}} \ar[swap]{d}{\psi^m} & \HH(-/k) \ar{d}{\psi^m}\\
\mathbb{S}_k[\Pic] \ar{r}{\mathrm{tr}} & \HH(-/k)^{[m]}.
 \end{tikzcd}
\end{equation}
of $S^1$-equivariant sheaves of $\bb E_\infty$-algebras on $\Sm_k$, i.e., in $\Shv_{\Zar}(\Sm_{k}, \CAlg(\mathrm{Sp})^{BS^1})$, commutes in a unique way.
\end{proposition}
\begin{proof}
It suffices to prove that the maps $\psi^m\circ\mathrm{tr}$ and $\mathrm{tr}\circ\psi^m$ are homotopic and that the mapping space \[\Map_{\Shv_{\Zar}(\Sm_k, \CAlg(\Spt)^{BS^1})}(\bb S_k[ \Pic], \HH(-/k)^{[m]})\] is discrete. We begin by proving these results after forgetting the $S^1$-action, i.e., in the category $\Shv_{\Zar}(\Sm_k, \CAlg(\Spt))$, where $\HH(-/k)^{[m]}$ is the same as $\HH(-/k)$.

We begin by rewriting the mapping space in question:
\begin{align*}
\Map_{\Shv_{\Zar}(\Sm_k, \CAlg(\Spt))}(\bb S_k[ \Pic], \HH(-/k))&\simeq
\Map_{\PShv(\Sm_k^\sub{aff}, \Spt)}(\scr O^{\times}[1], \mathrm{gl}_1\HH(-/k)))\\
&\simeq \Map_{\PShv(\Sm_k^\sub{aff}, \Spt)}(\scr O^{\times}, (\tau_{\ge1}\mathrm{gl}_1\HH(-/k))[-1])\\
&\quis \Map_{\PShv(\Sm_k^\sub{aff}, \Spt)}(\scr O^{\times}, \pi_1\mathrm{gl}_1\HH(-/k)).
\end{align*}
Here the first equivalence is the adjunction~\eqref{eq:gl1}, restriction to smooth affines, and the adjunction between sheaves and presheaves; the second equivalence comes the fact that $\roi^\times[1]$ is supported in homological degrees $1$. We claim that the arrow induced by the truncation map of presheaves $(\tau_{\ge1}\mathrm{gl}_1\HH(-/k))[-1]\to \pi_1\mathrm{gl}_1\HH(-/k)$ is an equivalence. Its fibre is $\Map_{\PShv(\Sm_k^\sub{aff}, \Spt)}(\scr O^{\times}, (\tau_{\ge2}\mathrm{gl}_1\HH(-/k))[-1])$, which we rewrite as \[\Map_{\PShv(\Sm_k^\sub{aff}, \mathrm{CMon}_{\bb E_{\infty}}(\Spc))}(\roi^\times, \Omega^{\infty}(\tau_{\geq 2}\mathrm{gl}_1\HH(-/k))[-1]))\] via the equivalence between connective spectra and grouplike $\bb E_\infty$-monoids in spaces. To prove that this is contractible, it is enough by Lemma~\ref{lem:algebra_vs_linear_maps} to show that the space of maps of presheaves of spaces
\[
\Map_{\PShv(\Sm_k^\sub{aff}, \Spc)}(\roi^\times, \Omega^{\infty}(\tau_{\geq 2}\mathrm{gl}_1\HH(-/k))[-1]))
\]
is contractible. But the presheaf of sets $\roi^\times$ is represented by $\Spec(k[t^{\pm1}])$, so this mapping space is given by $\Omega^{\infty}$ of $(\tau_{\geq 2} \mathrm{gl}_1\HH(k[t^{\pm1}])/k))[-1]$; this complex is indeed $0$ as desired, since the homotopy groups of $\mathrm{gl}_1\HH(k[t^{\pm1}]/k)$ and $\HH(k[t^{\pm1}]/k)$ are isomorphic in degrees $\geq 1$ (see Remark~\ref{rem:hpty}), and the latter vanish in degrees $\ge2$ by the Hochschild--Kostant--Rosenberg theorem. This completes the proof of the claim and so establishes an equivalence of mapping spaces \begin{equation}\Map_{\Shv_{\Zar}(\Sm_k, \CAlg(\Spt))}(\bb S_k[ \Pic], \HH(-/k))\simeq \Map_{\PShv(\Sm_k^\sub{aff}, \Spt)}(\scr O^{\times}, \pi_1\mathrm{gl}_1\HH(-/k)),\label{eqn_adams_discrete}\end{equation} where the right side is the discrete space consisting of the set of maps of abelian presheaves from $\scr O^{\times}$ to $\pi_1\mathrm{gl}_1\HH(-/k))$ on $\Sm_k^\sub{aff}$.

This not only establishes discreteness of the left side of \eqref{eqn_adams_discrete} but also shows that any map $\bb S_k[ \Pic]\to \HH(-/k)$ of presheaves of $\bb E_\infty$-algebras on $\Sm_k$ is uniquely compatible with the Adams operators on each side: indeed, this follows from naturality of \eqref{eqn_adams_discrete} for endomorphisms of $\bb S_k[ \Pic]$ and $\HH(-/k)$, that $\psi^m$ acts on both $\scr O^{\times}$ and $\pi_1\mathrm{gl}_1\HH(-/k)$ as multiplication by $m$ (in the latter case use the natural log isomorphism $\pi_1\mathrm{gl}_1\HH(-/k)\simeq \HH_1(-/k)$ of Remark~\ref{rem:hpty} and note that $\HH_1(-/k)$ is isomorphic to $(\gr^1_\sub{HKR}\HH(-/k))[-1]$ on smooth affine $k$-schemes), and the fact that any map of abelian presheaves automatically commutes with multiplication by $m$ on the source and target. In particular, the maps $\psi^m\circ\mathrm{tr},\,\mathrm{tr}\circ\psi^m\in \Map_{\Shv_{\Zar}(\Sm_k, \CAlg(\Spt))}(\bb S_k[ \Pic], \HH(-/k))$ are equal.

It is now easy to take account of the $S^1$-action. The mapping space \[\Map_{\Shv_{\Zar}(\Sm_k, \CAlg(\Spt)^{BS^1})}(\bb S_k[ \Pic], \HH(-/k)^{[m]})\] is equivalent to \[\Map_{\Shv_{\Zar}(\Sm_k, \CAlg(\Spt))}(\bb S_k[ \Pic], \HH(-/k))^{hS^1},\] where $\Map_{\Shv_{\Zar}(\Sm_k, \CAlg(\Spt))}(\bb S_k[ \Pic], \HH(-/k))$ has been equipped with the usual structure of an $S^1$-equivariant space (induced by the trivial action on $\bb S_k[\Pic]$ and the rescaled action on $\HH(-/k)$). But we have shown that the latter mapping space is discrete, so its $S^1$-homotopy-fixed points is again discrete and consists of the naive $S^1$-fixed points of the set. That is, forgetting the $S^1$-action \[\Map_{\Shv_{\Zar}(\Sm_k, \CAlg(\Spt)^{BS^1})}(\bb S_k[ \Pic], \HH(-/k)^{[m]})\to \Map_{\Shv_{\Zar}(\Sm_k, \CAlg(\Spt))}(\bb S_k[ \Pic], \HH(-/k))\] is an inclusion of discrete spaces. The maps $\psi^m\circ\mathrm{tr},\,\mathrm{tr}\circ\psi^m\in \Map_{\Shv_{\Zar}(\Sm_k, \CAlg(\Spt)^{BS^1})}(\bb S_k[ \Pic], \HH(-/k)^{[m]})$ are therefore homotopic (since we have shown that the same is true after forgetting the $S^1$-action), even in a unique way (since the mapping space is discrete).
\end{proof}

We finally arrive at the main result of this appendix. 

\begin{corollary}\label{corol:key-q}
Let $k$ be a $\bb Q$-algebra. Then the diagram
\[
\begin{tikzcd}
\K[\tfrac{1}{m}] \ar{r}{\mathrm{tr}} \ar[swap]{d}{\psi^m} & \HC^-(-/k) \ar{d}{\psi^m}\\
\K[\tfrac{1}{m}] \ar{r}{\mathrm{tr}} & \HC^-(-/k).
 \end{tikzcd}
\]
of sheaves of $\bb E_\infty$-algebras on $\Sm_k$ commutes
\end{corollary}
\begin{proof}
Taking $S^1$-fixed points of \eqref{key-hc1} yields a commutative diagram
\begin{equation}\label{key-hc2}
\begin{tikzcd}
\bb S_k[\Pic] \ar{r}{\mathrm{tr}} \ar[swap]{d}{\psi^m} & \HC^-(-/k) \ar{d}{\psi^m}\\
\bb S_k[\Pic] \ar{r}{\mathrm{tr}} & \HC^-(-/k)
 \end{tikzcd}
\end{equation}
of sheaves of $\bb E_\infty$-algebras on $\Sm_k$. We view it as a diagram of $\bb S_k[\Pic]$-modules via the indicated structure maps from the top left corner. Applying $L_\sub{pbf}$ turns the top left corner into $K$ by the Annala--Iwasa Theorem \ref{thm:annala-iwasa}, and similarly turns the bottom left (namely the $\bb S_k[\Pic]$-module $\bb S_k[\Pic]^{[m]}$) into $\K[\tfrac1m]$ by Lemma \ref{lem:nbeta}. Meanwhile, $L_\sub{pbf}$ does not affect $\HC^-(-/k)$ since the latter comes from a localising invariant and so satisfies $\bb P^1$-periodicity and elementary blow-up excision (see Remark \ref{eq:k-theory}); this also shows that the bottom right satisfies elementary blow-up excision since it does not depend on the $\bb S_k[\Pic]$-module structure. Finally, using Lemma \ref{lem:nbeta}, the question of $\bb P^1$-periodicity for the bottom right is whether the map \[\HC^-(-/k)\xto{m(1-[\roi(-1)])}\mathrm{fib}(\HC^-(\bb P^1_{\ph}/k)\xto{\infty^*}\HC^-(-/k))\] is an equivalence on $\Sm_k$; but this follows from the usual $\bb P^1$-bundle formula for $\HC^-$ (i.e., the case $m=1$) since $m$ acts invertibly on $\HC^-(-/k)$.

The outcome of applying $L_\sub{pbf}$ to \eqref{key-hc2} is therefore a commutative diagram
\[
\begin{tikzcd}
\K \ar{r}{\mathrm{tr}} \ar[swap]{d}{\psi^m} & \HC^-(-/k) \ar{d}{\psi^m}\\
\K[\tfrac{1}{m}] \ar{r}{\mathrm{tr}} & \HC^-(-/k).
 \end{tikzcd}
\]
of sheaves of $\bb E_\infty$-algebras on $\Sm_k$; harmlessly inverting $m$ in the top left completes the proof.
\end{proof}

\color{black}
\bibliographystyle{acm}
\bibliography{../bibliography-zar, bibliography-zar2, Bibliography}
\end{document}